\documentclass[final,12pt,times]{elsarticle}
\usepackage{lineno,hyperref}
\modulolinenumbers[5]

\journal{Journal des Math\'{e}matiques Pures et
Appliqu\'{e}es}

\usepackage{amsfonts}
\usepackage{amssymb}
\usepackage{color}
\usepackage{bbding}
\usepackage{pifont}
\usepackage{}
\usepackage{mathrsfs}
\usepackage{amssymb}
\usepackage{amsfonts}
\usepackage{amssymb,latexsym,amsmath, amsthm}
\usepackage{hyperref}
\usepackage{amscd}
\usepackage[all]{xy}
\usepackage{xcolor}
\usepackage{amsmath}
\numberwithin{equation}{section}
=0pt
\def\sqr#1#2{{\vcenter{\vbox{\hrule height.#2pt
              \hbox{\vrule width.#2pt height#1pt \kern#1pt \vrule
width.#2pt}
              \hrule height.#2pt}}}}
\def\3n{\negthinspace \negthinspace \negthinspace }
\def\2n{\negthinspace \negthinspace }
\def\1n{\negthinspace }

\def\ep{\epsilon}

\def\bn{\mathcal{N}}

\def\ba{\mathcal{A}}

\def\be{\mathcal{E}}

\def\supp{\hbox{\rm supp$\,$}}

\def\({\Big (}
\def\){\Big )}
\def\[{\Big[}
\def\]{\Big]}

\def\={\buildrel \triangle \over =}

\def\be{\begin{equation}}
\def\bel{\begin{equation}\label}
\def\ee{\end{equation}}
\def\bea{\begin{eqnarray}}
\def\eea{\end{eqnarray}}
\def\bt{\begin{theorem}}
\def\et{\end{theorem}}
\def\bc{\begin{corollary}}
\def\ec{\end{corollary}}
\def\bl{\begin{lemma}}
\def\el{\end{lemma}}
\def\bp{\begin{proposition}}
\def\ep{\end{proposition}}
\def\br{\begin{remark}}
\def\er{\end{remark}}
\def\ba{\begin{array}}
\def\ea{\end{array}}
\def\bd{\begin{definition}}
\def\ed{\end{definition}}

\newtheorem{lemma}{Lemma}[section]
\newtheorem{remark}{Remark}[section]

\newtheorem{theorem}{Theorem}[section]
\newtheorem{corollary}{Corollary}[section]

\newtheorem{definition}{Definition}[section]
\newtheorem{proposition}{Proposition}[section]
\newtheorem{condition}{Condition}[section]

\journal{...}

\begin{document}

\begin{frontmatter}

\title{$L^2$ estimates and existence theorems for the $\overline{\partial}$ operators in infinite dimensions, II}


\author[0]{Zhouzhe Wang}
\address[0]{School of Mathematics,
Sichuan University, Chengdu, 610064, China. E-mail
address: wangzhouzhe@stu.scu.edu.cn.}

\author[1]{Jiayang Yu}
\address[1]{School of Mathematics,
Sichuan University, Chengdu, 610064, China. E-mail
address: jiayangyu@scu.edu.cn.}

\author[2]{Xu Zhang}
\address[2]{School of Mathematics,
Sichuan University, Chengdu, 610064, China. E-mail
address: zhang\_xu@scu.edu.cn.}


\begin{abstract}
This paper is the second part of our series of works to establish $L^2$ estimates and existence theorems for the $\overline{\partial}$ operators in infinite dimensions. In this part, we consider the most difficult
case, i.e., the underlying space is a general pseudo-convex domain. In order to solve this longstanding open problem, we introduce several new concepts and techniques, which have independent interest and pave the way for
research that investigates some other issues in infinite-dimensional analysis.
\end{abstract}


\end{frontmatter}



\section{Introduction}

Quite many of rich, deep and beautiful results in modern mathematics are limited to finite dimensional spaces
and their subspaces/submanifolds or function spaces defined on them. A fundamental reason for this is
that, in finite dimensional spaces there exist Lebesgue measures (or more generally Haar measures on locally compact Hausdorff topological groups), which enjoy nearly perfect properties. Unfortunately, in the setting of infinite dimensions, though after the works of great mathematicians such as
D.~Hilbert (\cite{Hil}), A.~N.~Kolmogorov (\cite{Kol56}), N.~Wiener (\cite{Wiener}), J.~von~Neumann (\cite{Neumann55}) and I.~M.~Gel'fand (\cite{Gel64}), so far people have NOT yet found an ideal substitute
for Lebesaue/Haar measures so that in-depth studies of infinite-dimensional analysis are always possible. Roughly
speaking, functional analysis is in some sense the part of linear algebra in infinite-dimensional spaces; while
the part of calculus in infinite-dimensional spaces, i.e., infinite-dimensional analysis, to our opinion, very likely, requires the
efforts of mathematicians of several generations or even many generations to be satisfactory. As a reference, the long process of classical calculus (on finite dimensions) from its birth to maturity, from I. Newton and G.~W.~Leibniz to H.~L.~Lebesgue, has lasted more than 200 years.

Infinite-dimensional analysis has a wide range of application backgrounds:
\begin{itemize}
\item[{\rm (1)}] In probability theory, many useful random variables and stochastic processes are actually defined on infinite-dimensional spaces (e.g., \cite{Gel64, Kol56, Wiener}), including typically the classical Brownian motion defined on the space of all continuous functions on
$[0, 1]$ which vanish at $0$, i.e. the Wiener space. Also, it is well-known that one of the most fundamental concepts in probability theory is the so-called ``independence", for which a proper measure for constructing nontrivial sequences of independent random variables is, simply to use the product measure of some given sequence of measures in finite dimensions;

\item[{\rm (2)}] In control theory, the value functions of optimal control problems for distributed parameter systems (including particularly controlled partial differential equations) are functions of infinitely many variables, and hence the corresponding Hamilton-Jacobi-Bellman equations are nonlinear partial differential equations with infinite number of variables (\cite{DaZ02, Li-Yong, Lions88}). On the other hand, quite naturally, functions of infinitely many variables are extensively used in the field of stochastic control, in both finite and infinite dimensions (\cite{Lu-Zhang, YZ99});

\item[{\rm (3)}] In physics, functions with infinite number of variables appear extensively in continuum mechanics, quantum mechanics, quantum field theory (\cite{Glimm, Neumann55}). The  most relevant example in this respect is the Feynman path integrals (\cite{Feynman}). As pointed by N.~Wiener (\cite[p. 131]{Wiener}), ``the physicist is equally concerned with systems the dimensionality of which, if not infinite, is so large that it invites the use of limit-process in which it is treated as infinite. These systems are the systems of statistical mechanics...".
\end{itemize}

In \cite[p. 14]{Atiyah}, M. Atiyah remarked that ``the 21st century might be the era of quantum
mathematics or, if you like, of infinite-dimensional mathematics". Similarly to the role of analysis in mathematics, infinite-dimensional analysis is of fundamental importance in infinite-dimensional mathematics, especially when precise computation or estimate is indispensable. Similarly to the finite dimensional setting, infinite-dimensional analysis is
the basis of infinite-dimensional mathematics.

Historically, infinite-dimensional analysis originated to a large extent from the study of probability theory and related fields (e.g., \cite{Kol56, Wiener} and so on). In the recent decades, thanks to the important and fundamental works by some leading probability experts including L. Gross (for abstract Wiener spaces, \cite{Gro65}), T. Hida (for white noise analysis, \cite{Hida}) and P. Malliavin (for Malliavin calculus, \cite{Mall}), significant progresses have been made in infinite-dimensional analysis.
Nevertheless, as far as we know, due to the lack of nontrivial translation invariance measure and local compactness in the setting of infinite dimensions,
there is still a long way to go before establishing an effective setting for infinite-dimensional analysis.
In addition, most of these probabilistic experts' works consider infinite-dimensional analysis on quite general probability spaces (or even abstract probability spaces), whose results inevitably appear not powerful enough to solve some complicated problems in this respect.

In our opinion, infinite-dimensional analysis has the significance of being independent of probability theory, and should and can fully become an independent discipline and mathematical branch. In particular, ``randomness" should be removed to a certain extent (of course, probability theory will always be one of the most important application backgrounds of infinite-dimensional analysis), and ``(mathematical) analysis nature" should be returned to. As a reference, functional analysis originated (to a large extent in history) from the study of integral equations, but as an independent branch of mathematics, the great development of functional analysis is possible only after getting rid of the shackles of integral equations.

As the first step of returning to the above mentioned ``(mathematical) analysis nature", according to our understanding, one should first deeply study analytic functions in infinite-dimensional space, that is, infinite-dimensional complex analysis. This is because, similarly to the finite dimensions, analytic functions in infinite-dimensional spaces are these functions enjoying many elegant and elaborate (or, in some sense even the best) analysis properties on these spaces. Without a deep understanding of this class of functions, the in-depth study of infinite-dimensional analysis seems difficult to be carried out.

Since the last sixties, infinite-dimensional complex analysis has been a rapid developing field, in which one can find many works, for example, the early survey by L.~Nachbin (\cite{Nachbin}), the monographes by A. Bayoumi (\cite{Bayoumi}), S.~B.~Chae (\cite{Chae}), G.~Coeur\'e (\cite{Coeure}), J.~Colombeau (\cite{Col}), S.~Dineen (\cite{Din81, Din99}), M.~Herv\'e (\cite{Herve}), J.-M. Isidro and L. L. Stach\'o (\cite{IS85}), P.~Mazet (\cite{Maz}), J.~Mujica (\cite{Mujica}) and P.~Noverraz (\cite{Noverraz}), and the rich references cited therein. In particular, around 2000, L.~Lempert revisited this field and made some remarkable progresses in a series of works (\cite{Lem98, Lem99, Lem00, Lem03}). On the other hand, infinite-dimensional complex analysis is also strongly linked with many other branches of mathematics (e.g., \cite{DG21, SS16}, just to mention a few).

Although great advance has already been made, to the best of our knowledge, compared with the situation of finite dimensions, infinite-dimensional complex analysis is still far from reaching the level of precision analysis. So far, a large number of results in infinite-dimensional complex analysis are still in the category of abstract analysis or soft analysis, which can be said to be very similarly to the situation of several complex analysis before the 1960s, as we shall explain below a little more.

As in the setting of several complex analysis, one of the most fundamental problems in infinite-dimensional complex analysis is the solvability of the following $\overline{\partial}$ equation:
\begin{equation}\label{d-bar equation}
\overline{\partial}u=f,
\end{equation}
where $f$ is an inhomogeneous term with $\overline{\partial}f=0$ and $u$ is an unknown, defined on some infinite-dimensional complex space or its suitable open subset. L. Lempert pointed out at \cite[p. 775]{Lem99} that, ``{\sl up to now not a single infinite dimensional Banach space and an open subset therein have been proposed where the equation $(\ref{d-bar equation})$ could be proved to be solvable under reasonably general conditions on $f$}". In our previous work \cite{YZ} (See \cite[Theorem 5.1, p. 542]{YZ} particularly), we provided an approach to solve (\ref{d-bar equation}) in the case of quite general forms $f$ on
the infinite-dimensional space $\ell^p$ (for $ p \in [1, \infty)$). However, our approach therein applies only to the whole space of $\ell^p$, while the case of general pseudo-convex domains is much more difficult and interesting, which turns out to be the main concern of the present paper.

In finite dimensions, a powerful basic tool to solve the counterpart of (\ref{d-bar equation}) is the $L^2$ estimates, developed systematically by C.~B.~Morrey (\cite{Mor}), J.~J.~Kohn (\cite{Koh}) and particularly by
L.~H\"omander in \cite{Hor65, Hor90} (See also two recent very interesting books by T.~Ohsawa (\cite{Ohsawa}) and by E.~J.~Straube (\cite{Straube})). Since the $L^2$ estimates found by L. H\"ormander are dimension-free, starting from the 1970s, very naturally, people (See P. Raboin \cite{Rab79} and the references therein)
have been looking forward to extending it to infinite dimensions, but without complete success for more than 40 years. Indeed, it is a longstanding unsolved problem to establish the $L^2$ estimates in infinite dimensions, particularly for general pseudo-convex domains therein (See \cite[Theorem 3.1 at p. 533 and Corollary 5.2 at p. 544]{YZ} for some partial positive results in this respect for the whole space $\ell^p$).

Based on the above observations, as a key step for our research project on infinite-dimensional analysis, we shall focus first on the solvability of the $\overline{\partial}$ equations over infinite-dimensional pseudo-convex domains, and especially the $L^2$ estimates for them. We shall choose the Hilbert space $\ell^2$ as a basic space. In some sense, $\ell^2$  is the simplest infinite-dimensional space, closest to Euclidean spaces. Note that the essential difficulties for analysis problems on $\ell^2$ are by no means reduced too much though sometimes working in this space does provide some convenience. To see this, we note that, similarly to the proof of \cite[Theorem 3.1, p. 533]{YZ}, it is easy to derive an $L^2$ estimate for smooth functions/forms on any pseudo-convex domain in $\ell^2$. In order to solve the corresponding $\overline{\partial}$ equation, people need to extend such an estimate to all functions/forms with suitably weak smoothness only. In the case of finite dimensions, this extension
was achieved by L.~H\"omander (See \cite[Proposition 2.1.1 at p. 100 and Theorem 2.1.4 at p. 104]{Hor65} and \cite[Lemma 4.1.3 at p. 80, and Lemma 4.2.1 and Theorem 4.2.2 at p. 84]{Hor90} for two different ways) with the classical Friedrichs mollifier technique which relies heavily upon an
argument not available in any infinite-dimensional space (including $\ell^2$, of course), namely, the translation invariance of the Lebesgue measure. To overcome this difficulty, we have to construct some suitable working spaces (See $L^{2}_{\left( s,t\right)  }\left( V,loc\right)$ and so on in Section \ref{secx3}) so that the involved problems in infinite dimensions can be reduced to that in finite dimensions, which is possible by introducing a simple but very useful reduction transformation in (\ref{Integration reduce dimension}) (See Proposition \ref{Reduce diemension} for its main properties). In some sense, such an idea to establish the desired  $L^2$ estimates in infinite dimensions is quite natural but the full process to realize it is delicate. Indeed, as mentioned before, there are two ways to derive the corresponding $L^2$ estimates (for functions/forms with some weak smoothness only) in finite dimensions, both of which are possible to be generalized to the setting of infinite dimensions:
\begin{itemize}
\item[{\rm (1)}] One way is that in \cite{Hor65} (See \cite[Proposition 2.1.1 at p. 100 and Theorem 2.1.4 at p. 104]{Hor65}), which is based on another key preliminary, i.e., \cite[Proposition 1.2.4, p. 96]{Hor65} to show the identity of weak and strong extensions of first order differential operators with some complicated boundary conditions for which one has to use the localization and flattening techniques (based respectively on partitions of unity and local coordinate systems) and in particular introduce some delicate modification of the Friedrichs mollifier technique (See \cite[Proof of Proposition 1.2.4, pp. 97--98]{Hor65}). We have tried a long time to extend \cite[Proposition 1.2.4, p. 96]{Hor65} to infinite dimensions but failed, for which the main difficulty is that the above mentioned techniques do not seem to be compatible with our reduction transformation (\ref{Integration reduce dimension}).

\item[{\rm (2)}] Another way is that in \cite{Hor90} (See \cite[Lemma 4.1.3 at p. 80, and Lemma 4.2.1 and Theorem 4.2.2 at p. 84]{Hor90}, which is based on a delicate choice of weight functions (\cite[(4.2.2) at p. 82 and Proof of Theorem 4.2.2 at p. 85]{Hor90}) so that for the corresponding $L^2$ estimate one only needs to consider those functions/forms with compact supports, and hence the boundary conditions for them become quite simple. Due to the lack of local compactness, when trying to extend such an argument to infinite dimensions, we have to modify essentially the usual spaces of smooth functions on Euclidean spaces. For this purpose, we introduce the concept of uniform inclusion for two sets in $\ell^2$ (See Definition \ref{def of bounded contained}), and via which we may introduce the desired spaces $C^{\infty}_0(V)$ (and also $C^{\infty}_{F}(V)$ and so on) of smooth functions (in a certain new sense) on any nonempty open set V of $\ell^2$ (See Subsection \ref{definition of basic symbols} for the details). We may then define pseudo-convex domains in $\ell^2$ by means of plurisubharmonic exhaustion functions (See Definition \ref{pseudocovexdomain}), via which the desired weight functions can be constructed (See (\ref{weight function}) and (\ref{230503e1})). With these machines in hand, for the corresponding $L^2$ estimate in infinite dimensions we only need to consider those functions/forms which vanish on the boundaries, and hence our reduction transformation (\ref{Integration reduce dimension}) works well to reduce the original infinite-dimensional problems to a sequence of approximate ones in finite dimensions so that the Friedrichs mollifier technique can be employed (See the proof of Theorem \ref{general density4}). In this way, we succeed in establishing the desired $L^2$ estimates for the $\overline{\partial}$ operators in general pseudo-convex domains of $\ell^2$, and via which we obtain the solvability of the corresponding infinite-dimensional $\overline{\partial}$ equations.
\end{itemize}

 The strategy adopted in this work (i.e., choosing $\ell^2$ and/or its suitable subsets as typical underlying spaces) can be refined to propose a convenient framework for infinite-dimensional analysis (including real and complex analysis), in which differentiation (in some weak sense) and integration operations can be easily performed, integration by parts can be conveniently established under rather weak conditions (\cite{WYZ0}),
 and especially some nice properties and consequences obtained by convolution (or even the Fourier transformation) in Euclidean spaces can be extended to infinite-dimensional spaces in some sense by taking the limit (See \cite{YZ-b} for more details). Compared to the existing tools in infinite-dimensional analysis (including abstract Wiener space, white noise analysis, Malliavin calculus and so on), our framework enjoys more convenient and clearer links with that of finite dimensions, and hence it is more suitable for computation and studying some analysis problems in infinite-dimensional spaces. Except for the $L^2$ estimates for the $\overline{\partial}$ operators (in general pseudo-convex domains of $\ell^2$) obtained in this paper, by modifying suitably the above strategy, we can also solve several other longstanding open problems, including Holmgren type theorem with infinite number of variables (\cite{YZ-a}) and the equivalence of Sobolev spaces over quite general domains of infinite-dimensional spaces (\cite{YZZ}).

It is well-known that the concepts of smooth functions in infinite-dimensional spaces are much more complicated than that in finite dimensions. The usual Fr\'echet and G\^ateaux derivatives, especially those of higher order, are actually quite inconvenient to use. In our framework, instead we use the partial derivatives (equivalent to some specific directional derivatives, which have lower requirements than all directional derivatives, i.e. the G\^ateaux derivative).
The related regularity issues for solutions to the $\overline{\partial}$ equations considered in this paper are much more difficult than that in finite dimensions, which,
as well as its applications in infinite-dimensional complex analysis, will be presented in our forthcoming papers.

The rest of this paper is organized as follows. Section \ref{secx2} is of preliminary nature, in which we introduce some basic notations and notions, including some important spaces of functions and particularly the definition of pseudo-convex domains of $\ell^2$.
In Section \ref{secx3}, we define the $\overline{\partial}$ operators on our working spaces and present some properties of these operators. Section \ref{secx4} is devoted to establishing the key $L^2$ estimates for these operators, while Section \ref{sec5} is for solving the $\overline{\partial}$ equations on pseudo-convex domains of $\ell^{2}$.

\section{Preliminaries}\label{secx2}

\subsection{Notations and definitions}\label{definition of basic symbols}
To begin with, we introduce some notations and definitions. Suppose that $G$ is a
nonempty set and $(G,\cal J)$ is a topological space on $G$, then the smallest $\sigma$-algebra
generated by all sets in $\cal J$ is called the
Borel $\sigma$-algebra of
$G$, denoted by $\mathscr{B}(G)$. For any $\sigma$-finite Borel measure $\nu$ on $G$ and $p\in [1,\infty)$, denote by $L^p(G,\nu)$ the Banach space of all
(equivalence classes of) complex-valued Borel-measurable functions $f$ on $G$ for which $\int_G|f|^pd\nu<\infty$ (Particularly, $L^2(G,\nu)$ is a Hilbert space with the canonical inner product).
For any subset $A$ of $G$, denote by $\overline{A}$ the closure of $A$ in $G$, by $\partial A$ the boundary of $A$, and  by $\chi_A(\cdot)$ the characteristic function of $A$ in $G$. Recall that $G$ is called a Lindel\"{o}f space if every open cover of $G$ has a countable subcover (e.g., \cite[p. 50]{Kel}).

Denote by $\mathbb{R}$ and $\mathbb{C}$ respectively the sets (or fields, or Euclidean spaces with the usual norms) of all real and complex numbers, and by $\mathbb{N}$ the set of all positive integers. Write $\mathbb{N}_0\triangleq \mathbb{N}\cup \{0\}$ and $\mathbb{C}^{\infty}=\prod\limits_{i=1}^{\infty}\mathbb{C}$, which is the countable Cartesian product of $\mathbb{C}$, i.e., the space of sequences of complex numbers (Nevertheless, in the sequel we shall write an element $\textbf{z}\in \mathbb{C}^{\infty }$ explicitly as $\textbf{z}=( z_{i})^{\infty }_{i=1}$ rather than $\textbf{z}=\{ z_{i}\}^{\infty }_{i=1}$, where $z_i\in\mathbb{C}$ for each $i\in \mathbb{N}$). Clearly, $\mathbb{C}^{\infty}$ can be identified with $(\mathbb{R}^2)^{\infty}=\prod\limits_{i=1}^{\infty}\mathbb{R}^2$. Set
$$
\ell^{2}\triangleq\left\{( z_{i})^{\infty }_{i=1} \in \mathbb{C}^{\infty } :\; \sum\limits^{\infty }_{i=1} \left| z_{i}\right|^{2}  <\infty \right\}.
$$
Then, $\ell^2$ is a separable Hilbert space with the canonic inner product $\left( \cdot,\cdot\right)$ (and hence  the canonic norm $\left| \left|\;\cdot\;\right|  \right|$). From \cite[Theorem 15, p. 49]{Kel}, we see that every separable metric space is a Lindel\"{o}f space. Hence, in particular every open subset of $\ell^2$ is a Lindel\"{o}f space. For any $\textbf{a}\in \ell^{2}$ and $r\in(0,+\infty)$, set
$$
B_{r}(\textbf{a})\triangleq \left\{ \textbf{z}\in \ell^{2} :\;\left| \left| \textbf{z}-\textbf{a}\right|  \right|<r\right\}  ,
$$
and we simply write $B_{r}$ for $B_{r}(0)$. For any $z\in \mathbb{C}$, denote by $\overline{z}$ the complex conjugation of $z$.

In the rest of this paper, unless otherwise stated, we will fix a nonempty open set $V$ of $\ell^{2}$ and a sequence $\{a_i\}_{i=1}^{\infty}$ of positive numbers  such that
$$\sum\limits_{i=1}^{\infty} a_i <1.
$$

Suppose that $f$ is a complex-valued function defined on $V$. As in \cite[p. 528]{YZ}, for each $\textbf{z}=(z_j)_{j=1}^{\infty}=(x_j+\sqrt{-1}y_j)_{j=1}^{\infty}\in V$ (here and henceforth $x_j,y_j\in \mathbb{R}$ for each $j\in \mathbb{N}$), we define the partial derivatives of $f(\cdot)$ (at $\textbf{z}$ with respect to $x_j$ and $y_j$):
$$
\begin{array}{ll}
\displaystyle\frac{\partial f(\textbf{z})}{\partial x_j}\triangleq\lim_{\mathbb{R}\ni \tau\to 0}\frac{f(z_1,\cdots,z_{j-1},z_j+\tau,z_{j+1},\cdots)-f(\textbf{z})}{\tau},\\[3mm]
\displaystyle\frac{\partial f(\textbf{z})}{\partial y_j}\triangleq\lim_{\mathbb{R}\ni \tau\to 0}\frac{f(z_1,\cdots,z_{j-1},z_j+\sqrt{-1}\tau,z_{j+1},\cdots)-f(\textbf{z})}{\tau},
\end{array}
$$
provided that the above limits exist.

In the sequel, for each $k\in\mathbb{N}_0$, we denote by $C^{k} ( V)$ the set of all complex-valued functions $f$ on $V$ for which $f$ itself and all of its  partial derivatives up to order $k$ are continuous on $V$. Write
$$
C^{\infty} ( V)\triangleq \bigcap_{j=1}^{\infty}C^{j} ( V).
$$
Note that throughout this work we use partial derivatives instead of the usual Fr\'echet derivatives. The following simple result reveals a big difference between these two notions of derivatives in infinite dimensions.
\begin{proposition}\label{SMOOTH BUT NOT Frechet}
There exists $f^0\in C^{\infty}(\ell^2)$, which is not Fr\'echet differentiable at some point $\textbf{z}_0\in\ell^2$.
\end{proposition}
\begin{proof}
Let
$$
f^0(\textbf{z})\triangleq \prod\limits_{j=1}^{\infty}(1+x_j^2\sin(2j^2\pi x_j))(1+y_j^2\sin(2j^2\pi y_j)),\quad \forall\;\textbf{z}=(x_j+\sqrt{-1}y_j)_{j=1}^{\infty}\in\ell^2.
$$
Note that $|x_j^2\sin(2j^2\pi x_j)|+|y_j^2\sin(2j^2\pi y_j)|\leqslant x_j^2+y_j^2$ for each $j\in\mathbb{N}$ and $\sum\limits_{j=1}^{\infty}(x_j^2+ y_j^2)<\infty$, by \cite[Theorem 15.4, p. 299]{Rud87}, the infinite product $\prod\limits_{j=1}^{\infty}(1+x_j^2\sin(2j^2\pi x_j))(1+y_j^2\sin(2j^2\pi y_j))$ exists for each $\textbf{z}=(x_j+\sqrt{-1}y_j)_{j=1}^{\infty}\in\ell^2.$ It is easy to see that all order of partial derivatives of $f^0$ exist and $\frac{\partial f^0(\textbf{x})}{\partial x_i}=2(x_i\sin(2i^2\pi x_i)+i^2\pi x_i^2\cos(2i^2\pi x_i))\prod\limits_{j\neq i}(1+x_j^2\sin(2j^2\pi x_j))\prod\limits_{j=1}^{\infty}(1+y_j^2\sin(2j^2\pi y_j))$. Put $\textbf{z}_0\triangleq\left(\frac{1}{j}\right)_{j=1}^{\infty}\in\ell^2$. Then we have $\frac{\partial f^0(\textbf{z}_0)}{\partial x_i}=2\pi$ for each $i\in\mathbb{N}$. Let us use the contradiction argument to show that $f^0$ is not Fr\'{e}chet differentiable at $\textbf{z}_0$. If $f^0$ was Fr\'{e}chet differentiable at $\textbf{z}_0$, then it would hold that $\sum\limits_{i=1}^{\infty}\left(\left|\frac{\partial f^0(\textbf{z}_0)}{\partial x_i}\right|^2+\left|\frac{\partial f^0(\textbf{z}_0)}{\partial y_i}\right|^2\right)<\infty$,  and hence $\lim\limits_{i\to\infty}\frac{\partial f^0(\textbf{z}_0)}{\partial x_i}=0$, which is a contradiction.

For simplicity, we only prove that $f^0\in C^0(\ell^2)$. For each $n\in\mathbb{N}$, let
$$
f^0_n(\textbf{z})\triangleq \prod\limits_{j=1}^{n}(1+x_j^2\sin(2j^2\pi x_j))(1+y_j^2\sin(2j^2\pi y_j)),\quad \forall\;\textbf{z}=(x_j+\sqrt{-1}y_j)_{j=1}^{\infty}\in\ell^2.
$$
Obviously, $\{f^0_n\}_{n=1}^{\infty}\subset C^0(\ell^2)$. By \cite[Proposition 2.1, p. 778]{Lem99}, we only need to prove that $f^0_n\to f^0$ uniformly on each compact subset $K$ of $\ell^2$. Firstly, since $K$ is compact and
$$|x_n^2\sin(2n^2\pi x_n)|+|y_n^2\sin(2n^2\pi y_n)|\leqslant x_n^2+y_n^2\leqslant ||\textbf{z}||^2,\quad \forall\;n\in\mathbb{N},
 $$
we have $\sup\limits_{\textbf{z}\in K}||\textbf{z}||<\infty$ and hence both $\{x_n^2\sin(2n^2\pi x_n)\}_{n=1}^{\infty}$ and $\{y_n^2\sin(2n^2\pi y_n)\}_{n=1}^{\infty}$ are sequences of bounded functions on $K$. By \cite[Theorem 15.4, p. 299]{Rud87} again, we only need to prove that the series $\sum\limits_{j=1}^{\infty}(x_j^2+y_j^2)$ converges uniformly on $K$. Since $K$ is compact, for any given $\varepsilon>0$, one can find $k_1\in\mathbb{N}$ and $\textbf{z}_1=(x_j^1+\sqrt{-1}y_j^1)_{j=1}^{\infty}\in K,\cdots,\textbf{z}_{k_1}=(x_j^{k_1}+\sqrt{-1}y_j^{k_1})_{j=1}^{\infty}\in K$ such that for any $\textbf{z}=(x_j+\sqrt{-1}y_j)_{j=1}^{\infty}\in K$ there exists an integer $k$ with $1\leqslant k\leqslant k_1$ such that $||\textbf{z}_k-\textbf{z}||<\frac{\sqrt{\varepsilon}}{2}$. Note that there exists $N\in\mathbb{N}$ such that $\sum\limits_{j=N}^{\infty}\left((x_j^l)^2+(y_j^l)^2\right)<\frac{\varepsilon}{4}$ for all integers $l$ satisfying $1\leqslant l\leqslant k_1$. Therefore, by the triangle inequality, we have
$$
 \sum\limits_{j=N}^{\infty}(x_j^2+y_j^2)\leqslant \left(\sqrt{\sum\limits_{j=N}^{\infty}\left((x_j^k)^2+(y_j^k)^2\right)}+\sqrt{\sum\limits_{j=N}^{\infty}\left[(x_j-x_j^k)^2+(y_j-y_j^k)^2\right]}\right)^2
 \leqslant \left(\frac{\sqrt{\varepsilon}}{2}+||\textbf{z}_k-\textbf{z}||\right)^2<\varepsilon
$$
for all $\textbf{z}\in K$, which implies that the series $\sum\limits_{i=1}^{\infty}(x_i^2+y_i^2)$ converges uniformly on $K$. This completes the proof of Proposition \ref{SMOOTH BUT NOT Frechet}.
\end{proof}

Recall that any continuous function on a finite dimensional space is bounded on each bounded and closed subset of its domain. However, this is NOT true anymore for continuous functions on $\ell^2$, as shown in the following result.
\begin{proposition}\label{Frechet D but NOT bounded on boudned set}
There exists a continuous function (on $\ell^2$) which is unbounded on some bounded, nonempty subset of $\ell^2$.
\end{proposition}
\begin{proof}
Let $\{\textbf{e}_k\}_{k=1}^{\infty}$ be an orthonormal basis on $\ell^2$. Then, $||\textbf{e}_k||=1$ for each $k\in\mathbb{N}$ and $||\textbf{e}_k-\textbf{e}_l||=\sqrt 2$ for all $k,l\in\mathbb{N}$ with $k\neq l$. Choose $\varphi\in C^{\infty}(\mathbb{R})$ such that $0\leqslant \varphi\leqslant 1$, $\varphi(x)=1$ for all $|x|\leqslant \frac{1}{256}$ and $\varphi(x)=0$ for all $|x|\geqslant \frac{1}{64}$. Let
$$
 f^0(\textbf{z})\triangleq \sum_{k=1}^{\infty}k\cdot \varphi(||\textbf{z}-\textbf{e}_k||^2),\quad \forall\;\textbf{z}\in\ell^2.
$$
If $\textbf{z}\in B(\textbf{e}_{k_0},\frac{1}{4})$ for some $k_0\in\mathbb{N}$, then $\textbf{z}\notin \bigcup\limits_{k\neq k_0}B(\textbf{e}_{k},\frac{1}{8})$ and hence $f^0(\textbf{z})=k_0\cdot\varphi(||\textbf{z}-\textbf{e}_{k_0}||^2)$. If $\textbf{z}\in \ell^2\setminus \overline{\bigcup\limits_{k=1}^{\infty}B(\textbf{e}_{k},\frac{1}{8})}$, then $f^0(\textbf{z})=0$. Thus, $f^0$ is continuous at each point of $\ell^2$ and $f^0(\textbf{e}_k)=k$ for each $k\in\mathbb{N}$, which implies that $\sup\limits_{||\textbf{z}||\leqslant 1}|f^0(\textbf{z})|=\infty$. The proof of Proposition \ref{Frechet D but NOT bounded on boudned set} is completed.
\end{proof}

\begin{remark}
 For any infinite-dimensional Banach space $X$ with the norm $||\cdot||_X$, by the Riesz Lemma (e.g., \cite[Lemma 2, p. 218]{Riesz}), there exists a sequence $\{h_k\}_{k=1}^{\infty}\subset X$ such that $||h_k||_X=1$ for each $k\in\mathbb{N}$ and $||h_k-h_l||_X>\frac{1}{2}$ for all $k,l\in\mathbb{N}$ with $k\neq l$. Then, similarly to the proof of Proposition \ref{Frechet D but NOT bounded on boudned set}, one can find a continuous function (on $X$) which is unbounded on some bounded, nonempty subset of $X$.
\end{remark}

Because of Proposition \ref{Frechet D but NOT bounded on boudned set}, we shall consider bounded and continuous functions on $\ell^2$ in this paper. More precisely, in the sequel, for each $k\in\mathbb{N}_0$, we denote by $C^{k}_{b}\left( V\right)$  the set of all complex-valued functions $f$ on $V$ for which $f$ itself and all of its  partial derivatives up to order $k$ are bounded and continuous on $V$. Write
$$C^{\infty}_b(V)\triangleq \bigcap_{j=1}^{\infty}C^{j}_b( V).
 $$

The following notion (of uniform inclusion) will play a fundamental role in the sequel.
\begin{definition}\label{def of bounded contained}
A set $S\subset \ell^2$ is said to be uniformly included in $V$, denoted by $S\stackrel{\circ}{\subset} V$, if there exist $r,R\in(0,+\infty)$ such that $\cup_{\textbf{z}\in S}B_r(\textbf{z})\subset V$ and $S\subset B_R$.
\end{definition}

For any set $S\subset \ell^2$, one can show that $S\stackrel{\circ}{\subset} V$ if and only if $S$ is a bounded subset of $V$ lying at a positive distance from the boundary of $V$ (See Lemma \ref{221018lem1} in Subsection \ref{subsection000}). On the other hand, it is easy to see that $S\stackrel{\circ}{\subset} V$ if and only if $\overline{S} \stackrel{\circ}{\subset} V$.

\begin{remark}
The above notion of uniform inclusion (or maybe its variants) has been used in previous literatures, e.g., \cite[p. 543]{AGGM} (addressing to some other issue in infinite dimensional analysis) and \cite[Definition 0.2, p. 1]{IS85} (for the case that $V$ is
a bounded domain). Due to its fundamental importance in the sequel, we give it a specified, illuminating new name and notation for convenience.
\end{remark}

For $f\in C^0(V)$, we call $\supp f\triangleq\overline{\{\textbf{z}\in V:\;f(\textbf{z})\neq 0\}}$ the support of $f$. For $k\in\mathbb{N}_0$, write
$$
C^{k}_{0}\left( V\right)  \triangleq \left\{ f\in C^{k}_{b}\left( V\right)  :\;\supp f\stackrel{\circ}{\subset} V\right\},\quad C^{\infty}_0(V)\triangleq \bigcap_{j=1}^{\infty}C^{j}_0( V).
$$
For each $n\in\mathbb{N}$, we denote by $C^{k}\left( \mathbb{C}^n\right)$  the set of all complex-valued functions $f$ on $\mathbb{C}^n$ for which $f$ itself and all of its  partial derivatives up to order $k$ are continuous on $\mathbb{C}^n$. For any $f\in C^0\left( \mathbb{C}^n\right)$, we may define $\supp f$ similarly as above. Write
$$
C^{k}_{c}\left( \mathbb{C}^n\right)  \triangleq \left\{ f\in C^{k}\left( \mathbb{C}^n\right)  :\;\supp f\hbox{ is compact in } \mathbb{C}^n\right\},\quad C^{\infty}_c\left(\mathbb{C}^n\right)\triangleq \bigcap_{j=1}^{\infty}C^{j}_c\left( \mathbb{C}^n\right).
$$
In a similar way, we denote by $C^{\infty}_c\left(\mathbb{R}^n\right)$ the set of all complex-valued functions $f$ on $\mathbb{R}^n$ for which $f$ itself and all of its  partial derivatives are continuous on $\mathbb{R}^n$, and $\supp f$ is compact in $\mathbb{R}^n$.

\begin{remark}
Note that, since $V$ is assumed to be a nonempty open set of the infinite-dimensional Hilbert space $\ell^2$, any non-zero function $g$ in $C^{0}_{0}\left( V\right)$ does NOT have a compact support in $V$. Nevertheless, such a function $g$ is bounded and continuous on $V$ and $\supp g\stackrel{\circ}{\subset} V$, and therefore, in a suitable sense and to a certain extend, functions in $C^{0}_{0}\left( V\right)$ can be regarded as ideal (infinite-dimensional) counterparts of compactly supported continuous functions on nonempty open sets of Euclidean spaces.
\end{remark}

For $f\in C^1(V)$ and $i\in\mathbb{N}$, write
\begin{equation}\label{eq2301161}
\begin{array}{ll}
\displaystyle D_{x_{i}}f\triangleq\frac{\partial f}{\partial x_{i}},\quad D_{y_{i}}f\triangleq\frac{\partial f}{\partial y_{i}},\quad  \partial_{i} f\triangleq\frac{1}{2}( D_{x_{i}}f-\sqrt{-1}D_{y_{i}}f),\\[3mm]
\displaystyle\overline{\partial }_{i} f\triangleq\frac{1}{2}( D_{x_{i}}f +\sqrt{-1}D_{y_{i}}f), \quad \delta_{i} f\triangleq \partial_{i} f-\frac{\overline{z_{i}} }{2a^{2}_{i}} \cdot f,\quad \overline{\delta_{i} } f\triangleq \overline{\partial_{i} } f-\frac{z_{i}}{2a^{2}_{i}}\cdot f.
\end{array}
\end{equation}
Denote by $\mathbb{N}_0^{(\mathbb{N})}$ the set of all finitely supported sequences of nonnegative integers, i.e., for each $\alpha =(\alpha_{j} )_{j=1}^\infty\in \mathbb{N}^{\left( \mathbb{N} \right)  }_{0} $ with $\alpha_{m}\in\mathbb{N}_{0}$ for any $ m\in\mathbb{N}$, there exists $n\in\mathbb{N}$ such that $\alpha_j=0$ for all $j\geqslant n$. Set $k\triangleq \left| \alpha \right|  \triangleq \sum\limits^{\infty }_{j=1} \alpha_{j}$. For $f\in C^k(V)$, write
$$
\delta^{\alpha }f  \triangleq \delta^{\alpha_{1} }_{1} \  \delta^{\alpha_{n} }_{2} \cdots \  \delta^{\alpha_{n} }_{n} f,\qquad  \overline{\delta }^{\alpha  } f\triangleq \overline{\delta }^{\alpha _{1} }_{1} \  \overline{\delta }^{\alpha _{2} }_{2} \  \cdots \  \overline{\delta }^{\alpha_{n} }_{n}  f,
$$
$$
\partial_{\alpha } f\triangleq \partial^{\alpha_{1} }_{1} \partial^{\alpha_{2} }_{2} \cdots \partial^{\alpha_{n} }_{n} f ,\qquad  \overline{\partial}_{\alpha } f \triangleq \overline{\partial}_{1} ^{\alpha_{1}}  \overline{\partial}_{2} ^{\alpha _{2} } \cdots \overline{\partial}_{n} ^{\alpha _{n} }f.
$$

Put
$$
C^{1}_{F}\left( V\right)\triangleq \left\{f\in C^{1}\left( V\right):\;\sup_{E} \left( \left| f\right|  +\sum^{\infty }_{i=1} \bigg(\left| D_{x_i} f\right|^{2}  +\left|D_{y_i} f\right|^{2}\bigg)  \right) <\infty,\; \forall\; E\stackrel{\circ}{\subset}V\right\}.
$$
Clearly, $C^{1}_{F}\left( V\right)\not\subset C^{1}_{b}\left( V\right)$ and $C^{1}_{b}\left( V\right)\not\subset C^{1}_{F}\left( V\right)$. For $k\in\mathbb{N}\cup \{\infty\}$, we write
$$
\begin{array}{ll}
\displaystyle C^{k}_{F}\left( V\right)\triangleq C^{k}\left( V\right)\cap C^{1}_{F}\left( V\right),\quad
C^{k}_{b,F}\left( V\right)\triangleq C^{k}_{b}\left( V\right)\cap C^{1}_{F}\left( V\right),\quad
C^{k}_{0,F}\left( V\right)\triangleq C^{k}_{0}\left( V\right)\cap C^{1}_{F}\left( V\right),\\[3mm]
\displaystyle C^{k}_{F^{k}}\left( V\right)\triangleq \left\{f\in C^{k}(V):\;\partial_{a} \overline{\partial_{\beta } } f\in C^{1}_{F}\left( V\right)\text{ for all }\alpha, \beta\in \mathbb{N}^{\left( \mathbb{N} \right)  }_{0}\text{ with }|\alpha|+|\beta|<k\right\},\\[3mm]
\displaystyle C^{k}_{b,F^{k}}\left( V\right)\triangleq C^{k}_{b}\left( V\right)\cap C^{k}_{F^{k}}\left( V\right),\quad
C^{k}_{0,F^{k}}\left( V\right)\triangleq C^{k}_{0}\left( V\right)\cap C^{k}_{F^{k}}\left( V\right).
\end{array}
$$
It is obvious that
$$
C^{k}_{F}\left( V\right)\supset C^{k}_{b,F}\left( V\right)\supset C^{k}_{0,F}\left( V\right),\quad
C^{k}_{F^{k}}\left( V\right)\supset C^{k}_{b,F^{k}}\left( V\right)\supset C^{k}_{0,F^{k}}\left( V\right).
$$

\subsection{Multiplicative System Theorems}
The material of this short subsection is from \cite[pp. 111--112]{Dri} (with some modification).

Suppose that $\Omega$ is a nonempty set and $\mathbb{H}$ is a set of some bounded $\mathbb{R}$-valued (or $\mathbb{C}$-valued) functions on $\Omega$. We call $\{f_n\}_{n=1}^{\infty}\subset \mathbb{H}$ a bounded convergent sequence in $\mathbb{H}$ if
$\displaystyle\sup_{(n,\omega)\in\mathbb{N}\times\Omega}|f_n(\omega)|< \infty$, and
$f(\omega)\triangleq\lim\limits_{n\to\infty}f_n(\omega)$ exists for every $\omega\in\Omega$.
We say that $\mathbb{H}$ is closed under the bounded convergence if the limit of every bounded convergent sequence in $\mathbb{H}$ is still in $\mathbb{H}$.
 A subset $\mathbb{M}$ of $\mathbb{H}$ is called a multiplicative system in $\mathbb{H}$ if  for any $f,g\in\mathbb{M}$, it holds that $f\cdot g\in\mathbb{M}$. We denote by $\sigma(\mathbb{M})$ the smallest $\sigma$-algebra on $\Omega$ such that for any $f\in \mathbb{M}$, $f$ is a $\sigma(\mathbb{M})$-measurable function. It is easy to see that, for the case of $\mathbb{R}$-valued ({\it resp.} $\mathbb{C}$-valued) functions, $\sigma(\mathbb{M})$ is exactly the $\sigma$-algebra generated by all subsets (of $\Omega$) in the form $h^{-1}(A)$, where $h\in \mathbb{M}$ and $A\in \mathscr{B}(\mathbb{R})$ ({\it resp.} $A\in \mathscr{B}(\mathbb{C})$).

\begin{theorem}\label{Dynkin's Multiplicative System Theorem}
\textbf{(Dynkin's Multiplicative System Theorem).} If $\mathbb{H}$  is a real linear space of some bounded $\mathbb{R}$-valued functions on $\Omega$, $1\in \mathbb{H}$, $\mathbb{H}$ is closed under the bounded convergence and $\mathbb{M}$ is a multiplicative system in $\mathbb{H}$ , then $\mathbb{H}$ contains all bounded $\mathbb{R}$-valued $\sigma(\mathbb{M})$-measurable functions.
\end{theorem}
The following is a complex version of Theorem \ref{Dynkin's Multiplicative System Theorem}.
\begin{theorem}\label{Complex Multiplicative System Theorem}
\textbf{(Complex Multiplicative System Theorem).}
If $\mathbb{H}$  is a complex linear space of some bounded $\mathbb{C}$-valued functions on $\Omega$, $1\in \mathbb{H}$, $\mathbb{H}$  is closed under the bounded convergence, $\mathbb{M}$ is a multiplicative system in $\mathbb{H}$, and both  $\mathbb{H}$  and $\mathbb{M}$ are closed under the complex conjugation, then $\mathbb{H}$ contains all bounded $\mathbb{C}$-valued $\sigma(\mathbb{M})$-measurable functions.
\end{theorem}

\subsection{A Borel probability measure on $\ell^2$}\label{s2-3}
For any given $a>0$, we define a probability measure $\bn_a$ in $(\mathbb{C},\mathscr{B}(\mathbb{C}))$ by
$$
 \bn_a(E)\triangleq \frac{1}{ 2\pi a^2 }\int_{E}e^{-\frac{x^2+y^2}{2a^2}}\,\mathrm{d}x\mathrm{d}y,\quad\,\forall\;E\in \mathscr{B}(\mathbb{C}).
$$
We need to use a product measure on the space $\mathbb{C}^{\infty}$, endowed with the usual product topology.  By the discussion in \cite[p. 9]{Da}, $\mathscr{B}(\mathbb{C}^{\infty})$ is precisely the product $\sigma$-algebra  generated by the following sets
$$
\bigcup_{k=1}^{\infty}\Bigg\{E_1\times E_2\times \cdots \times E_k\times \Bigg(\prod_{i=k+1}^{\infty}\mathbb{C}\Bigg):\;E_i\in \mathscr{B}(\mathbb{C}),\,1\leqslant i\leqslant k\Bigg\}.
$$
Let
$$
 \bn\triangleq\prod_{i=1}^{\infty}\bn_{a_i}
$$
be the product measure on $(\mathbb{C}^{\infty},\mathscr{B}(\mathbb{C}^{\infty}))$. Just as \cite[Exercise 1.10, p. 12]{Da}, one can show that every closed ball in $\ell^2$ lies in $\mathscr{B}(\mathbb{C}^{\infty})$ and
$\mathscr{B}(\ell^2)=\{E\cap\ell^2:\; E\in \mathscr{B}(\mathbb{C}^{\infty})\}$.
We recall the following known simple but useful result (See \cite[Proposition 1.11, p. 12]{Da}).

\begin{proposition}\label{lem1}
$\bn(\ell^2)=1$.
\end{proposition}

Thanks to Proposition \ref{lem1}, we obtain a Borel probability measure $P$ on $\ell^2$ by setting
$$
P(E)\triangleq \bn(F),\quad\forall\;E\in \mathscr{B}(\ell^2),
$$
where $F$ is any subset in $\mathscr{B}(\mathbb{C}^{\infty})$ so that $E=F\cap\ell^2$.
For any $n\in\mathbb{N}$, we denote by $C_c^{\infty}(\mathbb{C}^n)$ the set of all $\mathbb{R}$-valued, $C^{\infty}$-functions on $\mathbb{C}^n(\equiv \mathbb{R}^{2n})$ with compact supports. Clearly, each function in $C_c^{\infty}(\mathbb{C}^n)$ can also be viewed as a cylinder function on $\ell^2$. Set
$$
\mathscr {C}_c^{\infty}(\mathbb{C}^n)\triangleq \left\{f+\sqrt{-1}g:\;f,g\in C_c^{\infty}(\mathbb{R}^{2n})\right\},\quad\mathscr {C}_c^{\infty}\triangleq \bigcup_{n=1}^{\infty}\mathscr {C}_c^{\infty}(\mathbb{C}^n).
$$
By \cite[Proposition 2.4, p. 528]{YZ}, $\mathscr {C}_c^{\infty}$ is dense in $L^2(\ell^2,P)$.

We shall need the following result.

\begin{lemma}\label{strictly positive measure}
For every nonempty open subset $O$ of $\ell^2$, it holds that $P(O)>0$.
\end{lemma}
\begin{proof}
We use the contradiction argument. Suppose that $P(O)>0$ was not true. Then, we would have $P(O)=0$.
Let
$$
\mathscr{Q}\triangleq \bigcup\limits_{n=1}^{\infty}\left\{(r_1+\sqrt{-1}r_2,\cdots,r_{2n-1}+\sqrt{-1}r_{2n},0,\cdots,0,\cdots):\;r_1,\cdots,r_{2n}\hbox{ are real rational numbers}\right\}.
  $$
One can show that $\mathscr{Q}$ is dense in $\ell^2$, and $\ell^2=\bigcup\limits_{\textbf{r}\in \mathscr{Q}}(\textbf{r}+O)$, where $\textbf{r}+O=\{\textbf{r}+\textbf{z}:\;\textbf{z}\in O\}$. By the conclusion (ii) of \cite[Theorem 2.8, p. 31]{Da} and $P(O)=0$,  for each $\textbf{r}\in \mathscr{Q}$, we have $P(\textbf{r}+O)=0$. Since $\mathscr{Q}$ is a countable set, this leads to $1=P(\ell^2)\leqslant \sum\limits_{\textbf{r}\in \mathscr{Q}}P(\textbf{r}+O)=0$, which is a contradiction. This completes the proof of Lemma \ref{strictly positive measure}.
\end{proof}

\begin{remark}
As in \cite[p. 125]{Com Neg}, a probability measure with the property stated in Lemma \ref{strictly positive measure} is called a strictly positive measure.
\end{remark}

As a consequence of Lemma \ref{strictly positive measure}, it is easy to see the following property.
\begin{corollary}\label{property of strictly positive measure}
Suppose that $E$ is a Borel subset of $\ell^2$ such that $P(E)=1$. Then $E$ is dense in $\ell^2$.
\end{corollary}

For each $n\in\mathbb{N}$, we define a probability measure $\bn^n$ in $(\mathbb{C}^n,\mathscr{B}(\mathbb{C}^n))$ by setting
$$
\bn^n\triangleq \prod_{i=1}^{n}\bn_{a_i}
$$
and a probability measure ${\widehat{ \bn}^n}$ in $(\mathbb{C}^{\infty},\mathscr{B}(\mathbb{C}^\infty))$ by setting
$$
\widehat{\bn}^n\triangleq \prod_{i=n+1}^{\infty}\bn_{a_i}.
$$
Let
\begin{equation}\label{220817e1zx}
P_n(E)\triangleq \widehat{\bn}^n(F),\quad\forall\;E\in \mathscr{B}(\ell^2),
\end{equation}
where $F$ is any subset in $\mathscr{B}(\mathbb{C}^{\infty})$ so that $E=F\cap\ell^2$.
Then, by Proposition \ref{lem1} again, we obtain a Borel probability measure $P_n$ on $\ell^2$. Obviously,
$$P=\bn^n\times P_n,\quad\forall\; n\in \mathbb{N}.$$
Further, for any $f\in L^2(\ell^2, P)$, we define
\begin{equation}\label{Integration reduce dimension}
f_n(\textbf{z}_n)\triangleq \int f(\textbf{z}_n,\textbf{z}^n)\,\mathrm{d}P_n(\textbf{z}^n),\quad \textbf{z}_n=(x_{i} +\sqrt{-1}y_{i})_{i=1}^{n}\in \mathbb{C}^n,
\end{equation}
where $\textbf{z}^n=(x_{i} +\sqrt{-1}y_{i})_{i=n+1}^\infty\in \ell^2$. The following result will play a crucial role in the sequel.

\begin{proposition}\label{Reduce diemension}
For $f\in L^{2} ( \ell^{2} ,P )$ and $n\in\mathbb{N}$, the function $f_n$ given by (\ref{Integration reduce dimension}) can be viewed as a cylinder function on $\ell^2$ with the following properties:
\begin{itemize}
\item[{\rm (1)}]\label{z1}
		\quad $\left| \left| f_{n}\right|  \right|_{L^{2} ( \ell^{2} ,P )  }  \leqslant \left| \left| f\right|  \right|_{L^{2} ( \ell^{2} ,P )  } $;
\item[{\rm (2)}] \label{z2}
		\quad $\lim\limits_{n\rightarrow \infty } \left| \left| f_{n}-f\right|  \right|_{L^{2} ( \ell^{2} ,P)  }  =0$;
\item[{\rm (3)}] \label{z3}
	\quad $\lim\limits_{n\rightarrow \infty } \int_{\ell^2}\big|| f_{n}|^{2}  -| f|^{2}\big|\,\mathrm{d}P =0.$	
\end{itemize}
\end{proposition}
\begin{proof}
(1) Note that
$$
\begin{array}{ll}
\displaystyle \int_{\ell^2}|f_{n}|^2\,\mathrm{d}P=\int_{\mathbb{C}^n}|f_n(\textbf{z}_n)|^2\,\mathrm{d}\mathcal{N}^n(\textbf{z}_n)\\[5mm]
\displaystyle =\int_{\mathbb{C}^n}\bigg|\int f(\textbf{z}_n,\textbf{z}^n)
\mathrm{d}P_n(\textbf{z}^n)\bigg|^2\mathrm{d}\mathcal{N}^n(\textbf{z}_n)\\[5mm]
\displaystyle \leqslant \int_{\mathbb{C}^n} \int |f(\textbf{z}_n,\textbf{z}^n)|^2
\mathrm{d}P_n(\textbf{z}^n)\mathrm{d}\mathcal{N}^n(\textbf{z}_n)\\[5mm]
\displaystyle =\int_{\ell^2}|f|^2\,\mathrm{d}P,
\end{array}
$$
where the inequality follows from the Jensen inequality (See \cite[Theorem 3.3, p. 62]{Rud87}) and the first and the third equalities follows from the fact that $P=\bn^n\times P_n$.

\medskip

(2) Since $\mathscr {C}_c^{\infty}$ is dense in $L^2(\ell^2,P)$, for any $\varepsilon>0$, there exists $m\in\mathbb{N}$ and $g\in \mathscr{C}_c^{\infty}(\mathbb{C}^m)$ such that
$$
\bigg(\int_{\ell^2} |f-g|^2\mathrm{d}P\bigg)^{\frac{1}{2}}<\frac{\varepsilon}{2}.
$$
For any $n\geqslant m$, we define $g_n$ as that in (\ref{Integration reduce dimension}), i.e.,
$
g_n(\cdot)\triangleq \int g(\cdot,\textbf{z}^n)\,\mathrm{d}P_n(\textbf{z}^n)$.
It is easy to check that $g_n=g$
and
$$
\begin{array}{ll}
\displaystyle \bigg(\int_{\ell^2} |f_n-f|^2\mathrm{d}P\bigg)^{\frac{1}{2}}
=\bigg(\int_{\ell^2} |f_n-g_n+g-f|^2\mathrm{d}P\bigg)^{\frac{1}{2}}\\[3mm]
\displaystyle \leqslant\bigg(\int_{\ell^2} |f_n-g_n|^2\mathrm{d}P\bigg)^{\frac{1}{2}}+\bigg(\int |g-f|^2\mathrm{d}P\bigg)^{\frac{1}{2}}\\[3mm]
\displaystyle \leqslant 2\bigg(\int_{\ell^2} |g-f|^2\mathrm{d}P\bigg)^{\frac{1}{2}}<\varepsilon,
\end{array}
$$
where the second inequality follows from the conclusion (1).

\medskip

(3) Note that
$$
\begin{array}{ll}
\displaystyle
\int_{\ell^2}\big|| f_{n}|^{2}  -| f|^{2}\big|\,\mathrm{d}P  =
\int_{\ell^2}\big|| f_{n}|-| f|\big|\cdot(| f_{n}|+| f|)\,\mathrm{d}P\\[3mm]
\displaystyle \leqslant \big|\big| |f_{n}|-| f|\big|\big|_{L^{2}\left( \ell^{2} ,P\right)}\cdot \big|\big| |f_{n}|+| f|\big|\big|_{L^{2}\left( \ell^{2} ,P\right)}\\[3mm]
\displaystyle \leqslant  || f_{n}- f||_{L^{2}\left( \ell^{2} ,P\right)}\cdot(|| f_n||_{L^{2}\left( \ell^{2} ,P\right)}+ || f||_{L^{2}\left( \ell^{2} ,P\right)})\\[3mm]
\displaystyle \leqslant 2|| f_{n}- f||_{L^{2}\left( \ell^{2} ,P\right)}\cdot || f||_{L^{2}\left( \ell^{2} ,P\right)},
\end{array}
$$
where the first inequality follows from the Cauchy-Schwarz inequality. The above inequality, combined with the conclusion (2), gives the desired conclusion (3). This completes the proof of Proposition \ref{Reduce diemension}.
\end{proof}

We shall need the following infinite-dimensional Gauss-Green type Theorem.

\begin{theorem}\label{Gauss-Green Theorem}
 If $f\in C^{1}_{0} ( \ell^2 )$, then for each $m\in\mathbb{N}$,
$$
\int_{\ell^2}D_{x_m}f\,\mathrm{d}P
=\int_{\ell^2}\frac{x_m }{a_m^2}\cdot f \,\mathrm{d}P.
$$
\end{theorem}
\begin{proof}
For simplicity, we consider only $m=1$. Since $f\in C^{1}_{0} ( \ell^2 )$, there exists $R\in(0,+\infty)$ such that $\supp f\subset B_R$. Write
$$
\ell^{2,1}\triangleq \left\{(\sqrt{-1}y_1,x_{2}+\sqrt{-1}y_{2},\cdots):\;y_{1},x_i,y_i\in\mathbb{R}\hbox{ for }i=2,3,\cdots,|y_{1}|^2+\sum_{i\not=1}(|x_i|^2+|y_i|^2)<\infty\right\}.
$$
For any $\textbf{z}=(x_1+\sqrt{-1}y_1,x_{2}+\sqrt{-1}y_{2},\cdots)\in\ell^2$, where $x_i,y_i\in\mathbb{R}$ for each $i\in\mathbb{N}$, write
$$
\textbf{z}^{x_1}\triangleq (\sqrt{-1}y_1,x_{2}+\sqrt{-1}y_{2},\cdots).
$$
Clearly, $\textbf{z}^{x_{1}}\in\ell^{2,1}$ and $\ell^{2,1}$ is a real Hilbert space (with the canonical inner product). Let
$$F(\textbf{z})\triangleq D_{x_1}\left(f(\textbf{z})\right)\cdot \frac{e^{-\frac{x_1^2}{2a_1^2}}}{\sqrt{2\pi a_1^2}}.$$
Then,
$$
\int_{\ell^2}D_{x_1}f\,\mathrm{d}P
=\int_{\mathbb{R}\times \ell^{2,1}}F(\textbf{z})\,\mathrm{d}x_1\mathrm{d}P^{\widehat{x_1}}(\textbf{z}^{x_1}),
$$
where $P^{\widehat{x_{1}}}$ is the product measure $\prod_{i=1}^{\infty}\bn_{a_i}$ without the $x_{1}$-component, i.e., it is the restriction of the product measure $\frac{1}{\sqrt{2\pi a_{1}^2}}\cdot e^{-\frac{ y_{1}^2}{2a_{1}^2}}\,\mathrm{d}y_{1}\times\Pi_{i\not= 1}\bn_{a_i}$ on $\left(\ell^{2,1},\mathscr{B}\big(\ell^{2,1}\big)\right)$ (similarly to that in (\ref{220817e1zx})). Applying Fubini's theorem \cite[Theorem 8.8, p. 164]{Rud87}, we have
$$
\int_{\mathbb{R}\times \ell^{2,1}}F(\textbf{z})\,\mathrm{d}x_1\mathrm{d}P^{\widehat{x_1}}(\textbf{z}^{x_1})=\int_{\ell^{2,1}}\left(\int_{\mathbb{R}}F_{\textbf{z}^{x_1}}(x_1)\,\mathrm{d}x_1\right)\mathrm{d}P^{\widehat{x_1}}(\textbf{z}^{x_1}),
$$
where $F_{\textbf{z}^{x_1}}(x_1)\triangleq F((x_i+\sqrt{-1}y_{i})_{i=1}^\infty)$. Then, by our assumption, it follows that $F_{\textbf{z}^{x_1}}\in C^1(\mathbb{R})$ and $\supp F_{\textbf{z}^{x_1}}\subset [-R,R]$ for each $\textbf{z}^{x_1}\in \ell^{2,1}$. Therefore,
$$
\begin{array}{ll}
\displaystyle \int_{\ell^{2,1}}\left(\int_{\mathbb{R}}F_{\textbf{z}^{x_1}}(x_1)\,\mathrm{d}x_1\right)\mathrm{d}P^{\widehat{x_1}}(\textbf{z}^{x_1})
=\int_{\ell^{2,1}}\left(\int_{-R}^{R}F_{\textbf{z}^{x_1}}(x_1)\,\mathrm{d}x_1\right)\mathrm{d}P^{\widehat{x_1}}(\textbf{z}^{x_1})\\[3mm]
\displaystyle =\int_{\ell^{2,1}}\left(\int_{-R}^{R}D_{x_1}\left(f(\textbf{z})\right)\cdot \frac{e^{-\frac{x_1^2}{2a_1^2}}}{\sqrt{2\pi a_1^2}}\,\mathrm{d}x_1\right)\mathrm{d}P^{\widehat{x_1}}(\textbf{z}^{x_1})
=\int_{\ell^{2,1}}\left(\int_{-R}^{R} \frac{x_1}{a_1^2}\cdot f(\textbf{z}) \cdot \frac{e^{-\frac{x_1^2}{2a_1^2}}}{\sqrt{2\pi a_1^2}}\,\mathrm{d}x_1\right)\mathrm{d}P^{\widehat{x_1}}(\textbf{z}^{x_1})\\[3mm]
\displaystyle =\int_{\ell^{2,1}}\left(\int_{\mathbb{R}}  \frac{x_1}{a_1^2}\cdot f(\textbf{z}) \cdot \frac{e^{-\frac{x_1^2}{2a_1^2}}}{\sqrt{2\pi a_1^2}}\,\mathrm{d}x_1\right)\mathrm{d}P^{\widehat{x_1}}(\textbf{z}^{x_1})
=\int_{\ell^{2}}  \frac{x_1}{a_1^2}\cdot f(\textbf{z}) \,\mathrm{d}P (\textbf{z}),
\end{array}
$$
where the third equality follows from the classical Newton-Leibniz formula. The proof of Theorem \ref{Gauss-Green Theorem} is completed.
\end{proof}

Recall (\ref{eq2301161}) for the definition of $\overline{\partial}_{i} $ and $\delta_i$. As a consequence of Theorem \ref{Gauss-Green Theorem}, we have the following result:
\begin{corollary}\label{integration by Parts or deltai}
For any $f,g\in C^{1}_{0}\left( V\right)$, it holds that
$$
\int_{V} \overline{\partial}_{i}  f\cdot \bar{g}  \,\mathrm{d}P=-\int_{V} f\cdot\overline{ \delta_{i} g}  \;\mathrm{d}P,\quad \forall\; i\in\mathbb{N}.
$$
\end{corollary}
\begin{proof}
Since $f,g\in C^{1}_{0}\left( V\right)$, there exists $r,R\in(0,+\infty)$ such that $\displaystyle\bigcup_{\textbf{z}\in  \supp f\bigcup \supp g}B_r(\textbf{z})\subset V$ and  $\supp f\cup\supp g\subset B_R$. Then both $f$ and $g$ can be viewed as elements in $ C^{1}_{0} ( \ell^2 )$ by extending their values to $\ell^2\setminus (\supp f\cup \supp g)$ by 0. Thus for each $i\in\mathbb{N}$, it holds that
$$
\int_{V} \overline{\partial}_{i}  f\cdot \overline{g} \,\mathrm{d}P
=\int_{\supp f\bigcup \supp g} \overline{\partial}_{i}  f\cdot \overline{g} \,\mathrm{d}P
=\int_{\ell^2} \overline{\partial}_{i}  f\cdot \overline{g}  \,\mathrm{d}P
=\int_{\ell^2}\frac{1}{2}( D_{x_{i}}f +\sqrt{-1}D_{y_{i}}f)\cdot \overline{g}  \,\mathrm{d}P.
$$
Then, by Theorem \ref{Gauss-Green Theorem}, we have
$$
\begin{array}{ll}
\displaystyle \int_{V} \overline{\partial}_{i}  f\cdot\overline{g} \,\mathrm{d}P\\[5mm]
\displaystyle
= -\int_{\ell^2}f\cdot \overline{ \frac{1}{2}( D_{x_{i}}g -\sqrt{-1}D_{y_{i}}g)}  \,\mathrm{d}P
+\int_{\ell^2}f\cdot \overline{ \frac{x_i-\sqrt{-1}y_i}{2a_i^2} g}  \,\mathrm{d}P
= \int_{\ell^2}f\cdot \overline{ \left(-\partial_{i}g  +\frac{\overline{z_i}}{2a_i^2}\cdot g\right) }\,\mathrm{d}P\\[5mm]
\displaystyle
=-\int_{\ell^2}f\cdot \overline{\delta_i g}\,\mathrm{d}P=-\int_{\supp f\bigcup \supp g}f\cdot \overline{\delta_i g}\,\mathrm{d}P
=-\int_{V}f\cdot \overline{\delta_i g}\,\mathrm{d}P.
\end{array}
$$
This completes the proof of Corollary \ref{integration by Parts or deltai}.
\end{proof}

\begin{remark}
We refer to \cite{WYZ0} for a more general infinite-dimensional Gauss-Green type Theorem.
\end{remark}

Now, we fix a real-valued function $\varphi \in C^{2}_{F}\left( V\right)$. For any $\phi \in C^{1}(V)$ and $i\in\mathbb{N}$,  write
\begin{equation}\label{230419e11}
\sigma_{i} \phi\triangleq \delta_{i} \phi-\phi \cdot\partial_{i}\varphi.
\end{equation}
Clearly, for any $i,j\in\mathbb{N}$,
\begin{equation}\label{commutaor formula}
\left( \overline{\partial_i }\sigma_j-\sigma_j \overline{\partial_i }\right)h=-h\cdot (\overline{\partial_i } \partial_j \varphi) -\frac{ \overline{\partial_i}( \overline{z_j})}{2a_j^2}\cdot h,\quad\forall\;h\in C^{2}(V).
\end{equation}
We have the following simple result:
\begin{lemma}\label{integration by Parts}
For $f,g\in C^{1}_{0}\left( V\right)$,
$$
\int_{V} (\overline{\partial_{i} } f)\cdot \overline{g}\cdot e^{-\varphi }\,\mathrm{d}P=-\int_{V} f\cdot\overline{(\sigma_{i} g)}\cdot e^{-\varphi }\,\mathrm{d}P,\quad \forall\; i\in\mathbb{N}.
$$
\end{lemma}
\begin{proof}
By Corollary \ref{integration by Parts or deltai}, it follows that
$$
\begin{array}{ll}
\displaystyle
\int_{V} (\overline{\partial_{i} } f)\cdot \overline{g}\cdot e^{-\varphi }\,\mathrm{d}P
=\int_{V} \overline{\partial_{i} } (f\cdot e^{-\varphi })\cdot \overline{g}\,\mathrm{d}P
+\int_{V} (\overline{\partial_{i} }\varphi)\cdot f\cdot e^{-\varphi }\cdot \overline{g}\,\mathrm{d}P\\[3mm]
\displaystyle =-\int_{V}f\cdot e^{-\varphi }\cdot \overline{\delta_i g}\,\mathrm{d}P
+\int_{V} (\overline{\partial_{i} }\varphi)\cdot f\cdot e^{-\varphi }\cdot \overline{g}\,\mathrm{d}P
=-\int_{V} f\cdot\overline{(\sigma_{i} g)}\cdot e^{-\varphi }\,\mathrm{d}P,
\end{array}
$$
which completes the proof of Lemma \ref{integration by Parts}.
\end{proof}

\subsection{Pseudo-convex domains in $\ell^{2}$}\label{subsection000}

In this subsection, we shall introduce the concept of pseudo-convex domains in $\ell^{2}$ and present some of their elementary properties.

For any two subsets $E_1,E_2\subset \ell^2$, we define the distance $d(E_1,E_2)$ between $E_1$ and $E_2$ by
 $$
 d(E_1,E_2)\triangleq
 \left\{
 \begin{array}{ll}
 \infty, &\hbox{if }E_1=\emptyset\hbox{ or }E_2=\emptyset,\\
 \inf\big\{||\textbf{z}_1-\textbf{z}_2||:\;\textbf{z}_1\in E_1,\,\textbf{z}_2\in E_2\big\},\quad\;&\hbox{otherwise}.
 \end{array}\right.
  $$
For a nonempty subset $S$ of $V$, set
$$
d_V(S)\triangleq \min\bigg\{d(S, \partial V),\;\frac{1}{\sup\limits_{\textbf{z}\in S}||\textbf{z}||}\bigg\}.
$$
Here and henceforth, we use the convention that $\frac{1}{0}=\infty$ and $\frac{1}{\infty}=0$. Particularly, for $\textbf{z}\in V$, let
$$
d_V(\textbf{z})\triangleq \min\bigg\{d(\{\textbf{z}\}, \partial V),\frac{1}{||\textbf{z}||}\bigg\}.
$$
It is clear that $d_V(\textbf{z})\geqslant d_V(S)$ for each $\textbf{z}\in S$. The following simple result is a characterization of $S \stackrel{\circ}{\subset}V$.

\begin{lemma}\label{221018lem1}
Suppose that $S$ is a nonempty subset of $V$. Then, $d_V(S)>0$ if and only if $S \stackrel{\circ}{\subset}V$.
\end{lemma}
\begin{proof}
If $S \stackrel{\circ}{\subset}V$, then by Definition \ref{def of bounded contained}, there exist $r,R\in(0,+\infty)$ such that $\cup_{\textbf{z}\in S}B_r(\textbf{z})\subset V$ and $S\subset B_R$. Thus $d(S, \partial V)\geqslant r$, $\frac{1}{\sup\limits_{\textbf{z}\in S}||\textbf{z}||}\geqslant \frac{1}{R}$ and hence $d_V(S)\geqslant \min\left\{r,\frac{1}{R}\right\}>0.$

Conversely, if $d_V(S)>0$, then for any $r\in \left(0,d_V(S)\right)$, it holds that $d(S, \partial V)> r$ and $\frac{1}{\sup\limits_{\textbf{z}\in S}||\textbf{z}||}> r$. Therefore, we have  $\cup_{\textbf{z}\in S}B_r(\textbf{z})\subset V$ and $S\subset B_{\frac{1}{r}}$, and hence $S \stackrel{\circ}{\subset}V$. This completes the proof of Lemma \ref{221018lem1}.
\end{proof}

The following result is useful to construct many subsets which are uniformly contained in $V$ (Recall Definition \ref{def of bounded contained}).
\begin{lemma}\label{Exhaustion function lemma}
Suppose that $\alpha$ is a real-valued function on $V$. Then, $\{ \textbf{v}\in V:\;\alpha (\textbf{v})  \leqslant \tau\} \stackrel{\circ}{\subset} V$ for all $ \tau\in \mathbb{R}$ if and only if $\lim\limits_{d_V(\textbf{z})\rightarrow 0 } \alpha (\textbf{z})=+\infty.$
\end{lemma}

\begin{proof}
If $\lim\limits_{d_V(\textbf{z})\rightarrow 0 } \alpha (\textbf{z})=+\infty$, then for any given $\tau\in \mathbb{R}$,  there exists $\delta>0$ such that $\alpha \left( \textbf{z}\right)  >\tau+1$ holds for any $\textbf{z}\in V$ satisfying $d_V(\textbf{z})<\delta $. Thus
$$
\left\{ \textbf{v}\in V:\;\alpha \left( \textbf{v}\right)  \leqslant \tau\right\}
\subset \left\{ \textbf{v}\in V:\;d_V(\textbf{v})  \geqslant \delta \right\}
\subset  \{ \textbf{v}\in V:\;d(\{\textbf{v}\},\partial V)  \geqslant \delta \}\cap \bigg\{ \textbf{v}\in V:\; ||\textbf{v}|| \leqslant \frac{1}{\delta} \bigg\}
\stackrel{\circ}{\subset} V.
$$

Conversely, suppose that $\left\{ \textbf{v}\in V:\;\alpha \left( \textbf{v}\right)  \leqslant \tau\right\} \stackrel{\circ}{\subset} V$ for all $\tau\in \mathbb{R}$. Given $M>0$, let
$$
\delta \triangleq d_V (  \{ \textbf{v}\in V:\;\alpha \left( \textbf{v}\right)  \leqslant M \}).
$$
By Lemma \ref{221018lem1}, we have $\delta>0$. Then for any $\textbf{z}\in V$ with $d_V(\textbf{z})  <\delta$, we claim that
\begin{equation}\label{230317e1}
\textbf{z}\notin \{\textbf{v}\in V:\;\alpha \left(\textbf{v}\right)  \leqslant M \}.
 \end{equation}
Otherwise, if $\textbf{z}\in \{\textbf{v}\in V:\;\alpha \left(\textbf{v}\right)  \leqslant M \}$, then $d_V(\textbf{z})\geqslant d_V (  \{ \textbf{v}\in V:\;\alpha \left( \textbf{v}\right)  \leqslant M \})=\delta$, which is a contradiction. Therefore, (\ref{230317e1}) holds, and hence $\alpha \left( \textbf{z}\right)  >M$, which implies that $\lim\limits_{d_V(\textbf{z})\rightarrow 0} \alpha \left( \textbf{z}\right)  =+\infty$. The proof of Lemma \ref{Exhaustion function lemma} is completed.
\end{proof}

Based on Definition \ref{def of bounded contained} and motivated by the definition of exhaustion functions in finite dimensions, we introduce the following notion.

\begin{definition}
A real-valued function $\alpha$ on $V$ is called an exhaustion function on $V$, if
\begin{equation}\label{230317e2}
\{\textbf{z}\in V:\;\alpha (\textbf{z})  \leqslant \tau\} \stackrel{\circ}{\subset} V, \quad\forall\;\tau\in \mathbb{R}.
\end{equation}
\end{definition}

\begin{remark}
Note that in view of \cite[p. 16]{Ohsawa}, in the setting of finite dimensions the counterpart of (\ref{230317e2}) takes the form that $\{ \textbf{z}\in V:\;\alpha (\textbf{z})  \leqslant \tau\}$ is compact in $V$ for all $\tau\in \mathbb{R}$. As we shall see later, (\ref{230317e2}) is more convenient in infinite-dimensional spaces.
\end{remark}

In this work, we use the following notion of pseudo-convex domains in $\ell^2$:
\begin{definition}\label{pseudocovexdomain}
$V$  is called a pseudo-convex domain in $\ell^2$, if there is an exhaustion function $\eta \in C^{\infty}_{F}\left( V\right)$ such that
for every $n\in \mathbb{N}$, the following inequality holds on $V$,
\begin{equation}\label{positivity condition}
		\sum^{n}_{i=1} \sum^{n}_{j=1} ( \partial_i \overline{\partial_j} \eta  )\cdot  \zeta_{i}\cdot\overline{\zeta_{j}} \geqslant 0,\quad\forall\;(\zeta_{1},\cdots,\zeta_n)\in \mathbb{C}^n.
\end{equation}
The above $\eta$ is called a plurisubharmonic exhaustion function on $V$.
\end{definition}

For the function $\eta$ in Definition \ref{pseudocovexdomain} and each $\tau\in\mathbb{R}$, write
\begin{equation}\label{230117e1}
V_{\tau}\triangleq \left\{ \textbf{z}\in V:\;\eta \left( \textbf{z}\right)  \leqslant \tau\right\}, \qquad V^{o}_{\tau}\triangleq \big\{\textbf{z}\in V:\; B_{r}(\textbf{z})\subset V_\tau\hbox{ for some }r>0\big\}.
\end{equation}
Clearly, $V_{\tau}\stackrel{\circ}{\subset} V$  and
$V^{o}_{\tau}$ is the interior of $V_\tau$ for all $\tau\in \mathbb{R}$. On the other hand, since $V_{0}\stackrel{\circ}{\subset} V$ and $\eta \in C^{\infty}_{F}\left( V\right)$, it is clear that $\inf\limits_{V_{0}} \eta >-\infty$ and $\inf\limits_{V} \eta=\min\left(\inf\limits_{V_{0}}\eta ,\;\inf\limits_{V\setminus V_{0}}\eta \right) >-\infty$.

\begin{remark}
Definition \ref{pseudocovexdomain} is essentially motivated by \cite[Theorem 2.6.2, p. 44]{Hor90} and \cite[Theorem 2.6.7, p. 46]{Hor90}. There exist several other equivalent definitions of pseudo-convex domains in infinite-dimensional spaces, e.g., \cite[p. 361]{Lem03}, \cite[Definition 37.3, p. 274]{Mujica} and \cite[Definition 2.1.3, p. 41]{Noverraz}. Our Definition \ref{pseudocovexdomain} is equivalent to these definitions but the proof of this fact is quite complicated, and therefore we shall present it in a forthcoming paper \cite{WZ}.

 Clearly, $\eta$ appeared in Definition \ref{pseudocovexdomain} is not unique. Since only second order partial derivatives are used in (\ref{positivity condition}), it seems more natural to require $\eta\in C^2_{F}\left( V\right)$ in Definition \ref{pseudocovexdomain}. Nevertheless, as we shall show in \cite{WZ} that, if one can find an exhaustion function $\eta\in C^2_{F}\left( V\right)$ satisfying (\ref{positivity condition}), then there exists another exhaustion function $\tilde\eta\in C^{\infty}_{F}\left( V\right)$ enjoying the same property. Furthermore, the $C^{\infty}_{F}\left( V\right)$-smoothness for plurisubharmonic exhaustion functions on $V$ will paly a key role in our main approximation results in the rest of this paper.
\end{remark}

\begin{remark}\label{230422r1}
If necessary, we may replace the plurisubharmonic exhaustion function $\eta$ in Definition \ref{pseudocovexdomain} by $\eta + \left| \left|\cdot\right|  \right|^{2}_{\ell^{2} }  -\inf\limits_{V_{0}} \eta  $. Thus, we may assume that $\eta\geqslant 0$ and for any $n\in\mathbb{N}$, instead of (\ref{positivity condition}), the following inequality holds on $V$:
\begin{equation}\label{regular condition for eta}
\sum_{1\leqslant i,j\leqslant n}\left( \partial_{i} \overline{\partial_{j} } \eta\right)\cdot \zeta_{i}\cdot\overline{\zeta_{j}} \geqslant \sum^{n}_{i=1} \left| \zeta_{i}\right|^{2}  ,\quad\forall\;(\zeta_{1},\cdots,\zeta_n)\in \mathbb{C}^n .
\end{equation}
\end{remark}

One can find many pseudo-convex domains in $\ell^2$, as shown by the following result.

\begin{proposition}\label{pseudo-convex domains in l2}
{\rm (1)} Suppose that $V$ is a pseudo-convex domain in $\ell^2$. For any $\textbf{a}\in \ell^{2}$ and $\textbf{c}=(c_{i})_{i=1}^{\infty } \in \mathbb{C}^{\infty}$ with $ 0<\inf\limits_{i\geqslant 1} \left| c_{i}\right|  \leqslant \sup\limits_{i\geqslant 1} \left| c_{i}\right|  <\infty  $, let
    $$V+\textbf{a } \triangleq \left\{ \textbf{z}+\textbf{a}:\;\textbf{z}\in V\right\}, \quad\textbf{c}V \triangleq \left\{ \left( c_{i}z_{i}\right)^{\infty }_{i=1} :\;\textbf{z}=\left(z_{i}\right)^{\infty }_{i=1} \in V\right\}.
    $$
Then both $V+\textbf{a }$ and $\textbf{c}V$ are pseudo-convex domains in $\ell^2$;

\medskip

{\rm (2)} For any $\textbf{a}\in \ell^{2}$ and $r\in(0,+\infty)$, $B_{r}(\textbf{a})$ is a pseudo-convex domain in $\ell^2$;

\medskip

{\rm (3)} If $S\subset \mathbb{C}^{m}$ (for some $m\in\mathbb{N}$) is a pseudo-convex domain in $\mathbb{C}^{m}$ (See  \cite[Definition 2.6.8, p. 47]{Hor90}), then $V\triangleq \{(\textbf{z}_m,\textbf{z}^m):\;\textbf{z}_m\in S,\,\textbf{z}^m\in \ell^2\}$ is a pseudo-convex domain in $\ell^{2}$;

\medskip

{\rm (4)} The Hilbert polydisc $\mathbb{D}_2^{\infty}\triangleq \big\{(z_i)_{i=1}^{\infty}\in\ell^2:\;|z_i|<1\text{ for all }i\in\mathbb{N}  \big\}  $ is a pseudo-convex domain in $\ell^{2}$.
\end{proposition}

\begin{proof}
(1) Suppose that $\eta$ is a plurisubharmonic exhaustion function on $V$. Let $\eta_{1} \left( \textbf{z}\right)\triangleq\eta \left( \textbf{z}-\textbf{a}\right)$ for $\textbf{z}\in V+\textbf{a}$ and $\eta_{2} \left( \textbf{z}\right)\triangleq\eta \left( \frac{z_{1}}{c_{1}} ,\frac{z_{2}}{c_{2}} ,\cdots \right)$ for $\textbf{z}=\left(z_{i}\right)^{\infty }_{i=1} \in \textbf{c}V$.
	Clearly, $\eta_{1}  \in C^{\infty}_{F}\left( V+\textbf{a}\right)$ and $\eta_{2}  \in C^{\infty}_{F}\left( \textbf{c}V\right)$. Then $\eta_{1}$ and $\eta_{2}$ are plurisubharmonic exhaustion functions on $V+\textbf{a }$ and $\textbf{c}V$, respectively.

\medskip

(2) Let $\eta \left( \textbf{z}\right)\triangleq-\ln \left( 1-\left| \left| \textbf{z}\right|  \right|^{2}  \right)$ for $\textbf{z}\in B_1.$ Then $\eta$ is a plurisubharmonic exhaustion function on $B_1$. By the conclusion (1), we see that $B_{r}(\textbf{a})$ is a pseudo-convex domain in $\ell^2$.

\medskip

(3)	By \cite[Theorem 2.6.11, p. 48]{Hor90}), there exists a strictly plurisubharmonic exhaustion function $\eta \in C^{\infty }\left(S \right)$ of $S$. Let $\tilde{\eta}(\textbf{z})\triangleq\eta(\textbf{z}_m)+\left| \left| \textbf{z}\right|  \right|^{2}$ for $\textbf{z}=(\textbf{z}_m,\textbf{z}^m)\in V$. It is easy to check that $\tilde{\eta}$ is a plurisubharmonic exhaustion function on $V$.

\medskip

(4) Let $\eta  ( \textbf{z})\triangleq\prod^{\infty }_{i=1} \frac{1}{1-|z_{i}|^{2}}$ for $\textbf{z}=(z_i)_{i=1}^{\infty}\in\mathbb{D}_2^{\infty}$. Then $\eta$ is a plurisubharmonic exhaustion function on $\mathbb{D}_2^{\infty}$. This completes the proof of Proposition \ref{pseudo-convex domains in l2}.
\end{proof}

We shall use the following assumption frequently.
\begin{condition}\label{230424c1}
$V$ is a pseudo-convex domain in $\ell^{2}$ and $\eta$ is a plurisubharmonic exhaustion function on $V$.
\end{condition}

The following result is about some elementary properties of pseudo-convex domains in $\ell^{2}$ and their plurisubharmonic exhaustion functions (Recall (\ref{230117e1}) for
$V_{\tau}$ and $V^{o}_{\tau}$).

\begin{proposition}\label{properties on pseudo-convex domain}
Let Condition \ref{230424c1} hold. Then,
\begin{itemize}
\item[$(1)$] For each $\tau\in\mathbb{R}$, there exists $C(\tau)\in(0,+\infty)$ such that
\begin{equation}\label{230319e1}
\left| \eta \left(\textbf{z}\right)  -\eta \left( \tilde{\textbf{z}} \right)  \right|  \leqslant C(\tau)\left| \left|\, \textbf{z}-\tilde{\textbf{z}}\,\right|  \right|,\quad\forall\; \textbf{z},\tilde{\textbf{z}}\in V_{\tau};
\end{equation}
\item[$(2)$] For any $\tau_1,\tau_2\in\mathbb{R}$ with $\tau_1<\tau_2$, it holds that $V_{\tau_1}\stackrel{\circ}{\subset} V^{o}_{\tau_2}$;
\item[$(3)$] For any $S\stackrel{\circ}{\subset} V$, there exists $\tau\in\mathbb{R}$ such that $S\stackrel{\circ}{\subset} V^{o}_{\tau}$.
\end{itemize}
\end{proposition}
\begin{proof}
(1) Choose a real-valued function $\psi \in C^{\infty }\left( \mathbb{R} \right)$ such that $0\leqslant \psi \leqslant 1,\psi \left( x\right)  =1$ for all $x\leqslant \tau$ and $\psi \left( x\right)  =0$ for all $x>\tau+1$. Let $\Psi \left( \textbf{z}\right)\triangleq \eta \left( \textbf{z}\right) \psi \left( \eta \left( \textbf{z}\right)  \right)$ for $\textbf{z}\in V$ and $\Psi \left( \textbf{z}\right)\triangleq 0$ for $\textbf{z}\in \ell^2\setminus V$. Then  $\Psi\in C^{\infty}_{0,F}( \ell^2).$ For each $n\in\mathbb{N}$, set
$$
\eta_n(\textbf{z}_n)\triangleq \int \Psi(\textbf{z}_n,\textbf{z}^n)\,\mathrm{d}P_n(\textbf{z}^n),\quad \textbf{z}_n=(x_{i} +\sqrt{-1}y_{i})_{i=1}^{n}\in \mathbb{C}^n,
$$
where $\textbf{z}^n=(x_{i} +\sqrt{-1}y_{i})_{i=n+1}^\infty\in \ell^2$.
 By H\"older's inequality, it follows that
\begin{equation}\label{230516e1}
\begin{array}{ll}
\displaystyle
\sum^{n}_{j=1}\big(| D_{x_j} \eta_{n}|^{2}+| D_{y_j} \eta_{n}|^{2}\big)
= \sum^{n}_{j=1}\bigg(\bigg|\int D_{x_j} \Psi\,\mathrm{d}P_{n}\bigg|^{2}+\bigg|\int D_{y_j} \Psi\,\mathrm{d}P_{n}\bigg|^{2}\bigg)\\[3mm]
\displaystyle \leqslant \int \sum^{n}_{j=1}\bigg( |D_{x_j} \Psi|^{2}+|D_{y_j} \Psi|^{2}\bigg)\,\mathrm{d}P_{n}\\[3mm]
\displaystyle \leqslant \sup_{V} \sum^{\infty }_{j=1} \big(| D_{x_j}\Psi|^{2}+| D_{y_j} \Psi|^{2}\big)<\infty,
\end{array}
\end{equation}
where the last inequality holds because $\Psi\in C^{\infty}_{0,F}( \ell^2)$.

Write $C(\tau)\triangleq \max\left\{1,\sqrt{\sup\limits_{V} \sum\limits^{\infty }_{j=1} \big(| D_{x_j}\Psi|^{2}+| D_{y_j} \Psi|^{2}\big)}\right\}$. Then by the classical mean value theorem, using (\ref{230516e1}), we have
$$
\left| \eta_{n}\left( \textbf{z}_n\right)  -\eta_{n}\left(\tilde{\textbf{z}}_n\right)  \right|  \leqslant C(\tau)\left| \left| \textbf{z}_n -\tilde{\textbf{z}}_n \right|  \right|_{\mathbb{C}^{n} }  ,\quad  \forall\;\textbf{z}_n ,\tilde{\textbf{z}}_n \in \mathbb{C}^{n}.
$$
Viewing $\eta_n$ as a cylinder function on $\ell^2$,  we  have
\begin{equation}\label{230516e2}
| \eta_{n}( \textbf{z})  -\eta_{n}(\tilde{\textbf{z}})  |=| \eta_{n}( \textbf{z}_n)  -\eta_{n}(\tilde{\textbf{z}}_n)  | \leqslant C(\tau)\left| \left| \textbf{z}_n -\tilde{\textbf{z}}_n \right|  \right|_{\mathbb{C}^{n} }   \leqslant C(\tau)\left| \left| \textbf{z} -\tilde{\textbf{z}} \right|  \right|,\quad \forall\;  \textbf{z},\tilde{\textbf{z}}\in\ell^2,
\end{equation}
where $\tilde{\textbf{z}}^n,\textbf{z}^n\in \ell^2,\,\, \tilde{\textbf{z}}_n,\textbf{z}_n\in \mathbb{C}^n$ and $\textbf{z}=(\textbf{z}_n,\textbf{z}^n)$, $\tilde{\textbf{z}}=(\tilde{\textbf{z}}_n,\tilde{\textbf{z}}^n)$.
Meanwhile, by the conclusion (2) of Proposition \ref{Reduce diemension},  there exists a subsequence $\{\eta_{n_{k}}\}_{k=1}^{\infty}$ and a Borel subset $E_0$ of $\ell^2$ with $P(E_0)=0$ such that $\lim\limits_{k\rightarrow \infty } \eta_{n_{k}}(\textbf{z})=\Psi(\textbf{z})$ for any $\textbf{z}\in \ell^2\setminus E_0$. Thus, letting $k\to\infty$, we obtain from (\ref{230516e2}) that
$$
| \Psi( \textbf{z})  -\Psi(\tilde{\textbf{z}})  | \leqslant C(\tau)\left| \left| \textbf{z} -\tilde{\textbf{z}} \right|  \right|,\quad\forall\;\textbf{z}, \tilde{\textbf{z}}\in\ell^2\setminus E_0.
$$
Combining the continuity of $\Psi$ with the above inequality and by Corollary \ref{property of strictly positive measure}, we see that
$$
| \Psi( \textbf{z})  -\Psi(\tilde{\textbf{z}})  | \leqslant C(\tau)\left| \left| \textbf{z} -\tilde{\textbf{z}} \right|  \right|,\quad \forall\;\textbf{z},\tilde{\textbf{z}}\in\ell^2,
$$
which leads to (\ref{230319e1}).

\medskip

(2) Obviously, $V_{\tau_1}\subset V^{o}_{\tau_2}$. Note that for $\textbf{z}\in V_{\tau_1}$ and $\tilde{\textbf{z}}\in \partial V^{o}_{\tau_2}$, we have $\eta(\textbf{z}) \leqslant \tau_1$ and $\eta(\tilde{\textbf{z}})=\tau_2$. By the above conclusion (1), we have
$$
0<\tau_2-\tau_1<\left| \eta \left( \textbf{z} \right)  -\eta \left( \tilde{\textbf{z}}\right)  \right|  \leqslant C\left( \tau_2\right)  \left| \left|\textbf{z}-\tilde{\textbf{z}}\right|  \right|,
$$
which implies that $d(V_{\tau_1},\partial V^{o}_{\tau_2})\geqslant \frac{\tau_2-\tau_1}{C(\tau_2)}>0$ and $V_{\tau_1}$ is a bounded subset of $\ell^2$. Thus $V_{\tau_1}\stackrel{\circ}{\subset}V^{o}_{\tau_2}$.

\medskip

(3) Firstly, we use the contradiction argument to prove that there exists $\tau_0\in\mathbb{R}$  such that $S\subset V_{\tau_0}$. Otherwise, for each $n\in\mathbb{N}$, there would exist $\textbf{z}_n\in S\backslash V_n$ and hence, by the definition of $V_\tau$ in (\ref{230117e1}), one has $\eta(\textbf{z}_n)>n$, which contradicts the fact that $\sup\limits_{\textbf{z}\in S}|\eta(\textbf{z})|<\infty$ (which follows from the assumption $\eta \in C^{\infty}_{F}\left( V\right)$ and the definition of $C^{\infty}_{F}\left( V\right)$). Then choosing $\tau>\tau_0$, by the above conclusion (2), we have $S\subset V_{\tau_0}\stackrel{\circ}{\subset}V^{o}_{\tau}$. The proof of Proposition \ref{properties on pseudo-convex domain} is thus completed.
\end{proof}

\section{$\overline{\partial}$ operators on pseudo-convex domains}\label{secx3}

In the sequel, unless otherwise stated, we fix $s,t\in \mathbb{N}_0$ with $s+t\geqslant 1$, a $\sigma$-finite Borel measure $\mu$ on $V$, a Borel subset $E$ of $V$ and a real-valued Borel-measurable function $w$ on $E$.

Write
$$
(\ell^2)^{s+t}=\underbrace{\ell^2\times\ell^2\times\cdots\times\ell^2}_{s+t\hbox{ \tiny times}}.
$$
For any strictly increasing multi-indices $I=(i_1,\cdots,i_s)$ and $J=(j_1,\cdots,j_t)$, i.e., $i_1,\cdots,i_s, j_1,\cdots,$ $j_t\in \mathbb{N}$, $i_1<\cdots<i_s$ and $j_1<\cdots<j_t$, we fix a number
\begin{equation}\label{defnition of general st froms}
c_{I,J}>0,
\end{equation}
and define $I\cup J\triangleq\{i_1,\cdots,i_s\}\cup \{j_1,\cdots,j_t\}$. For $j\in \mathbb{N}$, we define $J\cup\{j\}=\{j_1,\cdots,j_t\}\cup\{j\}$.
As in \cite[p. 530]{YZ}, we define a complex-valued function $\mathrm{d}z^I\wedge \mathrm{d}\overline{z}^J$ on $(\ell^2)^{s+t}$ by
$$
(\mathrm{d}z^I\wedge \mathrm{d}\overline{z}^J)(\textbf{z}^1,\cdots,\textbf{z}^{s+t})\triangleq \frac{1}{\sqrt{(s+t)!}}\sum_{\sigma\in S_{s+t}}(-1)^{s(\sigma)}\cdot\prod_{k=1}^{s}z_{i_k}^{\sigma_k}\cdot\prod_{l=1}^{t}\overline{z_{j_l}^{\sigma_{s+l}}},
$$
where $\textbf{z}^l=(z_j^l)_{j=1}^\infty\in\ell^2,\,\,l=1,\cdots, s+t$, $S_{s+t}$ is the permutation group of $\{1,\cdots,s+t\}$, $s(\sigma)$ is the sign of $\sigma=(\sigma_1,\cdots,\sigma_s,\sigma_{s+1},\cdots,\sigma_{s+t})$, and we have the usual convention that $0!=1$. We call the following (formal) summation an $(s,t)$-form on $E$:
 \begin{equation}\label{230320e1}
 \sum^{\prime }_{\left| I\right|  =s} \sum^{\prime }_{\left| J\right|  =t} f_{I,J}\,\mathrm{d}z_{I}\wedge \,\mathrm{d}\overline{z_{J}},
 \end{equation}
where the sum $\sum\limits^{\prime }_{\left| I\right|  =s} \sum\limits_{\left| J\right|  =t}^{\prime } $ is taken only over all strictly increasing multi-indices $I$ and $J$ with $\left| I\right|  =s$ and $\left| J\right|  =t$ (for which $\left| I\right|$ and $\left| J\right|$ stand for respectively the cardinalities of sets $\{i_1,\cdots,i_s\}$ and $\{j_1,\cdots,j_t\}$), and $f_{I,J}$ is a function on $E$. Clearly, each $f_{I,J}\,\mathrm{d}z_{I}\wedge \,\mathrm{d}\overline{z_{J}}$ can be viewed as a function on $E\times (\ell^2)^{s+t}$. Nevertheless, in the present setting of infinite dimensions, (\ref{230320e1}) is an infinite series for which the convergence is usually not guaranteed, and therefore it is a formal summation (unless further conditions are imposed).

Next, we introduce some working spaces which will play key roles in the sequel. Denote by $L^{2}\left(V,loc\right)$ the set of all Borel-measurable functions $f$ on $V$ for which
$$
\int_{S} \left| f\right|^{2} \,\mathrm{d}P<\infty,\quad\forall\; S\in \mathcal{B}(V)\hbox{ with }S\stackrel{\circ}{\subset} V;
$$
by $L^{2}_{\left( s,t\right)  }\left( V,loc\right)$ the set of all $(s,t)$-forms $\sum^{\prime }_{\left| I\right|  =s} \sum^{\prime }_{\left| J\right|  =t} f_{I,J}\,\mathrm{d}z_{I}\wedge \,\mathrm{d}\overline{z_{J}}$ on $V$ for which
$$
f_{I,J}\in L^{2}\left(V,loc\right)\hbox{ for each } I,J,\hbox{ and }\sum^{\prime }_{\left| I\right|  =s} \sum^{\prime }_{\left| J\right|  =t}c^{ }_{I,J} \int_{S} \left| f_{I,J}\right|^{2}\,\mathrm{d}P<\infty ,\;\forall\; S\in \mathcal{B}(V)\hbox{ with } S\stackrel{\circ}{\subset} V;
$$
by $L^{2}_{\left( s,t\right)  }\left( E,\mu \right)$ the set of all $(s,t)$-forms $\sum^{\prime }_{\left| I\right|  =s} \sum^{\prime }_{\left| J\right|  =t} f_{I,J}\,\mathrm{d}z_{I}\wedge \,\mathrm{d}\overline{z_{J}}$ on $E$ for which
$$
f_{I,J}\in L^2(E,\mu)\hbox{ for each } I,J,\hbox{ and }\sum^{\prime }_{\left| I\right|  =s} \sum^{\prime }_{\left| J\right|  =t} c^{ }_{I,J}\int_{E} \left| f_{I,J}\right|^{2} \,\mathrm{d}\mu <\infty;
$$
by $L^{2}_{\left( s,t\right)}\left(E,w\right)$ the set  of all $(s,t)$-forms $\sum^{\prime }_{\left| I\right|  =s} \sum^{\prime }_{\left| J\right|  =t} f_{I,J}\,\mathrm{d}z_{I}\wedge \,\mathrm{d}\overline{z_{J}}$ on $E$ for which
$$
f_{I,J}\in L^2(E,e^{-w}P)\hbox{ for each } I,J,\hbox{ and }\sum^{\prime }_{\left| I\right|  =s} \sum^{\prime }_{\left| J\right|  =t} c^{ }_{I,J}\int_{E} \left| f_{I,J}\right|^{2}  e^{-w}dP<\infty .
$$

\begin{remark}
Clearly, both $L^{2}_{(s,t)}\left(E,w\right)$ and $L^{2}_{(s,t)}\left(E,\mu\right)$, endowed respectively with the following norms
$$
\left|\left|f\right|\right|_{L^{2}_{\left( s,t\right) \left( E,\mu \right)}}=\sqrt{\sum^{\prime }_{\left| I\right|  =s} \sum^{\prime }_{\left| J\right|  =t} c^{ }_{I,J}\int_{E} \left| f_{I,J}\right|^{2} \,\mathrm{d}\mu},\quad\forall\;f=\sum^{\prime }_{\left| I\right|  =s} \sum^{\prime }_{\left| J\right|  =t} f_{I,J}\,\mathrm{d}z_{I}\wedge \,\mathrm{d}\overline{z_{J}}\in L^{2}_{\left( s,t\right) \left( E,\mu \right)}
$$
and
$$
\left|\left|f\right|\right|_{L^{2}_{\left( s,t\right) \left( E,w \right)}}=\sqrt{\sum^{\prime }_{\left| I\right|  =s} \sum^{\prime }_{\left| J\right|  =t} c^{ }_{I,J}\int_{E} \left| f_{I,J}\right|^{2} \,e^{-w}dP},\quad\forall\;f=\sum^{\prime }_{\left| I\right|  =s} \sum^{\prime }_{\left| J\right|  =t} f_{I,J}\,\mathrm{d}z_{I}\wedge \,\mathrm{d}\overline{z_{J}}\in L^{2}_{\left( s,t\right) \left( E,w \right)},
$$
are Hilbert spaces. Moreover, $L^{2}_{\left( s,t\right)  }\left( V,loc\right)$ can be viewed as the set of all locally square-integrable $(s,t)$-forms on $V$ (In particular, $L^{2}\left( V,loc\right)$ is the set of all locally square-integrable functions on $V$).
\end{remark}

\begin{lemma}\label{density1}
Assume that $\textbf{z}_0\in\ell^2$, $r\in(0,+\infty)$ and $\mu $ is a Borel measure on $B_r(\textbf{z}_0)$ such that $\mu \left( E\right)<\infty$ for all Borel subset $E\stackrel{\circ}{\subset} B_r(\textbf{z}_0)$. Then $C^{\infty }_{0,F^{\infty}}\left(B_r(\textbf{z}_0)\right)$ is dense in $L^{p}\left(B_r(\textbf{z}_0),\mu \right)$ for each $ p\in [1,\infty) $.
\end{lemma}

\begin{proof}
Let $\eta \left( \textbf{z}\right)\triangleq-\ln \left( 1-\left| \left| \frac{\textbf{z}-\textbf{z}_0}{r}\right|  \right|^{2}  \right),\,\textbf{z}\in B_r(\textbf{z}_0)$ and $(B_r(\textbf{z}_0))_n^o\triangleq \{\textbf{z}'\in B_r(\textbf{z}_0):\;\eta(\textbf{z}')<n\},\,n\in\mathbb{N}$. Choose $\varphi_n\in C^{\infty}(\mathbb{R})$ such that $0\leqslant \varphi_n\leqslant 1$, $\varphi_n(x)=1$ for all $x\leqslant n-1$ and $\varphi_n(x)=0$ for all $x\geqslant n$. Let $\Phi_n(\cdot)\triangleq \varphi_n(\eta(\cdot))$ and $\mu_n(E)\triangleq\mu(E)$,  $\forall\;E\in \mathscr{B}((B_r(\textbf{z}_0))_n^o)$. Then, $\Phi_n(\cdot)\in C^{\infty }_{0,F^{\infty}}\left(B_r(\textbf{z}_0)\right)$. For each $f\in L^{p}\left( B_r(\textbf{z}_0),\mu \right)$, one can view $\Phi_n f$ as an element of $L^p((B_r(\textbf{z}_0))_n^o,\mu_n)$. Thus, according to Lebesgue's Dominated Convergence Theorem, we have
\begin{equation}\label{230123e1}
\lim_{n\to\infty}\int_{ B_r(\textbf{z}_0)}|\Phi_n f-f|^p\,\mathrm{d}\mu=0.
\end{equation}
For each $n\in\mathbb{N}$, we choose $\{\psi_m\}_{m=1}^{\infty}\subset C^{\infty}(\mathbb{R})$ such that $0\leqslant \psi_m\leqslant 1$, $\psi_m(x)=1$ for all $x<n-\frac{2}{m}$ and $\psi_m(x)=0$ for all $x\geqslant n-\frac{1}{m},\,m\in\mathbb{N}$. Then $\lim\limits_{m\to\infty}\psi_m(x)=1$ for all $x<n$ and $\lim\limits_{m\to\infty}\psi_m(x)=0$ for all $x\geqslant n$. Let $\Psi_m(\cdot)\triangleq \psi_m(\eta(\cdot))$. Then $\Psi_m(\cdot)\in  C^{\infty }_{0,F^{\infty}}\left( (B_r(\textbf{z}_0))_n^o\right)$.

Write
$$
\begin{array}{ll}
\displaystyle
\mathbb{M}_n \triangleq  C^{\infty }_{0,F^{\infty}}\left( (B_r(\textbf{z}_0))_n^o\right),\\[3mm]
\displaystyle
\mathbb{N}_n\triangleq
\left\{f\in L^p((B_r(\textbf{z}_0))_n^o,\mu_n):\;\exists\; \{g_k\}_{k=1}^{\infty}\subset \mathbb{M}_n\text{ such that }\lim_{k\to\infty}\int_{(B_r(\textbf{z}_0))_n^o}|g_k-f|^p\,\mathrm{d}\mu_n=0\right\},\\[3mm]
\displaystyle
 \mathbb{H}_n\triangleq \left\{f\in \mathbb{N}_n:\; f\text{ is a bounded Borel-measurable function on }(B_r(\textbf{z}_0))_n^o \right\}.
\end{array}
$$
Then $\mathbb{H}_n$ and $\mathbb{M}_n$ are two families of bounded complex-valued functions on $(B_r(\textbf{z}_0))_n^o$ such that:

\begin{enumerate}

\item[(1)] $\mathbb{H}_n$ is a complex linear subspace which is closed under the complex conjugation and the bounded convergence;

\item[(2)] $\mathbb{M}_n\subset \mathbb{H}_n$ and $\mathbb{M}_n$ is a multiplicative system;

\item[(3)] $1\in \mathbb{H}_n$ (because $\{\Psi_m\}_{m=1}^{\infty}\subset \mathbb{M}_n$ and $\lim\limits_{m\to\infty}\Psi_m(\textbf{z})=1$ for all $\textbf{z}\in (B_r(\textbf{z}_0))_n^o$).

\end{enumerate}

\noindent By Theorem \ref{Complex Multiplicative System Theorem}, $\mathbb{H}_n$ contains all bounded complex valued $\sigma(\mathbb{M}_n)$-measurable functions.

We claim that
\begin{equation}\label{230321e1}
\sigma(\mathbb{M}_n)=\mathscr{B}((B_r(\textbf{z}_0))_n^o).
\end{equation}
Indeed, it is easy to see that $\sigma(\mathbb{M}_n)\subset\mathscr{B}((B_r(\textbf{z}_0))_n^o)$. Conversely, for each $\textbf{z}_0'\in (B_r(\textbf{z}_0))_n^o$ and $r_0>0$ so that the open ball $B_{r_0}(\textbf{z}_0')\subset (B_r(\textbf{z}_0))_n^o$, we choose an integer $k_0>\frac{1}{r_0^2}$. Then for each integer $k\geqslant k_0$, choose $\varphi_k \in C_c^{\infty}(\mathbb{R})$ such that $0\leqslant \varphi_k\leqslant 1$,  $\varphi_k(x)=1$ for all $x\in \mathbb{R}$ with $0\leqslant |x|\leqslant r_0^2-\frac{1}{k}$, and  $\varphi_k(x)=0$ for all $x\in \mathbb{R}$ satisfying $|x|\geqslant r_0^2-\frac{1}{2k}$. Let $\Phi_k(\textbf{z})\triangleq \varphi_k(||\textbf{z}-\textbf{z}_0'||^2),\,\textbf{z}\in B_r(\textbf{z}_0),\,k\geqslant k_0$. It is easy to show that $\{\Phi_k\}_{k=k_0}^{\infty}\subset \mathbb{M}_n $, and hence $B_{r_0}(\textbf{z}_0')=\bigcup\limits_{k=k_0}^{\infty}\left\{\textbf{z}\in B_r(\textbf{z}_0):\Phi_k(\textbf{z})=1\right\}\in \sigma(\mathbb{M}_n)$. Since $(B_r(\textbf{z}_0))_n^o$ is a Lindel\"{o}f space, each non-empty open subset of $(B_r(\textbf{z}_0))_n^o$ is a countable union of open balls contained in $(B_r(\textbf{z}_0))_n^o$. Therefore, $\sigma(\mathbb{M}_n)$ contains all open subset of $(B_r(\textbf{z}_0))_n^o$, and hence (\ref{230321e1}) is true.

By (\ref{230321e1}), we conclude that $\mathbb{H}_n$ contains all bounded complex valued Borel-measurable functions. With the fact that the set of all bounded complex valued Borel-measurable functions is dense in $L^p((B_r(\textbf{z}_0))_n^o,\mu_n)$, this yields that
$L^p((B_r(\textbf{z}_0))_n^o,\mu_n)\subset \mathbb{N}_n$, and hence $L^p((B_r(\textbf{z}_0))_n^o,\mu_n)= \mathbb{N}_n$, which indicates that $ C^{\infty }_{0,F^{\infty}}\left( (B_r(\textbf{z}_0))_n^o\right)$ is dense in $L^p((B_r(\textbf{z}_0))_n^o,\mu_n)$. By (\ref{230123e1}) and noting that $\Phi_n f\in L^p((B_r(\textbf{z}_0))_n^o,\mu_n)$ for each $n$, we conclude that $C^{\infty }_{0,F^{\infty}}\left( B_r(\textbf{z}_0) \right)$ is dense in $L^{p}\left( B_r(\textbf{z}_0),\mu \right)$. The proof of Lemma \ref{density1} is completed.
\end{proof}

As a consequence of Lemma \ref{density1}, we have the following result.
\begin{corollary}\label{Test functions}
Suppose that $ p\in(1,+\infty)$, $\mu$ is a Borel measure on $V$ such that $\mu \left( S\right)<\infty$ for all Borel subset $S \stackrel{\circ}{\subset} V$, and $f$ is a Borel-measurable function on $V$ so that for each $\textbf{z}_0\in V$, there exists $r\in(0,+\infty)$ for which $B_{r}(\textbf{z}_0)\stackrel{\circ}{\subset} V$, $\chi_{B_{r}(\textbf{z}_0)}\cdot f\in L^{p}\left(V,\mu \right)$ and
\begin{equation}\label{jichuchou}
		\int_{V} f\bar\phi d\mu =0,\quad\forall\;\phi\in C^{\infty }_{0,F^{\infty }}\left( B_{r}(\textbf{z}_0)\right).
\end{equation}
Then, $f=0$ almost everywhere with respect to $\mu$.
\end{corollary}
\begin{proof}
By Lemma \ref{density1}, $C^{\infty }_{0,F^{\infty }}\left( B_{r}(\textbf{z}_0)\right)$ is dense in $L^q(B_{r}(\textbf{z}_0),\mu\mid_{B_{r}(\textbf{z}_0)})$, where $q\triangleq\frac{1}{1-\frac{1}{p}}$ and $\mu\mid_{B_{r}(\textbf{z}_0)}(E)\triangleq \mu(E)$ for all $E\in \mathscr{B}\left(B_{r}(\textbf{z}_0)\right)$. Since $\overline{\text{sgn}(f)}\cdot |f|^{p-1}\in L^q(B_{r}(\textbf{z}_0),\mu\mid_{B_{r}(\textbf{z}_0)})$ where sgn$(z)\triangleq \frac{z}{|z|}$ if $z\in\mathbb{C}\setminus\{0\}$ and  sgn$(0)\triangleq 0$, the above density result, along with (\ref{jichuchou}), implies that
$$
\int_{B_{r}(\textbf{z}_0)}|f|^p\,\mathrm{d}\mu=0.
$$
Since $V$ is a Lindel\"{o}f space, we may cover $V$ by countably many balls as above. Thus, we can use the Monotone Convergence Theorem to conclude that
$$
\int_{V}|f|^p\,\mathrm{d}\mu=0,
$$
which implies that $f=0$ almost everywhere with respect to $\mu$. Hence we completed the proof of Corollary \ref{Test functions}.
\end{proof}

As a further consequence of Lemma \ref{density1}, we have the following density result for the space of square-integrable functions on $V$.
\begin{corollary}\label{density2}
Assume that $\mu $ is a Borel measure on $V$ such that $\mu \left( S\right)<\infty$ for all Borel subset $S \stackrel{\circ}{\subset} V$. Then $C^{\infty }_{0,F^{\infty}}\left(V  \right)$ is dense in $L^{2}\left( V,\mu \right)$.
\end{corollary}
\begin{proof}
Note that $L^{2}\left( V,\mu \right)$ is a Hilbert space and obviously $C^{\infty }_{0,F^{\infty }}\left( V\right) \subset L^{2}\left( V,\mu \right)$. Write
 $$
 \left( C^{\infty }_{0,F^{\infty }}\left( V\right)  \right)^{\bot }=\left\{f\in L^{2}\left( V,\mu \right):\;\int_Vf\bar gd\mu=0,\quad\forall\;g\in  C^{\infty }_{0,F^{\infty }}\left( V\right)\right\}.
 $$
For any $f\in \left( C^{\infty }_{0,F^{\infty }}\left( V\right)  \right)^{\bot }$, we have $\int_{V} f\bar gd\mu=0$ for all $g\in C^{\infty }_{0,F^{\infty }}\left( V\right)$. Then, by Corollary \ref{Test functions},  $f=0\in L^2\left( V,\mu \right)$, which implies that $C^{\infty }_{0,F^{\infty}}\left(V  \right)$ is dense in $L^{2}\left( V,\mu \right)$.
The proof of Corollary \ref{density2} is completed.
	\end{proof}

Also, we have the following density result for the space of square-integrable $(s,t)$-forms on $V$.
\begin{lemma}\label{0orderSobolev}
Suppose that $w$ is a real-valued Borel-measurable function on $V$ such that $\int_Ee^{-w}\,\mathrm{d}P<\infty$ for all Borel subset $E\stackrel{\circ}{\subset} V$. Then, the following set
$$
\mathfrak{F}\triangleq \bigcup^{\infty }_{n=1}  \bigg\{\sum^{\prime }_{\left| I\right|  =s,\left| J\right|  =t,\max \left\{ I\bigcup J\right\}  \leqslant n} f_{I,J}dz_{I}\wedge d\overline{z_{J}} :\;f_{I,J}\in C^{\infty }_{0,F^{\infty}}\left( V\right)\text{ for all involved }I,J \bigg\}
$$
is dense in $L^{2}_{(s,t)  }( V,w)$.
\end{lemma}
\begin{proof}
Obviously,
$\mathfrak{F}\subset L^{2}_{(s,t)}\left( V,w\right) $.
Conversely, for any $ f=\sum\limits^{\prime }_{\left| I\right|  =s,\left| J\right|  =t} f_{I,J}dz_{I}\wedge d\overline{z_{J}} \in L^{2}_{\left( s,t\right)  }\left( V,w\right)$ and $\varepsilon >0$, there exists $n\in \mathbb{N}$ such that
$$
\sum^{\prime }_{\left| I\right|  =s,\left| J\right|  =t,\max \left\{ I\bigcup J\right\}  >n} c_{I,J}\int_{V} \left| f_{I,J}\right|^{2}  e^{-w}dP<\frac{\varepsilon }{2}.
$$
By Corollary \ref{density2},  $C_{0,F^{\infty}}^{\infty}(V)$ is dense in $L^{2}(V,e^{-w}P)$. Then, for each  strictly
increasing multi-indices $I$ and $J$ with $\left| I\right|  =s,\left| J\right|  =t$ and $\max \left\{ I\bigcup J\right\}  \leqslant n$, choose $\varphi_{I,J} \in C^{\infty }_{0,F^{\infty}}\left( V\right)$ such that
$$
\int_{V} \left| \varphi_{I,J} -f_{I,J}\right|^{2}  e^{-w}dP<\frac{\varepsilon }{2n^{t+s}\max\limits_{\left| I\right|  =s,\left| J\right|  =t,\max \left\{ I\bigcup J\right\}  \leqslant n} c_{I,J}}
$$
and let $\varphi \triangleq\sum\limits^{\prime }_{\left| I\right|  =s,\left| J\right|  =t,\max \left\{ I\bigcup J\right\}  \leqslant n} \varphi_{I,J} dz_{I}\wedge d\overline{z_{J}}.$ It is clear that $\phi\in \mathfrak{F} $ and $\left| \left| \varphi -f\right|  \right|^{2}_{L^{2}_{(s,t)}\left( V,w\right)  } <\varepsilon $. This completes the proof of Lemma \ref{0orderSobolev}.
\end{proof}

Motivated by Corollary \ref{Test functions},  we introduce the following notions of supports of locally square-integrable functions/forms.
\begin{definition}\label{230124def1}
For any $f\in L^2(V,loc)$, the support of $f$, denoted by $\supp f$, is
the set of all points $\textbf{z}_0$ in $V$ for which there exists no $r>0$ so that $B_{r}(\textbf{z}_0)\subset V$ and $\int_{B_{r}(\textbf{z}_0)} f\bar\phi dP=0$ for all $\phi \in C^{\infty }_{0,F^{\infty}}\left( B_{r}(\textbf{z}_0)\right).$
	
Generally, for any
$$
f=\sum^{\prime }_{\left| I\right|  =s} \sum^{\prime }_{\left| J\right|  =t} f_{I,J}\,\mathrm{d}z_{I}\wedge \,\mathrm{d}\overline{z_{J}}\in L^{2}_{\left( s,t\right)  }\left( V,loc\right),
$$
the support of $f$, denoted by $\supp f$, is the set $\overline{\bigcup\limits^{\prime }_{\left| I\right|  =s,\left| J\right|  =t} \supp  f_{I,J}}$, where the union $\bigcup\limits^{\prime }$ is taken only over all strictly increasing multi-indices $I$ and $J$ with $\left| I\right|  =s$ and $\left| J\right|  =t$.
\end{definition}

For each $n\in\mathbb{N}$, choose a real-valued function
\begin{equation}\label{eq230131e1}
\gamma_{n} \in C^{\infty }_{c}\left( \mathbb{C}^{n} \right)
\end{equation}
such that
\begin{itemize}
\item[(1)]
$\gamma_{n} \geqslant 0$;

\item[(2)]
$\int_{\mathbb{C}^{n} } \gamma_{n} \left( \textbf{z}_n\right)  d\textbf{z}_n=1$, where $d\textbf{z}_n$ is the Lebesuge measure on $\mathbb{C}^n$;

\item[(3)]
 $\supp\gamma_{n} \subset \{ \textbf{z}_n \in \mathbb{C}^{n} :\;\left| \left| \textbf{z}_n\right|  \right|_{\mathbb{C}^{n} }  \leqslant 1\} ;$

\item[(4)] $\gamma_{n} ( \textbf{z}_n)  =\gamma_{n} ( \textbf{z}_n')$ for all $\textbf{z}_n,\textbf{z}_n'\in\mathbb{C}^{n}$ with $| | \textbf{z}_n  |  |_{\mathbb{C}^{n}}  =\left| \left| \textbf{z}_n'\right|  \right|_{\mathbb{C}^{n}}  $.
\end{itemize}
For $\delta\in(0,+\infty)$, set
$$\gamma_{n,\delta } \left( \textbf{z}_n\right)  \triangleq \frac{1}{\delta^{2n} } \gamma_n \left( \frac{\textbf{z}_n  }{\delta } \right),\quad\forall\;\textbf{z}_n  \in \mathbb{C}^{n}.
$$
Suppose that $f\in L^2(\ell^2,P)$.  By the conclusion (1) of Proposition \ref{Reduce diemension}, we see that $f_n\in L^2(\mathbb{C}^n,\mathcal{N}^n)$, where $f_n$ is given by (\ref{Integration reduce dimension}). Suppose further that $\supp f\stackrel{\circ}{\subset} \ell^2$. Then, for any $\delta\in(0,+\infty)$, we may define the convolution $f_{n,\delta}(\cdot)$  of $f_n$ and $\gamma_{n,\delta}$ by
\begin{equation}\label{Convolution after reduce dimension}
f_{n,\delta}(\textbf{z}_n)\triangleq  \int_{\mathbb{C}^n}f_n(\textbf{z}_n')\gamma_{n,\delta }  ( \textbf{z}_n -\textbf{z}_n')\mathrm{d}\textbf{z}_n',\quad\;\textbf{z}_n  \in \mathbb{C}^{n}.
\end{equation}
It is clear that $f_{n,\delta}\in C_c^{\infty}(\mathbb{C}^n)$. Set
\begin{equation}\label{def for Gauss weight}
\varphi_n(\textbf{z}_n)\triangleq\prod_{i=1}^{n}\left(\frac{1}{2\pi a_i^2}\cdot e^{-\frac{|z_i|^2}{2a_i^2}}\right),\quad \; \textbf{z}_n=(z_1,\cdots,z_n)\in \mathbb{C}^n.
\end{equation}
We have the following result.
\begin{proposition}\label{convolution properties}
Suppose that $f\in L^2(\ell^2,P)$ with $\supp f\stackrel{\circ}{\subset} \ell^2$. Then, for each $n\in\mathbb{N}$, it holds that (Recall (\ref{Integration reduce dimension}) for $ f_n$)
\begin{itemize}
\item[{\rm (1)}] $f_n\in  L^2(\mathbb{C}^n, d\textbf{z}_n)\cap L^1(\mathbb{C}^n, d\textbf{z}_n)$, $f_{n,\delta}\in L^2(\ell^2,P)$ for any $\delta>0$ and $\lim\limits_{\delta\to 0+}||f_{n,\delta}-f_n||_{L^2(\ell^2,P)}=0$;
\item[{\rm (2)}] $\lim\limits_{\delta\rightarrow 0+} \int_{\ell^2}\big|| f_{n,\delta}|^{2}  -| f_n|^{2}\big|\,\mathrm{d}P =0;$
\item[{\rm (3)}] For any $g\in L^2(\ell^2,P)$ with $\supp g\stackrel{\circ}{\subset} \ell^2$,
$$
\int_{\ell^2}f_{n,\delta}g\,\mathrm{d}P=\int_{\ell^2}f\varphi_n^{-1}\cdot(g_{n}\varphi_n)_{n,\delta}\,\mathrm{d}P.
$$
\end{itemize}
\end{proposition}
\begin{proof}
(1) Since $f_{n,\delta}\in C_c^{\infty}(\mathbb{C}^n)$, we have $f_{n,\delta}\in L^2(\ell^2,P)$.
By our assumption there exists $r>0$ such that $\supp f\subset B_{r}$. By (\ref{Integration reduce dimension}) and (\ref{Convolution after reduce dimension}), we see that $\supp f_{n}\subset \left\{ \textbf{z}_n\in \mathbb{C}^{n} :\;\left| \left| \textbf{z}_n\right|  \right|_{\mathbb{C}^{n} }  \leqslant r\right\}$, and hence $\supp f_{n,\delta }\subset \left\{ \textbf{z}_n\in \mathbb{C}^{n} :\;\left| \left| \textbf{z}_n\right|  \right|_{\mathbb{C}^{n} }  \leqslant r+\delta \right\}$. By the conclusion (1) of Proposition \ref{Reduce diemension}, we have $f_{n}\in L^2(\mathbb{C}^n,d\textbf{z}_n)\cap L^1(\mathbb{C}^n,d\textbf{z}_n)$ and
$$
\int_{\ell^2}|f_{n,\delta}  -f_n|^{2}\,\mathrm{d}P
=\int_{\mathbb{C}^n} |  f_{n,\delta} - f_n |^2\,\mathrm{d}\mathcal{N}^n
\leqslant \frac{1}{\left( 2\pi \right)^{n}  a^{2}_{1}a^{2}_{2}\cdots a^{2}_{n}} \int_{\mathbb{C}^{n} } \left| f_{n,\delta }\left(  \textbf{z}_n\right)  -f_{n}\left( \textbf{z}_n\right)  \right|^{2}  \mathrm{d}\textbf{z}_n.
$$
Since $\lim\limits_{\delta\to 0+}\int_{\mathbb{C}^{n} } \left| f_{n,\delta }\left(  \textbf{z}_n\right)  -f_{n}\left( \textbf{z}_n\right)  \right|^{2}  \mathrm{d}\textbf{z}_n=0$, we conclude that $\lim\limits_{\delta\to 0+}||f_{n,\delta}-f_n||_{L^2(\ell^2,P)}=0$.

\medskip

(2) The proof is similar to that of the conclusion (3) of Proposition \ref{Reduce diemension}, and therefore we omit the details.

\medskip

(3) Note that
$$
 	\begin{array}{ll}
\displaystyle
 		\int_{\ell^{2} } f_{n,\delta }g\mathrm{d}P=\int_{\ell^{2} } \left(\int_{\mathbb{C}^{n} } f_{n}\left( \textbf{z}_n'\right)  \gamma_{n,\delta } \left( \textbf{z}_n-\textbf{z}_n'\right)  \mathrm{d}\textbf{z}_n'\right)g\left( \textbf{z}_n,\textbf{z}^n\right) \mathrm{d}P\left( \textbf{z}_n,\textbf{z}^n\right)  \\[3mm]
\displaystyle =\int_{\ell^{2} } \left(\int_{\ell^{2} } f\left( \textbf{z}_n',\tilde{\textbf{z}}^n\right)  \gamma_{n,\delta } \left( \textbf{z}_n-\textbf{z}_n'\right)  \varphi^{-1}_{n} \left( \textbf{z}_n'\right)  \mathrm{d}P\left( \textbf{z}_n',\tilde{\textbf{z}}^n\right)  \right) g\left( \textbf{z}_n,\textbf{z}^n\right) \mathrm{d}P\left(\textbf{z}_n,\textbf{z}^n\right) \\[3mm]
\displaystyle =\int_{\ell^{2} }f\left( \textbf{z}_n',\tilde{\textbf{z}}^n\right)\varphi^{-1}_{n} \left( \textbf{z}_n'\right)\left( \int_{\ell^{2} }g\left( \textbf{z}_n,\textbf{z}^n\right) \gamma_{n,\delta } \left( \textbf{z}_n-\textbf{z}_n'\right)\mathrm{d}P\left(\textbf{z}_n,\textbf{z}^n\right)\right)\mathrm{d}P\left( \textbf{z}_n',\tilde{\textbf{z}}^n\right) \\[3mm]
\displaystyle =\int_{\ell^{2} }f\left( \textbf{z}_n',\tilde{\textbf{z}}^n\right)\varphi^{-1}_{n} \left( \textbf{z}_n'\right)\left( \int_{\mathbb{C}^{n}  }g_n\left( \textbf{z}_n\right)\varphi_n(\textbf{z}_n) \gamma_{n,\delta } \left( \textbf{z}_n-\textbf{z}_n'\right)\mathrm{d}\textbf{z}_n\right)\mathrm{d}P\left( \textbf{z}_n',\tilde{\textbf{z}}^n\right)  \\[3mm]
\displaystyle  =\int_{\ell^{2} } f\varphi^{-1}_{n} \cdot\left( g_{n}\varphi_{n} \right)_{n,\delta } \mathrm{d}P,
 	\end{array}
$$
where the third equality follows from Fubini's theorem. This completes the proof of Proposition \ref{convolution properties}.
\end{proof}

The following result provides a useful link between the forms on $V$ and that on finite dimensions.
\begin{lemma}\label{approximation for st form}
Suppose that $\sup\limits_{\textbf{z}\in E}|w(\textbf{z})|<\infty$ for each $E\stackrel{\circ}{\subset} V$ and $f=\sum\limits^{\prime }_{\left| I\right|  =s,\left| J\right|  =t} f_{I,J}dz_{I}\wedge d\overline{z_{J}}\in L^{2}_{(s,t)  }( V,w)$. Then, the following conclusions hold:

\medskip

 {\rm (1)} $L^{2}_{(s,t)  }( V,w)\subset L^{2}_{(s,t)  }( V,loc)$;

 \medskip

 {\rm (2)}  If $\supp f\stackrel{\circ}{\subset} V$ (in this case $f$  can be viewed as an element of $L^{2}_{(s,t)  }( \ell^2,P)\left(\subset L^{2}_{\left( s,t\right)  }\left( \ell^2,loc\right)\right)$ by extending the value of $f_{I,J}$ by 0 on $\ell^2\setminus V$ for all strictly
increasing multi-indices $I$ and $J$, and $\supp f\stackrel{\circ}{\subset} \ell^2$),
for each $n\in\mathbb{N}$ and $\delta\in(0,+\infty)$, letting
$$
f_{n,\delta}\triangleq\sum\limits^{\prime }_{\left| I\right|  =s,\left| J\right|  =t,\max\{I\cup J\}\leqslant n} (f_{I,J})_{n,\delta}dz_{I}\wedge d\overline{z_{J}},
$$
where $(f_{I,J})_{n,\delta}$ is defined as that in (\ref{Convolution after reduce dimension}), then
$$
\lim_{n\to\infty}\lim_{\delta\to 0+}||f_{n,\delta}-f||_{L^{2}_{(s,t)  }( \ell^2,P)}=0.
$$
\end{lemma}
\begin{proof}
The proof of the conclusion (1) is obvious, and hence we only prove the conclusion (2).
For each $n\in\mathbb{N}$, write
$$
f_{n}\triangleq\sum\limits^{\prime }_{\left| I\right|  =s,\left| J\right|  =t,\max\{I\cup J\}\leqslant n} (f_{I,J})_{n}dz_{I}\wedge d\overline{z_{J}},
$$
where $(f_{I,J})_{n}$ is defined similarly to that in (\ref{Integration reduce dimension}). By $f\in L^{2}_{(s,t)  }( \ell^2,P)$, for any $\varepsilon>0$, there exists $n_0\in\mathbb{N}$ such that
$$
\sum\limits^{\prime }_{\left| I\right|  =s,\left| J\right|  =t,\max\{I\cup J\}> n_0} c_{I,J}\int_{\ell^2}|f_{I,J}|^2\,\mathrm{d}P<\frac{\varepsilon}{8}.
$$
By the  conclusion (2) of Proposition \ref{Reduce diemension}, there exists $n_1\in\mathbb{N}$ such that $n_1\geqslant n_0$ and for all $n\geqslant n_1$, it holds that
$$
\sum\limits^{\prime }_{\left| I\right|  =s,\left| J\right|  =t,\max\{I\cup J\}\leqslant n_0} c_{I,J}\int_{\ell^2}|(f_{I,J})_{n}-f_{I,J}|^2\,\mathrm{d}P<\frac{\varepsilon}{4},
$$
and by the  conclusion (1) of Proposition \ref{Reduce diemension},
$$
\begin{array}{ll}
\displaystyle
||f_{n}-f||_{L^{2}_{(s,t)  }( \ell^2,P)}^2\\[3mm]
\displaystyle =
\sum\limits^{\prime }_{\left| I\right|  =s,\left| J\right|  =t,\max\{I\cup J\}\leqslant n_0} c_{I,J}\int_{\ell^2}|(f_{I,J})_{n}-f_{I,J}|^2\,\mathrm{d}P+\sum\limits^{\prime }_{\left| I\right|  =s,\left| J\right|  =t,\max\{I\cup J\}> n} c_{I,J}\int_{\ell^2}|f_{I,J}|^2\,\mathrm{d}P\\[3mm]
\displaystyle \quad+\sum\limits^{\prime }_{\left| I\right|  =s,\left| J\right|  =t,n_0<\max\{I\cup J\}\leqslant n} c_{I,J}\int_{\ell^2}|(f_{I,J})_{n}-f_{I,J}|^2\,\mathrm{d}P\\[3mm]
\displaystyle \leqslant\sum\limits^{\prime }_{\left| I\right|  =s,\left| J\right|  =t,\max\{I\cup J\}\leqslant n_0} c_{I,J}\int_{\ell^2}|(f_{I,J})_{n}-f_{I,J}|^2\,\mathrm{d}P+
5\sum\limits^{\prime }_{\left| I\right|  =s,\left| J\right|  =t,\max\{I\cup J\}> n_0} c_{I,J}\int_{\ell^2}|f_{I,J}|^2\,\mathrm{d}P\\[3mm]
\displaystyle <\varepsilon.
\end{array}
$$
Therefore, $\lim\limits_{n\to\infty}||f_{n}-f||_{L^{2}_{(s,t)  }( \ell^2,P)}=0$. For each $n\in\mathbb{N}$ and $\delta\in(0,+\infty)$,
$$
||f_{n,\delta}-f_n||_{L^{2}_{(s,t)  }( \ell^2,P)}^2=
\sum\limits^{\prime }_{\left| I\right|  =s,\left| J\right|  =t,\max\{I\cup J\}\leqslant n} c_{I,J}\int_{\ell^2}|(f_{I,J})_{n,\delta}-(f_{I,J})_n|^2\,\mathrm{d}P.
$$
Combining this with the conclusion (1) of Proposition \ref{convolution properties}, we see that $\lim\limits_{\delta\to 0+}||f_{n,\delta}-f_n||_{L^{2}_{(s,t)  }( \ell^2,P)}=0$, which completes the proof of Lemma \ref{approximation for st form}.
\end{proof}

As a consequence of Lemmas \ref{0orderSobolev} and \ref{approximation for st form}, we have the following density result.
\begin{corollary}\label{cylinder function is dense}
Suppose that $\sup\limits_{\textbf{z}\in V} |w(\textbf{z})|<\infty$. Then
$$
\bigcup^{\infty }_{n=1}  \bigg\{\sum^{\prime }_{\left| I\right|  =s,\left| J\right|  =t,\max \left\{ I\bigcup J\right\}  \leqslant n} f_{I,J}dz_{I}\wedge d\overline{z_{J}} :\; f_{I,J}\in C^{\infty }_{c}\left(\mathbb{C}^n\right)\text{ for all involved }I,J \bigg\}
$$
is dense in $L^{2}_{(s,t)  }( V,w)$.
\end{corollary}

We also need the following simple approximation result (which should be known but we do not find an exact reference).
\begin{lemma}\label{F}
Assume that $f\in C^{1 }_{b}( \ell^{2})$. Then there exists $\{h_{k }\}_{k=1}^{\infty}\subset C^{\infty }_{b,F^{\infty}}( \ell^{2})$ such that $\lim\limits_{k \rightarrow \infty } h_{k}=f$, $\lim\limits_{k \rightarrow \infty } D_{x_{j}}h_{k}=D_{x_{j}}f$  and $\lim\limits_{k \rightarrow \infty } D_{y_{j}}h_{k}=D_{y_{j}}f$ almost everywhere respect to $P$ for each $j\in\mathbb{N}$. Meanwhile, for each $j,k\in\mathbb{N}$, the following inequalities hold on $\ell^2$,
\begin{equation}\label{6zqqqx}
\left\{
\begin{array}{ll}
\displaystyle\left|h_{k }\right|  \leqslant \sup_{\ell^{2} } \left| f\right|  ,\  \left| \partial_{j} h_{k}\right|  \leqslant \sup_{ \ell^{2} } \left| \partial_{j} f\right|  ,\left| \overline{\partial_{j} } h_{k }\right|  \leqslant \sup_{ \ell^{2} } \left| \overline{\partial_{j} } f\right|,  \\[5mm]
\displaystyle\left| D_{x_{j}}h_{k }\right|  \leqslant \sup_{\ell^{2} } \left| D_{x_{j}}f\right|\text{ and } \left| D_{y_{j}}h_{k}\right|  \leqslant \sup_{\ell^{2} } \left| D_{y_{j}}f\right|.
\end{array}\right.
\end{equation}
Furthermore, if $f\in C^{1}_{0,F}( \ell^{2})$, then
\begin{equation}\label{6}		
\sum^{\infty }_{j=1} \left| \partial_{j} h_{k}\right|^{2}  \leqslant \sup_{\ell^{2} } \sum^{\infty }_{j=1} \left| \partial_{j} f\right|^{2} ,\quad\sum^{\infty }_{j=1} \left| \overline{\partial_{j} } h_{k}\right|^{2}  \leqslant \sup_{\ell^{2} } \sum^{\infty }_{j=1} \left| \overline{\partial_{j} } f\right|^{2},\quad\forall\;k\in\mathbb{N}.
\end{equation}
\end{lemma}

\begin{proof}
For each $n\in\mathbb{N}$, let $f_n$ be given by (\ref{Integration reduce dimension}). Then, $D_{x_j}f_n(\textbf{z}_n)=\int D_{x_j}f(\textbf{z}_n,\textbf{z}^n)\,\mathrm{d}P_n(\textbf{z}^n)$ for $j=1,\cdots, n$. By the conclusion (2) of Proposition \ref{Reduce diemension}, we have $\lim\limits_{n\rightarrow \infty } \int_{\ell^{2} } \left| f_{n}-f\right|^{2}  dP=0$ and $\lim\limits_{n\rightarrow \infty } \int_{\ell^{2} } \left| D_{x_{j}} f_{n}  -D_{x_{j}} f\right|^{2}dP=0.$  For the sequence of functions $\{\gamma_n\}_{n=1}^{\infty}$ given in (\ref{eq230131e1}) and each $\delta\in(0,+\infty)$, we write
$$
f_{n,\delta}(\textbf{z}_n)\triangleq  \int_{\mathbb{C}^n}f_n(\textbf{z}_n')\gamma_{n,\delta }  ( \textbf{z}_n -\textbf{z}_n')d\textbf{z}_n',\quad\textbf{z}_n  \in \mathbb{C}^{n}.
$$
Then, $f_{n,\delta}\in C^{\infty }_{b,F^{\infty}}( \ell^{2})$, $|f_{n,\delta}|\leqslant \sup\limits_{\mathbb{C}^n}|f_n|\leqslant \sup\limits_{\ell^2}|f|$, $|D_{x_j}f_{n,\delta}|\leqslant \sup\limits_{\mathbb{C}^n}|D_{x_j}f_n|\leqslant \sup\limits_{\ell^2}|D_{x_j}f|$, $|\partial_{j}f_{n,\delta}|\leqslant \sup\limits_{\mathbb{C}^n}|\partial_{j}f_n|\leqslant \sup\limits_{\ell^2}|\partial_{j}f|$,
$$
\lim\limits_{\delta\rightarrow 0} \int_{\ell^{2} } \left| f_{n,\delta}-f_n\right|^{2}  dP=0,\qquad
\lim\limits_{\delta\rightarrow 0} \int_{\ell^{2} } \left| D_{x_{j}} f_{n,\delta}  -D_{x_{j}} f_n\right|^{2}dP=0
$$
for each fixed $j\in\mathbb{N}$. By the diagonal method, we can pick a sequence of functions $\{g_k\}_{k=1}^{\infty}$ from $\{f_{n,\delta}:\;n\in\mathbb{N},\,\delta\in(0,+\infty)\}$ such that $\lim\limits_{k\rightarrow \infty } \int_{\ell^{2} } \left| g_{k}-f\right|^{2}  dP=0$ and $\lim\limits_{k\rightarrow \infty } \int_{\ell^{2} } \left| D_{x_{j}} g_{k}  -D_{x_{j}} f\right|^{2}dP=0$ for each $j\in\mathbb{N}$. By \cite[Theorem 6.3.1 (b) at pp. 171--172 and Section 6.5 at pp. 180--181]{Res} and using  the diagonal method again, we can find a subsequence $\{h_k\}_{k=1}^{\infty}$ of $\{g_k\}_{k=1}^{\infty}$ such that $\lim\limits_{k \rightarrow \infty } h_{k}=f$ and $\lim\limits_{k \rightarrow \infty } D_{x_{j}}h_{k}=D_{x_{j}}f$  almost everywhere respect to $P$ for each $j\in\mathbb{N}$. Obviously, the sequence $\{h_k\}_{k=1}^{\infty}$ satisfies the conditions in (\ref{6zqqqx}).

Next, if $f\in C^{1}_{0,F}( \ell^{2})$, then by the Jensen inequality (See \cite[Theorem 3.3, p. 62]{Rud87}),
$$
\begin{array}{ll}
\displaystyle 		
\sum^{\infty }_{j=1} \left| \partial_{j} f_{n,\delta}( \textbf{z}_n)\right|^{2}
= \sum^{n}_{j=1} \left| \partial_{j} f_{n,\delta}( \textbf{z}_n )\right|^{2}
\leqslant\int_{\mathbb{C}^n}\sum^{n}_{j=1} \left| \partial_{j} f_{n }(\textbf{z}_n')\right|^{2}\gamma_{n,\delta }  ( \textbf{z}_n -\textbf{z}_n')d\textbf{z}_n'\\[3mm]
\displaystyle \leqslant \sup_{\textbf{z}_n'\in\mathbb{C}^n}\sum^{n}_{j=1} \left| \partial_{j} f_{n }(\textbf{z}_n')\right|^{2}
= \sup_{\textbf{z}_n'\in\mathbb{C}^n}\sum^{n}_{j=1} \left| \int\partial_{j} f (\textbf{z}_n',\textbf{z}^n)\,\mathrm{d}P_n(\textbf{z}^n)\right|^{2}\\[3mm]
\displaystyle \leqslant \sup_{\textbf{z}_n'\in\mathbb{C}^n} \int\sum^{n}_{j=1} \left|\partial_{j} f (\textbf{z}_n',\textbf{z}^n)\right|^{2}\,\mathrm{d}P_n(\textbf{z}^n)
\leqslant \sup_{\textbf{z}\in\ell^2}  \sum^{n}_{j=1} \left|\partial_{j} f (\textbf{z})\right|^{2}\leqslant \sup_{\textbf{z}\in\ell^2}  \sum^{\infty}_{j=1} \left|\partial_{j} f (\textbf{z})\right|^{2},
\end{array}
$$
and hence we obtain \eqref{6}. This completes the proof of Lemma \ref{F}.
\end{proof}

 For $i\in\mathbb{N} $ and strictly increasing multi-indices $L=(l_1,\cdots,l_t)$ and $K$ with $\left| L\right|  =t$ and $\left| K\right|  =t+1$,
	we set
\begin{equation}\label{230204e3}
\varepsilon_{iL}^{K}\triangleq\left\{
\begin{array}{ll}0, &\hbox{ if }K\neq  \{i\}\cup L,\\
\hbox{the sign of the permutation }\left(\begin{array}{ll}i,l_1,\cdots,l_t\\ \qquad K\end{array}\right), &\hbox{ if }K= \{i\}\cup L.
\end{array}\right.
\end{equation}

\begin{definition}\label{def for db}
We say that $f=\sum\limits^{\prime }_{\left| I\right|  =s} \sum\limits^{\prime }_{\left| J\right|  =t} f_{I,J}\,\mathrm{d}z_{I}\wedge \,\mathrm{d}\overline{z_{J}}\in L^{2}_{(s,t)}\left( V,loc\right)$ is in the domain $ D_{\overline{\partial } }$ of $\overline{\partial }$ if there exists $g=\sum\limits^{\prime }_{\left| I\right|  =s} \sum\limits^{\prime }_{\left| K\right|  =t+1} g_{I,K}\,\mathrm{d}z_{I}\wedge \,\mathrm{d}\overline{z_{K}}\in L^{2}_{(s,t+1)}\left( V,loc\right)$ such that for all strictly increasing multi-indices $I$ and $K$ with $\left| I\right|  =s$ and $\left| K\right|  =t+1$, it holds that
\begin{equation}\label{230204e5}
\left( -1\right)^{s+1}  \int_{V} \sum^{\infty }_{i=1} \sum^{\prime }_{\left| J\right|  =t} \varepsilon^{K}_{iJ} f_{I,J}\overline{\delta_{i} \phi }\,\mathrm{d}P=\int_{V} g_{I,K}\bar{\phi }\,\mathrm{d}P,\quad \forall \;\phi \in C^{\infty }_{0,F}\left( V\right).
\end{equation}
In this case, we define $\overline{\partial } f\triangleq g$.
\end{definition}

\begin{remark}\label{230327r1}
Note that, for any fixed strictly increasing multi-index $K$ with $\left| K\right|=t+1$, by the definition of $\varepsilon^{K}_{iJ}$ in (\ref{230204e3}), there exist at most finitely many non-zero terms in the series $\sum\limits^{\infty }_{i=1} \sum\limits^{\prime }_{\left| J\right|  =t} \varepsilon^{K}_{iJ} f_{I,J}\overline{\delta_{i} \phi }$ in (\ref{230204e5}).
\end{remark}

\begin{remark}\label{230407r1}
If $f=\sum\limits^{\prime }_{\left| I\right|  =s} \sum\limits^{\prime }_{\left| J\right|  =t} f_{I,J}\,\mathrm{d}z_{I}\wedge \,\mathrm{d}\overline{z_{J}}\in L^{2}_{(s,t)}\left( V,loc\right)$ satisfies $ f_{I,J}\in C_0^1(V)$ for each strictly increasing multi-indices $I$ and $J$ with $\left| I\right|  =s$ and $\left| J\right|  =t$, and  $\sum\limits^{\prime }_{\left| I\right|  =s} \sum\limits^{\prime }_{\left| K\right|  =t+1} \sum\limits^{\infty }_{i=1} \sum\limits^{\prime }_{\left| J\right|  =t} \varepsilon^{K}_{iJ} \overline{\partial}_{i}  f_{I,J}\,\mathrm{d}z_{I}\wedge \,\mathrm{d}\overline{z_{K}}\in L^{2}_{(s,t+1)}\left( V,loc\right)$, then, by Corollaries \ref{integration by Parts or deltai} and \ref{density2}, $ f\in D_{\overline{\partial } }$, and
$$
\overline{\partial } f=\left( -1\right)^{s} \sum\limits^{\prime }_{\left| I\right|  =s} \sum\limits^{\prime }_{\left| K\right|  =t+1} \sum^{\infty }_{i=1} \sum^{\prime }_{\left| J\right|  =t} \varepsilon^{K}_{iJ} \overline{\partial}_{i}  f_{I,J}\,\mathrm{d}z_{I}\wedge \,\mathrm{d}\overline{z_{K}}.
$$
\end{remark}

We have the following result.
\begin{proposition}\label{weak equality for more test functions}
Under Condition \ref{230424c1}, for any $f=\sum\limits^{\prime }_{\left| I\right|  =s} \sum\limits^{\prime }_{\left| J\right|  =t} f_{I,J}\,\mathrm{d}z_{I}\wedge \,\mathrm{d}\overline{z_{J}}\in L^{2}_{(s,t)}\left( V,loc\right)$, the following assertions hold:

{\rm (1)} If $f\in  D_{\overline{\partial } }$, then for all strictly increasing multi-indices $I$ and $K$ with $\left| I\right|=s$ and $\left| K\right|=t+1$,
\begin{equation}\label{230203e1}
\left( -1\right)^{s+1}  \int_{V} \sum^{\infty }_{i=1} \sum^{\prime }_{\left| J\right|  =t} \varepsilon^{K}_{iJ} f_{I,J}\overline{\delta_{i} \phi }\,\mathrm{d}P=\int_{V} \left( \overline{\partial } f\right)_{I,K}  \bar{\phi }\,\mathrm{d}P,\quad  \forall\; \phi \in C^{1 }_{0}\left( V\right);
\end{equation}

{\rm (2)}
If $f\in  D_{\overline{\partial } }$ and $\supp f \stackrel{\circ}{\subset}V$, then for all strictly increasing multi-indices $I$ and $K$  with $\left| I\right|=s$ and $\left| K\right|=t+1$,
\begin{equation}\label{230204e2}	
\left( -1\right)^{s+1}  \int_{V} \sum^{\infty }_{i=1} \sum^{\prime }_{\left| J\right|  =t} \varepsilon^{K}_{iJ} f_{I,J}\overline{\delta_{i} \phi }\,\mathrm{d}P=\int_{V} \left( \overline{\partial } f\right)_{I,K}  \bar{\phi }\,\mathrm{d}P,\quad  \forall\; \phi \in C^{1 }_{b}\left( V \right);
\end{equation}

{\rm (3)}
If $\eta\in C^{\infty}_{F^{\infty}}\left( V\right)$, $g=\sum\limits^{\prime }_{\left| I\right|  =s} \sum\limits^{\prime }_{\left| K\right|  =t+1} g_{I,K}\,\mathrm{d}z_{I}\wedge \,\mathrm{d}\overline{z_{K}}\in L^{2}_{(s,t+1)}\left( V,loc\right)$, and the following equality
$$
\left( -1\right)^{s+1}  \int_{V} \sum^{\infty }_{i=1} \sum^{\prime }_{\left| J\right|  =t} \varepsilon^{K}_{iJ} f_{I,J}\overline{\delta_{i} \phi }\,\mathrm{d}P=\int_{V} g_{I,K}\bar{\phi } \,\mathrm{d}P,\quad  \forall\; \phi \in C^{\infty }_{0,F^{\infty }}\left( V\right)
$$
holds for any strictly increasing multi-indices $I$ and $K$  with $\left| I\right|=s$ and $\left| K\right|=t+1$, then $f\in  D_{\overline{\partial } }$ and $\overline{\partial } f\triangleq g$.
\end{proposition}

\begin{proof}
(1) Choose $\{\varphi_{k}\}_{k=1}^{\infty} \subset C^{\infty }\left( \mathbb{R}\right)$ such that for each $k\in\mathbb{N}$, $0\leqslant\varphi_{k} \leqslant1$, $\varphi_{k} \left( x\right)  =1$ for all $x\leqslant k$ and $\varphi_{k} \left( x\right)  =0$ for all $x>k+1$. Recall Condition \ref{230424c1} for $\eta$. Let
\begin{equation}\label{230204e1}
\Phi_{k} \triangleq\varphi_{k}( \eta ).
\end{equation}
Then  $\Phi_{k}  \in C^{\infty }_{0,F}( V).$
For any fixed $\phi \in C^{1 }_{0}\left( V\right)$, by Proposition \ref{properties on pseudo-convex domain}, there exists $r\in(0,+\infty)$ such that $\supp \phi \subset V_{r}$ (Recall (\ref{230117e1}) for the notation $V_r$). We may view $\phi$ as an element of $C^{1 }_{0}( \ell^2)\subset C^{1 }_{b}( \ell^2)$ by the zero extension. By Lemma \ref{F}, there exists $\{\phi_{k}\}_{k=1}^{\infty}\subset C^{\infty }_{b,F^\infty}( \ell^{2})$ such that the conditions corresponding to (\ref{6zqqqx}) are satisfied (for which $f$ and $h_k$ therein are replaced by $\phi$ and $\phi_k$, respectively). For any $k,l\in\mathbb{N}$ with $k>r$, we have $\phi_{l}\Phi_{k} \in C^{\infty }_{0,F}\left( V\right)$, and hence, by  $f\in  D_{\overline{\partial } }$ and Definition \ref{def for db},
\begin{equation}\label{230204e4}
\left( -1\right)^{s+1}  \int_{V} \sum^{\infty }_{i=1} \sum^{\prime }_{\left| J\right|  =t} \varepsilon^{K}_{iJ} f_{I,J}\overline{\delta_{i} \left( \phi_{l } \Phi_{k} \right)  }\,\mathrm{d}P=\int_{V} \left( \overline{\partial } f\right)_{I,K}  \overline{\phi_{l } \Phi_{k} }\,\mathrm{d}P.
\end{equation}
Letting $l\rightarrow \infty $ in (\ref{230204e4}) and noting Remark \ref{230327r1}, we have
\begin{equation}\label{230204e25}
\left( -1\right)^{s+1}  \int_{V} \sum^{\infty }_{i=1} \sum^{\prime }_{\left| J\right|  =t} \varepsilon^{K}_{iJ} f_{I,J}\overline{\delta_{i} \left( \phi \Phi_{k} \right)  }\,\mathrm{d}P=\int_{V} \left( \overline{\partial } f\right)_{I,K}  \overline{\phi \Phi_{k} }\,\mathrm{d}P.
\end{equation}
Since $\supp \phi\subset V_k$ and $\Phi_k=1$ on $V_k$, by (\ref{230204e25}) we obtain the desired (\ref{230203e1}).

\medskip

(2) Since $\supp f \stackrel{\circ}{\subset}V$, by Proposition \ref{properties on pseudo-convex domain} again, there exists $r\in(0,+\infty)$ such that $\supp f \subset V_{r}.$  Then for any $k\in\mathbb{N}$ with $k>r$ and $\phi \in C^{1 }_{b}\left( V \right)$, $\phi \Phi_{k}\in C^{1 }_{0}( V)$ (where $\Phi_{k}$ is given by (\ref{230204e1})). By the above conclusion (1), we see that (\ref{230204e25}) holds for the present $\phi \Phi_{k}$. Hence, by $\supp f\subset V_k$ and $\Phi_k=1$ on $V_k$, we obtain the desired (\ref{230204e2}).

\medskip

(3) By our assumption, $\{\Phi_{k}\}_{k=1}^{\infty} \subset C^{\infty }_{0,F^{\infty}}( V)  $ (See (\ref{230204e1}) for $\Phi_{k}$). For any $\phi \in C^{\infty}_{0, F}( V)$, we choose $\{\phi_{k}\}_{k=1}^{\infty}\subset C^{\infty }_{b,F^{\infty}}( \ell^{2})$ as in the proof of the conclusion (1). For any $k,l\in\mathbb{N}$ such that $k>r$, we have $\phi_{l}\Phi_{k} \in C^{\infty }_{0,F^{\infty}}\left( V\right)$ and hence
$$
\left( -1\right)^{s+1}  \int_{V} \sum^{\infty }_{i=1} \sum^{\prime }_{\left| J\right|  =t} \varepsilon^{K}_{iJ} f_{I,J}\overline{\delta_{i} \left( \phi_{l} \Phi_{k} \right)  }\,\mathrm{d}P=\int_{V}g_{I,K}\overline{\phi_{l } \Phi_{k} }\,\mathrm{d}P.
$$
By the same argument as that in the proof of the conclusion (1), from the above equality we obtain (\ref{230204e5}).
The proof of Proposition \ref{weak equality for more test functions} is completed.
\end{proof}

\begin{definition}\label{definition of T S}
For any real-valued functions $w_{1},w_{2},w_{3}\in C^{2}_{F}\left( V\right)$, we define two linear unbounded operators $T:\;L^{2}_{(s,t)}\left( V,w_{1}\right) \to L^{2}_{(s,t+1)}\left( V,w_{2}\right)$ and $S:\;L^{2}_{(s,t+1)}\left( V,w_{2}\right)\to L^{2}_{(s,t+2)}\left( V,w_{3}\right)$ as follows:
$$
\left\{
\begin{array}{ll}
D_{T}\triangleq \left\{ u\in L^{2}_{(s,t)}\left( V,w_{1}\right)  :\;u\in D_{\overline{\partial }}\;\text{ and }\;\overline{\partial } u\in L^{2}_{(s,t+1)}\left( V,w_{2}\right)  \right\},\\[3mm]
Tu\triangleq  \overline{\partial } u,\quad\forall\;u\in D_T,
\end{array}\right.
$$
$$
\quad\;\left\{
\begin{array}{ll}
D_{S}\triangleq \left\{ f\in L^{2}_{\left( s,t+1\right)  }\left( V,w_{2}\right)  :\;f\in D_{\overline{\partial }}\;\text{ and }\;\overline{\partial } f\in L^{2}_{(s,t+2)}\left( V,w_{3}\right)  \right\},\\[3mm]
Sf\triangleq \overline{\partial } f,\quad\forall\;f\in D_S.
\end{array}\right.
$$
We call $
R_{T}\triangleq \left\{ Tu :\;u\in D_{T}  \right\}$ and $
N_{S}\triangleq \left\{ f :\;f\in D_{S}\text{ and }Sf=0  \right\}$ the range of $T$ and the kernel of $S$, respectively.
\end{definition}

In the rest of this section, to simplify the presentation, unless otherwise stated, we fix $w_{1},w_{2},w_{3}$, $T$ and $S$ as that in Definition \ref{definition of T S}.
\begin{lemma}\label{densly defined closed}
The operator $T$ (resp. $S$) is  densely defined and closed from $L^{2}_{(s,t)}\left( V,w_{1}\right)$ (resp. $L^{2}_{(s,t+1)}\left( V,w_{2}\right)$) into $L^{2}_{(s,t+1)}\left( V,w_{2}\right)$ (resp. $L^{2}_{(s,t+2)}\left( V,w_{3}\right)$).
\end{lemma}
\begin{proof}
It is easy to see that
$$
\bigcup^{\infty }_{n=1}  \bigg\{ \sum^{\prime }_{\left| I\right|  =s,\left| J\right|  =t,\max \left\{ I\bigcup J\right\}  \leqslant n} f_{I,J}dz_{I}\wedge d\overline{z_{J}} :\;f_{I,J}\in C^{\infty }_{0,F}\left( V\right)\text{ for all }I,J \bigg\}\subset D_T
$$
and
$$
\bigcup^{\infty }_{n=1}  \bigg\{ \sum^{\prime }_{\left| I\right|  =s,\left| J\right|  =t+1,\max \left\{ I\bigcup J\right\}  \leqslant n} f_{I,J}dz_{I}\wedge d\overline{z_{J}} :\;f_{I,J}\in C^{\infty }_{0,F}\left( V\right)\text{ for all }I,J \bigg\}\subset D_S.
$$
Combining Lemma \ref{0orderSobolev} and the fact that $ C^{\infty }_{0,F^{\infty}}\left( V\right)\subset  C^{\infty }_{0,F}\left( V\right)$, we see that both $T$ and $S$ are densely defined operators. Similarly to the proof of \cite[Lemma 2.5, p. 529]{YZ}, we can prove that both $T$ and $S$ are closed operators. The proof of Lemma \ref{densly defined closed} is completed.
\end{proof}

In the sequel, we shall denote by $T^*$ and $S^*$ respectively the adjoint operators of $T$ and $S$ (e.g., \cite[Definition 4.2.2, p. 177]{Kra}). Meanwhile, we denote by $D_{T^*}$ and $D_{S^*}$ respectively the domains of $T^*$ and $S^*$.

We have the following simple result.
\begin{proposition}\label{Pre-exactness}
For each $u\in D_T$, it holds that $Tu\in D_S$ and $S(Tu)=0$, i.e., $R_T\subset N_S$.
\end{proposition}

\begin{proof}
The proof is very similar to that of \cite[Lemma 2.6, p. 531]{YZ}, and therefore we omit the details.
\end{proof}

For any fixed $f=\sum\limits^{\prime }_{\left| I\right|  =s,\left| J\right|  =t+1} f_{I,J}dz_{I}\wedge d\overline{z_{J}}\in L^{2}_{(s,t+1)}\left(V,loc\right)$, $i\in\mathbb{N} $ and strictly increasing multi-indices $K$ and $L$ with $\left| K\right|  =s$ and $\left| L\right|  =t$, we write:
 \begin{equation}\label{gener111}
 f_{K,iL}\triangleq \sum\limits^{\prime }_{\left| J\right|  =t+1}\varepsilon^{J}_{iL}\cdot f_{K,J},\quad c_{K,iL}\triangleq \sum\limits^{\prime }_{\left| J\right|  =t+1}|\varepsilon^{J}_{iL}|\cdot c_{K,J},
 \end{equation}
where   $\varepsilon^{J}_{iL}$ and $c_{K,J}$ are given in (\ref{230204e3}) and (\ref{defnition of general st froms}), respectively.

 \begin{remark}\label{230329r1}
 By (\ref{230204e3}), it is easy to see that, for any $i\in\mathbb{N} $ and strictly increasing multi-indices $K$ and $L$ with $\left| K\right|  =s$ and $\left| L\right|  =t$, there is at most one non-zero term in each of the series $\sum\limits^{\prime }_{\left| J\right|  =t+1}\varepsilon^{J}_{iL}\cdot f_{K,J}$ and $\sum\limits^{\prime }_{\left| J\right|  =t+1}|\varepsilon^{J}_{iL}|\cdot c_{K,J}$ in (\ref{gener111}).
 \end{remark}
 We need the following technical result.
\begin{lemma}\label{T*formula}
Assume that $ f=\sum\limits^{\prime }_{\left| I\right|  =s,\left| J\right|  =t+1} f_{I,J}dz_{I}\wedge d\overline{z_{J}} \in D_{T^{\ast }}$, $\phi \in C^{1}_0\left( V\right)$, and
\begin{equation}\label{phi begin to D condition}
\int_{V} \left| \phi \right|^{2}  e^{-w_{1}}\,\mathrm{d}P<\infty \ \ \text{ and } \ \ \sum^{\infty }_{j=1} c_{K,jL}\int_{V} \left| \partial_{j} \phi \right|^{2}  e^{-w_{2}}\,\mathrm{d}P<\infty
\end{equation}
for some given strictly increasing multi-indices $K$ and $L$  with $\left| K\right|=s$ and $\left| L\right|=t$.
Then,
$$
c_{K,L}\int_{V} \left( T^{\ast }f\right)_{K,L}  \phi e^{-w_{1}}\,\mathrm{d}P=\left( -1\right)^{s}  \sum_{j=1}^{\infty} c_{K,jL} \int_{V}\partial_{j} \phi \cdot f_{K,jL}e^{-w_{2}}\,\mathrm{d}P.
$$
\end{lemma}
\begin{proof}
Let $g\triangleq \sum\limits^{\prime }_{\left| I_2\right|  =s,\left| L_2\right|  =t} g_{I_2,L_2}dz_{I_2}\wedge d\overline{z_{L_2}}$, where $g_{K,L}\triangleq \overline{\phi}$ and $g_{I_2,L_2}\triangleq 0$ for any other strictly increasing multi-indices $I_2$ and $L_2$ with $\left| I_2\right|=s$ and $\left| L_2\right|=t$.  Fix any $ \varphi \in C^{\infty }_{0,F}\left( V\right)$ and strictly increasing multi-indices $I_3$ and $J_3$ with $\left| I_3\right|=s$ and $\left| J_3\right|=t+1$. If $I_3\neq K$ or there does not exist $j_0\in\mathbb{N}$ such that $\{j_0\}\cup L=J_3$, then we have
$$
\left( -1\right)^{s+1}  \int_{V} \sum^{\infty }_{i=1} \sum^{\prime }_{\left| L_3\right|  =t} \varepsilon^{J_3}_{iL_3} g_{I_3,L_3}\overline{\delta_{i} \varphi }\,\mathrm{d}P=0.
$$
If $I_3=K$ and there exists $j_0\in\mathbb{N}$ such that $\{j_0\}\cup L=J_3$, then we have
$$
\begin{array}{ll}
\displaystyle \left( -1\right)^{s+1}  \int_{V} \sum^{\infty }_{i=1} \sum^{\prime }_{\left| L_3\right|  =t} \varepsilon^{J_3}_{iL_3} g_{K,L_3}\overline{\delta_{i} \varphi }\,\mathrm{d}P\\[3mm]
\displaystyle =\left( -1\right)^{s+1}  \int_{V} \varepsilon^{J_3}_{j_0L}\cdot \overline{\phi}\cdot\overline{\delta_{j_0} \varphi }\,\mathrm{d}P=\left( -1\right)^{s}  \int_{V} \varepsilon^{J_3}_{j_0L}\cdot \left(\overline{\partial_{j_0}}\overline{\phi}\right)\cdot\overline{ \varphi }\,\mathrm{d}P=\left( -1\right)^{s}  \int_{V}\sum_{j=1}^{\infty} \varepsilon^{J_3}_{jL}\cdot \left(\overline{\partial_{j}}\overline{\phi}\right)\cdot\overline{ \varphi }\,\mathrm{d}P,
\end{array}
$$
where the second equality follows from Corollary \ref{integration by Parts or deltai}.
Hence, by the condition \eqref{phi begin to D condition}, we see that $g\in L^{2}_{(s,t)}\left( V,w_{1}\right)$ and
$$
(-1)^s\sum^{\prime }_{\left| J\right|  =t+1} \sum_{j=1}^{\infty}\varepsilon^{J}_{jL}\overline{\partial_{j} }\overline{\phi} dz_{K}\wedge d\,\overline{z_{J}}\in L^{2}_{(s,t+1)}\left( V,w_{2}\right).
$$
Then according to  Definition \ref{def for db}, as well as Definition \ref{definition of T S}, we conclude that $g\in D_T$ and
$
Tg = (-1)^s\sum^{\prime }_{\left| J\right|  =t+1} \sum_{j=1}^{\infty}\varepsilon^{J}_{jL}\overline{\partial_{j} }\overline{\phi} dz_{K}\wedge d\,\overline{z_{J}}$.
By $(T^{\ast }f,g)_{L^2_{(s,t)}(V,w_1)}=
(f,Tg)_{L^2_{(s,t+1)}(V,w_2)}$, using the fact that $\varepsilon^{J}_{jL}=\varepsilon^{J}_{jL}|\varepsilon^{J}_{jL}|$ and noting Remark \ref{230329r1}, we have
$$
\begin{array}{ll}
\displaystyle
c_{K,L}\int_{V} \left( T^{\ast }f\right)_{K,L} \phi e^{-w_{1}}\,\mathrm{d}P
=\left( -1\right)^{s}  \sum^{\prime }_{\left| J\right|  =t+1} \sum_{j=1}^{\infty} c_{K,J}\varepsilon^{J}_{jL}\int_{V}f_{K,J}\cdot\overline{\overline{ \partial_{j}} \overline{\phi}} e^{-w_{2}}\,\mathrm{d}P\\[3mm]
\displaystyle =\left( -1\right)^{s}  \sum^{\prime }_{\left| J\right|  =t+1} \sum_{j=1}^{\infty} c_{K,J}\varepsilon^{J}_{jL}\int_{V}f_{K,J}\cdot\partial_{j} \phi e^{-w_{2}}\,\mathrm{d}P =\left( -1\right)^{s} \sum_{j=1}^{\infty} c_{K,jL}\int_{V}f_{K,jL}\cdot\partial_{j}\phi e^{-w_{2}}\,\mathrm{d}P,
\end{array}
$$
which completes the proof of Lemma \ref{T*formula}.
\end{proof}


Now we derive an explicit expression of $T^*$ (Recall (\ref{gener111}) for $f_{I,iL}$):
\begin{proposition}\label{general formula of T*}
Suppose that $f=\sum\limits^{\prime }_{\left| I\right|  =s,\left| J\right|  =t+1} f_{I,J}dz_{I}\wedge d\overline{z_{J}}\in L^{2}_{(s,t+1)}\left( V,w_{2}\right)$, where $f_{I,J}=0$ for all  strictly
increasing multi-indices $I$ and $J$ with $\max \left\{ I\bigcup J\right\}  >n_{0}$ for some $n_{0}\in \mathbb{N}$ and $f_{I,J}\in C^{1}_{0,F}\left( V\right)$ for all  strictly
increasing multi-indices $I$ and $J$ satisfying $\max \left\{ I\bigcup J\right\} \leqslant n_{0}$.
Then, $f\in D_{T^{\ast }}$ and
$$
T^{\ast }f= \left( -1\right)^{s+1}  e^{w_{1}-w_{2}} \sum^{\prime }_{\left| I\right|  =s,\left| L\right|  =t,\max \left\{ I\bigcup L\right\}  \leqslant n_{0}}\sum^{n_{0}}_{i=1} \frac{c_{I,iL}}{c_{I,L}}\cdot (\delta_{i} f_{I,iL}-f_{I,iL}\cdot \partial_{i} w_{2})dz_{I}\wedge d\overline{z_{L}}.
$$
\end{proposition}
\begin{proof}
Let
$$
g\triangleq  \left( -1\right)^{s+1}  e^{w_{1}-w_{2}} \sum^{\prime }_{\left| I\right|  =s,\left| L\right|  =t,\max \left\{ I\bigcup L\right\}  \leqslant n_{0}}\sum^{n_{0}}_{i=1} \frac{c_{I,iL}}{c_{I,L}}\cdot(\delta_{i} f_{I,iL}-f_{I,iL}\cdot\partial_{i} w_{2})dz_{I}\wedge d\overline{z_{L}}.
$$
Then, $g\in L^{2}_{(s,t)}\left( V,w_{1}\right)$. We will prove that $\left( Tu,f\right)_{L^{2}_{(s,t+1)}\left( V,w_{2}\right)  }  =\left( u,g\right)_{L^{2}_{(s,t)}\left( V,w_{1}\right)  } $ for all $u\in D_{T}$, which implies that $f\in D_{T^*}$ and $T^*f=g$. Since $f_{I,J}e^{-w_{2}}\in C_0^1(V)$ for all strictly
increasing multi-indices $I$ and $J$, by the conclusion (1) of Proposition \ref{weak equality for more test functions}, using the equality $\varepsilon^{J}_{jL}=\varepsilon^{J}_{jL}|\varepsilon^{J}_{jL}|$ and recalling Remark \ref{230329r1}, we obtain that
$$
\begin{array}{ll}
\displaystyle
			\left( Tu,f\right)_{L^{2}_{(s,t+1)}\left( V,w_{2}\right)  }   =\sum^{\prime }_{\left| I\right|  =s} \sum^{\prime }_{\left| J\right|  =t+1} c_{I,J}\int_{V} \left( Tu\right)_{I,J}  \overline{f_{I,J}} e^{-w_{2}}\,\mathrm{d}P\\[3mm]
\displaystyle =\left( -1\right)^{s+1}  \sum^{\prime }_{\left| I\right|  =s,\left| J\right|  =t+1} \sum^{\prime }_{\left| L\right|  =t} \sum^{\infty }_{i=1} c_{I,J}\varepsilon^{J}_{iL} \int_{V} u_{I,L}\overline{\delta_{i} (f_{I,J}e^{-w_{2}})} \,\mathrm{d}P\\[3mm]
\displaystyle =\left( -1\right)^{s+1}  \sum^{\prime }_{\left| I\right|  =s,\left| J\right|  =t+1} \sum^{\prime }_{\left| L\right|  =t} \sum^{\infty }_{i=1} c_{I,J}\varepsilon^{J}_{iL} \int_{V} u_{I,L}e^{w_{1}-w_{2}}(\overline{\delta_{i} f_{I,J}-f_{I,J}\partial_{i} w_{2}}) e^{-w_{1}}\,\mathrm{d}P\\[3mm]
\displaystyle =\left( -1\right)^{s+1}  \sum^{\prime }_{\left| I\right|  =s} \sum^{\prime }_{\left| L\right|  =t} c_{I,L}\int_{V} u_{I,L}\sum^{\prime }_{\left| J\right|  =t+1} \sum^{\infty }_{i=1}\varepsilon^{J}_{iL}\cdot \frac{c_{I,J}}{c_{I,L}}\cdot e^{w_{1}-w_{2}}(\overline{\delta_{i} f_{I,J}-f_{I,J}\partial_{i} w_{2}} )e^{-w_{1}}\,\mathrm{d}P\\[3mm]
\displaystyle =\left( -1\right)^{s+1}  \sum^{\prime }_{\left| I\right|  =s} \sum^{\prime }_{\left| L\right|  =t} c_{I,L}\int_{V} u_{I,L} \sum^{\infty }_{i=1}  \frac{c_{I,iL}}{c_{I,L}}\cdot e^{w_{1}-w_{2}}(\overline{\delta_{i} f_{I,iL}-f_{I,iL}\partial_{i} w_{2}} )e^{-w_{1}}\,\mathrm{d}P.
\end{array}
$$
Meanwhile, since $f_{I,J}=0$ for all  strictly
increasing multi-indices $I$ and $J$ with $\max \left\{ I\cup J\right\}  >n_{0}$, and by noting the definition of $f_{I,iL}$ in (\ref{gener111}), we obtain that
$$
\begin{array}{ll}
\displaystyle \left( -1\right)^{s+1}  \sum^{\prime }_{\left| I\right|  =s} \sum^{\prime }_{\left| L\right|  =t} c_{I,L}\int_{V} u_{I,L} \sum^{\infty }_{i=1}  \frac{c_{I,iL}}{c_{I,L}}\cdot e^{w_{1}-w_{2}}(\overline{\delta_{i} f_{I,iL}-f_{I,iL}\partial_{i} w_{2}} )e^{-w_{1}}\,\mathrm{d}P\\[3mm]
\displaystyle =\left( -1\right)^{s+1} \sum^{\prime }_{\left| I\right|  =s,\left| L\right|  =t,\max \left\{ I\bigcup L\right\}  \leqslant n_{0}} c_{I,L}\int_{V} u_{I,L} \sum^{n_0 }_{i=1}  \frac{c_{I,iL}}{c_{I,L}}\cdot e^{w_{1}-w_{2}}(\overline{\delta_{i} f_{I,iL}-f_{I,iL}\partial_{i} w_{2}} )e^{-w_{1}}\,\mathrm{d}P\\[3mm]
\displaystyle =\left( u,g\right)_{L^{2}_{(s,t)}\left( V,w_{1}\right)  }.
\end{array}
$$
Thus, the proof of Proposition \ref{general formula of T*} is completed.
\end{proof}

The following proposition will be useful in the sequel.
\begin{proposition}\label{support argument}
Under Condition \ref{230424c1}, the following results hold.
	\begin{itemize}
		\item[$(1)$] Suppose that $f\in D_{T^{\ast}}$, $\supp f \stackrel{\circ}{\subset} V_{r}^{o}$ for some $r\in(0,+\infty)$, and $\sup\limits_{i\in\mathbb{N}}c_{I,iL}<\infty$ for any strictly increasing multi-indices $I$ and $L$ with $\left| I\right|=s$ and $\left| L\right|=t$ (See (\ref{gener111}) for $ c_{I,iL}$).
Then $\supp (T^{\ast}f) \stackrel{\circ}{\subset} V_{r}^{o}$;
		\item[$(2)$] If $f\in D_{T}$ and $\supp f \stackrel{\circ}{\subset} V_{r}^{o}$ for some $r\in(0,+\infty)$, then $\supp (Tf)  \stackrel{\circ}{\subset} V_{r}^{o}$.
	\end{itemize}
\end{proposition}
\begin{proof}
(1) Without loss of generality, we assume that $f\not\equiv0$. Hence $\supp f\not=\emptyset$. By Lemma \ref{221018lem1}, $d\left( \supp f,\partial V^{o}_{r}\right) >0$. Let $U\triangleq\left\{ \textbf{z}\in V^{o}_{r}:\;d\left( \textbf{z},\supp f\right)  <\frac{1}{2} d\left( \supp f,\partial V^{o}_{r}\right)  \right\}$.  We claim that $d\left( U,\partial V^{o}_{r}\right)  \not=0$. Otherwise, we could find $\{\textbf{z}^{1 }_{n}\}_{n=1}^{\infty}\subset U$ and $\{\textbf{z}^{2 }_{n}\}_{n=1}^{\infty}\subset \partial V^{o}_{r} $ so that $\lim\limits_{n\rightarrow \infty } d\left( \textbf{z}^{1  }_{n},\textbf{z}^{2  }_{n}\right)  =0.$ Then, there exists $n_0\in\mathbb{N}$ such that for all $n\geqslant n_0$, it holds that $d\left( \textbf{z}^{1 }_{n},\textbf{z}^{2 }_{n}\right)<\frac{1}{2} d\left( \supp f,\partial V^{o}_{r}\right)$, and hence
$$
 d\left( \textbf{z}^{1  }_{n},\supp f\right)  \geqslant d\left( \textbf{z}^{2  }_{n},\supp f\right)  -d\left( \textbf{z}^{1  }_{n},\textbf{z}^{2  }_{n}\right)
 \geqslant d\left( \partial V^{o}_{r},\supp f\right)  -d\left( \textbf{z}^{1}_{n},\textbf{z}^{2  }_{n}\right)
>\frac{1}{2} d\left(\supp f,\partial V^{o}_{r}\right),
$$
which leads to a contradiction. Thus $ d\left( U,\partial V^{o}_{r}\right)  >0$. Hence $ d\left( \overline{U},\partial V^{o}_{r}\right)  >0$, and by  Lemma \ref{221018lem1} again, $\overline{U}\stackrel{\circ}\subset V^{o}_{r}$. Further, for all strictly increasing multi-indices $I$ and $L$ with $\left| I\right|  =s$ and $\left| L\right|  =t$ and $ \phi \in C^{\infty }_{0,F}\left( V\backslash \overline{U}\right)$ (which can be viewed as a subset of $C^{1 }_{0,F}\left( V\right)$ by the zero extension to $\overline{U}$), noting that
$$
\int_{V} \left| \phi \right|^{2}  e^{-w_{1}}\,\mathrm{d}P<\infty \text{ and } \sum^{\infty }_{j=1} c_{I,jL}\int_{V} \left| \partial_{j} \phi \right|^{2}  e^{-w_{2}}\,\mathrm{d}P\leqslant \sup_{i\in\mathbb{N}}c_{I,iL} \sum^{\infty }_{j=1}\int_{V} \left| \partial_{j} \phi \right|^{2}  e^{-w_{2}}\,\mathrm{d}P<\infty ,
$$
by Lemma \ref{T*formula}, we have  (See (\ref{gener111}) for $ f_{I,jL}$)
 $$
 c_{I,L}\int_{V\backslash\overline{U}  } \left( T^{\ast }f\right)_{I,L}  \phi e^{-w_{1}}\,\mathrm{d}P=\left( -1\right)^{s}  \sum^{\infty }_{j=1} \int_{V} c_{I,jL}f_{I,jL}\partial_{j} \phi e^{-w_{2}}\,\mathrm{d}P=0.
 $$
By Corollary \ref{density2}, we have $\left( T^{\ast }f\right)_{I,L}  =0$ almost everywhere on $V/\overline{U}$ with respect to $P$. Therefore, $\supp T^{\ast}f  \subset \overline{U}\stackrel{\circ}{\subset} V_{r}^{o}$.

\medskip

(2) Similarly to the above and the proof of the conclusion (1) of Proposition \ref{weak equality for more test functions}, we can show the conclusion (2) of Proposition \ref{support argument}.
\end{proof}

Combining Lemma \ref{T*formula} and Proposition \ref{support argument}, we have the following variant of Lemma \ref{T*formula}.
\begin{corollary}\label{weak T*formula}
Assume that Condition \ref{230424c1} holds, $ f=\sum\limits^{\prime }_{\left| I\right|  =s,\left| J\right|  =t+1} f_{I,J}dz_{I}\wedge d\overline{z_{J}} \in D_{T^{\ast }}$ with $\supp f \stackrel{\circ}{\subset} V$,  $\sup\limits_{i\in\mathbb{N}}c_{I,iL}<\infty$, $\phi \in C^{1}_0\left( V\right)$, and for all $E\in \mathscr{B}(V)$ with $E \stackrel{\circ}{\subset}  V$,
\begin{equation}\label{phi begin to D condition'}
\int_{E} \left| \phi \right|^{2}  e^{-w_{1}}\,\mathrm{d}P<\infty \text{ and } \sum^{\infty }_{j=1} c_{I,jL}\int_{E} \left| \partial_{j} \phi \right|^{2}  e^{-w_{2}}\,\mathrm{d}P<\infty
\end{equation}
hold for some strictly increasing multi-indices $I$ and $L$ with $\left| I\right|=s$ and $\left| L\right|=t$.
Then,
\begin{equation}\label{230221e1}
c_{I,L}\int_{V} \left( T^{\ast }f\right)_{I,L}  \phi e^{-w_{1}}\,\mathrm{d}P=\left( -1\right)^{s}  \sum_{j=1}^{\infty} c_{I,jL} \int_{V}\partial_{j} \phi \cdot f_{I,jL}e^{-w_{2}}\,\mathrm{d}P.
\end{equation}
\end{corollary}
\begin{proof}
Recall that $V$ is a pseudo-convex domain in $\ell^{2}$ and $\eta(\in C_F^\infty(U))$ is a plurisubharmonic exhaustion function on $V$ (See Condition \ref{230424c1}). By the conclusion (3) of Proposition \ref{properties on pseudo-convex domain}, we may find $r\in(0,+\infty)$ such that $\supp f \stackrel{\circ}{\subset} V_r^o$. Choose $h\in C^{\infty}(\mathbb{R})$ such that $0\leqslant h(\cdot)\leqslant 1$, $h(x)=0$ for all $x\geqslant r+1$ and $h(x)=1$ for all $x\leqslant r$. Let $\Phi\triangleq h(\eta)$. Then $\Phi\in C_{0,F}^{\infty}(V)$. By (\ref{phi begin to D condition'}), we see that
$$
\int_{V} \left| \phi \cdot\Phi \right|^{2}  e^{-w_{1}}\,\mathrm{d}P=\int_{V_{r+1}^o} \left| \phi \cdot\Phi \right|^{2}  e^{-w_{1}}\,\mathrm{d}P
\leqslant \int_{V_{r+1}^o} \left| \phi \right|^{2}  e^{-w_{1}}\,\mathrm{d}P<\infty
$$
and
$$
\begin{array}{ll}
\displaystyle \sum^{\infty }_{j=1} c_{I,jL}\int_{V} \left| \partial_{j} (\phi \cdot \Phi )\right|^{2}  e^{-w_{2}}\,\mathrm{d}P \\[3mm]
\displaystyle =\sum^{\infty }_{j=1} c_{I,jL}\int_{V} \left| \Phi\cdot\partial_{j}  \phi +\phi \cdot \partial_{j} \Phi \right|^{2}  e^{-w_{2}}\,\mathrm{d}P \\[3mm]
\displaystyle \leqslant 2\sum^{\infty }_{j=1} c_{I,jL}\int_{V} \left| \Phi\cdot\partial_{j}  \phi  \right|^{2}  e^{-w_{2}}\,\mathrm{d}P+2\sum^{\infty }_{j=1} c_{I,jL}\int_{V} \left|  \phi \cdot \partial_{j} \Phi \right|^{2}  e^{-w_{2}}\,\mathrm{d}P\\[3mm]
\displaystyle =2\sum^{\infty }_{j=1} c_{I,jL}\int_{V_{r+1}^o} \left| \Phi\cdot\partial_{j}  \phi  \right|^{2}  e^{-w_{2}}\,\mathrm{d}P+2\sum^{\infty }_{j=1} c_{I,jL}\int_{V_{r+1}^o} \left|  \phi \cdot \partial_{j} \Phi \right|^{2}  e^{-w_{2}}\,\mathrm{d}P\\[3mm]
\displaystyle \leqslant 2\sum^{\infty }_{j=1} c_{I,jL}\int_{V_{r+1}^o} \left|\partial_{j}  \phi  \right|^{2}  e^{-w_{2}}\,\mathrm{d}P+2 \cdot \left(\sup_{j\in\mathbb{N}} c_{I,jL} \right)\cdot \left(\sup_{V_{r+1}^o}\sum^{\infty }_{j=1}\left|\partial_{j} \Phi \right|^{2}\right)\cdot\int_{V_{r+1}^o} \left|  \phi  \right|^{2}  e^{-w_{2}}\,\mathrm{d}P<\infty.
\end{array}
$$
Thus $\phi\cdot \Phi$ satisfies the condition \eqref{phi begin to D condition} in Lemma \ref{T*formula}, and hence
$$
c_{I,L}\int_{V} \left( T^{\ast }f\right)_{I,L}  \phi\cdot \Phi \cdot e^{-w_{1}}\,\mathrm{d}P=\left( -1\right)^{s}  \sum_{j=1}^{\infty} c_{I,jL} \int_{V}\partial_{j}( \phi\cdot \Phi)\cdot f_{I,jL}e^{-w_{2}}\,\mathrm{d}P.
$$
Since $\supp f \stackrel{\circ}{\subset} V_r^o$ and $\sup\limits_{i\in\mathbb{N}}c_{I,iL}<\infty$, by the conclusion (1) of Proposition \ref{support argument}, we see that $\supp T^*f \stackrel{\circ}{\subset} V_r^o$. Then,
$$
\begin{array}{ll}
\displaystyle
c_{I,L}\int_{V} \left( T^{\ast }f\right)_{I,L}  \phi e^{-w_{1}}\,\mathrm{d}P
= c_{I,L}\int_{V_r^o} \left( T^{\ast }f\right)_{I,L}  \phi e^{-w_{1}}\,\mathrm{d}P\\[3mm]
\displaystyle = c_{I,L}\int_{V_r^o} \left( T^{\ast }f\right)_{I,L}  \phi\cdot \Phi \cdot e^{-w_{1}}\,\mathrm{d}P= c_{I,L}\int_{V} \left( T^{\ast }f\right)_{I,L}  \phi\cdot \Phi \cdot e^{-w_{1}}\,\mathrm{d}P,
\end{array}
$$
and similarly $\left( -1\right)^{s}  \sum\limits_{j=1}^{\infty} c_{I,jL} \int_{V}\partial_{j}( \phi\cdot \Phi)\cdot f_{I,jL}e^{-w_{2}}\,\mathrm{d}P
=
\left( -1\right)^{s}  \sum\limits_{j=1}^{\infty} c_{I,jL} \int_{V}(\partial_{j} \phi )\cdot f_{I,jL}e^{-w_{2}}\,\mathrm{d}P$.
Thus we obtain the desired equality (\ref{230221e1}).
This completes the proof of Corollary \ref{weak T*formula}.
\end{proof}

To end this section, we put
\begin{equation}\label{230419e21}
\mathscr{M}\triangleq\left\{ \mathfrak{m}\in C^{1}_b\left( V\right)  :\;\sup_{V}  \bigg(\sum^{\infty }_{i=1}  (|\partial_{i} \mathfrak{m} |^{2}+ | \overline{\partial_{i} } \mathfrak{m}|^{2} )\cdot e^{w_{1}-w_{2}} \bigg) <\infty \right\}.
\end{equation}
Also, we introduce the following assumption (Recall (\ref{gener111}) for $c_{I,iJ}$).
\begin{condition}\label{230423ass1}
$\displaystyle c_1^{s,t}\triangleq \sup_{|I|=s,|J|=t,i\in\mathbb{N}}\frac{ c_{I,iJ} }{c_{I,J}}<\infty$ for all $s,t\in\mathbb{N}_0$.
\end{condition}

The following result will be useful later.
\begin{proposition}\label{general multiplitier}
Under Conditions \ref{230424c1} and \ref{230423ass1}, for any $\mathfrak{m}\in \mathscr{M}$, it holds that
\begin{itemize}
		\item[$(1)$]  If $f\in D_{T}$, then $\mathfrak{m}\cdot f\in D_{T}$ and $T\left( \mathfrak{m}\cdot f\right)  =\mathfrak{m}\cdot (Tf)+(T\mathfrak{m}) \wedge f,$ where
$$
(T\mathfrak{m})\wedge f\triangleq \left( -1\right)^{s}  \sum^{\prime }_{\left| I\right|  =s} \sum^{\prime }_{\left| K\right|  =t+1} \left(\sum^{\infty }_{i=1} \sum^{\prime }_{\left| J\right|  =t} \varepsilon^{K}_{iJ} f_{I,J}\overline{\partial_{i} } \mathfrak{m} \right)dz_{I}\wedge d\overline{z_{K}};
$$

\item[$(2)$] If $g\in D_{T^{\ast }}$, then $\mathfrak{m} \cdot g\in D_{{T}^{\ast }}$.
\end{itemize}
\end{proposition}

\begin{proof}
(1) Note that for any $\phi \in C^{\infty }_{0,F}\left( V\right)$, $\mathfrak{m}\phi\in C_0^1(V)$. Hence, for any strictly increasing multi-indices $I$ and $K$ with $\left| I\right|  =s$ and $\left| K\right|  =t+1$,  by the conclusion (1) of Proposition \ref{weak equality for more test functions}, we have
$$
		\begin{array}{ll}
\displaystyle
			\left( -1\right)^{s+1}  \int_{V} \sum^{\infty }_{i=1} \sum^{\prime }_{\left| J\right|  =t} \varepsilon^{K}_{iJ} \mathfrak{m} f_{I,J}\overline{\delta_{i} \phi } \,\mathrm{d}P
=\left( -1\right)^{s+1}  \int_{V} \sum^{\infty }_{i=1} \sum^{\prime }_{\left| J\right|  =t} \varepsilon^{K}_{iJ} f_{I,J}\overline{(\delta_{i} \left( \overline{\mathfrak{m} } \phi \right)  -\partial_{i} \overline{\mathfrak{m} } \cdot \phi )}\,\mathrm{d}P\\[3mm]
\displaystyle =\int_{V} \mathfrak{m} (T f)_{I,K}\overline{\phi } \,\mathrm{d}P+\left( -1\right)^{s}  \int_{V} \sum^{\infty }_{i=1} \sum^{\prime }_{\left| J\right|  =t} \varepsilon^{K}_{iJ} f_{I,J}\overline{\partial_{i} } \mathfrak{m} \cdot \overline{\phi } \,\mathrm{d}P\\[3mm]
\displaystyle =\int_{V} \mathfrak{m} (T f)_{I,K}\overline{\phi }\,\mathrm{d}P+  \int_{V}((T\mathfrak{m} )\wedge f)_{I,K} \cdot \overline{\phi } \,\mathrm{d}P.
		\end{array}
$$
By (\ref{gener111}), it follows that (Recall Condition \ref{230423ass1} for $c_1^{s,t}$)
\begin{equation}\label{general inequality for Teta wedege f}
\begin{array}{ll}
\displaystyle\left| \left|(T\mathfrak{m}) \wedge f\right|  \right|_{L^{2}_{\left( s,t+1\right)  }\left( V,w_{2}\right)  }^2  \\[3mm]\displaystyle
\leqslant  \left( t+1\right)  \sum^{\prime }_{\left| I\right|  =s,\left| K\right|  =t+1} c_{I,K}\int_{V} \sum^{\infty }_{i=1} \sum^{\prime }_{\left| J\right|  =t} |\varepsilon^{K}_{iJ} f_{I,J}\overline{\partial_{i} } \mathfrak{m} |^{2}e^{-w_{2}}\,\mathrm{d}P \\[3mm]\displaystyle=\left( t+1\right)  \sum^{\prime }_{\left| I\right|  =s,\left| K\right|  =t+1}\int_{V} \sum^{\infty }_{i=1} \sum^{\prime }_{\left| J\right|  =t}\frac{ c_{I,K}\cdot|\varepsilon^{K}_{iJ}|}{c_{I,J}}|\overline{\partial_{i} } \mathfrak{m}|^2\cdot c_{I,J}\cdot|f_{I,J}|^{2}e^{-w_{2}}\,\mathrm{d}P \\[3mm]\displaystyle= \left( t+1\right)  \sum^{\prime }_{\left| I\right|  =s }\int_{V} \sum^{\infty }_{i=1} \sum^{\prime }_{\left| J\right|  =t}\frac{ c_{I,iJ} }{c_{I,J}}|\overline{\partial_{i} } \mathfrak{m}|^2\cdot c_{I,J}\cdot|f_{I,J}|^{2}e^{-w_{2}}\,\mathrm{d}P \\[3mm]\displaystyle\leqslant \left( t+1\right) \cdot c_1^{s,t}\cdot\int_{V} \left(\sum^{\infty }_{i=1} |\overline{\partial_{i} } \mathfrak{m} |^{2}\cdot e^{w_{1}-w_{2}}\right)\cdot \left(\sum^{\prime }_{\left| I\right|  =s}\sum^{\prime }_{\left| J\right|  =t} c_{I,J}\cdot\left| f_{I,J}\right|^{2}\right)  e^{-w_{1}}\,\mathrm{d}P\\[3mm]\displaystyle\leqslant \left( t+1\right) \cdot c_1^{s,t}\cdot\sup_{V} \sum^{\infty }_{i=1} (|\overline{\partial_{i} } \mathfrak{m} |^{2}e^{w_{1}-w_{2}})\cdot \left| \left| f\right|  \right|^{2}_{L^{2}_{\left( s,t\right)  }\left( V,w_{1}\right)  } <\infty.
\end{array}
\end{equation}
Also, $||\mathfrak{m}\cdot(Tf)||_{L^{2}_{\left( s,t+1\right)  }\left( V,w_{2}\right)  }  \leqslant \sup\limits_{V}|\mathfrak{m}|\cdot ||Tf||_{L^{2}_{\left( s,t+1\right)  }\left( V,w_{2}\right)  }$. Therefore, $\mathfrak{m} \cdot f \in D_{T}$ and $T ( \mathfrak{m}\cdot f )  =\mathfrak{m}\cdot (Tf)+(T\mathfrak{m}) \wedge f.$

\medskip

(2) Clearly, $\overline{\mathfrak{m}}\in \mathscr{M}$. For any $u\in D_{T}$, by the conclusion (1), it holds that $\overline{\mathfrak{m}}\cdot u\in D_T$ and $T\left( \overline{\mathfrak{m}}\cdot u\right)  =\overline{\mathfrak{m}}\cdot (Tu)+(T\overline{\mathfrak{m}}) \wedge u$ and hence
$$
\begin{array}{ll}
\displaystyle
			\left( Tu,\mathfrak{m} \cdot g\right)_{L^{2}_{(s,t+1)}\left( V,w_{2}\right)  }  =\left( \overline{\mathfrak{m} }\cdot (Tu),g\right)_{L^{2}_{(s,t+1)}\left( V,w_{2}\right)  } \\[3mm]
\displaystyle =\left(T\left( \overline{\mathfrak{m}}\cdot u\right),g\right)_{L^{2}_{\left( s,t+1\right)  }\left( V,w_{2}\right)  }  -\left( (T\overline{\mathfrak{m}}) \wedge u,g\right)_{L^{2}_{\left( s,t+1\right)  }\left( V,w_{2}\right)  } \\[3mm]
\displaystyle =\left( u,\mathfrak{m}\cdot T^{\ast }g\right)_{L^{2}_{\left( s,t\right)  }\left( V,w_{1}\right)  }  -\left( (T\overline{\mathfrak{m}}) \wedge u,g\right)_{L^{2}_{\left( s,t+1\right)  }\left( V,w_{2}\right)  }.
		\end{array}
$$
With \eqref{general inequality for Teta wedege f}, the above implies that
$$
\begin{array}{ll}
\displaystyle
| ( Tu,\mathfrak{m}\cdot g )_{L^{2}_{(s,t+1)}\left( V,w_{2}\right)  }  |
\leqslant |( u,\mathfrak{m}\cdot T^{\ast }g)_{L^{2}_{\left( s,t\right)  }\left( V,w_{1}\right)  }| +|( (T\overline{\mathfrak{m}}) \wedge u,g)_{L^{2}_{\left( s,t+1\right)  }\left( V,w_{2}\right)  }|\\[3mm]
\displaystyle \leqslant \left(  || \mathfrak{m}\cdot T^{\ast }g||_{L^{2}_{\left( s,t\right)  }\left( V,w_{1}\right)  }  +C\cdot\left| \left| g\right|  \right|_{L^{2}_{\left( s,t+1\right)  }\left( V,w_{2}\right)  }  \right) \cdot \left| \left| u\right|  \right|_{L^{2}_{\left( s,t\right)  }\left( V,w_{1}\right)  }\\[3mm]
\displaystyle \leqslant \left(  \sup_V|\mathfrak{m}|\cdot || T^{\ast }g ||_{L^{2}_{\left( s,t\right)  }\left( V,w_{1}\right)  }  +C  \cdot\left| \left| g\right|  \right|_{L^{2}_{\left( s,t+1\right)  }\left( V,w_{2}\right)  }  \right) \cdot \left| \left| u\right|  \right|_{L^{2}_{\left( s,t\right)  }\left( V,w_{1}\right)  },
\end{array}
$$
where $C\triangleq \sqrt{\left( t+1\right) \cdot c_1^{s,t}\cdot \sup\limits_{V} \sum\limits^{\infty }_{i=1} (\left| \partial_{i} \mathfrak{m}\right|^{2}  e^{w_{1}-w_{2}})}$. Then, by the definition of $T^*$ (e.g., \cite[Definition 4.2.2, p. 177]{Kra}), we have $\mathfrak{m}\cdot g\in D_{T^{\ast }}$. This completes the proof of Proposition \ref{general multiplitier}.
\end{proof}

\section{$L^2$ estimates for the $\overline{\partial}$ equations on pseudo-convex domains in $\ell^{2}$}\label{secx4}

The purpose of this section is to establish the key $L^2$ estimates for the $\overline{\partial}$ equations (on pseudo-convex domains in $\ell^{2}$) by means of some suitable approximations.



From now on, we fix a pseudo-convex domain  $V$ in $\ell^{2}$ and a plurisubharmonic exhaustion function $\eta(\in C_F^\infty(V))$ on $V$. We need to choose suitably the weight functions $w_{1},w_{2}$ and $w_{3}$ in Definition \ref{definition of T S}. For this purpose, motivated by  \cite[p. 543]{AGGM} (for the definition of $H_b(U)$ therein), \cite[p. 164]{Din99} (for the holomorphic functions of bounded type) and \cite[p. 798]{Rya87} (for the holomorphic mappings of bounded type), we introduce the following notion.
\begin{definition}\label{230726def1}
A complex valued function $f$ on a nonempty open subset $U$ of $\ell^2$ is called a locally bounded function on $U$ if $\sup\limits_{E} \left| f\right|  <\infty$ for all $E\stackrel{\circ}{\subset}U$.
\end{definition}

We also need the following simple result.
\begin{lemma}\label{majority funtion}
Suppose that Condition \ref{230424c1} holds and $f$ is a locally bounded function on $V$. Then, there exists $g\in C^{\infty}_{F}(V)$ such that $|f(\textbf{z})|\leqslant g(\textbf{z})$ for all $\textbf{z}\in V$.
\end{lemma}
\begin{proof}
 For each $j\in \mathbb{N}$, define $V_j$ as that in (\ref{230117e1}).  Then $V=\bigcup\limits^{\infty }_{j=1} V_{j}=V_1\bigcup \bigcup\limits^{\infty }_{j=1} \left(V_{j+1}\setminus V_j\right)$. Since $f$ is a locally bounded function on $V$, we have $c_j\triangleq \sup\limits_{\textbf{z}\in V_{j}}|f(\textbf{z})|<\infty$ for each $j\in\mathbb{N}$. Choose $h \in C^{\infty }\left( \mathbb{R} \right)$ such that $h(x)\geqslant c_{j+1}$ for any $x\in( j,j+1],\,j\in\mathbb{N}$, while $h(x)=c_1$ for any $x\leqslant 1$. Let $g\triangleq h (\eta)$. It is obvious that $g\in C^{\infty}_{F}(V)$ such that $|f(\textbf{z})|\leqslant g(\textbf{z})$ for all $\textbf{z}\in V$. This completes the proof of Lemma \ref{majority funtion}.
\end{proof}

For every $k\in\mathbb{N}$, let
$$
h_{k} \left( t\right) \triangleq\begin{cases}0,&t>k +1,\\ \left( t-k -1\right)^{2}  \left( 2t+1-2k \right),&k \leqslant t\leqslant k +1,\\ 1,&t<k.
\end{cases}
$$
Then, $h_{k }\in C^{1 }\left( \mathbb{R}\right)$, $0\leqslant h_{k}\leqslant 1$, $h_{k}(x)=1$ for all $x<k$, $h_{k } \left( x\right)  =0$ for all $x>k+1$ and $\left| h^{\prime }_{k} \right|  \leqslant  \frac{3}{2}$. We construct a sequence $\{X_{k}\}_{k=1}^\infty$ of cut-off functions on $V$ as follows:
\begin{equation}\label{section2}
	X_{k}\triangleq h_{k }( \eta)  \in C^{1}_{0,F}\left( V\right).
\end{equation}
Note that for each $k\in\mathbb{N}$,
$$
\ln \left( 1+\sum\limits^{\infty }_{i=1} |\overline{\partial_{i} } X_{k}|^{2}\right)
=\ln \left( 1+\sum\limits^{\infty }_{i=1} |h_k^{'}(\eta)|^2\cdot|\overline{\partial_{i} }\eta|^{2}\right)
\leqslant \ln \left( 1+\frac{9}{4}\cdot\sum\limits^{\infty }_{i=1}|\overline{\partial_{i} }\eta|^{2}\right).
$$
Since $\ln \left( 1+\frac{9}{4}\cdot\sum\limits^{\infty }_{i=1}|\overline{\partial_{i} }\eta|^{2}\right)$ is a locally bounded function on $V$, by  Lemma \ref{majority funtion}, there exists $\psi \in C^{\infty }_{F}\left( V\right)$ such that $\ln \left( 1+\frac{9}{4}\cdot\sum\limits^{\infty }_{i=1}|\overline{\partial_{i} }\eta|^{2}\right)  \leqslant \psi $. Therefore,
\begin{equation}\label{majority function}
\sum^{\infty }_{i=1} |\overline{\partial_{i} } X_{k}|^{2}\leqslant e^{\psi },\quad \forall\; k\in\mathbb{N}.
\end{equation}
In the rest of this paper, unless otherwise stated, we shall choose the weight functions in Definition \ref{definition of T S} as follows:
\begin{equation}\label{weight function}
	w_{1} =\varphi -2\psi ,\quad w_{2} =\varphi -\psi ,\quad w_{3} =\varphi ,
\end{equation}
where $\varphi \in C^{2}_{F}\left( V\right)$ is a given real-valued function (satisfying appropriate conditions to be given later).

Recall that $T$ and $S$ are given as that in Definition \ref{definition of T S}. We need to approximate the elements in $ D_{S}\bigcap D_{T^{\ast }}$ in a suitable way and the whole process is divided into two parts. The following proposition is the first part of our approximation process (Recall \eqref{weight function} for $w_1,w_2$ and $w_3$ and \eqref{section2} for $\{X_{k}\}_{k=1}^{\infty}$).
\begin{proposition}\label{general cut-off density3}
 Under Conditions \ref{230424c1} and \ref{230423ass1}, for each $f\in D_{S}\bigcap D_{T^{\ast }}$, it holds that $X_k\cdot f\in D_{S}\bigcap D_{T^{\ast }}$ for any $k\in\mathbb{N}$ and
$$
\lim_{k \rightarrow \infty }\Big( \left| \left| S\left( X_{k}\cdot f\right)  -X_{k}\cdot (Sf)\right|  \right|_{w_{3}}  +\left| \left| T^{\ast }\left( X_{k}\cdot f\right)  -X_{k }\cdot (T^{\ast }f)\right|  \right|_{w_{1}}  +\left| \left| X_{k}\cdot f-f\right|  \right|_{w_{2}} \Big) =0,
$$
where we simply denote $ \left| \left| \cdot \right|  \right|_{L^{2}_{\left( s,t+i-1\right)  }\left( V,w_{i}\right)  }$ by $\left| \left| \cdot \right|  \right|_{w_{i}}$ for each $i=1,2,3$.
\end{proposition}
\begin{proof}
By \eqref{section2}, we see that $\{X_{k}\}_{k=1}^{\infty}\subset \mathscr{M}$ (Recall (\ref{230419e21}) for
$\mathscr{M}$). Hence, similarly to the proof of the conclusion (1) of Proposition \ref{general multiplitier}, one can show that $X_{k}\cdot f\in D_{S}$  for each $k\in\mathbb{N}$ and $S\left( X_{k}\cdot f\right)  =X_{k}\cdot (Sf)+(SX_{k}) \wedge f,$ where
$$
(SX_{k})\wedge f\triangleq \left( -1\right)^{s}  \sum^{\prime }_{\left| I\right|  =s} \sum^{\prime }_{\left| K\right|  =t+2} \left(\sum^{\infty }_{i=1} \sum^{\prime }_{\left| J\right|  =t+1} \varepsilon^{K}_{iJ} f_{I,J}\overline{\partial_{i} } X_{k} \right)dz_{I}\wedge d\overline{z_{K}}.
$$
By the conclusion (2) of Proposition \ref{general multiplitier}, we have $X_k\cdot f\in D_{T^{\ast }}$, and hence $X_k\cdot f\in D_{S}\bigcap D_{T^{\ast }}$, and (Recall Condition \ref{230423ass1} for $c_1^{s,t+1}$)
\begin{eqnarray*}
&&\left| \left| S\left( X_{k}\cdot f\right)  -X_{k}\cdot(Sf)\right|  \right|^{2}_{w_{3} }
=\sum^{\prime }_{\left| I\right|  =s}\sum^{\prime }_{\left| K\right|  =t+2} c_{I,K}\int_{V} \left|\sum^{\infty }_{i=1} \sum^{\prime }_{\left| J\right|  =t+1} \varepsilon^{K}_{i,J} f_{I,J}\overline{\partial_{i} } X_{k }\right|^{2}e^{-w_{3}}\,\mathrm{d}P\\
&&\leqslant (t+2)\sum^{\prime }_{\left| I\right|  =s}\sum^{\prime }_{\left| K\right|  =t+2} c_{I,K}\int_{V}\sum^{\infty }_{i=1} \sum^{\prime }_{\left| J\right|  =t+1} \left|\varepsilon^{K}_{i,J} f_{I,J}\overline{\partial_{i} } X_{k }\right|^{2}e^{-w_{3}}\,\mathrm{d}P\\
&&=(t+2)\sum^{\prime }_{\left| I\right|  =s}\sum^{\prime }_{\left| K\right|  =t+2}\int_{V}\sum^{\infty }_{i=1} \sum^{\prime }_{\left| J\right|  =t+1} \frac{c_{I,K}\cdot|\varepsilon^{K}_{i,J}|}{c_{I,J}} \cdot c_{I,J}\cdot |f_{I,J}|^2\cdot|\overline{\partial_{i} }X_{k } |^{2}e^{-w_{3}}\,\mathrm{d}P\\
&&=(t+2)\sum^{\prime }_{\left| I\right|  =s}\int_{V}\sum^{\infty }_{i=1} \sum^{\prime }_{\left| J\right|  =t+1} \frac{c_{I,iJ}}{c_{I,J}} \cdot c_{I,J}\cdot |f_{I,J}|^2\cdot|\overline{\partial_{i} }X_{k } |^{2}e^{-w_{3}}\,\mathrm{d}P\\
&&\leqslant (t+2)\cdot c_1^{s,t+1}\cdot\sum^{\prime }_{\left| I\right|  =s}\sum^{\prime }_{\left| J\right|  =t+1}c_{I,J}\cdot\int_{V} |f_{I,J}|^2\cdot\sum^{\infty }_{i=1} |\overline{\partial_{i} }X_{k } |^{2}\cdot e^{-w_{3}}\,\mathrm{d}P.
\end{eqnarray*}
Recalling \eqref{majority function} and \eqref{weight function}, we have $\sum\limits^{\infty }_{i=1} |\overline{\partial_{i} } X_{k}|^{2}\leqslant e^{\psi }=e^{w_{3}-w_{2}}$ for all $k\in\mathbb{N}.$ Thus
$$
		 \sum^{\prime }_{\left| I\right|  =s,\left| J\right|  =t+1} c_{I,J} \left| f_{I,J}\right|^{2} \cdot\sum^{\infty }_{i=1}  |\overline{\partial_{i} } X_{k }|^{2} e^{-w_{3}} \leqslant\sum^{\prime }_{\left| I\right|  =s,\left| J\right|  =t+1} c_{I,J}  \left| f_{I,J}\right|^{2}  e^{-w_{2}}  .
$$
Note that for each $\textbf{z}\in V$, if $k>\eta(\textbf{z})$, then $|\overline{\partial_{i} } X_{k}\left( \textbf{z}\right)  |^{2}=0$ for all $i\in\mathbb{N}$.
Hence,
$$
\lim_{k \rightarrow \infty } \sum^{\infty }_{i=1} |\overline{\partial_{i} } X_{k }(\textbf{z})|^{2}=0,\quad  \forall\; \textbf{z}\in V.\
$$
By Lebesgue's Dominated Convergence Theorem, it follows that
$$
		\lim_{k\rightarrow \infty }\sum^{\prime }_{\left| I\right|  =s,\left| J\right|  =t+1} c_{I,J}\int_{V} \sum^{\infty }_{i=1} |\overline{\partial_{i} } X_{k}|^{2}\left| f_{I,J}\right|^{2}  e^{-w_{3}}\,\mathrm{d}P=0,
		$$
and hence $\lim\limits_{k\rightarrow \infty } \left| \left| S\left( X_{k}\cdot f\right)  -X_{k}\cdot(Sf)\right|  \right|^{2}_{w_{3}}  =0.$

On the other hand, for every $u\in D_{T}$, by Proposition \ref{general multiplitier}, and noting  \eqref{majority function} and \eqref{weight function} again, we obtain that
\begin{eqnarray*}
&&\left|\left( T^{\ast }\left( X_{k}\cdot f\right)  -X_{k}\cdot(T^{\ast }f),u\right)_{L^{2}_{\left( s,t\right)  }\left( V,w_{1}\right)  }  \right|
=\left|\left(f, X_{k}\cdot(Tu)\right)_{L^{2}_{\left( s,t+1\right)  }\left( V,w_{2}\right)  }  -\left( f,T(X_{k}\cdot u)\right)_{L^{2}_{\left( s,t+1\right)  }\left( V,w_{2}\right)  }  \right|\\
&&=\left|\left( f,(T X_{k})\wedge u\right)_{L^{2}_{\left( s,t+1\right)  }\left( V,w_{2}\right)  }  \right|=\left| \sum^{\prime }_{\left| I\right|  =s}\sum^{\prime }_{\left| J\right|  =t+1} c_{I,J}\int_{V} f_{I,J}\overline{((TX_{k})\wedge u)_{I,J}} e^{-w_{2}}\,\mathrm{d}P\right|  \\
&&\leqslant \sum^{\prime }_{\left| I\right|  =s}\sum^{\prime }_{\left| J\right|  =t+1} c_{I,J}\int_{V} \left|f_{I,J}e^{-\frac{w_{2}}{2} }\right|\cdot \left|((TX_{k})\wedge u)_{I,J}e^{-\frac{w_{2}}{2} }\right|\,\mathrm{d}P\\
&&\leqslant \int_{V} \left(\sum^{\prime }_{\left| I\right|  =s}\sum^{\prime }_{\left| J\right|  =t+1}c_{I,J} |f_{I,J}|^{2}e^{-w_{2}}\right)^{\frac{1}{2} }\cdot \left(\sum^{\prime }_{\left| I\right|  =s}\sum_{\left| J\right|  =t+1}^{\prime }  c_{I,J}|((TX_{k})\wedge u)_{I,J}|^{2}e^{-w_{2}}\right)^{\frac{1}{2} }\,\mathrm{d}P\\
&&\leqslant \sqrt{(t+1)\cdot c_1^{s,t}}  \cdot\int_{V\setminus V_{k-1}} \left(\sum^{\prime }_{\left| I\right|  =s}\sum^{\prime }_{\left| J\right|  =t+1} c_{I,J} |f_{I,J}|^{2}e^{-w_{2}}\right)^{\frac{1}{2} }\cdot \left(\sum^{\prime }_{\left| I\right|  =s}\sum_{\left| L\right|  =t}^{\prime } c_{I,J}|u_{I,L}|^{2}e^{-w_{1}}\right)^{\frac{1}{2} }\,\mathrm{d}P\\
&&\leqslant \sqrt{(t+1)\cdot c_1^{s,t}} \cdot\left(\int_{V\setminus V_{k-1}} \sum^{\prime }_{\left| I\right|  =s}\sum^{\prime }_{\left| J\right|  =t+1} c_{I,J}|f_{I,J}|^{2}e^{-w_{2}}\,\mathrm{d}P\right)^{\frac{1}{2} }\cdot \left| \left| u\right|  \right|_{L^{2}_{\left( s,t\right)  }\left( V,w_{1}\right)  } ,
\end{eqnarray*}
where the third inequality follows from the facts that $((TX_{k})\wedge u)_{I,J}=0$ on $V_{k-1}$ for all $I,J$,
\begin{eqnarray*}
\sum^{\infty }_{i=1} |\overline{\partial_{i} } X_{k}|^{2}\leqslant e^{\psi }=e^{w_{2}-w_{1}}
\end{eqnarray*}
and
$$
		\begin{array}{ll}
\displaystyle
			\sum_{\left| I\right|  =s}\sum^{\prime }_{\left| J\right|  =t+1}c_{I,J} |((T X_{k })\wedge u)_{I,J}|^{2}e^{-w_{2}}=\sum^{\prime }_{\left| I\right|  =s}\sum^{\prime }_{\left| J\right|  =t+1} c_{I,J}\left|\sum^{\infty }_{i=1} \sum^{\prime }_{\left| L\right|  =t} \varepsilon^{J}_{iL} u_{I,L}\overline{\partial_{i} } X_{k}\right|^{2}e^{-w_{2}}\\[3mm]
\displaystyle \leqslant \left( t+1\right)\cdot  \sum^{\prime }_{\left| I\right|  =s}\sum^{\prime }_{\left| J\right|  =t+1}c_{I,J} \sum^{\infty }_{i=1} \sum^{\prime }_{\left| L\right|  =t} |\varepsilon^{J}_{iL} u_{I,L}\overline{\partial_{i} } X_{k}|^{2}e^{-w_{2}}\\[3mm]
\displaystyle \leqslant \left( t+1\right)\cdot  \sum^{\prime }_{\left| I\right|  =s}\sum^{\prime }_{\left| J\right|  =t+1}\sum^{\infty }_{i=1} \sum^{\prime }_{\left| L\right|  =t} \frac{c_{I,J} |\varepsilon^{J}_{iL}|}{c_{I,L}}\cdot c_{I,L} \cdot|u_{I,L}|^2\cdot|\overline{\partial_{i} } X_{k}|^{2}e^{-w_{2}}\\[3mm]
\displaystyle \leqslant \left( t+1\right) \cdot \sum^{\prime }_{\left| I\right|  =s} \sum^{\infty }_{i=1} \sum^{\prime }_{\left| L\right|  =t} \frac{c_{I,iL} }{c_{I,L}}\cdot c_{I,L} \cdot|u_{I,L}|^2\cdot|\overline{\partial_{i} } X_{k}|^{2}e^{-w_{2}}\\[3mm]
\displaystyle \leqslant \left( t+1\right)\cdot c_1^{s,t}\cdot \sum^{\prime }_{\left| I\right|  =s} \sum^{\prime }_{\left| L\right|  =t} c_{I,L}|u_{I,L}|^2\sum^{\infty }_{i=1}|\overline{\partial_{i} } X_{k}|^{2}e^{-w_{2}}\\[3mm]
\displaystyle \leqslant \left( t+1\right)\cdot c_1^{s,t}\cdot   \sum^{\prime }_{\left| I\right|  =s}\sum^{\prime }_{\left| L\right|  =t} c_{I,L}|u_{I,L}|^{2}e^{-w_{1}}.
		\end{array}
$$
Since $D_T$ is dense in $L^{2}_{\left( s,t\right)  }\left( V,w_{1}\right)$, we deduce that
$$
\left| \left| T^{\ast }\left( X_{k}\cdot f\right)  -X_{k}\cdot( T^{\ast }f)\right|  \right|_{L^{2}_{(s,t)}\left( V,w_{1}\right)  }\leqslant \sqrt{(t+1)\cdot c_1^{s,t}}  \left(\sum^{\prime }_{\left| I\right|  =s} \sum^{\prime }_{\left| J\right|  =t+1} c_{I,J}\int_{V\setminus V_{k-1}} |f_{I,J}|^{2}e^{-w_{2}}\,\mathrm{d}P\right)^{\frac{1}{2} }.
$$
Note that the right side of the above inequality tends to zero as $k\to\infty$, which implies that $\lim\limits_{k\rightarrow \infty } \left| \left| T^{\ast }\left( X_{k}\cdot f\right)  -X_{k} \cdot(T^{\ast }f)\right|  \right|_{L^{2}_{(s,t)}\left( V,w_{1}\right)  }  =0.$

Since $|X_{k}|\leqslant 1$, we have $\sum\limits^{\prime }_{\left| I\right|  =s}\sum\limits^{\prime }_{\left| J\right|  =t+1} c_{I,J}|\left( X_{k}-1\right)  f_{I,J}|^{2}e^{-w_{2}}\leqslant 4 \sum\limits^{\prime }_{\left| I\right|  =s}\sum\limits^{\prime }_{\left| J\right|  =t+1} c_{I,J}\left| f_{I,J}\right|^{2}  e^{-w_{2}}$. We also note that for each $\textbf{z}\in V$, if $k>\eta(\textbf{z})$, then $X_{k}\left( \textbf{z}\right)=1$, and hence
$$
\lim_{k \rightarrow \infty }  X_{k }(\textbf{z}) =1,\  \forall\; \textbf{z}\in V.\
$$
By Lebesgue's Dominated Convergence Theorem, we arrive at
$$
		\begin{array}{ll}
\displaystyle
			\lim_{k\rightarrow \infty } \left| \left| X_{k} \cdot f-f\right|  \right|^{2}_{L^{2}_{(s,t+1)}\left( V,w_{2}\right)  }
=\lim_{k \rightarrow \infty } \int_{V} \sum^{\prime }_{\left| I\right|  =s}\sum^{\prime }_{\left| J\right|  =t+1} c_{I,J}|\left( X_{k}-1\right)  f_{I,J}|^{2}e^{-w_{2}}\,\mathrm{d}P \\[3mm]
\displaystyle = \int_{V} \lim_{k\rightarrow \infty }\sum^{\prime }_{\left| I\right|  =s}\sum^{\prime }_{\left| J\right|  =t+1} c_{I,J}|\left( X_{k}-1\right)  f_{I,J}|^{2}e^{-w_{2}}\,\mathrm{d}P=0 .
		\end{array}
$$
The proof of Proposition \ref{general cut-off density3} is completed.
\end{proof}

The following technical result will be useful later.
\begin{lemma}\label{general I,P}
Let Conditions \ref{230424c1} and \ref{230423ass1} hold. Then, for any $g=\sum\limits^{\prime }_{\left| I\right|  =s,\left| J\right|  =t+1 } g_{I,J}dz_{I}\wedge d\overline{z_{J}}\in D_{T^{\ast}}$ with $\supp g\stackrel{\circ}{\subset}  V$, the following conclusions hold:

\begin{itemize}
		\item[$(1)$]  For each strictly increasing multi-indices $I$ and $L$ with $\left| I\right|  =s$ and $\left| L\right|  =t$, the series $$\sum\limits^{\infty }_{i=1}\frac{c_{I,iL}}{c_{I,L}}\cdot g_{I,iL}\cdot\partial_{i} w_2$$ converges almost everywhere (with respect to $P$) on $V$, where $g_{I,iL}\triangleq \sum\limits^{\prime }_{\left| J\right|  =t+1}\varepsilon^{J}_{iL}g_{I,J}$ for $i\in\mathbb{N}$;

\item[$(2)$] The following two forms
\begin{equation}\label{general def of I and P}
\mathbb{I}g\triangleq \sum^{\prime }_{\left| I\right|  =s} \sum^{\prime }_{\left| L\right|  =t} \left(\sum^{\infty }_{i=1} \frac{c_{I,iL}}{c_{I,L}}\cdot g_{I,iL}\cdot\partial_{i} w_{2}\right)dz_{I}\wedge d\overline{z_{L}} ,\quad \mathbb{P}g\triangleq \mathbb{I}g-\left( -1\right)^{s}\cdot  e^{w_{2}-w_{1}}\cdot (T^{\ast }g)
\end{equation}
belong to $ L^{2}_{(s,t)}\left( V,w_{1}\right)$. Furthermore, if there exists $n_{0}\in \mathbb{N}$ such that $g_{I,J}=0$ for all strictly increasing multi-indices $I$ and $J$ with $\left| I\right|  =s$, $\left| J\right|  =t+1$ and $\max \left\{ I\cup J\right\}  >n_{0}$. Then,
$\left( \mathbb{I}g\right)_{I,L}  =0$ and $\left( \mathbb{P}g\right)_{I,L}  =0$ for all strictly increasing multi-indices $I$ and $L$ with $\left| I\right|  =s$, $\left| L\right|  =t$ and $\max \left\{ I\cup L\right\}  >n_{0}.$
\end{itemize}
\end{lemma}

\begin{proof}
(1) Since $T^*g\in L^{2}_{(s,t)}\left( V,w_{1}\right)$, by Proposition \ref{support argument}, we have $\supp (T^*g)\stackrel{\circ}{\subset}V$.
Since $\supp g\stackrel{\circ}{\subset}V$, by the conclusion (3) of Proposition \ref{properties on pseudo-convex domain}, there exists $r\in(0,+\infty)$ such that $\supp g\subset V_r^o\subset V_r$. Then for any strictly increasing multi-indices $I$ and $L$  with $\left| I\right|  =s$ and $\left| L\right|  =t$, by the Cauchy-Schwarz inequality, we have
\begin{eqnarray*}
&&\int_{V} \sum^{\infty }_{i=1}\frac{c_{I,iL}}{c_{I,L}}\cdot |g_{I,iL}\cdot \partial_{i}w_{2}|\,\mathrm{d}P\\
&&=\sum^{\infty }_{i=1}\int_{V}\frac{c_{I,iL}}{c_{I,L}}\cdot |g_{I,iL}|\cdot e^{-\frac{w_{2}}{2}}\cdot e^{\frac{w_{2}}{2}}\cdot|\partial_{i}w_{2}|\,\mathrm{d}P\\
&&\leqslant \left(\sum^{\infty }_{i=1} \left(\frac{c_{I,iL}}{c_{I,L}}\right)^2\cdot\int_{V} \left| g_{I,iL}\right|^{2}  e^{-w_{2}}\,\mathrm{d}P\right)^{\frac{1}{2} }\cdot \left(\sum^{\infty }_{i=1} \int_{V_{r}} \left| \partial_{i}w_{2}\right|^{2}  e^{w_{2}}\,\mathrm{d}P\right)^{\frac{1}{2} }\\
&&\leqslant \sqrt{\frac{1}{c_{I,L}}\cdot\left(\sup_{i\in\mathbb{N}}\frac{c_{I,iL}}{c_{I,L}}\right)}\cdot\left(\sum^{\infty }_{i=1} c_{I,iL} \cdot\int_{V} \left| g_{I,iL}\right|^{2}  e^{-w_{2}}\,\mathrm{d}P\right)^{\frac{1}{2} }\cdot \left(\sum^{\infty }_{i=1} \int_{V_{r}} \left| \partial_{i}w_{2}\right|^{2}  e^{w_{2}}\,\mathrm{d}P\right)^{\frac{1}{2} }.
\end{eqnarray*}
Also,
$$
		\begin{array}{ll}
\displaystyle \sum^{\infty }_{i=1} c_{I,iL} \cdot\int_{V} \left| g_{I,iL}\right|^{2}\cdot e^{-w_{2}}\,\mathrm{d}P=\sum^{\infty }_{i=1}\int_{V} c_{I,iL} \cdot g_{I,iL}\cdot\overline{g_{I,iL}} \cdot e^{-w_{2}}\,\mathrm{d}P\\[3mm]
\displaystyle =\sum^{\infty }_{i=1} \int_{V}\sum^{\prime }_{\left| J_1\right|  =t+1,\left| J_2\right|  =t+1,\left| J_3\right|  =t+1}\left|\varepsilon^{J_1}_{iL}\right|\cdot\varepsilon^{J_2}_{iL}\cdot \varepsilon^{J_3}_{iL}\cdot c_{I,J_1}\cdot g_{I,J_2}\cdot \overline{g_{I,J_3}}\cdot e^{-w_{2}}\,\mathrm{d}P\\[3mm]
\displaystyle =\sum^{\infty }_{i=1} \int_{V}\sum^{\prime }_{\left| J\right|  =t+1}\left|\varepsilon^{J}_{iL}\right|^3\cdot c_{I,J}\cdot \left| g_{I,J}\right|^2\cdot e^{-w_{2}}\,\mathrm{d}P,
\end{array}
$$
which gives
\begin{equation}\label{230414e3}
		\begin{array}{ll}
\displaystyle \sum^{\prime }_{\left| I\right|  =s,\left| L\right|  =t}\sum^{\infty }_{i=1} c_{I,iL} \cdot\int_{V} \left| g_{I,iL}\right|^{2}\cdot e^{-w_{2}}\,\mathrm{d}P\\[3mm]
\displaystyle =\sum^{\prime }_{\left| I\right|  =s,\left| J\right|  =t+1}\int_{V}\sum^{\infty }_{i=1} \sum^{\prime }_{\left| L\right|  =t}\left|\varepsilon^{J}_{iL}\right|^3\cdot c_{I,J}\cdot \left| g_{I,J}\right|^2\cdot e^{-w_{2}}\,\mathrm{d}P\\[3mm]
\displaystyle =(t+1)\sum^{\prime }_{\left| I\right|  =s,\left| J\right|  =t+1}c_{I,J}\cdot \int_{V}\left| g_{I,J}\right|^2\cdot e^{-w_{2}}\,\mathrm{d}P.
\end{array}
\end{equation}
Hence,
$$\int_{V} \sum^{\infty }_{i=1}\frac{c_{I,iL}}{c_{I,L}}\cdot |g_{I,iL}\cdot \partial_{i}w_{2}|\,\mathrm{d}P<\infty,
$$
which implies that the series $\sum\limits^{\infty }_{i=1}\frac{c_{I,iL}}{c_{I,L}}\cdot g_{I,iL}\partial_{i} w_{2}$ converges almost everywhere (with respect to $P$) on $V$.

\medskip

(2) \ \ By (\ref{230414e3}), it follows that (Recall Condition \ref{230423ass1} for $c_1^{s,t}$)
\begin{eqnarray*}
&&\sum^{\prime }_{\left| I\right|  =s,\left| L\right|  =t}c_{I,L}\cdot\int_{V} \left|\sum^{\infty }_{i=1} \frac{c_{I,iL}}{c_{I,L}}\cdot g_{I,iL}\cdot\left(\partial_{i} w_{2}\right)\right|^{2}\cdot e^{-w_1}\,\mathrm{d}P\\
&&=\sum^{\prime }_{\left| I\right|  =s,\left| L\right|  =t}c_{I,L}\cdot\int_{V_r} \left|\sum^{\infty }_{i=1} \frac{c_{I,iL}}{c_{I,L}}\cdot g_{I,iL}\cdot\left(\partial_{i} w_{2}\right)\right|^{2} e^{-w_1}\,\mathrm{d}P\\
& &\leqslant \sum^{\prime }_{\left| I\right|  =s,\left| L\right|  =t}c_{I,L}\cdot\int_{V_{r}} \left(\sum^{\infty }_{i=1} \left(\frac{c_{I,iL}}{c_{I,L}}\right)^2\cdot\left| g_{I,iL}\right|^{2}\right)\cdot e^{-w_2}\cdot \left(  \sum^{\infty }_{i=1} \left| \partial_{i}w_{2}\right|^{2} \right)\cdot e^{w_2-w_1} \,\mathrm{d}P\\
& &\leqslant \sup_{V_{r}}\left(  \sum^{\infty }_{i=1} \left|\partial_{i}w_{2}\right|^{2}\cdot e^{w_2-w_1} \right)\cdot c_1^{s,t}\cdot\sum^{\prime }_{\left| I\right|  =s,\left| L\right|  =t}\int_{V} \left(\sum^{\infty }_{i=1}c_{I,iL} \cdot\left| g_{I,iL}\right|^{2}\right)\cdot e^{-w_2} \,\mathrm{d}P\\
&&= \sup_{V_{r}}\left(  \sum^{\infty }_{i=1} \left|\partial_{i}w_{2}\right|^{2}\cdot e^{w_2-w_1} \right)\cdot c_1^{s,t}\cdot(t+1)\cdot \sum^{\prime }_{\left| I\right|  =s,\left| J\right|  =t+1}c_{I,J}\cdot\int_{V}\left| g_{I,J}\right|^{2}\cdot e^{-w_2} \,\mathrm{d}P\\
&&<\infty,
\end{eqnarray*}
and hence $\mathbb{I}g\in L^{2}_{(s,t)}\left( V,w_{1}\right)$.

By the definition of $\mathbb{I}g$, it is easy to see that $\left( \mathbb{I}g\right)_{I,L}  =0$ for all strictly increasing multi-indices $I$ and $L$ with $\left| I\right|  =s,\left| L\right|  =t$ and $\max \left\{ I\cup L\right\}  >n_{0}$. For any $\phi\in C^{\infty }_{0,F}\left( V\right)$ and strictly increasing multi-indices $I$ and $L$ with $\left| I\right|  =s,\left| L\right|  =t$ and $\max \left\{ I\cup L\right\}  >n_{0}$, noting that
\begin{eqnarray*}
\sum^{\infty }_{j=1} |c_{I,jL}|\int_{V} \left| \partial_{j} \phi \right|^{2}\cdot  e^{-w_{2}}\,\mathrm{d}P
\leqslant c_{I,L}\cdot \left(\sup_{i\in\mathbb{N}} \frac{c_{I,iL}}{c_{I,L}}\right)\cdot \sum^{\infty }_{j=1}\int_{V} \left| \partial_{j} \phi \right|^{2} \cdot e^{-w_{2}}\,\mathrm{d}P<\infty ,
\end{eqnarray*}
using Lemma \ref{T*formula}, we obtain that
\begin{eqnarray*}
 \int_{V} \left( T^{\ast }g\right)_{I,L} \cdot \phi \cdot e^{-w_{1}}dP=\frac{1}{c_{I,L}}\left( -1\right)^{s}  \sum_{j=1}^{\infty} c_{I,jL} \int_{V}\partial_{j} \phi \cdot g_{I,jL}\cdot e^{-w_{2}}\,\mathrm{d}P=0.
\end{eqnarray*}
By Corollary \ref{density2}, $C^{\infty }_{0,F}\left( V\right)$ is dense in $L^2(V,e^{-w_1}P)$, and hence $\left( T^{\ast }g\right)_{I,L}=0$ and $\left( \mathbb{P}g\right)_{I,L}=0$ for above indices $I$ and $L$. This completes the proof of Lemma \ref{general I,P}.
\end{proof}


The following theorem is the second part of our approximation process (Recall Definition \ref{definition of T S} for the operators $T$ and $S$).
\begin{theorem}\label{general density4}
Under Conditions \ref{230424c1} and  \ref{230423ass1}, for any $f\in D_{S}\bigcap D_{T^{\ast }}$ with $\supp f\stackrel{\circ}{\subset} V_{r}^{o}$ for some $r\in(0,+\infty)$, there exists $\{h_{k}\}_{k=1}^{\infty}\subset D_{S}\cap D_{T^{\ast }} $ and $\{n_k\}_{k=1}^{\infty}\subset\mathbb{N}$ such that for each $k\in\mathbb{N}$, $\left( h_{k }\right)_{I,J}  =0$ for all strictly increasing multi-indices $I$ and $J$ with $\left| I\right|  =s,\left| J\right|  =t+1$ and $\max \left\{ I\bigcup J\right\}  >n_{k}$ while $\left( h_{k }\right)_{I,J}  \in C^{\infty}_{0,F}(V)$ for all strictly increasing multi-indices $I$ and $J$ with $\left| I\right|  =s,\left| J\right|  =t+1$ and $\max \left\{ I\bigcup J\right\}\leqslant n_{k}$. Moreover,
$$
\lim_{k \rightarrow \infty } \left(\left| \left| T^{\ast }  h_k   -T^{\ast }f\right|  \right|_{L^{2}_{(s,t)}\left( V,w_{1}\right)  }  +\left| \left| h_k-f\right|  \right|_{L^{2}_{(s,t+1)}\left( V,w_{2}\right)  }  +\left| \left| S h_k  -Sf\right|  \right|_{L^{2}_{(s,t+2)}\left( V,w_{3}\right)  }  \right)  =0.
$$
\end{theorem}
\begin{proof}
Note that for each strictly increasing multi-indices $I,J$ and $K$ with $\left| I\right|=s,\left| J\right|  =t+1$ and $\left| K\right|=t+2$, $\supp f_{I,J}\stackrel{\circ}{\subset} V_{r}^{o}$ and by Proposition \ref{support argument}, $\supp (Sf)_{I,K}\stackrel{\circ}{\subset} V_{r}^{o}$. Since $\sup\limits_{\textbf{z}\in V_{r}^{o}}|w_2(\textbf{z})|<\infty$ and $\sup\limits_{\textbf{z}\in V_{r}^{o}}|w_3(\textbf{z})|<\infty$,  we have $f_{I,J},(Sf)_{I,K}\in L^2(V,P)$. For $n\in\mathbb{N}$ and $\delta\in(0,+\infty)$, let
\begin{eqnarray*}
f_{n}&\triangleq &\sum^{\prime }_{\left| I\right|  =s,\left| J\right|  =t+1,\max \left\{ I\cup J\right\}  \leqslant n} (f_{I,J})_{n}dz_{I}\wedge d\overline{z_{J}}, \\
f_{n,\delta }&\triangleq &\sum^{\prime }_{\left| I\right|  =s,\left| J\right|  =t+1,\max \left\{ I\cup J\right\}  \leqslant n} (f_{I,J})_{n,\delta }dz_{I}\wedge d\overline{z_{J}}, \\
\left( Sf\right)_{n,\delta } & \triangleq &\sum^{\prime }_{\left| I\right|  =s,\left| K\right|  =t+2,\max \left\{ I\cup K\right\}  \leqslant n} (\left( Sf\right)_{I,K})_{n,\delta }  dz_{I}\wedge d\overline{z_{K}},
\end{eqnarray*}
where $(f_{I,J})_{n}$ is defined as that in (\ref{Integration reduce dimension}), and $(f_{I,J})_{n,\delta }$ and $(\left( Sf\right)_{I,K})_{n,\delta }$ are defined as that in (\ref{Convolution after reduce dimension}). Clearly, $f_{n}$ and $f_{n,\delta }$ (resp. $\left( Sf\right)_{n,\delta }$) may be viewed $(s,t+1)$-forms (resp. an $(s,t+2)$-form) on $V$ (and also on $\ell^2$), and
\begin{equation}\label{230409e1}
\begin{aligned}
		&\left(f_{n}\right)_{I,J}=\begin{cases}
\left(f_{I,J}\right)_{n},&\max \left\{ I\cup J\right\}  \leqslant n,\\ 0,&\max \left\{ I\cup J\right\}  > n,
\end{cases}\qquad
 \left(f_{n,\delta}\right)_{I,J}=\begin{cases}
\left(f_{I,J}\right)_{n,\delta},&\max \left\{ I\cup J\right\}  \leqslant n,\\ 0,&\max \left\{ I\cup J\right\}  > n,
\end{cases}\\[5mm]
		&
		\left(\left(Sf\right)_{n,\delta}\right)_{I,K}=\begin{cases}
\left(\left(Sf\right)_{I,K}\right)_{n,\delta},&\max \left\{ I\cup K\right\}  \leqslant n,\\ 0,&\max \left\{ I\cup K\right\}  > n.
\end{cases}
\end{aligned}
\end{equation}

Let
$$
h \left( x\right)\triangleq
\begin{cases}
1,&x\in (-\infty ,0],\\ \left( e^{\frac{1}{x-1} }-1\right)  e^{-\frac{e^{\frac{1}{x-1} }}{x} }+1,&x\in \left( 0,1\right),  \\ 0,&x\in [1,+\infty ).
\end{cases}
$$
For each $\rho>r$ and $x\in\mathbb{R}$, set $h_{\rho}(x)\triangleq h(x-\rho)$ and $\eta_{\rho }\triangleq h_{\rho } \left( \eta \right)$. Then $\eta_{\rho }\in C^{\infty }_{0,F}\left( V\right)$, $h,h_{\rho}\in C^{\infty }\left( \mathbb{R}\right)$, $0\leqslant h\leqslant 1$, $0\leqslant h_{\rho}\leqslant 1$, $\  h_{\rho } \left( x\right)  =1$ for all $x<\rho ,\  h_{\rho } \left( x\right)  =0$ for all $x>\rho +1$ and $\left| h_{\rho}^{\prime } \right|  \leqslant C$ where $C\triangleq \sup\limits_{\mathbb{R}}| h^{\prime } |<\infty.$
Note that $(\eta_{\rho }\cdot f_{n,\delta })_{I,J}\in C^{\infty }_{0,F}\left( V\right)$ for each strictly increasing multi-indices $I$ and $J$ with $\left| I\right|  =s,\left| J\right|  =t+1$ and $\max\{I\cup J\}\leqslant n$, while $(\eta_{\rho }\cdot f_{n,\delta })_{I,J}=0$ for each strictly increasing multi-indices $I$ and $J$ with $\left| I\right|  =s,\left| J\right|  =t+1$ and $\max\{I\cup J\}>n$.  Since $\supp \eta_{\rho}\subset V_{\rho+1}$, by the conclusion (2) of Proposition \ref{properties on pseudo-convex domain},  for each $\rho'>\rho+1$, it holds that $\supp \eta_{\rho } \stackrel{\circ}{\subset} V_{\rho^{\prime } }^{o}$.  By Proposition \ref{general formula of T*} and the proof of Lemma \ref{densly defined closed}, we see that $\eta_{\rho }\cdot f_{n,\delta }\in D_{T^{\ast }}\cap D_S$. We will prove that
\begin{equation}\label{230418e03}
\begin{array}{ll}
\displaystyle \lim\limits_{n\rightarrow \infty } \lim\limits_{\delta \rightarrow 0^{+}}\left(\left| \left| T^{\ast }\left( \eta_{\rho }\cdot f_{n,\delta }\right)  -T^{\ast }f\right|  \right|_{L^{2}_{(s,t)}\left( V,w_{1}\right)  }  +\left| \left| \eta_{\rho }\cdot f_{n,\delta }-f\right|  \right|_{L^{2}_{(s,t+1)}\left( V,w_{2}\right)  }\right.\\[3mm]
\displaystyle \left.\qquad\qquad  +\left| \left| S\left( \eta_{\rho }\cdot f_{n,\delta }\right)  -Sf\right|  \right|_{L^{2}_{(s,t+2)}\left( V,w_{3}\right)  }\right)\\[3mm]
\displaystyle   = 0.
\end{array}
\end{equation}
As a consequence of (\ref{230418e03}), the desired sequence $\{h_k\}_{k=1}^{\infty}$ in Theorem \ref{general density4} can be chosen from $\{\eta_{\rho }\cdot f_{n,\delta }:\;\rho>r,\,n\in\mathbb{N},\,\delta\in(0,+\infty)\}$. The proof of the equality (\ref{230418e03}) is quite long, and therefore we divide it into several steps.

\medskip

\textbf{Step 1:}  We shall prove that
\begin{equation}\label{230406e1}
\lim\limits_{n\rightarrow \infty } \lim\limits_{\delta \rightarrow 0^{+}} \left| \left| \eta_{\rho }\cdot f_{n,\delta }-f\right|  \right|_{L^{2}_{(s,t+1)}\left( V_{\rho^{\prime } },w_{2}\right)  }  =0.
\end{equation}
Since $\supp f \stackrel{\circ}{\subset} V^{o}_{r}$, $\rho>r$ and $\eta_{\rho }=1$ on $V^{o}_{r}$, we have $\eta_{\rho }\cdot f=f$. Note also that
\begin{eqnarray*}
&& \left| \left| \eta_{\rho }\cdot f_{n,\delta }-f\right|  \right|_{L^{2}_{(s,t+1)}\left( V_{\rho^{\prime } },w_{2}\right)  } \\
&&\leqslant \left| \left| \eta_{\rho }\cdot f_{n,\delta }-\eta_{\rho }\cdot f_{n}\right|  \right|_{L^{2}_{(s,t+1)}\left( V_{\rho^{\prime } },w_{2}\right)  }  +\left| \left| \eta_{\rho }\cdot f_{n}-\eta_{\rho } \cdot f\right|  \right|_{L^{2}_{(s,t+1)}\left( V_{\rho^{\prime } },w_{2}\right)  }  +\left| \left| \eta_{\rho }\cdot f-f\right|  \right|_{L^{2}_{(s,t+1)}\left( V_{r},w_{2}\right)  }\\
&&\leqslant \left| \left|  f_{n,\delta }- f_{n}\right|  \right|_{L^{2}_{(s,t+1)}\left( V_{\rho^{\prime } },w_{2}\right)  }  +\left| \left|   f_{n}-  f\right|  \right|_{L^{2}_{(s,t+1)}\left( V_{\rho^{\prime } },w_{2}\right)  }   \\
&&\leqslant \sqrt{\sup_{V_{\rho^{\prime } }} e^{-w_2}}\cdot\left| \left|  f_{n,\delta }- f_{n}\right|  \right|_{L^{2}_{(s,t+1)}\left( V_{\rho^{\prime } },P\right)  }  +\sqrt{\sup_{V_{\rho^{\prime } }} e^{-w_2}}\cdot\left| \left|   f_{n}-  f\right|  \right|_{L^{2}_{(s,t+1)}\left( V_{\rho^{\prime } },P\right)  }.
\end{eqnarray*}
Since $\sup\limits_{V_{\rho^{\prime } }} e^{-w_2}<\infty$ and by the proof of Lemma \ref{approximation for st form}, we obtain the desired equality (\ref{230406e1}).

\medskip

\textbf{Step 2:} In this step, we shall show that
\begin{equation}\label{230406e2}
\lim\limits_{n\rightarrow \infty } \lim\limits_{\delta \rightarrow 0^{+}} \left| \left| S(\eta_{\rho }\cdot f_{n,\delta })-Sf\right|  \right|_{L^{2}_{(s,t+2)}\left( V_{\rho^{\prime } },w_{3}\right)  }  =0.
\end{equation}
Since $\eta_{\rho }\cdot (Sf)=Sf$ and
\begin{eqnarray*}
 &&\left| \left| S\left( \eta_{\rho }\cdot f_{n,\delta }\right)  -Sf\right|  \right|_{L^{2}_{(s,t+2)}\left( V_{\rho^{\prime } },w_{3}\right)  }\\
&&\leqslant \left| \left| S\left( \eta_{\rho }\cdot f_{n,\delta }\right)  -\eta_{\rho }\cdot \left( Sf\right)_{n,\delta }  \right|  \right|_{L^{2}_{(s,t+2)}\left( V_{\rho^{\prime } },w_{3}\right)  } +\left| \left| \eta_{\rho }\cdot \left( Sf\right)_{n,\delta }  -\eta_{\rho }\cdot (Sf)\right|  \right|_{L^{2}_{(s,t+2)}\left( V_{\rho^{\prime } },w_{3}\right)  }\\
&&\quad+\left| \left| \eta_{\rho }\cdot (Sf)-Sf\right|  \right|_{L^{2}_{(s,t+2)}\left( V_{r},w_{3}\right)  }\\
&&=\left| \left| S\left( \eta_{\rho }\cdot f_{n,\delta }\right)  -\eta_{\rho }\cdot \left( Sf\right)_{n,\delta }  \right|  \right|_{L^{2}_{(s,t+2)}\left( V_{\rho^{\prime } },w_{3}\right)  } +\left| \left| \eta_{\rho }\cdot \left( Sf\right)_{n,\delta }  -\eta_{\rho }\cdot (Sf)\right|  \right|_{L^{2}_{(s,t+2)}\left( V_{\rho^{\prime } },w_{3}\right)  },
\end{eqnarray*}
it suffices to prove that
 \begin{equation}\label{230407e1}
 \lim\limits_{n\rightarrow \infty } \lim\limits_{\delta \rightarrow 0^{{}_{+}}} \left| \left| \eta_{\rho }\cdot \left( Sf\right)_{n,\delta }  -\eta_{\rho }\cdot (Sf)\right|  \right|_{L^{2}_{(s,t+2)}\left( V_{\rho^{\prime } },w_{3}\right)  }  =0
 \end{equation}
 and
 \begin{equation}\label{230407e2}
 \lim\limits_{n\rightarrow \infty } \lim\limits_{\delta \rightarrow 0^{+}}\left| \left| S\left( \eta_{\rho }\cdot f_{n,\delta }\right)  -\eta_{\rho } \cdot\left( Sf\right)_{n,\delta }  \right|  \right|_{L^{2}_{(s,t+2)}\left( V_{\rho^{\prime } },w_{3}\right)  }=0.
\end{equation}

By the conclusion (2) of Proposition \ref{support argument}, $\supp Sf \stackrel{\circ}{\subset} V^{o}_{r}$. Since $0\leqslant \eta_{\rho } \leqslant 1$ and $\sup\limits_{V_{\rho^{\prime } }}e^{-w_3}<\infty$, we have
$$
\left| \left| \eta_{\rho } \cdot\left( Sf\right)_{n,\delta }  -\eta_{\rho }\cdot (Sf)\right|  \right|_{L^{2}_{(s,t+2)}\left( V_{\rho^{\prime } },w_{3}\right)  } ^2
\leqslant \left(\sup_{V_{\rho^{\prime } }}e^{-w_3}\right)\cdot \left| \left|\left( Sf\right)_{n,\delta }  - Sf\right|  \right|_{L^{2}_{(s,t+2)}\left( V,P\right)  } ^2.
$$
By the conclusion (2) of Lemma \ref{approximation for st form}, we obtain (\ref{230407e1}).

To show (\ref{230407e2}), we claim that
\begin{equation}\label{230407e9}
S( \eta_{\rho }\cdot f_{n,\delta })  =\eta_{\rho }\cdot \left( Sf\right)_{n,\delta }  +(S\eta_{\rho }) \wedge f_{n,\delta },
 \end{equation}
where
$$
(S\eta_{\rho }) \wedge f_{n,\delta }\triangleq \left( -1\right)^{s}  \sum^{\prime }_{\left| I\right|  =s,\max\{I\}\leqslant n} \sum^{\prime }_{\left| K\right|  =t+2} \left(\sum^{\infty }_{i=1} \sum^{\prime }_{\left| J\right|  =t+1,\max\{J\}\leqslant n} \varepsilon^{K}_{iJ} (f_{n,\delta })_{I,J}\overline{\partial_{i} } (\eta_{\rho }) \right)dz_{I}\wedge d\overline{z_{K}}.
$$
Indeed, for each strictly increasing multi-indices $I$ and $K$ with $\left| I\right|  =s$ and $\left| K\right|  =t+2$, we show first that
\begin{equation}\label{230407e5}
(-1)^s\sum^{\prime }_{\left| J\right|  =t+1,\max\{J\}\leqslant n} \sum^{n}_{i=1} \varepsilon^{K}_{iJ} \overline{\partial_{i} } (f_{n,\delta })_{I,J}=(( Sf )_{n,\delta })_{I,K}.
\end{equation}
For this purpose, we consider several cases. The first case is $\max\{I\cup K\}\leqslant n$. In this case, by (\ref{230409e1}), we have
$$
\begin{array}{ll}
\displaystyle (-1)^s\sum^{\prime }_{\left| J\right|  =t+1,\max\{J\}\leqslant n} \sum^{n}_{i=1} \varepsilon^{K}_{iJ} \overline{\partial_{i} } (f_{n,\delta })_{I,J}(\textbf{z}_n)=(-1)^s\sum^{\prime }_{\left| J\right|  =t+1} \sum^{\infty}_{i=1} \varepsilon^{K}_{iJ} \overline{\partial_{i} } (f_{n,\delta })_{I,J}(\textbf{z}_n)\\[3mm]
\displaystyle =(-1)^s\sum^{\prime }_{\left| J\right|  =t+1} \sum^{\infty}_{i=1} \varepsilon^{K}_{iJ} \int_{\mathbb{C}^{n} } (f_{I,J})_{n} ( \textbf{z}_n^{'})  \overline{\partial_{i} } \gamma_{n,\delta } (\textbf{z}_n-\textbf{z}_n^{'})  \,\mathrm{d}\textbf{z}_n^{'}\\[3mm]
\displaystyle =(-1)^s\sum^{\prime }_{\left| J\right|  =t+1} \sum^{\infty}_{i=1} \varepsilon^{K}_{iJ} \int_{\ell^{2}} f_{I,J} ( \textbf{z}_n',{\textbf{z}^n}' )  \overline{\partial_{i} } \gamma_{n,\delta }  (\textbf{z}_n- \textbf{z}_n' )  \varphi^{-1}_{n}  (\textbf{z}_n')  \,\mathrm{d}P( \textbf{z}_n',{\textbf{z}^n}' )\\[3mm]
\displaystyle =(-1)^{s+1}\sum^{\prime }_{\left| J\right|  =t+1} \sum^{\infty}_{i=1} \varepsilon^{K}_{iJ} \int_{\ell^{2}} f_{I,J} \overline{\delta_{i} \left( \gamma_{n,\delta }(\textbf{z}_n-\cdot) \varphi^{-1}_{n} \right)  } \,\mathrm{d}P\\[3mm]
\displaystyle =\int_{\ell^{2}} \left( Sf\right)_{I,K}    \gamma_{n,\delta }(\textbf{z}_n-\cdot) \varphi^{-1}_{n} \,\mathrm{d}P=(( Sf )_{n,\delta })_{I,K}(\textbf{z}_n),
\end{array}
$$
where $\textbf{z}_n\in\mathbb{C}^n$, $\varphi_{n} $ is given by (\ref{def for Gauss weight}) and the fifth equality follows from the fact that $ \gamma_{n,\delta }(\textbf{z}_n-\cdot) \varphi^{-1}_{n}\in C_b^1(V)$ and the conclusion (2) of Proposition \ref{weak equality for more test functions}. The second case is $\max\{I\}> n$, for which, by (\ref{230409e1}) again, each side of (\ref{230407e5}) is equal to $0$.  The third case is $\max\{K\}> n$. In this case, by the definition of $\varepsilon^{K}_{iJ} $, the left side of (\ref{230407e5}) is equal to $0$; while, by (\ref{230409e1}) once more, so is its right side.

In view of Remark \ref{230407r1}, it follows that
\begin{equation}\label{230410e1}
S( \eta_{\rho }\cdot f_{n,\delta })  =\left( -1\right)^{s} \sum\limits^{\prime }_{\left| I\right|  =s} \sum\limits^{\prime }_{\left| K\right|  =t+2} \sum^{\prime }_{\left| J\right|  =t+1} \sum^{\infty}_{i=1} \varepsilon^{K}_{iJ}\cdot \overline{\partial_{i} } \left(\eta_{\rho }\cdot(f_{n,\delta })_{I,J}\right)\,\mathrm{d}z_{I}\wedge \,\mathrm{d}\overline{z_{K}}.
\end{equation}
Also, by the conclusion (1) of Proposition \ref{general multiplitier}, for each strictly increasing multi-indices $I$ and $K$ with $\left| I\right|  =s$ and $\left| K\right|  =t+2$, noting (\ref{230407e5}), we have
\begin{eqnarray*}
&&(-1)^s\sum^{\prime }_{\left| J\right|  =t+1} \sum^{\infty}_{i=1} \varepsilon^{K}_{iJ}\cdot \overline{\partial_{i} } \left(\eta_{\rho }\cdot(f_{n,\delta })_{I,J}\right)=(-1)^s\sum^{\prime }_{\left| J\right|  =t+1,\max\{J\}\leqslant n} \sum^{\infty}_{i=1} \varepsilon^{K}_{iJ}\cdot \overline{\partial_{i} } \left(\eta_{\rho }\cdot(f_{n,\delta })_{I,J}\right)\\
&&=\eta_{\rho }\cdot(-1)^s\sum^{\prime }_{\left| J\right|  =t+1,\max\{J\}\leqslant n} \sum^{\infty}_{i=1} \varepsilon^{K}_{iJ} \cdot \overline{\partial_{i} } (f_{n,\delta })_{I,J}
+(-1)^s\sum^{\prime }_{\left| J\right|  =t+1,\max\{J\}\leqslant n} \sum^{\infty}_{i=1} \varepsilon^{K}_{iJ}\cdot  \overline{\partial_{i} } (\eta_{\rho })\cdot(f_{n,\delta })_{I,J}\\
&&=\eta_{\rho }\cdot(-1)^s\sum^{\prime }_{\left| J\right|  =t+1,\max\{J\}\leqslant n} \sum^{n}_{i=1} \varepsilon^{K}_{iJ}\cdot  \overline{\partial_{i} }  (f_{n,\delta })_{I,J}
+(-1)^s\sum^{\prime }_{\left| J\right|  =t+1,\max\{J\}\leqslant n} \sum^{\infty}_{i=1} \varepsilon^{K}_{iJ}\cdot  \overline{\partial_{i} } (\eta_{\rho })\cdot(f_{n,\delta })_{I,J}\\
&&=\eta_{\rho }\cdot(( Sf )_{n,\delta })_{I,K}+\left((S\eta_{\rho }) \wedge f_{n,\delta }\right)_{I,K},
\end{eqnarray*}
which, combined with (\ref{230410e1}), yields (\ref{230407e9}).

By (\ref{230407e9}), noting that $\varepsilon_{iJ}^K=\left|\varepsilon_{iJ}^K\right|\varepsilon_{iJ}^K$ for each $i\in\mathbb{N}$ and strictly increasing multi-indices $J$ and $K$ with $\left| J\right|  =t+1$ and $\left| K\right|  =t+2$, and recalling (\ref{gener111}) for  $c_{I,iJ}$,  we obtain that
\begin{equation}\label{230407e6}
\begin{aligned}
&\left| \left| S\left( \eta_{\rho } \cdot f_{n,\delta }\right)  -\eta_{\rho }\cdot \left( Sf\right)_{n,\delta }  \right|  \right|^{2}_{L^{2}_{(s,t+2)}\left( V_{\rho^{\prime } },w_{3}\right)  }\\
&=\left| \left|(S\eta_{\rho }) \wedge f_{n,\delta }\right|  \right|^{2}_{L^{2}_{t+2}\left( V_{\rho^{\prime } },w_{3}\right)  } \\
&= \sum^{\prime }_{\left| I\right|  =s,\left| K\right|  =t+2,\max \left\{ I \right\}  \leqslant n} c_{I,K}\int_{V_{\rho^{\prime } }}\left|\sum^{\prime }_{\left| J\right|  =t+1,\max\{J\}\leqslant n} \sum^{\infty}_{i=1} \varepsilon^{K}_{iJ}\cdot \overline{\partial_{i} } (\eta_{\rho })\cdot(f_{I,J})_{n,\delta }\right|^2e^{-w_{3}}\,\mathrm{d}P\\
&\leqslant \sum^{\prime }_{\left| I\right|  =s,\left| K\right|  =t+2,\max \left\{ I \right\}  \leqslant n} c_{I,K}\int_{V_{\rho^{\prime } }}\left(\sum^{\prime }_{\left| J\right|  =t+1,\max\{J\}\leqslant n} \sum^{\infty}_{i=1} \left|\varepsilon^{K}_{iJ}\right|^2\right)\\
&\qquad\qquad\qquad\qquad\qquad\quad\cdot\left(\sum^{\prime }_{\left| J\right|  =t+1,\max\{J\}\leqslant n} \sum^{\infty}_{i=1} \left|\varepsilon^{K}_{iJ}\cdot \overline{\partial_{i} } (\eta_{\rho })\cdot(f_{I,J})_{n,\delta }\right|^2\right)e^{-w_{3}}\,\mathrm{d}P\\
&=(t+2)\sum^{\prime }_{\left| I\right|  =s,\left| K\right|  =t+2,\max \left\{ I \right\}  \leqslant n} c_{I,K}\int_{V_{\rho^{\prime } }}\sum^{\prime }_{\left| J\right|  =t+1,\max\{J\}\leqslant n} \sum^{\infty}_{i=1} \left|\varepsilon^{K}_{iJ}\cdot \overline{\partial_{i} } (\eta_{\rho })\cdot(f_{I,J})_{n,\delta }\right|^2e^{-w_{3}}\,\mathrm{d}P\\
&= (t+2)\int_{V_{\rho^{\prime } }} \sum^{\infty }_{i=1} \sum^{\prime }_{\left| I\right|  =s,\left| J\right|  =t+1,\max \left\{ I\bigcup J\right\}  \leqslant n} c_{I,iJ}\cdot\left| \overline{\partial_{i} } \eta_{\rho } \right|^{2} \cdot\left| (f_{I,J})_{n,\delta }\right|^{2} \cdot e^{-w_{3}}\,\mathrm{d}P\\
&\leqslant C_1\int_{V_{\rho^{\prime } }} \left| h^{\prime }_{\rho } \left( \eta \right)  \right|^{2}  \sum^{\prime }_{\left| I\right|  =s,\left| J\right|  =t+1,\max \left\{ I\bigcup J\right\}  \leqslant n} c_{I,J}\left| (f_{I,J})_{n,\delta }\right|^{2}\,\mathrm{d}P,
\end{aligned}
\end{equation}
where $C_1\triangleq (t+2)\cdot \sup\limits_{V_{\rho^{\prime } }}\left(\sum\limits^{\infty }_{i=1} \left| \overline{\partial_{i} } \eta \right|^{2}\cdot e^{-w_3}\right)\cdot c_1^{s,t+1}<\infty$ (Recall Condition \ref{230423ass1} for $c_1^{s,t+1}$).

Since $f_{I,J}\in L^2(V,P)$ for each strictly increasing multi-indices $I$ and $J$ with $\left| I\right|  =s$ and $\left| J\right|  =t+1$, by the conclusion (2) of Proposition \ref{convolution properties}, for each $n\in\mathbb{N}$, it holds that
$$
\lim_{\delta \rightarrow 0^{+}} \int_{V_{\rho^{\prime } }} \left| | (f_{I,J})_{n,\delta } |^{2}  - | (f_{I,J})_{n}|^{2}  \right|dP=0.
$$
Therefore,
$$
\lim_{\delta \rightarrow 0^{+}} \int_{V_{\rho^{\prime } }} \left|\sum^{\prime }_{\left| I\right|  =s,\left| J\right|  =t+1,\max \left\{ I\bigcup J\right\}  \leqslant n} c_{I,J} | (f_{I,J})_{n,\delta } |^{2}  -\sum^{\prime }_{\left| I\right|  =s,\left| J\right|  =t+1,\max \left\{ I\bigcup J\right\}  \leqslant n} c_{I,J} | (f_{I,J})_{n}|^{2}  \right|dP=0.
$$
Note that
\begin{eqnarray*}
\sum^{\prime }_{\left| I\right|  =s,\left| J\right|  =t+1} c_{I,J} \int_{V_{\rho^{\prime } }} | f_{I,J}|^{2}\,\mathrm{d}P
\leqslant \left(\sup_{V_{\rho^{\prime } }}e^{w_2}\right)\cdot\sum^{\prime }_{\left| I\right|  =s,\left| J\right|  =t+1} c_{I,J} \int_{V_{\rho^{\prime } }} | f_{I,J}|^{2}e^{-w_2}\,\mathrm{d}P<\infty.
\end{eqnarray*}
For any given $\varepsilon>0$, there exists $n_0\in\mathbb{N}$ such that
\begin{eqnarray*}
\sum^{\prime }_{\left| I\right|  =s,\left| J\right|  =t+1,\,\max\{I\cup J\}>n_0} c_{I,J} \int_{V_{\rho^{\prime } }} | f_{I,J}|^{2}\,\mathrm{d}P
<\frac{\varepsilon}{4}.
\end{eqnarray*}
By the conclusion (3) of Proposition \ref{Reduce diemension}, there exits an integer $n_1>n_0$ such that for all $n\geqslant n_1$, it holds that
\begin{eqnarray*}
 \int_{V_{\rho^{\prime } }} \sum^{\prime }_{\left| I\right|  =s,\left| J\right|  =t+1,\,\max\{I\cup J\}\leqslant n_0} c_{I,J}\left| | (f_{I,J})_n|^{2}- | f_{I,J}|^{2}\right|\,\mathrm{d}P
<\frac{\varepsilon}{2}.
\end{eqnarray*}
Then for all $n\geqslant n_1$, noting that $\supp f\subset V_{\rho'}$, we have
\begin{eqnarray*}
&&\int_{V_{\rho^{\prime } }} \left|\sum^{\prime }_{\left| I\right|  =s,\left| J\right|  =t+1,\max \left\{ I\bigcup J\right\}  \leqslant n} c_{I,J}\left| (f_{I,J})_{n}\right|^{2}  -\sum^{\prime }_{\left| I\right|  =s,\left| J\right|  =t+1} c_{I,J}\left| f_{I,J}\right|^{2}  \right|\,\mathrm{d}P\\
&&\leqslant  \sum^{\prime }_{\left| I\right|  =s,\left| J\right|  =t+1,\,\max\{I\cup J\}\leqslant n_0} c_{I,J}\int_{V_{\rho^{\prime } }} \left|\left| (f_{I,J})_{n}\right|^{2}  -\left| f_{I,J}\right|^{2}  \right|\,\mathrm{d}P\\
&&\quad+\sum^{\prime }_{\left| I\right|  =s,\left| J\right|  =t+1,\,n_0<\max\{I\cup J\}\leqslant n } c_{I,J}\int_{V_{\rho^{\prime } }} \left|\left| (f_{I,J})_{n}\right|^{2}  -\left| f_{I,J}\right|^{2}  \right|\,\mathrm{d}P\\
&&\quad+\sum^{\prime }_{\left| I\right|  =s,\left| J\right|  =t+1,\max \left\{ I\bigcup J\right\}  >n} c_{I,J}\int_{V_{\rho^{\prime } }} \left| f_{I,J}\right|^{2}  \,\mathrm{d}P\\
&&<\frac{\varepsilon}{2}+2\sum^{\prime }_{\left| I\right|  =s,\left| J\right|  =t+1,\max \left\{ I\bigcup J\right\}  >n_0} c_{I,J} \int_{V_{\rho^{\prime } }} | f_{I,J}|^{2}\,\mathrm{d}P<\varepsilon,
\end{eqnarray*}
where the second inequality follows from the conclusion (1) of Proposition \ref{Reduce diemension}. Thus
$$
\lim_{n\rightarrow \infty } \int_{V_{\rho^{\prime } }} \left|\sum^{\prime }_{\left| I\right|  =s,\left| J\right|  =t+1,\max \left\{ I\bigcup J\right\}  \leqslant n} c_{I,J}\left| (f_{I,J})_{n}\right|^{2}  -\sum^{\prime }_{\left| I\right|  =s,\left| J\right|  =t+1} c_{I,J}\left| f_{I,J}\right|^{2}  \right|\,\mathrm{d}P=0,
$$
and hence
$$
\lim_{n\rightarrow \infty } \lim_{\delta \rightarrow 0^{+}} \int_{V_{\rho^{\prime } }} \left|\sum^{\prime }_{\left| I\right|  =s,\left| J\right|  =t+1,\max \left\{ I\bigcup J\right\}  \leqslant n} c_{I,J}\left| (f_{I,J})_{n,\delta }\right|^{2}  -\sum^{\prime }_{\left| I\right|  =s,\left| J\right|  =t+1} c_{I,J}\left| f_{I,J}\right|^{2}  \right|\,\mathrm{d}P=0.
$$
Since $\sup_{x\in\mathbb{R}}\left| h^{\prime }_{\rho }(x) \right|  =C<\infty$, we have
\begin{equation}\label{230407e7}
\lim_{n\rightarrow \infty } \lim_{\delta \rightarrow 0^{+}} \int_{V_{\rho^{\prime } }} \left| h^{\prime }_{\rho } \left( \eta \right)  \right|^{2}  \cdot\left|\sum^{\prime }_{\left| I\right|  =s,\left| J\right|  =t+1,\max \left\{ I\bigcup J\right\}  \leqslant n} c_{I,J}\left| (f_{I,J})_{n,\delta }\right|^{2}  -\sum^{\prime }_{\left| I\right|  =s,\left| J\right|  =t+1} c_{I,J}\left| f_{I,J}\right|^{2}  \right|\,\mathrm{d}P=0.
\end{equation}

Combining (\ref{230407e6}) and (\ref{230407e7}), we arrive at
\begin{equation}\label{end of STEP 2}
\begin{array}{ll}
\displaystyle\lim_{n\rightarrow \infty } \lim_{\delta \rightarrow 0^{+}}\left| \left| S\left( \eta_{\rho }\cdot f_{n,\delta }\right)  -\eta_{\rho } \cdot\left( Sf\right)_{n,\delta }  \right|  \right|^{2}_{L^{2}_{(s,t+2)}\left( V_{\rho^{\prime } },w_{3}\right)  }\\[5mm]\displaystyle
\leqslant  C_1\lim_{n\rightarrow \infty } \lim_{\delta \rightarrow 0^{+}}\int_{V_{\rho^{\prime } }} \left| h^{\prime }_{\rho } \left( \eta \right)  \right|^{2}  \sum^{\prime }_{\left| I\right|  =s,\left| J\right|  =t+1,\max \left\{ I\bigcup J\right\}  \leqslant n} c_{I,J}\left| (f_{I,J})_{n,\delta }\right|^{2}\,\mathrm{d}P\\[5mm]\displaystyle
=  C_1\int_{V_{\rho^{\prime } }} \left| h^{\prime }_{\rho } \left( \eta \right)  \right|^{2}  \sum^{\prime }_{\left| I\right|  =s,\left| J\right|  =t+1} c_{I,J}\left| f_{I,J} \right|^{2}\,\mathrm{d}P
=  C_1\int_{V_{r}} \left| h^{\prime }_{\rho } \left( \eta \right)  \right|^{2}  \sum^{\prime }_{\left| I\right|  =s,\left| J\right|  =t+1} c_{I,J}\left| f_{I,J} \right|^{2}\,\mathrm{d}P=0,
\end{array}
\end{equation}
which gives (\ref{230407e2}).

\medskip

\textbf{Step 3:} In this step, we shall prove that $\lim\limits_{n\rightarrow \infty } \lim\limits_{\delta \rightarrow 0^{+}} \left| \left| \mathbb{P}\left( \eta_{\rho } \cdot f_{n,\delta }\right)  -\mathbb{P}f\right|  \right|_{L^{2}_{(s,t)}\left( V_{\rho^{\prime } },w_{1}\right)  }  =0$.\\
Since $\eta_{\rho } \cdot(\mathbb{P}f)=\mathbb{P}f$ and
\begin{eqnarray*}
&&\left| \left| \mathbb{P}\left( \eta_{\rho }\cdot f_{n,\delta }\right)  -\mathbb{P}f\right|  \right|_{L^{2}_{(s,t)}\left( V_{\rho^{\prime } },w_{1}\right)  }\\
&&\leqslant \left| \left|\mathbb{P}\left( \eta_{\rho }\cdot f_{n,\delta }\right)  -\eta_{\rho } \cdot\left(\mathbb{P}f\right)_{n,\delta }  \right|  \right|_{L^{2}_{(s,t)}\left( V_{\rho^{\prime } },w_{1}\right)  }+\left| \left| \eta_{\rho }\cdot \left(\mathbb{P}f\right)_{n,\delta }  -\eta_{\rho }\cdot (\mathbb{P}f)\right|  \right|_{L^{2}_{(s,t)}\left( V_{\rho^{\prime } },w_{1}\right)  }\\
&& \quad +\left| \left| \eta_{\rho }\cdot (\mathbb{P}f)-\mathbb{P}f\right|  \right|_{L^{2}_{(s,t)}\left( V_{r},w_{1}\right)  }\\
&&=\left| \left|\mathbb{P}\left( \eta_{\rho }\cdot f_{n,\delta }\right)  -\eta_{\rho } \cdot\left(\mathbb{P}f\right)_{n,\delta }  \right|  \right|_{L^{2}_{(s,t)}\left( V_{\rho^{\prime } },w_{1}\right)  }+\left| \left| \eta_{\rho }\cdot \left(\mathbb{P}f\right)_{n,\delta }  -\eta_{\rho }\cdot (\mathbb{P}f)\right|  \right|_{L^{2}_{(s,t)}\left( V_{\rho^{\prime } },w_{1}\right)  },
\end{eqnarray*}
it suffices to prove that
\begin{equation}\label{230407e4}\lim\limits_{n\to\infty}\lim\limits_{\delta\to 0+}\left| \left|\mathbb{P}\left( \eta_{\rho }\cdot f_{n,\delta }\right)  -\eta_{\rho } \cdot\left( \mathbb{P}f\right)_{n,\delta }  \right|  \right|_{L^{2}_{(s,t)}\left( V_{\rho^{\prime } },w_{1}\right)  }=0.
\end{equation}
and
\begin{equation}\label{230407e3}\lim\limits_{n\to\infty}\lim\limits_{\delta\to 0+}\left| \left| \eta_{\rho }\cdot \left(\mathbb{P}f\right)_{n,\delta }  -\eta_{\rho }\cdot (\mathbb{P}f)\right|  \right|_{L^{2}_{(s,t)}\left( V_{\rho^{\prime } },w_{1}\right)  }=0
\end{equation}

The proof of (\ref{230407e3}) is similarly to that of (\ref{230407e1}), and therefore we omit the details. In the rest of this step, we shall prove (\ref{230407e4}).

By Proposition \ref{general formula of T*} and Lemma \ref{general I,P}, we have $\eta_{\rho }\cdot f_{n,\delta }\in D_{T^*},\mathbb{P}f,\mathbb{P}\left( \eta_{\rho } \cdot f_{n,\delta }\right)  ,\mathbb{I}f,\mathbb{I}\left( \eta_{\rho }\cdot f_{n,\delta }\right)  \in L^{2}_{\left( s,t\right)  }\left( V,w_{1}\right) $ (Recall (\ref{general def of I and P}) for the definition of $\mathbb{I}$ and $\mathbb{P}$). For every $\phi \in C^{\infty }_{0,F}\left(V_{\rho^{\prime} }^{o}\right) $ and sufficiently small $\delta>0$, each strictly increasing multi-indices $I$ and $L$ with $\left| I\right|  =s,\left| L\right|  =t ,\max\{I\cup L\}\leqslant n$, noting that
$
 \left(f_{n,\delta }\right)_{I,iL}=\sum_{|J|=t+1}'\varepsilon_{iL}^J \left(f_{n,\delta }\right)_{I,J}=\sum_{|J|=t+1,\max\{J\}\leqslant n}'\varepsilon_{iL}^J \left(f_{n,\delta }\right)_{I,J}$ for each $n,i\in\mathbb{N}$,
we have
$$
	\begin{array}{ll}
\displaystyle \int_{V} \left(\mathbb{P} ( \eta_{\rho }\cdot f_{n,\delta } ) \right)_{I,L}\cdot\phi \,\mathrm{d}P
=\int_{V} \left(\mathbb{ I}(\eta_{\rho}\cdot f_{n,\delta })  \right)_{I,L}\cdot  \phi \,\mathrm{d}P-\left( -1\right)^{s}  \int_{V} \left( T^{\ast } ( \eta_{\rho}\cdot f_{n,\delta } )  \right)_{I,L}\cdot  e^{w_{2}}\cdot\phi\cdot e^{-w_{1}}\,\mathrm{d}P\\[3mm]
\displaystyle =\int_{V} \sum^{n}_{i=1} \frac{c_{I,iL}}{c_{I,L}}\cdot\eta_{\rho }\cdot (f_{n,\delta })_{I,iL}\cdot(\partial_{i}w_{2})\cdot\phi \,\mathrm{d}P-\sum^{n}_{i=1} \int_{V}\frac{c_{I,iL}}{c_{I,L}}\cdot \eta_{\rho }\cdot (f_{n,\delta })_{I,iL}\cdot\partial_{i}(e^{w_{2}}\cdot\phi )\cdot e^{-w_{2}}\,\mathrm{d}P\\[3mm]
\displaystyle =-\sum^{n}_{i=1} \int_{V} \frac{c_{I,iL}}{c_{I,L}}\cdot\eta_{\rho }\cdot (f_{n,\delta })_{I,iL}\cdot(\partial_{i}\phi )\,\mathrm{d}P\\[3mm]
\displaystyle =\sum^{n}_{i=1} \int_{V}\frac{c_{I,iL}}{c_{I,L}}\cdot \eta_{\rho } \cdot \delta_{i} (f_{n,\delta })_{I,iL}\cdot\phi \,\mathrm{d}P+\sum^{n}_{i=1} \int_{V}\frac{c_{I,iL}}{c_{I,L}}\cdot \partial_{i}(\eta_{\rho }) \cdot (f_{n,\delta })_{I,iL}\cdot\phi\,\mathrm{d}P,
	\end{array}
$$
where the second equality follows from Lemma \ref{T*formula} and the final one follows from Corollary \ref{integration by Parts or deltai}. By (\ref{general def of I and P}) and the conclusion (3) of Proposition \ref{convolution properties}, we have
$$
\begin{aligned}
&\int_{V} \left( (\mathbb{P}f)_{I,L}\right)_{n,\delta }\cdot  \phi \,\mathrm{d}P
=\int_{V}  \left( (\mathbb{I}f)_{I,L}\right)_{n,\delta }\cdot  \phi \,\mathrm{d}P-\left( -1\right)^{s}  \int_{V}   \left(  e^{w_{2}-w_{1}}\cdot (T^{\ast }f)_{I,L}\right)_{n,\delta }\cdot\phi \,\mathrm{d}P\\
&=\int_{V} \left( \mathbb{I}f\right)_{I,L}\cdot  \varphi^{-1}_{n} \cdot\left( \phi_{n} \cdot\varphi_{n} \right)_{n,\delta }\,\mathrm{d}P-\left( -1\right)^{s}  \int_{V} \left( T^{\ast }f\right)_{I,L}\cdot  \varphi^{-1}_{n} \cdot e^{w_{2}-w_{1}}\cdot \left( \phi_{n} \cdot\varphi_{n} \right)_{n,\delta }  \,\mathrm{d}P\\
&=\int_{V} \left( \sum^{\infty }_{i=1}  \frac{c_{I,iL}}{c_{I,L}}\cdot f_{I,iL}\cdot\partial_{i}w_{2}\right) \cdot \varphi^{-1}_{n} \cdot\left( \phi_{n}\cdot \varphi_{n} \right)_{n,\delta } \,\mathrm{d}P\\
&\quad-\left( -1\right)^{s}  \int_{V} \left( T^{\ast }f\right)_{I,L}\cdot  \varphi^{-1}_{n}\cdot e^{w_{2}}\cdot\left( \phi_{n}\cdot \varphi_{n} \right)_{n,\delta } \cdot e^{-w_{1}}\,\mathrm{d}P.
\end{aligned}
$$
Hence, noting that $\supp f\stackrel{\circ}{\subset} V_{r}^{o}$ and the function $\varphi^{-1}_{n}\cdot e^{w_{2}}\cdot \eta_\rho\cdot\left( \phi_{n} \cdot\varphi_{n} \right)_{n,\delta }$ satisfies the assumption \eqref{phi begin to D condition'}, by the
conclusion (1) of Proposition \ref{support argument}, $ \left( T^{\ast }f\right)_{I,L}= \left( T^{\ast }f\right)_{I,L}\cdot \eta_\rho$, $\partial_{i} \phi_{n} =\left(\partial_{i} \phi\right)_{n} $, $ f_{I,iL}\cdot \eta_\rho= f_{I,iL}$ and $ f_{I,iL}\cdot (\partial_{i} \eta_\rho)= 0$ for $i=1,2,\cdots,n$, and by Corollary \ref{weak T*formula}, we obtain that
$$
\begin{aligned}
&\int_{V} \left( (\mathbb{P}f)_{I,L}\right)_{n,\delta } \cdot \phi \,\mathrm{d}P\\
&=\int_{V} \left( \sum^{\infty }_{i=1} \frac{c_{I,iL}}{c_{I,L}}\cdot f_{I,iL}\cdot\partial_{i}w_{2}\right)\cdot  \varphi^{-1}_{n}\cdot  \left( \phi_{n} \cdot \varphi_{n} \right)_{n,\delta }\,\mathrm{d}P\\
&\quad-\sum^{\infty }_{i=1} \int_{V} \frac{c_{I,iL}}{c_{I,L}}\cdot f_{I,iL}\cdot \partial_{i}\left(\varphi^{-1}_{n}\cdot  e^{w_{2}}\cdot \eta_\rho\cdot \left( \phi_{n}\cdot  \varphi_{n} \right)_{n,\delta } \right)\cdot e^{-w_{2}}\,\mathrm{d}P\\
&=\int_{V} \left( \sum^{\infty }_{i=1}\frac{c_{I,iL}}{c_{I,L}}\cdot f_{I,iL}\cdot \partial_{i}w_{2}\right)\cdot   \varphi^{-1}_{n}\cdot  \left( \phi_{n} \cdot \varphi_{n} \right)_{n,\delta }\,\mathrm{d}P\\
&\quad-\sum^{\infty }_{i=1} \int_{V}\frac{c_{I,iL}}{c_{I,L}}\cdot f_{I,iL}\cdot \partial_{i}\left(\varphi^{-1}_{n}\cdot  e^{w_{2}}\cdot \eta_\rho\right)\cdot \left( \phi_{n}\cdot  \varphi_{n} \right)_{n,\delta } \cdot e^{-w_{2}}\,\mathrm{d}P\\
&\quad-\sum^{\infty }_{i=1} \int_{V}\frac{c_{I,iL}}{c_{I,L}}\cdot f_{I,iL}\cdot \varphi^{-1}_{n}\cdot \partial_{i}\left(\left( \phi_{n} \cdot \varphi_{n} \right)_{n,\delta } \right)\,\mathrm{d}P\\
&=\int_{V} \left( \sum^{\infty }_{i=1}\frac{c_{I,iL}}{c_{I,L}}\cdot  f_{I,iL}\cdot \partial_{i}w_{2}\right) \cdot  \varphi^{-1}_{n} \cdot \left( \phi_{n} \cdot \varphi_{n} \right)_{n,\delta } \,\mathrm{d}P\\
&\quad-\sum^{\infty }_{i=1} \int_{V} \frac{c_{I,iL}}{c_{I,L}}\cdot f_{I,iL}\cdot \left(\partial_{i}\varphi^{-1}_{n} +\left(\partial_{i}w_{2}+\partial_{i}\eta_\rho\right)\cdot \varphi^{-1}_{n} \right)\cdot \left( \phi_{n} \cdot \varphi_{n} \right)_{n,\delta }  \,\mathrm{d}P\\
&\quad-\sum^{n}_{i=1} \int_{V}\frac{c_{I,iL}}{c_{I,L}}\cdot  f_{I,iL}\cdot \varphi^{-1}_{n}\cdot  \left((\partial_{i} \phi_{n} )\cdot \varphi_{n} \right)_{n,\delta } \,\mathrm{d}P-\sum^{n}_{i=1} \int_{V} \frac{c_{I,iL}}{c_{I,L}}\cdot f_{I,iL}\cdot \varphi^{-1}_{n}\cdot  \left( \phi_{n} \cdot \partial_{i}\varphi_{n} \right)_{n,\delta }\,\mathrm{d}P\\
&=\int_{V} \left( \sum^{n}_{i=1} \frac{c_{I,iL}}{c_{I,L}}\cdot  f_{I,iL}\cdot \partial_{i} w_{2}\right) \cdot  \varphi^{-1}_{n} \left( \phi_{n} \cdot \varphi_{n} \right)_{n,\delta } \,\mathrm{d}P\\
&\quad	 -\sum^{n}_{i=1} \int_{V} \frac{c_{I,iL}}{c_{I,L}}\cdot  f_{I,iL}\cdot \left(\frac{\overline{z_{i}} }{2a^{2}_{i}} \cdot \varphi^{-1}_{n} +\left(\partial_{i} w_{2} \right)\cdot \varphi^{-1}_{n}\right)\cdot \left( \phi_{n} \cdot \varphi_{n} \right)_{n,\delta } \,\mathrm{d}P\\
&\quad-\sum^{n}_{i=1} \int_{V} \frac{c_{I,iL}}{c_{I,L}}\cdot  f_{I,iL}\cdot \varphi^{-1}_{n}\cdot  \left( (\partial_{i} \phi_{n} ) \cdot \varphi_{n} \right)_{n,\delta } \,\mathrm{d}P+\sum^{n}_{i=1} \int_{V}\frac{c_{I,iL}}{c_{I,L}}\cdot   f_{I,iL}\cdot \varphi^{-1}_{n} \cdot \left( \phi_{n} \cdot \frac{\overline{z_{i}} }{2a^{2}_{i}}\cdot  \varphi_{n} \right)_{n,\delta } \,\mathrm{d}P\\
&=-\sum^{n}_{i=1} \int_{V} \frac{c_{I,iL}}{c_{I,L}}\cdot f_{I,iL}\cdot \frac{\overline{z_{i}} }{2a^{2}_{i}} \cdot  \varphi^{-1}_{n}\cdot \left( \phi_{n} \cdot \varphi_{n} \right)_{n,\delta } \,\mathrm{d}P\\
&\quad-\sum^{n}_{i=1} \int_{V}\frac{c_{I,iL}}{c_{I,L}}\cdot f_{I,iL}\cdot \varphi^{-1}_{n} \cdot \left( (\partial_{i} \phi)_{n} \cdot \varphi_{n} \right)_{n,\delta }\,\mathrm{d}P+\sum^{n}_{i=1} \int_{V}\frac{c_{I,iL}}{c_{I,L}}\cdot f_{I,iL}\cdot \varphi^{-1}_{n} \cdot \left( \phi_{n} \cdot \frac{\overline{z_{i}} }{2a^{2}_{i}} \cdot \varphi_{n} \right)_{n,\delta } \,\mathrm{d}P.
\end{aligned}
$$
Hence, by the conclusion (3) of Proposition \ref{convolution properties} again and noting that $\partial_{i}\phi=\eta_{\rho'+1}\cdot \left(\partial_{i}\phi\right)$ for $i=1,2,\cdots,n$ , we have
\begin{eqnarray*}
&&\int_{V} \left( (\mathbb{P}f)_{I,L}\right)_{n,\delta }\cdot   \phi \,\mathrm{d}P\\
&&=-\sum^{n}_{i=1} \int_{V}\frac{c_{I,iL}}{c_{I,L}}\cdot \left(f_{I,iL}\cdot\frac{\overline{z_{i}} }{2a^{2}_{i}} \right)_{n,\delta }\cdot\phi \,\mathrm{d}P-\sum^{n}_{i=1} \int_{V}\frac{c_{I,iL}}{c_{I,L}}\cdot (f_{I,iL})_{n,\delta }\cdot \eta_{\rho'+1}\cdot\partial_{i} \phi \,\mathrm{d}P\\
&&\quad+\sum^{n}_{i=1} \int_{V}\frac{c_{I,iL}}{c_{I,L}}\cdot (f_{I,iL})_{n,\delta }\cdot\frac{\overline{z_{i}} }{2a^{2}_{i}}\cdot \phi \,\mathrm{d}P\\
&&=-\sum^{n}_{i=1} \int_{V}\frac{c_{I,iL}}{c_{I,L}}\cdot \left(f_{I,iL}\cdot\frac{\overline{z_{i}} }{2a^{2}_{i}}\right)_{n,\delta }\cdot\phi \,\mathrm{d}P+\sum^{n}_{i=1} \int_{V}\frac{c_{I,iL}}{c_{I,L}}\cdot \delta_{i} \left((f_{I,iL})_{n,\delta }\cdot\eta_{\rho'+1}\right) \cdot\phi \,\mathrm{d}P\\
&&\quad+\sum^{n}_{i=1} \int_{V}\frac{c_{I,iL}}{c_{I,L}}\cdot (f_{I,iL})_{n,\delta }\cdot\frac{\overline{z_{i}} }{2a^{2}_{i}} \cdot\phi \,\mathrm{d}P,\\
&&=-\sum^{n}_{i=1} \int_{V}\frac{c_{I,iL}}{c_{I,L}}\cdot \left(f_{I,iL}\cdot\frac{\overline{z_{i}} }{2a^{2}_{i}} \right)_{n,\delta }\cdot\phi \,\mathrm{d}P+\sum^{n}_{i=1} \int_{V}\frac{c_{I,iL}}{c_{I,L}}\cdot \delta_{i} (f_{I,iL})_{n,\delta }\cdot\phi \,\mathrm{d}P\\
&&\quad+\sum^{n}_{i=1} \int_{V}\frac{c_{I,iL}}{c_{I,L}}\cdot (f_{I,iL})_{n,\delta }\cdot\frac{\overline{z_{i}} }{2a^{2}_{i}} \cdot\phi \,\mathrm{d}P,
\end{eqnarray*}
where the second equality follows from Corollary \ref{integration by Parts or deltai}.  By Corollary \ref{density2}, we have
\begin{equation}\label{Pfndelta}
\begin{aligned}
&\left( \mathbb{P}f\right)_{n,\delta } =\sum^{\prime }_{\left| I\right|  =s,\left| L\right|  =t,\max \left\{ I\bigcup L\right\}  \leqslant n} \sum^{n}_{i=1} \frac{c_{I,iL}}{c_{I,L}}\cdot\left( -\left(f_{I,iL}\cdot\frac{ \overline{z_{i}}}{2a^{2}_{i}} \right)_{n,\delta }+ \delta_{i} (f_{I,iL})_{n,\delta }+ (f_{I,iL})_{n,\delta }\cdot\frac{ \overline{z_{i}} }{2a^{2}_{i}} \right) dz_{I}\wedge d\overline{z_{L}}.
\end{aligned}
\end{equation}

On the other hand, by Proposition \ref{general formula of T*} and the definition of $\mathbb{P}$ in (\ref{general def of I and P}), it follows that
\begin{eqnarray}\label{Petarhofndelta}
\mathbb{P}\left( \eta_{\rho }\cdot f_{n,\delta }\right)  =\sum^{\prime }_{\left| I\right|  =s,\left| L\right|  =t,\max \{ I\bigcup L\} \leqslant n} \left(\sum^{n}_{i=1}\frac{c_{I,iL}}{c_{I,L}}\cdot\left( \eta_{\rho } \cdot\delta_{i} (f_{I,iL})_{n,\delta }+(\partial_{i} \eta_{\rho }) \cdot (f_{I,iL})_{n,\delta }\right)\right)dz_{I}\wedge d\overline{z_{L}} .
\end{eqnarray}

By \eqref{Pfndelta} and \eqref{Petarhofndelta}, we have
$$
\begin{aligned}
&\left| \left| \mathbb{P}\left( \eta_{\rho }\cdot f_{n,\delta }\right) -\eta_{\rho }\cdot (\mathbb{P}f)_{n,\delta } \right|  \right|_{L^{2}_{(s,t)}\left( V_{\rho^{\prime } },w_{1}\right)  } \\
&\leqslant\left|\left|\eta_{\rho }\cdot \sum^{\prime }_{\left| I\right|  =s,\left| L\right|  =t,\max \left\{ I\bigcup L\right\}  \leqslant n} \sum^{n}_{i=1}\frac{c_{I,iL}}{c_{I,L}}\cdot \left( -  \left(f_{I,iL}\cdot \frac{\overline{z_{i}} }{2a^{2}_{i}}\right)_{n,\delta }+ (f_{I,iL})_{n,\delta }\cdot \frac{ \overline{z_{i}}}{2a^{2}_{i}} \right) dz_{I}\wedge d\overline{z_{L}} \right|\right|_{L^{2}_{\left( s,t\right)  }\left( V_{\rho^{\prime } },w_{1}\right)  }\\
&\quad +\left| \left| \sum^{\prime }_{\left| I\right|  =s,\left| L\right|  =t,\max \{ I\bigcup L\} \leqslant n} \left(\sum^{n}_{i=1}\frac{c_{I,iL}}{c_{I,L}}\cdot(\partial_{i} \eta_{\rho }) \cdot (f_{I,iL})_{n,\delta }\right)dz_{I}\wedge d\overline{z_{L}}  \right|  \right|_{L^{2}_{(s,t)}\left( V_{\rho^{\prime } },w_{1}\right)  }\\
&\leqslant\sqrt{\sup_{V_{\rho^{\prime } }}e^{-w_1}}\cdot\left[\left|\left|\sum^{\prime }_{\left| I\right|  =s,\left| L\right|  =t,\max \left\{ I\bigcup L\right\}  \leqslant n} \sum^{n}_{i=1}\frac{c_{I,iL}}{c_{I,L}}\cdot\left( (f_{I,iL})_{n,\delta }\cdot \frac{ \overline{z_{i}}}{2a^{2}_{i}}\right.-\left(f_{I,iL}\cdot \frac{\overline{z_{i}} }{2a^{2}_{i}}\right)_{n,\delta }\right) dz_{I}\wedge d\overline{z_{L}} \right|\right|_{L^{2}_{\left( s,t\right)  }\left( V_{\rho^{\prime } },P\right)  }\\
&\quad +\left.\left| \left| \sum^{\prime }_{\left| I\right|  =s,\left| L\right|  =t,\max \{ I\bigcup L\} \leqslant n} \left(\sum^{n}_{i=1}\frac{c_{I,iL}}{c_{I,L}}\cdot(\partial_{i} \eta_{\rho }) \cdot (f_{I,iL})_{n,\delta }\right)dz_{I}\wedge d\overline{z_{L}}  \right|  \right|_{L^{2}_{(s,t)}\left( V_{\rho^{\prime } },P\right)  }\right].
\end{aligned}
$$
Note that
\begin{eqnarray*}
&&\left|\left|\sum^{\prime }_{\left| I\right|  =s,\left| L\right|  =t,\max \left\{ I\bigcup L\right\}  \leqslant n} \sum^{n}_{i=1}\frac{c_{I,iL}}{c_{I,L}}\cdot\left( (f_{I,iL})_{n,\delta }\cdot \frac{ \overline{z_{i}}}{2a^{2}_{i}} -\left(f_{I,iL}\cdot \frac{\overline{z_{i}} }{2a^{2}_{i}}\right)_{n,\delta }\right) dz_{I}\wedge d\overline{z_{L}} \right|\right|^2_{L^{2}_{\left( s,t\right)  }\left( V_{\rho^{\prime } },P\right)  }\\
&&=\sum^{\prime }_{\left| I\right|  =s,\left| L\right|  =t,\max \left\{ I\bigcup L\right\}  \leqslant n} c_{I,L}\cdot\int_{V_{\rho^{\prime } }} \left|\sum^{n}_{i=1}\frac{c_{I,iL}}{c_{I,L}}\cdot\left( (f_{I,iL})_{n,\delta }\cdot \frac{ \overline{z_{i}}}{2a^{2}_{i}} -\left(f_{I,iL}\cdot \frac{\overline{z_{i}} }{2a^{2}_{i}}\right)_{n,\delta }\right)\right|^2\,\mathrm{d}P\\
&&\leqslant n\cdot\sum^{\prime }_{\left| I\right|  =s,\left| L\right|  =t,\max \left\{ I\bigcup L\right\}  \leqslant n} c_{I,L}\cdot\sum^{n}_{i=1}\left(\frac{c_{I,iL}}{c_{I,L}}\right)^2\cdot\int_{V_{\rho^{\prime } }}\left| (f_{I,iL})_{n,\delta }\cdot \frac{ \overline{z_{i}}}{2a^{2}_{i}} -\left(f_{I,iL}\cdot \frac{\overline{z_{i}} }{2a^{2}_{i}}\right)_{n,\delta }\right|^2\,\mathrm{d}P,
\end{eqnarray*}
$(f_{I,iL})_{n}\cdot \overline{z_{i}} =(f_{I,iL}\cdot \overline{z_{i}} )_{n}$ for all $1\leqslant i\leqslant n$ and
\begin{eqnarray*}
&&\left(\int_{V_{\rho^{\prime } }}\left| (f_{I,iL})_{n,\delta }\cdot \frac{ \overline{z_{i}}}{2a^{2}_{i}} -\left(f_{I,iL}\cdot \frac{\overline{z_{i}} }{2a^{2}_{i}}\right)_{n,\delta }\right|^2\,\mathrm{d}P\right)^{\frac{1}{2}}\\
&&\leqslant\left(\int_{V_{\rho^{\prime } }}\left| (f_{I,iL})_{n,\delta }\cdot \frac{ \overline{z_{i}}}{2a^{2}_{i}} -(f_{I,iL})_{n}\cdot \frac{\overline{z_{i}}}{2a_i^2}\right|^2\,\mathrm{d}P\right)^{\frac{1}{2}}+\left(\int_{V_{\rho^{\prime } }}\left| \left(f_{I,iL}\cdot \frac{\overline{z_{i}}}{2a_i^2} \right)_{n}-\left(f_{I,iL}\cdot \frac{\overline{z_{i}} }{2a^{2}_{i}}\right)_{n,\delta }\right|^2\,\mathrm{d}P\right)^{\frac{1}{2}}\\
&&\leqslant\sup_{V_{\rho^{\prime } }}\left|\frac{ \overline{z_{i}}}{2a^{2}_{i}}\right|\cdot\left(\int_{V_{\rho^{\prime } }}\left| (f_{I,iL})_{n,\delta } -(f_{I,iL})_{n}\right|^2\,\mathrm{d}P\right)^{\frac{1}{2}}+\left(\int_{V_{\rho^{\prime } }}\left| \left(f_{I,iL}\cdot \frac{\overline{z_{i}}}{2a_i^2} \right)_{n}-\left(f_{I,iL}\cdot \frac{\overline{z_{i}} }{2a^{2}_{i}}\right)_{n,\delta }\right|^2\,\mathrm{d}P\right)^{\frac{1}{2}},
\end{eqnarray*}
which tends zero as $\delta\to 0+$ (by the conclusion (1) of Proposition \ref{convolution properties}). Thus, for all $n\in\mathbb{N}$,
\begin{eqnarray*}
\lim_{\delta\to 0+}\left|\left|\sum^{\prime }_{\left| I\right|  =s,\left| L\right|  =t,\max \left\{ I\bigcup L\right\}  \leqslant n} \sum^{n}_{i=1}\frac{c_{I,iL}}{c_{I,L}}\cdot\left( (f_{I,iL})_{n,\delta }\cdot \frac{ \overline{z_{i}}}{2a^{2}_{i}} -\left(f_{I,iL}\cdot \frac{\overline{z_{i}} }{2a^{2}_{i}}\right)_{n,\delta }\right) dz_{I}\wedge d\overline{z_{L}} \right|\right|^2_{L^{2}_{\left( s,t\right)  }\left( V_{\rho^{\prime } },P\right)  }=0.
\end{eqnarray*}
We also note that
\begin{eqnarray*}
&& \lim_{\delta\to 0+}\left| \left| \sum^{\prime }_{\left| I\right|  =s,\left| L\right|  =t,\max \{ I\bigcup L\} \leqslant n} \left(\sum^{n}_{i=1}\frac{c_{I,iL}}{c_{I,L}}\cdot(\partial_{i} \eta_{\rho }) \cdot (f_{I,iL})_{n,\delta }\right)dz_{I}\wedge d\overline{z_{L}}  \right|  \right|_{L^{2}_{(s,t)}\left( V_{\rho^{\prime } },P\right)  }^2\\
&&= \lim_{\delta\to 0+}\sum^{\prime }_{\left| I\right|  =s,\left| L\right|  =t,\max \{ I\bigcup L\} \leqslant n} c_{I,L}\cdot\int_{V_{\rho^{\prime } }}\left|\sum^{n}_{i=1}\frac{c_{I,iL}}{c_{I,L}}\cdot(\partial_{i} \eta_{\rho }) \cdot (f_{I,iL})_{n,\delta }\right|^2\,\mathrm{d}P\\
&&=\sum^{\prime }_{\left| I\right|  =s,\left| L\right|  =t,\max \{ I\bigcup L\} \leqslant n} c_{I,L}\cdot\int_{V_{\rho^{\prime } }}\left|\sum^{n}_{i=1}\frac{c_{I,iL}}{c_{I,L}}\cdot(\partial_{i} \eta_{\rho }) \cdot (f_{I,iL})_{n }\right|^2\,\mathrm{d}P\\
&&\leqslant\sum^{\prime }_{\left| I\right|  =s,\left| L\right|  =t,\max \{ I\bigcup L\} \leqslant n} c_{I,L}\cdot\int_{V_{\rho^{\prime } }}\left(\sum^{n}_{i=1}\frac{c_{I,iL}}{c_{I,L}}\cdot|\partial_{i} \eta_{\rho }|^2\right) \cdot \left(\sum^{n}_{i=1}\frac{c_{I,iL}}{c_{I,L}}\cdot(f_{I,iL})_{n }|^2\right)\,\mathrm{d}P\\
&&\leqslant C_2\cdot\sum^{\prime }_{\left| I\right|  =s,\left| L\right|  =t,\max \{ I\bigcup L\} \leqslant n} c_{I,L}\cdot\int_{V_{\rho^{\prime } }}|h_{\rho }'(\eta)|^2 \cdot \left(\sum^{n}_{i=1}\frac{c_{I,iL}}{c_{I,L}}\cdot|(f_{I,iL})_{n }|^2\right)\,\mathrm{d}P\\
&&=C_2\cdot(t+1)\cdot\sum^{\prime }_{\left| I\right|  =s,\left| J\right|  =t+1,\max \{ I\bigcup J\} \leqslant n} c_{I,J}\cdot\int_{V_{\rho^{\prime } }}|h_{\rho }'(\eta)|^2 \cdot |(f_{I,J})_{n }|^2\,\mathrm{d}P,
\end{eqnarray*}
where $C_2\triangleq c_1^{s,t}\cdot\sup\limits_{V_{\rho^{\prime } }}\left(\sum\limits^{\infty }_{i=1} \left| \overline{\partial_{i} } \eta \right|^{2}\right)<\infty$. Similarly to \eqref{end of STEP 2}, we have
\begin{eqnarray*}
&&\lim_{n\to\infty}\sum^{\prime }_{\left| I\right|  =s,\left| J\right|  =t+1,\max \{ I\bigcup L\} \leqslant n} c_{I,J}\cdot\int_{V_{\rho^{\prime } }}|h_{\rho }'(\eta)|^2 \cdot |(f_{I,J})_{n }|^2\,\mathrm{d}P\\
&&=\sum^{\prime }_{\left| I\right|  =s,\left| J\right|  =t+1} c_{I,J}\cdot\int_{V_{\rho^{\prime } }}|h_{\rho }'(\eta)|^2 \cdot |f_{I,J}|^2\,\mathrm{d}P=0,
\end{eqnarray*}
and hence
$$
\lim_{n\to\infty}\lim_{\delta\to 0+}\left| \left| \sum^{\prime }_{\left| I\right|  =s,\left| L\right|  =t,\max \{ I\bigcup L\} \leqslant n} \left(\sum^{n}_{i=1}\frac{c_{I,iL}}{c_{I,L}}\cdot(\partial_{i} \eta_{\rho }) \cdot (f_{I,iL})_{n,\delta }\right)dz_{I}\wedge d\overline{z_{L}}  \right|  \right|_{L^{2}_{(s,t)}\left( V_{\rho^{\prime } },P\right)  }=0.
$$
Therefore,
we obtain the desired equality (\ref{230407e3}).

\medskip

\textbf{Step 4:} In this step, we prove that $\lim\limits_{n\rightarrow \infty } \lim\limits_{\delta \rightarrow 0^{+}} \left| \left| \mathbb{I}\left( \eta_{\rho }\cdot f_{n,\delta }\right)  -\mathbb{I}f\right|  \right|_{L^{2}_{(s,t)}\left( V_{\rho^{\prime } },w_{1}\right)  }  =0$.
Since $\supp (\mathbb{I}f)\stackrel{\circ}{\subset} V_{r}^{o}$, $\eta_{\rho } \cdot(\mathbb{I}f)=\mathbb{I}f$ and
\begin{eqnarray*}
&&\left| \left| \mathbb{I}\left( \eta_{\rho }\cdot f_{n,\delta }\right)  -\mathbb{I}f\right|  \right|_{L^{2}_{(s,t)}\left( V_{\rho^{\prime } },w_{1}\right)  }\leqslant \left| \left| \mathbb{I}\left( \eta_{\rho }\cdot f_{n,\delta }\right)  -\eta_{\rho }\cdot (\mathbb{I}f)\right|  \right|_{L^{2}_{(s,t)}\left( V_{\rho^{\prime } },w_{1}\right)  } +\left| \left| \eta_{\rho }\cdot (\mathbb{I}f)-\mathbb{I}f\right|  \right|_{L^{2}_{(s,t)}\left( V_{r},w_{1}\right)  },
\end{eqnarray*}
it suffices to prove that
\begin{equation}\label{230417e1}
\lim\limits_{n\to\infty}\lim\limits_{\delta\to 0+}\left| \left| \mathbb{I}\left( \eta_{\rho }\cdot f_{n,\delta }\right)  -\eta_{\rho } \cdot\left(\mathbb{I}f\right)   \right|  \right|_{L^{2}_{(s,t)}\left( V_{\rho^{\prime } },w_{1}\right)  }=0.
\end{equation}

By the definition of $\mathbb{I}$ in (\ref{general def of I and P}), it follows that
$$
\mathbb{I}\left( \eta_{\rho }\cdot f_{n,\delta }\right)  =\sum^{\prime }_{\left| I\right|  =s,\left| L\right|  =t,\max \left\{ I\bigcup L\right\}  \leqslant n} \left(\sum^{n}_{i=1} \frac{c_{I,iL}}{c_{I,L}}\cdot\eta_{\rho }\cdot (f_{I,iL})_{n,\delta }\cdot\partial_{i} w_{2}\right)dz_{I}\wedge d\overline{z_{L}} .
$$
By the conclusion (2) of Proposition \ref{convolution properties} and noting that $\sup\limits_{V_{\rho^{\prime } }}|\eta_{\rho }\cdot (\partial_{i} w_{2})\cdot e^{-w_1}|<\infty$ for each $i=1,2,\cdots,n$, we deduce that
\begin{equation}\label{230418e01}
\lim_{\delta \rightarrow 0^{+}} \left| \left| \mathbb{I}\left( \eta_{\rho }\cdot f_{n,\delta }\right)  -\sum^{\prime }_{\left| I\right|  =s,\left| L\right|  =t,\max \left\{ I\bigcup L\right\}  \leqslant n} \left(\sum^{n}_{i=1}\frac{c_{I,iL}}{c_{I,L}}\cdot \eta_{\rho }\cdot (f_{I,iL})_n\cdot\partial_{i} w_{2}\right )dz_{I}\wedge d\overline{z_{L}}\right|  \right|_{L^{2}_{\left( s,t\right)  }\left( V_{\rho^{\prime } },w_{1}\right)  }  =0.
\end{equation}

Let $\theta \triangleq h_{\rho^{\prime }+1 } (\eta )$. Then $\theta\in C^{\infty }_{0,F}\left( V\right)$.
For any $g=\sum\limits^{\prime }_{\left| I\right|  =s,\left| J\right|=t+1}g_{I,J}dz_{I}\wedge d\overline{z_{L}}\in L^{2}_{(s,t+1)}\left( \ell^{2} ,P\right)$, similarly to the proof of (\ref{230414e3}), we have
\begin{equation}\label{general commutator estimation}
\begin{aligned}
&\Bigg|\Bigg|\sum^{\prime }_{\left| I\right|  =s,\left| L\right|  =t} \left(\sum^{\infty }_{i=1}  \frac{c_{I,iL}}{c_{I,L}}\cdot g_{I,iL}\cdot\partial_{i} \left( \theta\cdot w_{2}\right) \right)dz_{I}\wedge d\overline{z_{L}}\\
&\quad-\sum^{\prime }_{\left| I\right|  =s,\left| L\right|  =t,\max \left\{ I\bigcup L\right\}  \leqslant n}\left(\sum^{n}_{i=1} \frac{c_{I,iL}}{c_{I,L}}\cdot  (g_{I,iL})_n\cdot\partial_{i} \left( \theta\cdot w_{2}\right)  \right)dz_{I}\wedge d\overline{z_{L}} \Bigg|\Bigg|^{2}_{L^{2}_{\left( s,t\right)  }\left( V_{\rho^{\prime } },w_{1}\right)  }\\
&\leqslant 2\left| \left| \sum^{\prime }_{\left| I\right|  =s,\left| L\right|  =t} \left(\sum^{\infty }_{i=1}\frac{c_{I,iL}}{c_{I,L}}\cdot  g_{I,iL}\cdot\partial_{i} \left( \theta\cdot w_{2}\right) \right)dz_{I}\wedge d\overline{z_{L}} \right|  \right|^{2}_{L^{2}_{\left( s,t\right)  }\left( V_{\rho^{\prime } },w_{1}\right)  } \\
&\quad +2\left| \left| \sum^{\prime }_{\left| I\right|  =s,\left| L\right|  =t,\max \left\{ I\bigcup L\right\}  \leqslant n} \left(\sum^{n}_{i=1}\frac{c_{I,iL}}{c_{I,L}}\cdot  (g_{I,iL})_{n}\cdot\partial_{i} \left( \theta\cdot w_{2}\right)  \right)dz_{I}\wedge d\overline{z_{L}} \right|  \right|^{2}_{L^{2}_{\left( s,t\right)  }\left( V_{\rho^{\prime } },w_{1}\right)  } \\
&= 2\sum^{\prime }_{\left| I\right|  =s,\left| L\right|  =t} c_{I,L}\int_{V_{\rho^{\prime } }} \left| \sum^{\infty }_{i=1} \frac{c_{I,iL}}{c_{I,L}}\cdot g_{I,iL}\cdot\partial_{i} \left( \theta\cdot w_{2}\right) \right|^{2}e^{-w_1}\,\mathrm{d}P\\
&\quad+2\sum^{\prime }_{\left| I\right|  =s,\left| L\right|  =t,\max \left\{ I\bigcup L\right\}  \leqslant n} c_{I,L}\int_{V_{\rho^{\prime } }} \left| \sum^{n}_{i=1}\frac{c_{I,iL}}{c_{I,L}}\cdot  (g_{I,iL})_{n}\cdot\partial_{i} \left( \theta\cdot w_{2}\right)  \right|^{2}e^{-w_1}\,\mathrm{d}P\\
&\leqslant 2C_3\sum^{\prime }_{\left| I\right|  =s,\left| L\right|  =t} c_{I,L}\int_{V_{\rho^{\prime } }}  \sum^{\infty }_{i=1} \frac{c_{I,iL}}{c_{I,L}}\cdot |g_{I,iL}|^{2}\,\mathrm{d}P
+2C_3\sum^{\prime }_{\left| I\right|  =s,\left| L\right|  =t,\max \left\{ I\bigcup L\right\}  \leqslant n} c_{I,L}\int_{V_{\rho^{\prime } }}  \sum^{n}_{i=1}\frac{c_{I,iL}}{c_{I,L}}\cdot  |(g_{I,iL})_{n} |^{2}\,\mathrm{d}P\\
&=2C_3\cdot(t+1)\cdot\left(\sum^{\prime }_{\left| I\right|  =s,\left| J\right|=t+1} c_{I,J}\int_{V_{\rho^{\prime } }}  \sum^{\infty }_{i=1} |g_{I,J}|^{2}\,\mathrm{d}P
+ \sum^{\prime }_{\left| I\right|  =s,\left| J\right|=t+1,\max \left\{ I\bigcup L\right\}  \leqslant n} c_{I,J}\int_{V_{\rho^{\prime } }}  \sum^{n}_{i=1} |(g_{I,J})_{n} |^{2}\,\mathrm{d}P\right)\\
&\leqslant 4C_3\cdot(t+1)\cdot\left(\sum^{\prime }_{\left| I\right|  =s,\left| J\right|=t+1} c_{I,J}\int_{\ell^2}  \sum^{\infty }_{i=1} |g_{I,J}|^{2}\,\mathrm{d}P\right)=4C_3\cdot(t+1)\cdot ||g||^2_{ L^{2}_{(s,t+1)}\left( \ell^{2} ,P\right) },
\end{aligned}
\end{equation}
where $C_3\triangleq \left(\sup\limits_{V_{\rho^{\prime } }}\sum\limits^{\infty }_{i=1}  \left| \partial_{i} \left( \theta\cdot w_{2}\right) \right|^{2}\right)\cdot c_1^{s,t}<\infty$ and the third inequality follows from the conclusion (1) of Proposition \ref{Reduce diemension}.

Write
$$
\mathscr{C}_{(s,t+1)}\triangleq\bigcup^{\infty }_{m=1} \left\{ \sum^{\prime }_{\left| I\right|  =s,\left| J\right|  =t+1,\max \left\{ I\bigcup J\right\}  \leqslant m} h_{I,J}dz_{I}\wedge d\overline{z_{J}}:\;h_{I,J}\in   C^{\infty }_{c}(\mathbb{C}^m)  \right\}.
$$
Combining the sequence of cut-off functions given in \eqref{section2}, the final part of the proof of Proposition \ref{general cut-off density3} and Lemma \ref{approximation for st form}, we see that $\mathscr{C}_{(s,t+1)}$ is dense in $L^{2}_{(s,t+1)}\left( \ell^{2} ,P\right)$. As a consequence, for every $\varepsilon>0$, there exists $m\in\mathbb{N}$ and $g=\sum\limits^{\prime }_{\left| I\right|  =s,\left| J\right|  =t+1,\max \left\{ I\bigcup J\right\}  \leqslant m} g_{I,J}dz_{I}\wedge d\overline{z_{J}}\in \mathscr{C}_{(s,t+1)}$ such that $\left| \left| f-g\right|  \right|^{2}_{L^{2}_{(s,t+1)}\left( \ell^{2} ,P\right)  }  \leqslant \frac{\varepsilon}{8C_3(t+1)+1} $. Then for any $n>m$, noting that $\theta(\textbf{z})=1$ for all $\textbf{z}\in V_{\rho^{\prime } }$, we have
\begin{eqnarray*}
&&\left| \left| \mathbb{I}f-\sum^{\prime }_{\left| I\right|  =s,\left| L\right|  =t, \max \left\{I\bigcup L\right\}  \leqslant n} \left(\sum^{n}_{i=1}\frac{c_{I,iL}}{c_{I,L}}\cdot  (f_{I,iL})_{n}\cdot\partial_{i} w_{2}\right)dz_{I}\wedge d\overline{z_{L}}\right|  \right|^{2}_{L^{2}_{\left( s,t\right)  }\left( V_{\rho^{\prime } },w_{1}\right)  } \\
&&=\Bigg| \Bigg| \sum^{\prime }_{\left| I\right|  =s,\left| L\right|  =t} \left(\sum^{\infty }_{i=1}\frac{c_{I,iL}}{c_{I,L}}\cdot  f_{I,iL}\cdot\partial_{i} w_{2}\right)dz_{I}\wedge d\overline{z_{L}}\\
&&\quad-\sum^{\prime }_{\left| I\right|  =s,\left| L\right|  =t, \max \left\{I\bigcup L\right\}  \leqslant n} \left(\sum^{n}_{i=1}\frac{c_{I,iL}}{c_{I,L}}\cdot  (f_{I,iL})_{n}\cdot\partial_{i} w_{2}\right)dz_{I}\wedge d\overline{z_{L}}\Bigg| \Bigg|^{2}_{L^{2}_{\left( s,t\right)  }\left( V_{\rho^{\prime } },w_{1}\right)  } \\
&&=\Bigg| \Bigg|\sum^{\prime }_{\left| I\right|  =s,\left| L\right|  =t} \left(\sum^{\infty }_{i=1}\frac{c_{I,iL}}{c_{I,L}}\cdot  f_{I,iL}\cdot\partial_{i} \left( \theta\cdot w_{2}\right) \right)dz_{I}\wedge d\overline{z_{L}} \\
&&\quad-\sum^{\prime }_{\left| I\right|  =s,\left| L\right|  =t,\max \left\{ I\bigcup L\right\}  \leqslant n} \left(\sum^{n}_{i=1}\frac{c_{I,iL}}{c_{I,L}}\cdot  (f_{I,iL})_{n}\cdot\partial_{i} \left( \theta\cdot w_{2}\right) \right)dz_{I}\wedge d\overline{z_{L}}\Bigg| \Bigg|^{2}_{L^{2}_{\left( s,t\right)  }\left( V_{\rho^{\prime } },w_{1}\right)  } \\
&&\leqslant  2\Bigg| \Bigg| \sum^{\prime }_{\left| I\right|  =s,\left| L\right|  =t} \left(\sum^{\infty }_{i=1}\frac{c_{I,iL}}{c_{I,L}}\cdot  (f_{I,iL}-g_{I,iL})\cdot\partial_{i} \left( \theta\cdot w_{2}\right)\right)dz_{I}\wedge d\overline{z_{L}}\\
&&\quad-\sum^{\prime }_{\left| I\right|  =s,\left| L\right|  =t,\max \left\{ I\bigcup L\right\}  \leqslant n} \left(\sum^{n}_{i=1} \frac{c_{I,iL}}{c_{I,L}}\cdot ((f_{I,iL})_{n}-(g_{I,iL})_{n})\cdot\partial_{i} \left( \theta\cdot w_{2}\right)\right)dz_{I}\wedge d\overline{z_{L}} \Bigg| \Bigg|^{2}_{L^{2}_{\left( s,t\right)  }\left( V_{\rho^{\prime } },w_{1}\right)  }  \\
&&\quad+2\Bigg| \Bigg| \sum^{\prime }_{\left| I\right|  =s,\left| L\right|  =t} \left(\sum^{\infty }_{i=1}\frac{c_{I,iL}}{c_{I,L}}\cdot  g_{I,iL}\cdot\partial_{i} \left( \theta\cdot w_{2}\right)\right)dz_{I}\wedge d\overline{z_{L}}\\
&&\quad-\sum^{\prime }_{\left| I\right|  =s,\left| L\right|  =t,\max \left\{ I\bigcup L\right\}  \leqslant n} \left(\sum^{n}_{i=1}\frac{c_{I,iL}}{c_{I,L}}\cdot  (g_{I,iL})_{n}\cdot\partial_{i} \left( \theta\cdot w_{2}\right) \right)dz_{I}\wedge d\overline{z_{L}} \Bigg| \Bigg|^{2}_{L^{2}_{\left( s,t\right)  }\left( V_{\rho^{\prime } },w_{1}\right)  }\\
&& \leqslant 8C_3\cdot(t+1)\cdot\left| \left| f-g\right|  \right|^{2}_{L^{2}_{\left( s,t+1\right)  }\left( \ell^{2} ,P\right)  }   \leqslant  \varepsilon,
\end{eqnarray*}
where the second inequality follows from \eqref{general commutator estimation} and the fact that $g=\sum\limits^{\prime }_{\left| I\right|  =s,\left| J\right|  =t+1,\max \left\{ I\bigcup J\right\}  \leqslant m} g_{I,J}dz_{I}\wedge d\overline{z_{J}}\in \mathscr{C}_{(s,t+1)}$, and hence
$$
	\begin{array}{ll}
\displaystyle \sum^{\prime }_{\left| I\right|  =s,\left| L\right|  =t} \left(\sum^{\infty }_{i=1}\frac{c_{I,iL}}{c_{I,L}}\cdot   g_{I,iL}\cdot\partial_{i} \left( \theta\cdot w_{2}\right)  \right)dz_{I}\wedge d\overline{z_{L}} \\[3mm]
\displaystyle \quad-\sum^{\prime }_{\left| I\right|  =s,\left| L\right|  =t,\max \left\{ I\bigcup L\right\}  \leqslant n} \left(\sum^{n}_{i=1}\frac{c_{I,iL}}{c_{I,L}}\cdot   (g_{I,iL})_{n}\cdot\partial_{i} \left( \theta\cdot w_{2}\right) \right)dz_{I}\wedge d\overline{z_{L}}  \\[3mm]
\displaystyle =\sum^{\prime }_{\left| I\right|  =s,\left| L\right|  =t,\max \left\{ I\bigcup L\right\}  \leqslant m}\left(\sum^{m}_{i=1}\frac{c_{I,iL}}{c_{I,L}}\cdot   g_{I,iL}\cdot\partial_{i} \left( \theta\cdot w_{2}\right) \right)dz_{I}\wedge d\overline{z_{L}}\\[3mm]
\displaystyle \quad-\sum^{\prime }_{\left| I\right|  =s,\left| L\right|  =t,\max \left\{ I\bigcup L\right\}  \leqslant m} \left(\sum^{m}_{i=1}\frac{c_{I,iL}}{c_{I,L}}\cdot   (g_{I,iL})_{n}\cdot\partial_{i} \left( \theta\cdot w_{2}\right) \right)dz_{I}\wedge d\overline{z_{L}} \\[3mm]
\displaystyle =\sum^{\prime }_{\left| I\right|  =s,\left| L\right|  =t,\max \left\{ I\bigcup L\right\}  \leqslant m} \left(\sum^{m}_{i=1}\frac{c_{I,iL}}{c_{I,L}}\cdot   g_{I,iL}\cdot\partial_{i} \left( \theta\cdot w_{2}\right) \right)dz_{I}\wedge d\overline{z_{L}}\\[3mm]
\displaystyle \quad-\sum^{\prime }_{\left| I\right|  =s,\left| L\right|  =t,\max \left\{ I\bigcup L\right\}  \leqslant m}\left(\sum^{m}_{i=1} \frac{c_{I,iL}}{c_{I,L}}\cdot  g_{I,iL}\cdot\partial_{i} \left( \theta\cdot w_{2}\right) \right)dz_{I}\wedge d\overline{z_{L}} \\[3mm]
\displaystyle =0.
	\end{array}
$$
By the arbitrariness of $\varepsilon$, we have
\begin{equation}\label{230418e02}
\lim_{n\rightarrow \infty } \left| \left|  \mathbb{I}f-\sum^{\prime }_{\left| I\right|  =s,\left| L\right|  =t,\max \left\{ I\bigcup L\right\}  \leqslant n} \left(\sum^{n}_{i=1}\frac{c_{I,iL}}{c_{I,L}}\cdot  (f_{I,iL})_{n}\cdot\partial_{i} w_{2}\right)dz_{I}\wedge d\overline{z_{L}} \right|  \right|_{L^{2}_{(s,t)}\left( V_{\rho^{\prime } },w_{1}\right)  }  =0.
\end{equation}
Since
\begin{eqnarray*}
&&\left| \left| \eta_{\rho } \cdot(\mathbb{I}f)-\mathbb{I}\left( \eta_{\rho }\cdot f_{n,\delta }\right)  \right|  \right|_{L^{2}_{(s,t)}\left( V_{\rho^{\prime } },w_{1}\right)  } \\
&&\leqslant \left| \left|   \eta_{\rho }\cdot (\mathbb{I}f)-\sum^{\prime }_{\left| I\right|  =s,\left| L\right|  =t,\max \left\{ I\bigcup L\right\}  \leqslant n} \left(\sum^{n}_{i=1}  \eta_{\rho }\cdot\frac{c_{I,iL}}{c_{I,L}}\cdot  (f_{I,iL})_{n}\cdot\partial_{i} w_{2}\right)dz_{I}\wedge d\overline{z_{L}} \right|  \right|_{L^{2}_{(s,t)}\left( V_{\rho^{\prime } },w_{1}\right)  }\\
&&\quad+\left| \left| \mathbb{I}\left( \eta_{\rho }\cdot f_{n,\delta }\right)  -\sum^{\prime }_{\left| I\right|  =s,\left| L\right|  =t,\max \left\{ I\bigcup L\right\}  \leqslant n} \left(\sum^{n}_{i=1} \eta_{\rho }\cdot\frac{c_{I,iL}}{c_{I,L}}\cdot  (f_{I,iL})_n\cdot\partial_{i} w_{2}\right )dz_{I}\wedge d\overline{z_{L}}\right|  \right|_{L^{2}_{\left( s,t\right)  }\left( V_{\rho^{\prime } },w_{1}\right)  }\\
&&\leqslant \left| \left| \mathbb{I}f-\sum^{\prime }_{\left| I\right|  =s,\left| L\right|  =t,\max \left\{ I\bigcup L\right\}  \leqslant n} \left(\sum^{n}_{i=1} \frac{c_{I,iL}}{c_{I,L}}\cdot  (f_{I,iL})_{n}\cdot\partial_{i} w_{2}\right)dz_{I}\wedge d\overline{z_{L}} \right|  \right|_{L^{2}_{(s,t)}\left( V_{\rho^{\prime } },w_{1}\right)  }\\
&&\quad+\left| \left| \mathbb{I}\left( \eta_{\rho }\cdot f_{n,\delta }\right)  -\sum^{\prime }_{\left| I\right|  =s,\left| L\right|  =t,\max \left\{ I\bigcup L\right\}  \leqslant n} \left(\sum^{n}_{i=1} \eta_{\rho }\cdot \frac{c_{I,iL}}{c_{I,L}}\cdot  (f_{I,iL})_n\cdot\partial_{i} w_{2}\right )dz_{I}\wedge d\overline{z_{L}}\right|  \right|_{L^{2}_{\left( s,t\right)  }\left( V_{\rho^{\prime } },w_{1}\right)  },
\end{eqnarray*}
by (\ref{230418e01}) and (\ref{230418e02}), we obtain the desired equality (\ref{230417e1}).

\medskip

\textbf{Step 5:} Finally, we come to show that
$$
\begin{aligned}
		&
\lim\limits_{n\rightarrow \infty } \lim\limits_{\delta \rightarrow 0^{+}}\left( \left| \left| T^{\ast }\left( \eta_{\rho }\cdot f_{n,\delta }\right)  -T^{\ast }f\right|  \right|_{L^{2}_{(s,t)}\left( V,w_{1}\right)  }  +\left| \left| \eta_{\rho }\cdot f_{n,\delta }-f\right|  \right|_{L^{2}_{(s,t+1)}\left( V,w_{2}\right)  }  +\left| \left|  S\left( \eta_{\rho }\cdot f_{n,\delta }\right)  -Sf\right|  \right|_{L^{2}_{(s,t+2)}\left( V,w_{3}\right)  }\right) \\
& = 0.
\end{aligned}
$$
Combining Steps 3 and 4, and noting $T^*\left( \eta_{\rho }\cdot f_{n,\delta }\right)=(-1)^{s}e^{w_1-w_2}\cdot\left(\mathbb{I}\left( \eta_{\rho }\cdot f_{n,\delta }\right)-\mathbb{P}\left( \eta_{\rho }\cdot f_{n,\delta }\right)\right)$, $T^*f=(-1)^{s}e^{w_1-w_2}\cdot(\mathbb{I}f-\mathbb{P}f)$ and $\sup\limits_{V_{\rho^{\prime } }} e^{w_{1}-w_{2}}<\infty$, we have
\begin{equation}
\begin{array}{ll}
\displaystyle \lim_{n\rightarrow \infty } \lim_{\delta \rightarrow 0^{+}} \left| \left| T^{\ast }\left( \eta_{\rho }\cdot f_{n,\delta }\right)  -T^{\ast }f\right|  \right|_{L^{2}_{(s,t)}\left( V_{\rho^{\prime } },w_{1}\right)  }  \\[3mm]
\displaystyle =\lim_{n\rightarrow \infty } \lim_{\delta \rightarrow 0^{+}} \left| \left| \left( -1\right)^{s}  e^{w_{1}-w_{2}}\cdot\left(\mathbb{I}\left( \eta_{\rho }\cdot f_{n,\delta }\right)  - \mathbb{I}f \right)-\left( -1\right)^{s}  e^{w_{1}-w_{2}}\cdot\left(\mathbb{P}\left( \eta_{\rho }\cdot f_{n,\delta }\right)-\mathbb{P}f\right)\right|  \right|_{L^{2}_{(s,t)}\left( V_{\rho^{\prime } },w_{1}\right)  }\\[3mm]
\displaystyle \leqslant \left(\sup\limits_{V_{\rho^{\prime } }} e^{w_{1}-w_{2}}\right)\cdot\left(\lim_{n\rightarrow \infty } \lim_{\delta \rightarrow 0^{+}} \left| \left| \mathbb{I}\left( \eta_{\rho }\cdot f_{n,\delta }\right)  - \mathbb{I}f \right|  \right|_{L^{2}_{(s,t)}\left( V_{\rho^{\prime } },w_{1}\right)  }+\lim_{n\rightarrow \infty } \lim_{\delta \rightarrow 0^{+}}\left| \left| \mathbb{P}\left( \eta_{\rho }\cdot f_{n,\delta }\right)-\mathbb{P}f\right|  \right|_{L^{2}_{(s,t)}\left( V_{\rho^{\prime } },w_{1}\right)  }\right)\\[3mm]
\displaystyle  =0.\label{general 11111}
\end{array}
\end{equation}
By Steps 1 and 2, and noting $\eqref{general 11111}$, we have
\begin{equation}\label{230418e5}
\begin{aligned}
		&\lim\limits_{n\rightarrow \infty } \lim\limits_{\delta \rightarrow 0^{+}}\left( \left| \left| T^{\ast }\left( \eta_{\rho }\cdot f_{n,\delta }\right)  -T^{\ast }f\right|  \right|_{L^{2}_{(s,t)}\left( V_{\rho^{\prime } },w_{1}\right)  }  +\left| \left| \eta_{\rho }\cdot f_{n,\delta }-f\right|  \right|_{L^{2}_{(s,t+1)}\left( V_{\rho^{\prime } },w_{2}\right)  }  \right.\\
&\qquad\qquad\left.+\left| \left|  S\left( \eta_{\rho }\cdot f_{n,\delta }\right)  -Sf\right|  \right|_{L^{2}_{(s,t+2)}\left( V_{\rho^{\prime } },w_{3}\right)  }\right)\\
&= 0.
\end{aligned}
\end{equation}
By the fact that $\supp \eta_{\rho }\subset V^{o}_{\rho^{\prime } }$ and  Proposition \ref{support argument}, we have
\begin{eqnarray*}
&&\supp (\eta_{\rho }\cdot f_{n,\delta }) \stackrel{\circ}{\subset}  V^{o}_{\rho^{\prime } }, \quad\supp \left(S\left( \eta_{\rho }\cdot f_{n,\delta }\right)\right) \stackrel{\circ}{\subset}  V^{o}_{\rho^{\prime } },\quad \supp \left(T^{\ast }\left( \eta_{\rho }\cdot f_{n,\delta }\right)  \right) \stackrel{\circ}{\subset}  V^{o}_{\rho^{\prime } },\\
&&\supp f \stackrel{\circ}{\subset}  V^{o}_{\rho^{\prime } }, \quad\supp (Sf) \stackrel{\circ}{\subset}  V^{o}_{\rho^{\prime } }, \quad\supp (T^{\ast }f) \stackrel{\circ}{\subset}  V^{o}_{\rho^{\prime } }.
\end{eqnarray*}
Hence
$$
	\begin{aligned}
	&\left| \left| T^{\ast }\left( \eta_{\rho }\cdot f_{n,\delta }\right)  -T^{\ast }f\right|  \right|_{L^{2}_{(s,t)}\left( V,w_{1}\right)  }  +\left| \left| \eta_{\rho }\cdot f_{n,\delta }-f\right|  \right|_{L^{2}_{(s,t+1)}\left( V,w_{2}\right)  }  +\left| \left| S\left( \eta_{\rho }\cdot f_{n,\delta }\right)  -Sf\right|  \right|_{L^{2}_{(s,t+2)}\left( V,w_{3}\right)  }  \\&=\left| \left| T^{\ast }\left( \eta_{\rho }\cdot f_{n,\delta }\right)  -T^{\ast }f\right|  \right|_{L^{2}_{(s,t)}\left( V_{\rho^{\prime } },w_{1}\right)  }  +\left| \left| \eta_{\rho }\cdot f_{n,\delta }-f\right|  \right|_{L^{2}_{(s,t+1)}\left( V_{\rho^{\prime } },w_{2}\right)  }  +\left| \left| S\left( \eta_{\rho }\cdot f_{n,\delta }\right)  -Sf\right|  \right|_{L^{2}_{(s,t+2)}\left( V_{\rho^{\prime } },w_{3}\right)  } .
	\end{aligned}
$$
Combining this with (\ref{230418e5}), we obtain the desired equality (\ref{230418e03}), which completes the proof of Theorem \ref{general density4}.
\end{proof}

\begin{remark}
We should note that $f_{n,\delta }$ may be not in $D_{T^{\ast }}\cap D_S$, but $\eta_{\rho }\cdot f_{n,\delta }\in D_{T^{\ast }}\cap D_S$.
\end{remark}

In the sequel, we need the following two assumptions on the coefficients in (\ref{defnition of general st froms}) (Recall (\ref{gener111}) for $c_{I,iJ}$).
\begin{condition}\label{230423ass2}
$\displaystyle c_0^{s,t}\triangleq \inf_{|I|=s,|J|=t,i\in\mathbb{N}}\frac{ c_{I,iJ} }{c_{I,J}}>0$.
\end{condition}

\begin{condition}\label{230423ass3}
The coefficients in (\ref{defnition of general st froms}) satisfy the following condition
\begin{eqnarray}\label{multiplictive condition for coefficient}
c_{I,J}\cdot c_{I,J'}=c_{I,L}\cdot c_{I,K}
\end{eqnarray}
for all strictly increasing multi-indices $I,J,J',L$ and $K$ with $|I|=s,|J|=t+1,|J'|=t+1,|L|=t,|K|=t+2, J\cup J'=K$ and $J\cap J'=L$.
\end{condition}

We also need the following assumption:

\begin{condition}\label{230423ass4}
The real-valued function $\varphi$ in \eqref{weight function} satisfies that for each $n\in\mathbb{N}$, the following inequality holds on $V$,
\begin{eqnarray}\label{230419e12}
\sum_{1\leqslant i,j\leqslant n} (\partial_{i} \overline{\partial_{j} } \varphi) \cdot\zeta_{i}\cdot\overline{\zeta_{j}}
\geqslant \left( 2\sum\limits^{n}_{i=1} |\partial_{i} \psi |^{2}+2e^{\psi }-\frac{1}{2} \right)\cdot\left(  \sum^{n}_{i=1} \left| \zeta_{i}\right|^{2}\right),\quad \forall\; (\zeta_{1},\cdots,\zeta_{n})\in\mathbb{C}^n.
\end{eqnarray}
\end{condition}

We are in a position to establish an infinite-dimensional version of \cite[Lemma 4.2.1, p. 84]{Hor90} as follows (Recall Condition \ref{230423ass2} for $c_0^{s,t}$):
\begin{lemma}\label{general estimation similar to Lemma 4.2.1}
Under Conditions \ref{230424c1}, \ref{230423ass1}, \ref{230423ass2}, \ref{230423ass3} and \ref{230423ass4}, it holds that
\begin{eqnarray}\label{230419e13}
\left| \left| T^{\ast }f\right|  \right|^{2}_{L^{2}_{\left( s,t\right)  }\left( V,w_{1}\right)  }  +\left| \left| Sf\right|  \right|^{2}_{L^{2}_{\left( s,t+2\right)  }\left( V,w_{3}\right)  }\geqslant
c_0^{s,t}\cdot\left| \left| f\right|  \right|^{2}_{L^{2}_{\left( s,t+1\right)  }\left( V,w_{2}\right)  }
\end{eqnarray}
for any $f=\sum\limits^{\prime }_{\left| I\right|  =s,\left| J\right|  =t+1 } f_{I,J}dz_{I}\wedge d\overline{z_{J}}\in D_{S}\bigcap D_{T^{\ast }}$ satisfying that there exists $m\in\mathbb{N}$ such that $f_{I,J}=0$ for all strictly increasing multi-indices $I$ and $J$ with $\max{\left\{ I\bigcup J\right\}  }>m$ and $f_{I,J}\in C^{2}_{0,F}\left( V\right)$ for all strictly increasing multi-indices $I$ and $J$ with $\max{\left\{ I\bigcup J\right\}  }\leqslant m$.
\end{lemma}

\begin{proof}
We proceed as in the proof of \cite[Lemma 4.2.1, p. 84]{Hor90}.
Since $w_3=\varphi$,
\begin{eqnarray*}
\left| \left| Sf\right|  \right|^{2}_{L^{2}_{\left( s,t+2\right)  }\left( V,w_{3}\right)  }
&=&\sum^{\prime }_{\left| I\right|  =s}\sum^{\prime }_{\left| K\right|  =t+2} c_{I,K}\int_{V} \left| \sum^{\infty }_{i=1} \sum^{\prime }_{\left| J\right|  =t+1} \varepsilon^{K}_{i,J}\cdot \overline{\partial_{i} } f_{I,J }\right|^{2}  e^{-w_{3}}\,\mathrm{d}P\\
&=&\sum^{\prime }_{\left| I\right|  =s}\sum^{\prime }_{\left| K\right|  =t+2} \int_{V} \sum_{1\leqslant i,j<\infty } \sum^{\prime }_{\left| J\right|=t+1  ,\left| J^{\prime }\right|  =t+1} c_{I,K}\cdot\varepsilon^{K}_{i,J}\cdot \varepsilon^{K}_{j,J^{\prime }}\cdot \overline{\partial_{j} } f_{I,J^{\prime }}\cdot\partial_{i} \overline{f_{I,J}}\cdot e^{-w_{3}}\,\mathrm{d}P
\\
&=&\sum^{\prime }_{\left| I\right|  =s}\sum^{\prime }_{\left| J\right|=t+1  ,\left| J^{\prime }\right|  =t+1}\sum_{1\leqslant i,j<\infty }\int_{V}   c_{I,iJ}\cdot\varepsilon^{j,J^{\prime }}_{i,J}\cdot \overline{\partial_{j} } f_{I,J^{\prime }}\cdot\partial_{i} \overline{f_{I,J}}\cdot e^{-w_{3}}\,\mathrm{d}P,
\end{eqnarray*}
where $c_{I,iJ}$ is defined as that in (\ref{gener111}), and we have used the fact that $\varepsilon^{K}_{i,J}\cdot \varepsilon^{K}_{j,J^{\prime }}\not=0$ if and only if $i\not\in J$, $j\not\in J^{\prime }$ and $K=\{i\}\cup J=\{j\}\cup J^{\prime }$ (in this case $\varepsilon^{K}_{i,J}\cdot \varepsilon^{K}_{j,J^{\prime }}=\varepsilon^{j,J^{\prime }}_{i,J}$, and $\varepsilon^{j,J^{\prime }}_{i,J}$ is the sign of the permutation from $j J^{\prime }$ to $i J$). We shall rearrange the terms in the above sum. Note that if $i=j$, then $\varepsilon^{j,J^{\prime }}_{i,J}\neq 0$ if and only if $j\notin J$ and $J^{\prime }=J$. If $i\neq j$, then $\varepsilon^{j,J^{\prime }}_{i,J}\neq 0$ if and only if there exists a strictly increasing multi-index $L$ with $|L|=t$ so that $i\not\in L$, $j\not\in L$, $J^{\prime }=L\cup \{i\}$ and $J=L\cup \{j\}$, in which case,
\begin{eqnarray*}
\varepsilon^{j,J^{\prime }}_{i,J}=\varepsilon^{j,J^{\prime }}_{j,i,L}\cdot\varepsilon^{j,i,L}_{i,j,L}\cdot\varepsilon^{i,j,L}_{i,J}=-\varepsilon^{j,J^{\prime }}_{j,i,L}\cdot\varepsilon^{i,j,L}_{i,J}=-\varepsilon^{J^{\prime }}_{i,L}\cdot\varepsilon^{j,L}_{J}.
\end{eqnarray*}
Combining \eqref{multiplictive condition for coefficient} in Condition \ref{230423ass3}, and noting that by (\ref{gener111}), $c_{I,iL}=0$ for $i\in L$, $c_{I,jL}=0$ for $j\in L$ and $c_{I,iJ}=0$ for $i\in J$, we then have
\begin{eqnarray*}
\left| \left| Sf\right|  \right|^{2}_{L^{2}_{\left( s,t+2\right)  }\left( V,w_{3}\right)  }  &=&\sum^{\prime }_{\left| I\right|  =s}\sum_{i\notin J} \sum^{\prime }_{\left| J\right|  =t+1}c_{I,iJ}\cdot \int_{V} |\overline{\partial_{i} } f_{I,J}|^{2}\cdot e^{-\varphi }\,\mathrm{d}P\\
&&-\sum^{\prime }_{\left| I\right|  =s}\sum_{i\neq j} \sum^{\prime }_{\left| J\right|  ,\left| J^{\prime }\right|  =t+1} \sum^{\prime }_{\left| L\right|  =t} \frac{c_{I,iL}\cdot c_{I,jL}}{c_{I, L}}\int_{V}\varepsilon^{J^{\prime }}_{i,L}\cdot\varepsilon^{j,L}_{J}\cdot \partial_{i} \overline{f_{I,J}} \cdot \overline{\partial_{j} } f_{I,J^{\prime }}\cdot e^{-\varphi }\,\mathrm{d}P\\
&=&\sum^{\prime }_{\left| I\right|  =s}\sum_{i\notin J} \sum^{\prime }_{\left| J\right|  =t+1} c_{I,iJ}\int_{V} |\overline{\partial_{i} } f_{I,J}|^{2}\cdot e^{-\varphi }\,\mathrm{d}P\\
&&-\sum^{\prime }_{\left| I\right|  =s}\sum_{i\neq j} \sum^{\prime }_{\left| L\right|  =t} \frac{c_{I,iL}\cdot c_{I,jL}}{c_{I, L}}\int_{V} \partial_{i} \overline{f_{I,jL}}\cdot\overline{\partial_{j} } f_{I,iL}\cdot e^{-\varphi }\,\mathrm{d}P\\
&\geqslant&\sum^{\prime }_{\left| I\right|  =s}\sum^{\infty }_{i=1} \sum^{\prime }_{\left| J\right|  =t+1}c_{I,iJ} \int_{V} |\overline{\partial_{i} } f_{I,J}|^{2}\cdot e^{-\varphi }\,\mathrm{d}P\\
&&-\sum^{\prime }_{\left| I\right|  =s}\sum_{1\leqslant i,j<\infty } \sum^{\prime }_{\left| L\right|  =t}\frac{c_{I,iL}\cdot c_{I,jL}}{c_{I, L}} \int_{V} \partial_{i} \overline{f_{I,jL}}\cdot \overline{\partial_{j} } f_{I,iL}\cdot e^{-\varphi }\,\mathrm{d}P,
\end{eqnarray*}
where the last inequality follows from the fact that
\begin{eqnarray*}
\sum^{\prime }_{\left| I\right|  =s}\sum_{ i=1}^{\infty } \sum^{\prime }_{\left| L\right|  =t}\frac{c_{I,iL}^2}{c_{I, L}} \int_{V} |\overline{\partial_{i} } f_{I,iL}|^2\cdot e^{-\varphi }\,\mathrm{d}P\geqslant 0.
\end{eqnarray*}
Recalling (\ref{230419e11}) for the operator $\sigma_{i}$, by Lemma \ref{integration by Parts}, we obtain that
\begin{equation}\label{middle equality}
\begin{aligned}
	&\sum^{\prime }_{\left| I\right|  =s}  \sum_{1\leqslant i,j<\infty } \sum^{\prime }_{\left| L\right|  =t} \frac{c_{I,iL}\cdot c_{I,jL}}{c_{I, L}}\int_{V} \left((\sigma_{i} f_{I,iL})\cdot\overline{(\sigma_{j} f_{I,jL})}-\overline{(\overline{\partial_{i}}f_{I,jL})}\cdot(\overline{\partial_{j}}f_{I,iL})\right)\cdot e^{-\varphi }\,\mathrm{d}P\\
&=-\sum^{\prime }_{\left| I\right|  =s}  \sum_{1\leqslant i,j<\infty } \sum^{\prime }_{\left| L\right|  =t}\frac{c_{I,iL}\cdot c_{I,jL}}{c_{I, L}} \int_{V}  f_{I,iL} \cdot \left(\overline{\overline{\partial_{i} } (\sigma_{j} f_{I,jL})- \sigma_{j} (\overline{\partial_{i} } f_{I,jL})}\right)\cdot e^{-\varphi }\,\mathrm{d}P\\
&=\sum^{\prime }_{\left| I\right|  =s}  \sum_{1\leqslant i,j<\infty } \sum^{\prime }_{\left| L\right|  =t} \frac{c_{I,iL}\cdot c_{I,jL}}{c_{I, L}}\int_{V} f_{I,iL}\cdot \overline{f_{I,jL}}\cdot \left( \partial_{i} \overline{\partial_{j} } \varphi +\frac{ \overline{\partial_i }(\overline{z_{j}})}{2a^{2}_{j}}  \right)  \cdot e^{-\varphi }\,\mathrm{d}P,
\end{aligned}
\end{equation}
where the last equality follows from \eqref{commutaor formula}.

By Proposition \ref{general formula of T*} , we have $T^{\ast }f =e^{-\psi }\cdot X+e^{-\psi }\cdot Y $, where
$$
\begin{aligned}
	&X\triangleq(-1)^{s+1}\sum^{\prime }_{\left| I\right|  =s} \sum^{\prime }_{\left| L\right|  =t} \sum^{\infty }_{i=1} \frac{c_{I,iL}}{c_{I,L}}\cdot \sigma_{i} f_{I,iL}dz_{I}\wedge d\overline{z_{L}} ,\\
&Y\triangleq(-1)^{s+1}\sum^{\prime }_{\left| I\right|  =s} \sum^{\prime }_{\left| L\right|  =t} \sum^{\infty }_{i=1} \frac{c_{I,iL}}{c_{I,L}}\cdot f_{I,iL}\cdot\partial_{i}\psi dz_{I}\wedge d\overline{z_{L}}.
\end{aligned}
$$
Then,
$$
\left| \left| e^{-\psi }\cdot X\right|  \right|^{2}_{L^{2}_{\left( s,t\right)  }\left( V,w_{1}\right)  }  \leqslant 2\left| \left| e^{-\psi }\cdot Y\right|  \right|^{2}_{L^{2}_{\left( s,t\right)  }\left( V,w_{1}\right)  }  +2\left| \left| T^{\ast }f\right|  \right|^{2}_{L^{2}_{\left( s,t\right)  }\left( V,w_{1}\right)  },
$$
$$
		\begin{aligned}
			\left| \left| e^{-\psi }\cdot X\right|  \right|^{2}_{L^{2}_{\left( s,t\right) }\left( V,w_{1}\right)  }  &=\sum^{\prime }_{\left| I\right|  =s} \sum^{\prime }_{\left| L\right|  =t} \int_{V} \sum_{1\leqslant i,j<\infty }\frac{c_{I,iL}\cdot c_{I,jL}}{c_{I,L}}\cdot \overline{ (\sigma_{i} f_{I,iL})}\cdot(\sigma_{j} f_{I,jL})\cdot e^{-w_{1}-2\psi}\,\mathrm{d}P\\
&=\sum^{\prime }_{\left| I\right|  =s} \sum^{\prime }_{\left| L\right|  =t} \int_{V} \sum_{1\leqslant i,j<\infty } \frac{c_{I,iL}\cdot c_{I,jL}}{c_{I,L}}\cdot \overline{(\sigma_{i} f_{I,iL})}\cdot(\sigma_{j} f_{I,jL})\cdot e^{-\varphi}\,\mathrm{d}P
		\end{aligned}
$$
and
$$
		\begin{aligned}
			\left| \left| e^{-\psi }\cdot Y\right|  \right|^{2}_{L^{2}_{\left( s,t\right) }\left( V,w_{1}\right)  }  &=\sum^{\prime }_{\left| I\right|  =s} \sum^{\prime }_{\left| L\right|  =t} \frac{1}{c_{I,L}}\int_{V} \bigg|\sum^{\infty }_{i=1} c_{I,iL}\cdot f_{I,iL}\cdot\partial_{i}\psi \bigg|^{2}\cdot e^{-w_{1}-2\psi }\,\mathrm{d}P\\
&\leqslant \sum^{\prime }_{\left| I\right|  =s}  \sum^{\prime }_{\left| L\right|  =t}\int_{V} \left(\sum^{m}_{i=1}\frac{|c_{I,iL}|^2}{c_{I,L}}\cdot \left| f_{I,iL}\right|^{2} \right) \cdot \left(\sum^{m }_{i=1} |\partial_{i}\psi|^{2}\right)e^{-\varphi }\,\mathrm{d}P,
		\end{aligned}
$$
where the inequality follows from Cauchy-Schwarz inequality and the assumption that $f_{I,J}=0$ for all strictly increasing multi-indices $I$ and $J$ with $\left| I\right|  =s$, $\left| L\right|  =t$ and $\max{\left\{ I\bigcup J\right\}  }>m$. Combining the above inequalities and equalities we have
$$
\begin{aligned}
& 2\left| \left| T^{\ast }f\right|  \right|^{2}_{L^{2}_{\left( s,t\right)  }\left( V,w_{1}\right)  }  +\left| \left| Sf\right|  \right|^{2}_{L^{2}_{\left( s,t+2\right)  }\left( V,w_{3}\right)  }\\
&\geqslant \sum^{\prime }_{\left| I\right|  =s} \sum^{\prime }_{\left| L\right|  =t} \sum_{1\leqslant i,j<\infty }\frac{c_{I,iL}\cdot c_{I,jL}}{c_{I,L}}\cdot  \int_{V}\overline{(\sigma_{i} f_{I,iL})}\cdot(\sigma_{j} f_{I,jL})\cdot e^{-\varphi}\,\mathrm{d}P\\
&\quad-2\sum^{\prime }_{\left| I\right|  =s}  \sum^{\prime }_{\left| L\right|  =t} \int_{V} \left(\sum^{m}_{i=1} \frac{|c_{I,iL}|^2}{c_{I,L}}\cdot \left| f_{I,iL}\right|^{2} \right) \cdot \left(\sum^{m }_{i=1} |\partial_{i}\psi |^{2}\right)\cdot e^{-\varphi }\,\mathrm{d}P\\
&\quad+\sum^{\prime }_{\left| I\right|  =s}  \sum^{\infty }_{i=1} \sum^{\prime }_{\left| J\right|  =t+1}c_{I,iJ} \int_{V} |\overline{\partial_{i} } f_{I,J}|^{2}\cdot e^{-\varphi }\,\mathrm{d}P-\sum^{\prime }_{\left| I\right|  =s}  \sum_{1\leqslant i,j<\infty } \sum^{\prime }_{\left| L\right|  =t} \frac{c_{I,iL}\cdot c_{I,jL}}{c_{I, L}}\int_{V} \partial_{i} \overline{f_{I,jL}}\cdot \overline{\partial_{j} } f_{I,iL}\cdot e^{-\varphi }\,\mathrm{d}P\\
&=\sum^{\prime }_{\left| I\right|  =s}  \sum^{\prime }_{\left| L\right|  =t} \int_{V} \sum_{1\leqslant i,j<\infty }\frac{c_{I,iL}\cdot c_{I,jL}}{c_{I, L}}\left( \overline{(\sigma_{i} f_{I,iL})}\cdot(\sigma_{j} f_{I,jL})-\partial_{i} \overline{f_{I,jL}}\cdot \overline{\partial_{j} } f_{I,iL}\right)\cdot e^{-\varphi}\,\mathrm{d}P
\\
&\quad+\sum^{\prime }_{\left| I\right|  =s}\sum^{\infty }_{i=1} \sum^{\prime }_{\left| J\right|  =t+1} c_{I,iJ}\int_{V} |\overline{\partial_{i} } f_{I,J}|^{2}\cdot e^{-\varphi }\,\mathrm{d}P- 2\sum^{\prime }_{\left| I\right|  =s} \sum^{\prime }_{\left| L\right|  =t} \int_{V} \left(\sum^{m}_{i=1} \frac{|c_{I,iL}|^2}{c_{I,L}}\cdot\left| f_{I,iL}\right|^{2}\right)  \cdot \left(\sum^{m }_{i=1} |\partial_{i}\psi |^{2}\right)\cdot e^{-\varphi }\,\mathrm{d}P\\
&=\sum^{\prime }_{\left| I\right|  =s} \sum^{\prime }_{\left| L\right|  =t}\frac{c_{I,iL}\cdot c_{I,jL}}{c_{I, L}}\int_{V} \sum_{1\leqslant i,j<\infty }f_{I,iL}\cdot\overline{f_{I,jL}}\cdot \left( \partial_{i} \overline{\partial_{j} } \varphi +\frac{\overline{\partial_i} (\overline{z_j})}{2a^{2}_{j}} \right)\cdot e^{-\varphi}\,\mathrm{d}P
\\
&\quad+\sum^{\prime }_{\left| I\right|  =s}\sum^{\infty }_{i=1} \sum^{\prime }_{\left| J\right|  =t+1} c_{I,iJ}\int_{V} |\overline{\partial_{i} } f_{I,J}|^{2}\cdot e^{-\varphi }\,\mathrm{d}P- 2\sum^{\prime }_{\left| I\right|  =s} \sum^{\prime }_{\left| L\right|  =t} \int_{V} \left(\sum^{m}_{i=1}\frac{|c_{I,iL}|^2}{c_{I,L}}\cdot \left| f_{I,iL}\right|^{2} \right) \cdot \left(\sum^{m }_{i=1} |\partial_{i}\psi |^{2}\right)\cdot e^{-\varphi }\,\mathrm{d}P\\
&=\sum^{\prime }_{\left| I\right|  =s}  \sum^{\prime }_{\left| L\right|  =t} \int_{V} \sum_{1\leqslant i,j\leqslant m }\frac{c_{I,iL}\cdot c_{I,jL}}{c_{I, L}}\cdot f_{I,iL}\cdot\overline{f_{I,jL}}\cdot \left( \partial_{i} \overline{\partial_{j} } \varphi +\frac{\overline{\partial_i} (\overline{z_j})}{2a^{2}_{j}}  \right)\cdot e^{-\varphi}\,\mathrm{d}P
\\
&\quad+\sum^{\prime }_{\left| I\right|  =s}  \sum^{\infty }_{i=1} \sum^{\prime }_{\left| J\right|  =t+1} c_{I,iJ}\int_{V} |\overline{\partial_{i} } f_{I,J}|^{2}\cdot e^{-\varphi }\,\mathrm{d}P- 2\sum^{\prime }_{\left| I\right|  =s}  \sum^{\prime }_{\left| L\right|  =t} \int_{V} \left(\sum^{m}_{i=1} \frac{|c_{I,iL}|^2}{c_{I,L}}\cdot \left| f_{I,iL}\right|^{2} \right) \cdot\left( \sum^{m }_{i=1} |\partial_{i}\psi |^{2}\right)\cdot e^{-\varphi }\,\mathrm{d}P,
\end{aligned}
$$
where the second equality follows from \eqref{middle equality}. Thus
\begin{equation}
	\begin{aligned}
				&2\left| \left| T^{\ast }f\right|  \right|^{2}_{L^{2}_{\left( s,t\right)  }\left( V,w_{1}\right)  }+\left| \left| Sf\right|  \right|^{2}_{L^{2}_{\left( s,t+2\right)  }\left( V,w_{3}\right)  }+2\sum^{\prime }_{\left| I\right|  =s}  \sum^{\prime }_{\left| L\right|  =t} \int_{V} \left(\sum^{m}_{i=1} \frac{|c_{I,iL}|^2}{c_{I,L}}\cdot \left| f_{I,iL}\right|^{2} \right) \cdot \left(\sum^{m }_{i=1} |\partial_{i}\psi |^{2}\right)\cdot e^{-\varphi }\,\mathrm{d}P \\
&\geqslant \sum^{\prime }_{\left| I\right|  =s}\sum^{\infty }_{i=1} \sum^{\prime }_{\left| J\right|  =t+1} c_{I,iJ}\int_{V} |\overline{\partial_{i} } f_{I,J}|^{2}\cdot e^{-\varphi }\,\mathrm{d}P\\
&\quad+\sum^{\prime }_{\left| I\right|  =s}  \sum_{1\leqslant i,j\leqslant m } \sum^{\prime }_{\left| L\right|  =t}\frac{c_{I,iL}\cdot c_{I,jL}}{c_{I, L}} \int_{V} f_{I,iL}\cdot\overline{f_{I,jL}} \cdot\left( \partial_{i} \overline{\partial_{j} } \varphi +\frac{\overline{\partial_i} (\overline{z_j})}{2a^{2}_{j}}  \right)\cdot  e^{-\varphi }\,\mathrm{d}P\\
&= \sum^{\prime }_{\left| I\right|  =s}\sum^{\infty }_{i=1} \sum^{\prime }_{\left| J\right|  =t+1} c_{I,iJ}\int_{V} |\overline{\partial_{i} } f_{I,J}|^{2}\cdot e^{-\varphi }\,\mathrm{d}P+\sum^{\prime }_{\left| I\right|  =s}  \sum_{1\leqslant i,j\leqslant m } \sum^{\prime }_{\left| L\right|  =t}\frac{c_{I,iL}\cdot c_{I,jL}}{c_{I, L}} \int_{V} f_{I,iL}\cdot\overline{f_{I,jL}} \cdot\left( \partial_{i} \overline{\partial_{j} } \varphi \right)\cdot  e^{-\varphi }\,\mathrm{d}P\\
&\quad+\sum^{\prime }_{\left| I\right|  =s}  \sum^{\prime }_{\left| L\right|  =t} \int_{V}\left( \sum_{i=1}^{m}\frac{|c_{I,iL}|^2}{c_{I, L}}\cdot|f_{I,iL}|^2\cdot\frac{1}{2a^{2}_{i}}\right)\cdot  e^{-\varphi }\,\mathrm{d}P.\label{general inequality for smooth functions}
	\end{aligned}
\end{equation}
By \eqref{general inequality for smooth functions} and the assumption (\ref{230419e12}) in Condition \ref{230423ass4}, we have (Recall Condition \ref{230423ass2} for $c_0^{s,t}$)
$$
\begin{aligned}
				&2\left| \left| T^{\ast }f\right|  \right|^{2}_{L^{2}_{\left( s,t+1\right)  }\left( V,w_{1}\right)  }  +\left| \left| Sf\right|  \right|^{2}_{L^{2}_{\left( s,t+2\right)  }\left( V,w_{3}\right)  }\\
&\geqslant \sum^{\prime }_{\left| I\right|  =s}\sum^{\infty }_{i=1} \sum^{\prime }_{\left| J\right|  =t+1} c_{I,iJ}\int_{V} |\overline{\partial_{i} } f_{I,J}|^{2}\cdot e^{-\varphi }\,\mathrm{d}P+
 \sum^{\prime }_{\left| I\right|  =s}  \sum^{\prime }_{\left| L\right|  =t} \sum^{m}_{i=1} \frac{|c_{I,iL}|^2}{c_{I,L}} \int_{V} \left| f_{I,iL}\right|^{2}\cdot\left( 2e^{\psi }+\frac{1}{2a_i^2}-\frac{1}{2} \right)\cdot e^{-\varphi } \,\mathrm{d}P \\
&\geqslant 2\sum^{\prime }_{\left| I\right|  =s} \sum^{\prime }_{\left| L\right|  =t} \sum^{m}_{i=1} \frac{|c_{I,iL}|^2}{c_{I,L}}\cdot\int_{V} \left| f_{I,iL}\right|^{2}\cdot e^{\psi-\varphi } \,\mathrm{d}P\\
&\geqslant 2 c_0^{s,t}\cdot\sum^{\prime }_{\left| I\right|  =s} \sum^{\prime }_{\left| L\right|  =t} \sum^{m}_{i=1}  |c_{I,iL}| \cdot\int_{V} \left| f_{I,iL}\right|^{2}\cdot e^{-w_2} \,\mathrm{d}P\\
&= 2c_0^{s,t}\cdot(t+1)\cdot\left| \left| f\right|  \right|^{2}_{L^{2}_{\left( s,t+1\right)  }\left( V,w_{2}\right)  },
\end{aligned}
$$
which gives the desired inequality (\ref{230419e13}). This completes the proof of Lemma \ref{general estimation similar to Lemma 4.2.1}.
\end{proof}

As an immediate consequence of Proposition \ref{general cut-off density3}, Theorem \ref{general density4} and Lemma \ref{general estimation similar to Lemma 4.2.1}, we have the following main result in this section.
\begin{theorem}\label{general key inequality}
Under Conditions \ref{230424c1}, \ref{230423ass1}, \ref{230423ass2}, \ref{230423ass3} and \ref{230423ass4}, it holds that
$$
\left| \left| T^{\ast }f\right|  \right|^{2}_{L^{2}_{\left( s,t\right)  }\left( V,w_{1}\right)  }  +\left| \left| Sf\right|  \right|^{2}_{L^{2}_{\left( s,t+2\right)  }\left( V,w_{3}\right)  }
\geqslant  c_0^{s,t}\cdot \left| \left| f\right|  \right|^{2}_{L^{2}_{\left( s,t+1\right)  }\left( V,w_{2}\right)  },\quad \forall \;f\in D_{S}\bigcap D_{T^{\ast }}.
$$
\end{theorem}
\begin{remark}
Recall that for any strictly increasing multi-indices $I$ and $J$, it was chosen in \cite{YZ}
$$C_{I,J}=2^{|I|+|J|}\cdot \left(\prod\limits_{i\in I}a_i^2\right)\cdot \left(\prod\limits_{j\in J}a_j^2\right)$$ for the working space (See \cite[p. 530]{YZ}). In that case,
$$
c_0^{s,t}=\inf\limits_{|I|=s,\,|L|=t,\,i\in \mathbb{N}}\frac{c_{I,iL} }{c_{I,L}}=2\inf\limits_{i\in \mathbb{N}}a_i^2=0.
$$
Therefore, if we choose the working space as that in \cite{YZ}, then the conclusion in Theorem \ref{general key inequality} becomes $
\left| \left| T^{\ast }f\right|  \right|^{2}_{L^{2}_{\left( s,t\right)  }\left( V,w_{1}\right)  }  +\left| \left| Sf\right|  \right|^{2}_{L^{2}_{\left( s,t+2\right)  }\left( V,w_{3}\right)  }
\geqslant  0$ for all $f\in D_{S}\bigcap D_{T^{\ast }}$, which is a trivial result.
\end{remark}

\section{Solving the $\overline{\partial}$ equations on pseudo-convex domains in $\ell^{2}$}\label{sec5}

In this section, we shall solve the $\overline{\partial}$ equations on pseudo-convex domains in $\ell^{2}$.

Combining Proposition \ref{Pre-exactness}, Theorem \ref{general key inequality} and \cite[Corollary 2.1, p. 522]{YZ}, we have the following result immediately.

\begin{proposition}
\label{general Solving d-bar on V}
Under Conditions \ref{230424c1}, \ref{230423ass1}, \ref{230423ass2}, \ref{230423ass3} and \ref{230423ass4}, it holds that $R_T=N_S$, and for each $f\in N_S$ there exists $g\in D_T$ so that $Tg=f$ and $ \sqrt{c_0^{s,t}}\cdot \left| \left| g\right|  \right|_{L^{2}_{(s,t)}\left( V,w_{1}\right)  }  \leqslant \left| \left| f\right|  \right|_{L^{2}_{(s,t+1)}\left( V,w_{2}\right)  }$.
\end{proposition}
In order to solve the $\overline{\partial}$ equation on a more general working space, we need the following result from \cite[Lemma 7.2.10, p. 179]{Tu} (This result was implicitly used in \cite[Proof of Theorem 4.2.2 at p. 85]{Hor90}).
\begin{lemma}\label{smooth convext increasing function}
Suppose that $g_0$ is an increasing nonnegative function on $[0,+\infty )$. Then there exists a convex, increasing, and real analytic function $g$ on $[0,+\infty )$ such that
$$
g^{\prime \prime } (x)\geqslant g^{\prime } (x) \geqslant g (x)\geqslant g_0 (x),\quad\forall\; x\in[0,+\infty).
$$
\end{lemma}
\begin{proof}
Since we have not found an exact reference (in English) for Lemma \ref{smooth convext increasing function}, for the readers' convenience, following \cite[Lemma 7.2.10, p. 179]{Tu}, we shall give below a detailed proof.
Without loss of generality, we assume that $g_0\geqslant1$.
Let $N_{1}=1$, for $k\geqslant 2$, we assume that $N_ {1},N_ {2},\cdots, N_ {k-1}\in\mathbb{N}$ have been chosen such that
$$
N_{l}> \max\left\{ \frac{\ln(g_0(l+1))}{\ln\frac{l}{l-1}},N_{l-1}\right\},\quad l=2,\cdots,k-1.
$$
Choose $N_k\in\mathbb{N}$ such that
$$
N_{k}> \max\left\{ \frac{\ln(g_0(k+1))}{\ln\frac{k}{k-1}},N_{k-1}\right\},
$$
then,
$$
\frac{1}{k}(g_0(k+1))^{\frac{1}{N_{k}}}\in \bigg[\frac{1}{k},\frac{1}{k-1}\bigg).
$$
By induction, we obtain a sequence $\{N_{k}\}_{k=1}^{\infty}$ of positive integers. Let
$$
a_{0}\triangleq  g_0(2)\cdot e,\quad a_{l}\triangleq\frac{1}{k}(g_0(k+1))^{\frac{1}{N_{k}}}\cdot e^{\frac{1}{\sqrt{l}}} \;\hbox{ for }\; l=N_{k},\cdots, N_{k+1}-1,\quad\forall\; k,l\in\mathbb{N}.
$$

Note that $\{a_{l}\}_{l=1}^{\infty} $ is a decreasing sequence and for $l=N_{k},\cdots, N_{k+1}-1$, we have
$$
\frac{1}{l}\leqslant\frac{1}{N_k}\leqslant\frac{1}{k}\leqslant\frac{e^{\frac{1}{\sqrt{N_{k+1}}}}}{k}\leqslant a_l\leqslant \frac{e^{\frac{1}{\sqrt{N_{k}}}}}{k-1}\leqslant \frac{e^{\frac{1}{\sqrt{k}}}}{k-1}
$$
which implies that $\lim\limits_{l\rightarrow \infty } a_{l}=0.$
Hence, $\sum\limits_{n=0}^{\infty }\bigg(\prod\limits_{i=0}^{n}a_{i}\bigg)\cdot|x|^{n}<\infty$ for all $x\in\mathbb{R}$. Let
$$
g(x)\triangleq \sum_{n=0}^{\infty }\bigg(\prod_{i=0}^{n}a_{i}\bigg)\cdot x^{n},\quad\forall\; x\in\mathbb{R}.
$$
Then $g$ is an increasing and real analytic function on $[0,\infty)$.

Since $a_{l}\geqslant \frac{1}{l} $ for all $l\in\mathbb{N}$, for each $x\geqslant 0$, we have
\begin{eqnarray*}
g^{\prime \prime } \left( x\right)
&=&\sum_{n=2}^{\infty }a_n\cdot n\cdot a_{n-1}\cdot(n-1)\cdot\bigg(\prod_{i=0}^{n-2}a_{i}\bigg)\cdot x^{n-2}\\
&\geqslant &\sum_{n=2}^{\infty } a_{n-1}\cdot(n-1)\cdot\bigg(\prod_{i=0}^{n-2}a_{i}\bigg)\cdot x^{n-2}
=\sum_{n=1}^{\infty } a_{n}\cdot n\cdot\bigg(\prod_{i=0}^{n-1}a_{i}\bigg)\cdot x^{n-1}
=g^{\prime } \left(x\right)\\
&\geqslant&\sum_{n=1}^{\infty } \bigg(\prod_{i=0}^{n-1}a_{i}\bigg)\cdot x^{n-1}
=\sum_{n=0}^{\infty } \bigg(\prod_{i=0}^{n}a_{i}\bigg)\cdot x^{n}
=g(x)>0.
\end{eqnarray*}
Thus $g$ is a convex function on $[0,+\infty)$.

For $x\in[0,2]$, it follows that
$$
g \left(x\right)  \geqslant a_{0}=g_0\left( 2\right)\cdot  e\geqslant g_0\left( 2\right)  \geqslant g_0 \left( x\right).
$$
For $x\in \left( 2,+\infty \right)  $, we have
$$
\begin{array}{ll}
\displaystyle
g\left( x\right)  \!\!\!&\geqslant g\left( \left[ x\right]  \right)  \geqslant \prod^{N_{\left[ x\right]  }}_{i=0} a_{i}\cdot \left[ x\right]^{N_{\left[ x\right]  }}  \geqslant \left( a_{N_{\left[x\right]  }}\cdot \left[x\right]  \right)^{N_{\left[x\right]  }} =\left( \left(g_0\left( 1+\left[x\right]  \right)\right)^{\frac{1}{N_{\left[x\right]  }} }\cdot  e^{\frac{1}{\sqrt{N_{\left[x\right]  }} } }\right)^{N_{\left[x\right]  }} \\
&\displaystyle \geqslant g_0\left( 1+\left[x\right]  \right)  \geqslant g_0\left(x\right),
\end{array}
$$
where $\left[x\right] $ represents the maximum integer that does not exceed $x$. Thus $g(x)\geqslant g_0(x)$ for all $x\in[0,+\infty)$. This completes the proof of Lemma \ref{smooth convext increasing function}.
\end{proof}

The following theorem is an infinite-dimensional version of \cite[Theorem 4.2.2, p. 84]{Hor90}.
\begin{theorem}\label{general LL}
Under Conditions \ref{230424c1}, \ref{230423ass1}, \ref{230423ass2} and \ref{230423ass3}, for any $f=\sum\limits^{\prime }_{\left| I\right|  =s,\left| J\right|  =t+1 } f_{I,J}dz_{I}\wedge d\overline{z_{J}}\in L^{2}_{(s,t+1)}\left( V,loc \right)$ with $\overline{\partial}f=0$, there exists $u\in L^{2}_{(s,t)}\left( V,loc\right)$ such that $\overline{\partial}u=f$.
\end{theorem}
\begin{proof}
By Remark \ref{230422r1}, we may assume that the plurisubharmonic exhaustion function $\eta$ satisfies $\eta\geqslant 0$ and (\ref{regular condition for eta}).
Since $V_{j+1}\setminus V_{j}\subset V_{j+1}\stackrel{\circ}{\subset}V$ (Recall (\ref{230117e1}) for $V_j$, where $j\in \mathbb{N}$), we have
$$
0\leqslant\int_{V_{j+1}\setminus V_{j}} \sum^{\prime }_{\left| I\right|  =s,\left| J\right|  =t+1} c_{I,J}\cdot\left| f_{I,J}\right|^{2}\,\mathrm{d}P<\infty,\quad \forall\;j\in\mathbb{N}.
$$
Choosing a sequence $\{ b_{j}\}^{\infty }_{j=1}$ of positive numbers such that
$$
\sum^{\infty }_{j=1}b_{j} \int_{V_{j+1}\setminus V_{j}} \sum^{\prime }_{\left| I\right|  =s,\left| J\right|  =t+1} c_{I,J}\cdot\left| f_{I,J}\right|^{2}\,\mathrm{d}P<\infty  .
$$
Let
$$
h\left( x\right)  \triangleq \begin{cases}\sum\limits^{j}_{i=1} \left| \ln \frac{1}{b_{i} } +\sup\limits_{V_{i+1}\setminus V_{i}} \psi \right| , &j<x\leqslant j+1,j\in\mathbb{N}, \\ 0,&0\leqslant x\leqslant 1,\end{cases}
$$
where the function $\psi\in C^{\infty}_{F}(V)$ is given between \eqref{section2} and \eqref{majority function}. Set
$$
g_0\left( x\right)  \triangleq1+h\left( x\right)  +\sup_{V_{x}} \left( 2\sum\limits^{\infty }_{i=1} \left| \partial_{i}\psi \right|^{2}  +2e^{\psi }\right),\quad\forall\; x\in[0,+\infty)   .
$$
Obviously, $g_0$ is an increasing nonnegative function on $[0,+\infty )$. By Lemma \ref{smooth convext increasing function}, there exists a convex, increasing, and real analytic function $g$ on $[0,+\infty )$ such that
$$
g^{\prime \prime} \left( x\right)  \geqslant g^{\prime } \left( x\right)  \geqslant g \left( x\right)  \geqslant g_0\left( x\right),\quad\forall\; x\in [0,+\infty).
$$
Let
\begin{equation}\label{230503e1}
\varphi   \triangleq g \left( \eta  \right).
 \end{equation}
By the properties of $g$ and $\eta$, noting (\ref{230503e1}), it is easy to see that $\varphi\in C^{2}_{F}\left( V\right)  $. Let $w_{1} =\varphi -2\psi $, $w_{2} =\varphi -\psi$ and $w_{3} =\varphi$ as that in \eqref{weight function}.
Note that
\begin{eqnarray}\label{condtion for D g eta}
g^{\prime } \left( \eta \left( \textbf{z}\right)  \right)  \geqslant g_0\left( \eta \left( \textbf{z}\right)  \right)  \geqslant \sup_{V_{\eta \left( \textbf{z}\right)  }} \left( 2\sum\limits^{\infty }_{i=1} \left| \partial_{i}\psi \right|^{2}  +2e^{\psi }\right)  \geqslant 2\sum\limits^{\infty }_{i=1} \left|\partial_{i}\psi \left( \textbf{z}\right)  \right|^{2}  +2e^{\psi \left( \textbf{z}\right)  } ,\quad\forall\;\textbf{z}\in V,
\end{eqnarray}
and
$$
g \left( \eta \left( \textbf{z}\right)  \right)  \geqslant h\left( \eta \left( \textbf{z}\right)  \right)  \geqslant \ln \frac{1}{b_{j} } +\sup_{  V_{j+1}\setminus V_{j}} \psi   \geqslant \ln \frac{1}{b_{j} } +\psi \left( \textbf{z}\right),\quad  \forall\; \textbf{z}\in V_{j+1}\setminus V_{j},
$$
and hence
$$
\sup_{ V_{j+1}\setminus V_{j}} e^{-\left( \varphi -\psi \right)  }\leqslant b_{j}.
$$
Then,
\begin{eqnarray*}
&&\sum^{\prime }_{\left| I\right|  =s}\sum^{\prime }_{\left| J\right|  =t+1}c_{I,J} \cdot\int_{V} \left| f_{I,J}\right|^{2}  e^{-w_{2}}\,\mathrm{d}P\\
&&=\sum^{\prime }_{\left| I\right|  =s}  \sum^{\prime }_{\left| J\right|  =t+1} c_{I,J} \cdot\int_{V_{1}} \left| f_{I,J}\right|^{2} \cdot e^{-\left( \varphi -\psi \right) }\,\mathrm{d}P+\sum^{\infty }_{j=1} \sum^{\prime }_{\left| I\right|  =s} \sum^{\prime }_{\left| J\right|  =t+1} c_{I,J} \cdot\int_{V_{j+1}\setminus V_{j}} \left| f_{I,J}\right|^{2}\cdot  e^{-\left( \varphi -\psi \right)  }\,\mathrm{d}P\\
&&\leqslant \left(\sup_{V_1}e^{-\left( \varphi -\psi \right) }\right)\cdot\sum^{\prime }_{\left| I\right|  =s} \sum^{\prime }_{\left| J\right|  =t+1}c_{I,J} \cdot \int_{V_{1}} \left| f_{I,J}\right|^{2}  \,\mathrm{d}P+\sum^{\infty }_{j=1} b_{j} \int_{V_{j+1}\setminus V_{j}} \sum^{\prime }_{\left| I\right|  =s,\left| J\right|  =t+1} c_{I,J} \cdot\left| f_{I,J}\right|^{2}\,\mathrm{d}P<\infty,
\end{eqnarray*}
which implies that $f\in L^{2}_{(s,t+1)}\left( V,w_{2}\right)$.

On the other hand, for each $n\in\mathbb{N}$, by \eqref{regular condition for eta} and \eqref{condtion for D g eta}, we have
$$
		\begin{aligned}
			\sum_{1\leqslant i,j\leqslant n} \left(\partial_{i} \overline{\partial_{j} }\left( g\left( \eta \right)  \right) \right)\cdot \zeta_{i}\cdot\overline{\zeta_{j}}&=g^{\prime \prime } \left( \eta \right)  \cdot\left| \sum^{n}_{i=1}\zeta_{i}\cdot \partial_{i}\eta\right|^{2}  +g^{\prime } \left( \eta \right)\cdot  \sum_{1\leqslant i,j\leqslant n}\left( \partial_{i} \overline{\partial_{j} }\eta\right)\cdot \zeta_{i}\cdot\overline{\zeta_{j}}\\
&\geqslant g^{\prime } \left( \eta \right) \cdot \sum_{1\leqslant i,j\leqslant n}\left(\partial_{i} \overline{\partial_{j} }\eta \right)\cdot\zeta_{i}\cdot\overline{\zeta_{j}}\\
&\geqslant \left( 2\sum\limits^{\infty }_{i=1} \left| \partial_{i} \psi \right|^{2}  +2e^{\psi }\right) \cdot\left(  \sum^{n}_{i=1} \left| \zeta_{i} \right|^{2}\right) ,\quad\forall\;(\zeta_{1},\cdots,\zeta_n)\in \mathbb{C}^n,
		\end{aligned}
$$
which indicates that the function $\varphi $ in (\ref{230503e1}) verifies Condition \ref{230423ass4}.
Recall Definition \ref{definition of T S} for the operators $T$ and $S$. Since $\overline{\partial}f=0$, we have $Sf=0$. Hence by Proposition \ref{general Solving d-bar on V}, there exists $u\in L^{2}_{(s,t)}\left( V,w_{1}\right)$ such that $Tu=f$ and
\begin{eqnarray}\label{L2 estimation for u}
\sqrt{c_0^{s,t}}\cdot \left| \left| u\right|  \right|_{L^{2}_{(s,t)}\left( V,w_{1}\right)  }  \leqslant \left| \left| f\right|  \right|_{L^{2}_{(s,t+1)}\left( V,w_{2}\right)  }.
\end{eqnarray}
Note that for any $r>0$,
\begin{eqnarray*}
\sum^{\prime }_{\left| I\right|  =s}\sum^{\prime }_{\left| L\right|  =t}c_{I,L} \cdot\int_{V_r} \left|u_{I,L}\right|^{2}  \,\mathrm{d}P
\!\!&\leqslant&\!\!\sum^{\prime }_{\left| I\right|  =s}\sum^{\prime }_{\left| L\right|  =t}c_{I,L} \cdot\int_{V_r} \left|u_{I,L}\right|^{2} \cdot e^{-w_{1}}\cdot e^{w_{1}}\,\mathrm{d}P\\
&\leqslant&\!\!\left(\sup_{V_r}e^{w_{1}}\right)\cdot\left(\sum^{\prime }_{\left| I\right|  =s}\sum^{\prime }_{\left| L\right|  =t}c_{I,L} \cdot\int_{V_r} \left|u_{I,L}\right|^{2} \cdot e^{-w_{1}}\,\mathrm{d}P\right)\\
&<&\!\!\infty,
\end{eqnarray*}
which implies that $u\in L^{2}_{(s,t)}\left( V_{r},P\right)$. By the conclusion (3) of Proposition \ref{properties on pseudo-convex domain} and the definition of $L^{2}_{(s,t)}\left( V,loc\right)$, we have $u\in L^{2}_{(s,t)}\left( V,loc\right)$. Recalling again the definition of the operator $T$ in Definition \ref{definition of T S}, we conclude that $\overline{\partial}u=f$. The proof of Theorem \ref{general LL} is completed.
\end{proof}

In order to extend the \emph{a priori} estimate in \eqref{L2 estimation for u}, we need the following technical lemma which is motivated by \cite[Lemma 41.5, p. 304]{Mujica}.

\begin{lemma}\label{calculus lemma}
If $0<x_1<x_2<+\infty$ and $g$ is a non-decreasing real-valued function on $[0,+\infty)$ so that $g(x)=0$ for all $x\in[0,x_2]$, then there exists $G\in C^2(\mathbb{R})$ such that $G(x)=0$ for all $x\leqslant x_1$, $G^{''}(x)\geqslant 0$ for all $x\in[0,+\infty)$, $G(x)\geqslant g(x)$ and $G'(x)\geqslant g(x)$ for all $x\in[0,+\infty)$.
\end{lemma}
\begin{proof}
 Choose $r_0\triangleq0<r_1\triangleq x_1<r_2\triangleq\frac{x_1+x_2}{2}<r_3\triangleq x_2<\cdots<r_j<r_{j+1}<\cdots$ for $j\in\{4,5,\cdots,\}$ so that $\lim\limits_{j\to\infty}r_j=\infty$. For each $j\in\mathbb{N}$, let $\lambda_j\triangleq \sup\limits_{x\in[0,r_j]}g(x)$. Meanwhile, we choose $\{\varphi_j\}_{j=1}^{\infty}\subset C^2(\mathbb{R})$ such that $0\leqslant\varphi_j\leqslant 1$, $\varphi_j'(x)\geqslant 0$ for all $x\in\mathbb{R}$, $\varphi_j(x)= 0$ for all $x \leqslant r_{j-1}$, while  $\varphi_j(x)= 1$ for all $x \geqslant r_{j}$.  Then $\lambda_1=\lambda_2=\lambda_3=0$. Write
\begin{eqnarray*}
\Phi \triangleq \lambda_1+\sum_{j=1}^{\infty}(\lambda_{j+1}-\lambda_j)\varphi_j=\sum_{j=3}^{\infty}(\lambda_{j+1}-\lambda_j)\varphi_j.
\end{eqnarray*}
Then $\Phi(x)=0$ for all $x\leqslant r_2$. If $x\in[r_j,r_{j+1}]$ for some $j\in\mathbb{N}$, then $\varphi_1(x)=\cdots=\varphi_j(x)=1$, $0=\varphi_{j+2}(x)=\cdots $, and hence
\begin{eqnarray*}
\Phi(x)=\lambda_1+\sum_{k=1}^{j+1}(\lambda_{k+1}-\lambda_k)\varphi_k(x)\geqslant \lambda_1+\sum_{k=1}^{j}(\lambda_{k+1}-\lambda_k)=\lambda_{j+1}\geqslant g(x).
\end{eqnarray*}
Obviously, $\Phi\in C^2(\mathbb{R})$, $\Phi'(x)\geqslant 0$ and $\Phi(x)\geqslant g(x)$ for all $x\in[0,+\infty)$. Set $d_j\triangleq \sup\limits_{[0,r_j]}\max\{\Phi',g\}$ for each $j\in\mathbb{N}$ and
\begin{eqnarray*}
h\triangleq d_1+\sum_{j=1}^{\infty}(d_{j+1}-d_j)\varphi_j .
\end{eqnarray*}
Similarly to the above, we can show that $h\in C^2(\mathbb{R})$, $h(x)\geqslant \max\{\Phi'(x),g(x)\}$ and $h'(x)\geqslant 0$ for all $x\in[0,+\infty)$, $d_1=d_2=0$, and $h(x)=0$ for all $x\leqslant x_1$. Let
\begin{eqnarray*}
G(x)\triangleq \int_{-\infty}^{x}h(\tau)\,\mathrm{d}\tau,\quad \forall\; x\in \mathbb{R}.
\end{eqnarray*}
Then, $G\in C^2(\mathbb{R})$, $G(x)\geqslant \int_0^x \Phi'(\tau)\,\mathrm{d}\tau=\Phi(x)\geqslant g(x)$, $G'(x)=h(x)\geqslant g(x),\,G^{''}(x)=h'(x)\geqslant 0$ for all $x\in[0,+\infty)$ and $G(x)=0$ for all $x\leqslant x_1$. The proof of Lemma \ref{calculus lemma} is completed.
\end{proof}
The following theorem is an infinite-dimensional version of \cite[Lemma 4.4.1, p. 92]{Hor90} (Recall Condition \ref{230423ass2} for $c_0^{s,t}$).
\begin{theorem}\label{general L}
Suppose that Conditions \ref{230424c1}, \ref{230423ass1}, \ref{230423ass2} and \ref{230423ass3} hold, and the real-valued function $\varphi$ in \eqref{weight function} is chosen so that for each $n\in\mathbb{N}$,
$$
\sum_{1\leqslant i,j\leqslant n}\left( \partial_{i} \overline{\partial_{j} } \varphi\left(\textbf{z}\right) \right)\cdot\zeta_{i}\cdot\overline{\zeta_{j}} \geqslant c \cdot\sum_{1\leqslant i\leqslant n} \left| \zeta_{i}\right|^{2},\quad\forall\; (\textbf{z},\zeta_{1},\cdots,\zeta_n)\in V\times\mathbb{C}^n,
$$
where $c$ is a positive, continuous function on $V$ such that $\sup\limits_{ E} c  <\infty$ for all $E\stackrel{\circ}{\subset} V $. Then, for any $f=\sum\limits^{\prime }_{\left| I\right|  =s,\left| J\right|  =t+1 } f_{I,J}dz_{I}\wedge d\overline{z_{J}}\in L^{2}_{(s,t+1)}\left( V,\varphi \right)  $ with $\overline{\partial}f=0$ and $\sum^{\prime }\limits_{\left| I\right|  =s}\sum^{\prime }\limits_{\left| J\right|  =t+1} c_{I,J}\int_{V} \frac{\left| f_{I,J}\right|^{2}  }{c} \cdot e^{-\varphi }\,\mathrm{d}P<\infty$, there exists $u\in L^{2}_{(s,t)}\left( V,\varphi \right)$ such that $\overline{\partial}u=f$  and
\begin{equation}\label{230423e1}
\sum^{\prime }_{\left| I\right|  =s}\sum^{\prime }_{\left| L\right|  =t} c_{I,L}\int_{V} \left| u_{I,L}\right|^{2}\cdot  e^{-\varphi }\,\mathrm{d}P\leqslant \frac{ 2\sum\limits^{\prime }_{\left| I\right|  =s,\left| J\right|  =t+1} c_{I,J}\int_{V} \frac{\left| f_{I,J}\right|^{2}  }{c} \cdot e^{-\varphi }\,\mathrm{d}P }{c_0^{s,t}\cdot(t+1)}.
\end{equation}
\end{theorem}
\begin{proof}
As in the proof of Theorem \ref{general LL}, we may assume that the plurisubharmonic exhaustion function $\eta$ satisfies $\eta\geqslant 0$ and (\ref{regular condition for eta}).

Similarly to the construction of $h_{\rho}(\cdot)$ at the beginning of the proof of Theorem \ref{general density4}, one can find a sequence of functions $\{h_k\}_{k=1}^{\infty}\subset C^{\infty }\left( \mathbb{R}\right)$ such that for each $k\in\mathbb{N}$, $0\leqslant h_k\leqslant 1$, $h_{k} \left( x\right)  =1$ for $ x<k $ and $h_{k } \left( x\right)  =0$ for $ x>k+1$, and $\left| h^{\prime }_{k} \right|  \leqslant C$ for a positive number $C$ which is independent of $k$.

Let $X_{k}\triangleq h_{k}(\eta)$ for $k\in\mathbb{N}$. Then $X_{k}\in C^{\infty}_{0,F}\left( V\right)$, and
$$
\sup_{k\in\mathbb{N}} \;\ln \bigg(1+\sum^{\infty }_{j=1} \left| \overline{\partial_{j} } X_{k }\right|^{2} \bigg) =\sup_{k\in\mathbb{N}} \; \ln \left(1+|h_{k}^{'} (\eta)|^2\cdot\left(\sum^{\infty }_{j=1} \left| \overline{\partial_{j} } \eta\right|^{2}\right)\right)
$$
is a locally bounded function on $V$. By Lemma \ref{majority funtion}, there exists $\Psi \in C^{\infty }_{F}\left( V\right)$ such that
$$
\ln \left(1+\sum\limits^{\infty }_{j=1} \left| \overline{\partial_{j} } X_{k }\right|^{2} \right)\leqslant \Psi,\quad\forall\; k\in\mathbb{N}
$$
and hence $\sum\limits^{\infty }_{j=1} \left| \overline{\partial_{j} } X_{k}\right|^{2}  \leqslant e^{\Psi }$ for all $k\in\mathbb{N}$.

For each $a>0$, choose $a_2>a_1>a$ and $\varphi_{a}\in C^{\infty }\left( \mathbb{R} \right)$ such that  $0\leqslant \varphi_{a}\leqslant 1 $, $\varphi_{a} \left( x\right)  =0$ for $x<a_1$ and $\varphi_{a} \left( x\right)  =1$ for $x>a_2$. Let $\eta_{a}  \triangleq\varphi_{a} \left( \eta    \right)$ and $\psi_a\triangleq \eta_{a} \cdot \Psi $. It is easy to see that $\psi_a\in C^{2}_{F}\left( V\right)$. For every $k >a_2$, we have $X_{k}\left( \textbf{z}\right)  =1$ for all $\textbf{z}\in V_{a_2}$ and $\psi_a \left( \textbf{z}\right)  =0$ for all $\textbf{z}\in V_{a_1}$. When $\eta(\textbf{z})\geqslant a_2$, we have $\eta_a(\textbf{z})=1$ and hence $e^{\psi_a (\textbf{z})}=e^{\Psi (\textbf{z})}\geqslant\sum\limits^{\infty }_{j=1} \left| \overline{\partial_{j} } X_{k }(\textbf{z})\right|^{2}$. While when $\eta(\textbf{z})< a_2$, we have $e^{\psi_a (\textbf{z})}\geqslant 0=\sum\limits^{\infty }_{j=1} \left| \overline{\partial_{j} } X_{k }(\textbf{z})\right|^{2}$. Hence, for each $k >a_2$,
$$e^{\psi_a (\textbf{z})}\geqslant\sum\limits^{\infty }_{j=1} \left| \overline{\partial_{j} } X_{k }(\textbf{z})\right|^{2},\quad \forall\;\textbf{z}\in V,
$$
and $\psi_a$ has the same property as $\psi$ in \eqref{majority function}.

Let
$$
h\left( x\right)  \triangleq 2\sup_{V_{x}} \psi_a +2\sup_{V_{x}}\left( \sum^{\infty }_{i=1} \left| \partial_{i}\psi_a \right|^{2}\right),\quad x\in[0,+\infty).
$$
Then $h\left( x\right)  =0$ for all $0\leqslant x\leqslant a_1$ and $h$ is a non-decreasing nonnegative function on $[0,+\infty )$. 
By Lemma \ref{calculus lemma}, there exists $H\in C^2(\mathbb{R})$ such that $H(x)=0$ for all $x\leqslant a$,
$H^{''}(x)\geqslant 0$ for all $x\in[0,+\infty)$, and $H(x)\geqslant h(x)$ and $H'(x)\geqslant h(x)$ for all $x\in[0,+\infty)$.
Note that for each $\textbf{z}\in V$, we have
\begin{eqnarray*}
H \left( \eta \left( \textbf{z}\right)  \right)  \geqslant 2\sup_{V_{\eta \left( \textbf{z}\right)  }} \psi_a   \geqslant 2\psi_a \left(  \textbf{z} \right),\quad H^{\prime } \left( \eta \left( \textbf{z}\right)  \right)  \geqslant 2\sup_{V_{\eta \left( \textbf{z}\right)  }} \left(\sum^{\infty }_{i=1} \left| \partial_{i}\psi_a \right|^{2}\right)  \geqslant 2\sum^{\infty }_{i=1} \left| \partial_{i}\psi_a \left( \textbf{z}\right)  \right|^{2},
\end{eqnarray*}
and hence, by (\ref{regular condition for eta}), for each $n\in\mathbb{N}$, the following inequality holds on $V$,
$$
H^{\prime } \left( \eta   \right) \cdot \sum_{1\leqslant i,j\leqslant n}\left( \partial_{i}\overline{\partial_{j}}\eta\right)\cdot \zeta_{i}\cdot\overline{\zeta_{j}}\geqslant 2\left(\sum^{\infty }_{i=1} \left| \partial_{i}\psi_a  \right|^{2} \right)\cdot\left( \sum^{n}_{i=1} \left| \zeta_{i}\right|^{2}\right),\quad \forall\;(\zeta_1,\cdots,\zeta_n)\in\mathbb{C}^n.
$$
Then $H(\eta)\in C^2_F(V)$ and $\varphi +H\left( \eta \right)$ enjoys the same property as $\varphi$. Moreover, for each $n\in\mathbb{N}$, the following inequality holds on $V$,
\begin{equation}\label{230424e1}
\begin{aligned}
&\sum_{1\leqslant i,j\leqslant n} \left(\partial_{i} \overline{\partial_{j} } (\varphi+H\left( \eta \right)) \right)\cdot\zeta_{i}\cdot\overline{\zeta_{j}}\\
&=\sum_{1\leqslant i,j\leqslant n} \left(\partial_{i} \overline{\partial_{j} } \varphi\right)\cdot\zeta_{i}\cdot\overline{\zeta_{j}}
+H^{\prime \prime }\left( \eta \right) \cdot\sum_{1\leqslant i,j\leqslant n}\partial_{i} \eta \cdot\overline{\partial_{j} } \eta\cdot \zeta_{i}\cdot\overline{\zeta_{j}}
+H^{\prime }\left( \eta\right) \cdot \sum_{1\leqslant i,j\leqslant n} \left(\partial_{i}\overline{\partial_{j}}\eta \right)\cdot\zeta_{i}\cdot\overline{\zeta_{j}}\\
&\geqslant c\cdot\sum^{n}_{i=1} \left| \zeta_{i}\right|^{2}+ 2\cdot\left(\sum^{\infty }_{i=1} \left| \partial_{i}\psi_a \right|^{2}  \right)\cdot\left(\sum^{n}_{i=1} \left| \zeta_{i}\right|^{2}\right),\quad \forall\;(\zeta_1,\cdots,\zeta_n)\in\mathbb{C}^n.
\end{aligned}
\end{equation}

Let
$$
\widetilde{w_{1}}\triangleq\varphi +H\left( \eta \right)  -2\psi_a,\quad\widetilde{w_{2}}\triangleq\varphi +H\left( \eta \right)  -\psi_a,\quad\widetilde{w_{3}}\triangleq\varphi +H\left( \eta \right).
$$
Suppose that $\widetilde{T}$ and $\widetilde{S}$ are the corresponding operators in Definition \ref{definition of T S} with $w_1, w_2$ and $w_3$ replaced respectively by $\widetilde{w_1},\widetilde{w_2}$ and $\widetilde{w_3}$. Then all properties for $T$ and $S$ also hold for $\widetilde{T}$ and $\widetilde{S}$. Now, in view of \eqref{general inequality for smooth functions} (and noting that the condition (\ref{230419e12}) is NOT used in the proof of \eqref{general inequality for smooth functions}) and using (\ref{230424e1}), we obtain that
\begin{equation}\label{general formula 14}
c_0^{s,t}\cdot(t+1)\cdot\sum^{\prime }_{\left| I\right|  =s,\left| J\right|  =t+1} c_{I,J}\int_{V} c\cdot\left| g_{I,J}\right|^{2} \cdot e^{-\left( \varphi +H\left( \eta \right)  \right)  }dP\leqslant 2\left| \left| \widetilde{T}^{\ast }g\right|  \right|^{2}_{L^{2}_{\left( s,t\right)  }\left( V,\widetilde{w_{1}}\right)  }  +\left| \left| \widetilde{S}g\right|  \right|^{2}_{L^{2}_{\left( s,t+2\right)  }\left( V,\widetilde{w_{3}}\right)  }
\end{equation}
for every $g=\sum\limits^{\prime }_{\left| I\right|  =s,\left| J\right|  =t+1 }g_{I,J}dz_{I}\wedge d\overline{z_{J}}\in D_{\widetilde{T}^{\ast}}\bigcap D_{\widetilde{S}} $ satisfying that there exists $k\in \mathbb{N}$ such that $g_{I,J}=0$ for all strictly increasing multi-indices $I$ and $J$ with $\left| I\right|  =s,\left| J\right|  =t+1$ and $\max \left\{ I\bigcup J\right\}  >k$, and $g_{I,J}\in C^{2}_{0,F}\left( V\right)$ for all strictly increasing multi-indices $I$ and $J$ with $\left| I\right|  =s,\left| J\right|  =t+1$ and $\max \left\{ I\bigcup J\right\} \leqslant k$. By Proposition \ref{general cut-off density3} and Theorem \ref{general density4}, for every $g\in D_{\widetilde{S}}\cap D_{\widetilde{T}^{\ast }}$, there exists
$$
\{g_{n}\}_{n=1}^{\infty}\subset D_{\widetilde{T}^{\ast }}\cap D_{\widetilde{S}}
$$
satisfying that, for each $n\in\mathbb{N}$, there exists $k_{n}\in \mathbb{N}$ such that $(g_{n})_{I,J}=0$ for all strictly increasing multi-indices $I$ and $J$ with $\left| I\right|  =s,\left| J\right|  =t+1$ and $\max \left\{ I\bigcup J\right\}  >k_{n}$, $(g_{n})_{I,J}\in C^{2}_{0,F}\left( V\right)$ for all strictly increasing multi-indices $I$ and $J$ with $\left| I\right|  =s,\left| J\right|  =t+1,\max \left\{ I\bigcup J\right\}\leqslant k_{n}$, and
$$
\begin{aligned}
&\lim_{n\rightarrow \infty }\left( \left| \left| \sqrt{c}\cdot e^{-\frac{\psi_a }{2} }\cdot g_{n}-\sqrt{c}\cdot e^{-\frac{\psi_a }{2} }\cdot g\right|  \right|_{L^{2}_{\left( s,t+1\right)  }\left( V_{r},\widetilde{w_{2}}\right)  }  +\left| \left| \widetilde{T}^{\ast }g_{n}-\widetilde{T}^{\ast }g\right|  \right|_{L^{2}_{\left( s,t\right)  }\left( V_{r},\widetilde{w_{1}}\right)}  +\left| \left| \widetilde{S}g_{n}-\widetilde{S}g\right|  \right|_{L^{2}_{\left( s,t+2\right)  }\left( V_{r},\widetilde{w_{3}}\right)}\right) \\
& =0,\qquad \forall\;r>0.
\end{aligned}
 $$
By \eqref{general formula 14}, for each $n\in\mathbb{N}$, we have
$$
c_0^{s,t}\cdot(t+1)\cdot\sum^{\prime }_{\left| I\right|  =s,\left| J\right|  =t+1} c_{I,J}\int_{V_r} c\cdot\left| (g_n)_{I,J}\right|^{2} \cdot e^{-\left( \varphi +H\left( \eta \right)  \right)  }\,\mathrm{d}P\leqslant 2\left| \left| \widetilde{T}^{\ast }g_n\right|  \right|^{2}_{L^{2}_{\left( s,t\right)  }\left( V,\widetilde{w_{1}}\right)  }  +\left| \left| \widetilde{S}g_n\right|  \right|^{2}_{L^{2}_{\left( s,t+2\right)  }\left( V,\widetilde{w_{3}}\right)  }.
$$
Letting $n\to\infty$ in the above inequality,  we obtain that
$$
\begin{aligned}
&c_0^{s,t}\cdot(t+1)\cdot\sum^{\prime }_{\left| I\right|  =s,\left| J\right|  =t+1} c_{I,J}\int_{V_{r}} c\cdot\left| g_{I,J}\right|^{2} \cdot e^{-\left( \varphi +H \left( \eta \right)  \right)  }\,\mathrm{d}P\\
&\leqslant 2\left| \left|\widetilde{T}^{\ast }g\right|  \right|^{2}_{L^{2}_{\left( s,t\right)  }\left( V,\widetilde{w_{1}}\right)  }  +\left| \left| \widetilde{S}g\right|  \right|^{2}_{L^{2}_{\left( s,t+2\right)  }\left( V,\widetilde{w_{3}}\right)  },\quad \forall\; g\in D_{\widetilde{S}}\cap D_{\widetilde{T}^{\ast }}.
\end{aligned}
$$
By the arbitrariness of $r$, we then have
\begin{equation}\label{a priori estimate}
\begin{aligned}
&c_0^{s,t}\cdot(t+1)\cdot\sum^{\prime }_{\left| I\right|  =s,\left| J\right|  =t+1} c_{I,J}\int_{V} c\cdot\left| g_{I,J}\right|^{2} \cdot e^{-\left( \varphi +H \left( \eta \right)  \right)  }\,\mathrm{d}P\\
&\leqslant 2\left| \left|\widetilde{T}^{\ast }g\right|  \right|^{2}_{L^{2}_{\left( s,t\right)  }\left( V,\widetilde{w_{1}}\right)  }  +\left| \left|\widetilde{S}g\right|  \right|^{2}_{L^{2}_{\left( s,t+2\right)  }\left( V,\widetilde{w_{3}}\right)  },\quad \forall\; g\in D_{\widetilde{S}}\cap D_{\widetilde{T}^{\ast }}.
\end{aligned}
\end{equation}
By $H \left( \eta \right) \geqslant 2\psi_a$, we have $e^{-\varphi -H \left( \eta \right)  +\psi_a }\leqslant e^{-\varphi -\frac{H\left( \eta \right)  }{2} }$. Since $\widetilde{w_{2}}=\varphi+H(\eta)-\psi_a\geqslant \varphi$ and $f\in L^{2}_{(s,t+1)}\left( V,\varphi \right)$ with $\overline{\partial}f=0$, we have $f\in L^{2}_{(s,t+1)}\left( V,\widetilde{w_{2}} \right)$ and $f\in N_{\widetilde{S}}$. Hence, for every $g=\sum\limits^{\prime }_{\left| I\right|  =s,\left| J\right|  =t+1 }g_{I,J}dz_{I}\wedge d\overline{z_{J}}\in D_{\widetilde{S}}\bigcap D_{\widetilde{T}^{\ast }}$, it holds that
 \begin{equation}
		\begin{aligned}
			&\left| \left( g,f\right)_{L^{2}_{(s,t+1)}\left( V,\widetilde{w_{2}}\right)  }  \right|^{2}  =\bigg|\sum^{\prime }_{\left| I\right|  =s,\left| J\right|  =t+1} c_{I,J}\int_{V} f_{I,J}\cdot\overline{g_{I,J}}\cdot e^{-\widetilde{w_{2}}}\,\mathrm{d}P\bigg|^{2}\\
&\leqslant \bigg|\sum^{\prime }_{\left| I\right|  =s,\left| J\right|  =t+1} c_{I,J}\int_{V} |f_{I,J}\cdot\overline{g_{I,J}}| \cdot e^{-\varphi -H \left( \eta \right)  +\psi_a }\,\mathrm{d}P\bigg|^{2}\\
&\leqslant \bigg|\sum^{\prime }_{\left| I\right|  =s,\left| J\right|  =t+1} c_{I,J}\int_{V} |f_{I,J}\cdot\overline{g_{I,J}}|\cdot e^{-\varphi -\frac{H\left( \eta \right)  }{2} }\,\mathrm{d}P\bigg|^{2}\\
&\leqslant \left(\sum^{\prime }_{\left| I\right|  =s,\left| J\right|  =t+1} c_{I,J}\int_{V} \frac{\left| f_{I,J}\right|^{2}  }{c}\cdot e^{-\varphi }\,\mathrm{d}P\right)\cdot \left(\sum^{\prime }_{\left| I\right|  =s,\left| J\right|  =t+1} c_{I,J}\int_{V} c\cdot\left| g_{I,J}\right|^{2}\cdot  e^{-\left( \varphi +H\left( \eta \right)  \right)  }\,\mathrm{d}P\right)\\
&\leqslant \left(2\sum^{\prime }_{\left| I\right|  =s,\left| J\right|  =t+1} c_{I,J}\int_{V} \frac{\left| f_{I,J}\right|^{2}  }{c} \cdot e^{-\varphi }\,\mathrm{d}P\right)\cdot \frac{\left(\left| \left|\widetilde{T}^{\ast }g\right|  \right|^{2}_{L^{2}_{\left( s,t\right)  }\left( V,\widetilde{w_{1}}\right)  }  +\left| \left|\widetilde{S}g\right|  \right|^{2}_{L^{2}_{\left( s,t+2\right)  }\left( V,\widetilde{w_{3}}\right)  }\right)}{c_0^{s,t}\cdot(t+1)},\label{general meidiation inequality}
		\end{aligned}
\end{equation}
where the third inequality follows from the Cauchy-Schwarz inequality and the fourth inequality follows from \eqref{a priori estimate}.

If $g\in N^{\bot }_{\widetilde{S}}$, then by $R_{\widetilde{T}}\subset N_{\widetilde{S}}$, we have $ N_{\widetilde{S}}^{\perp}\subset R_{\widetilde{T}}^{\perp}=N_{\widetilde{T}^{\ast }}$ (where the equality $R_{\widetilde{T}}^{\perp}=N_{\widetilde{T}^{\ast }}$ follows from \cite[Exercise 2, p. 178]{Kra}), and hence $g\in N_{\widetilde{T}^{\ast }}$, which gives $\widetilde{T}^{\ast }g=0$. Recall that $f\in N_{\widetilde{S}}$, which, combined with $g\in N^{\bot }_{\widetilde{S}}$, gives $\left( g,f\right)_{L^{2}_{(s,t+1)}\left( V,\widetilde{w_{2}}\right)  }=0$.

If $g\in N_{\widetilde{S}} \bigcap D_{\widetilde{T}^{\ast }}$, then by \eqref{general meidiation inequality}, we have
\begin{eqnarray}
\left| \left( g,f\right)_{L^{2}_{(s,t+1)}\left( V,\widetilde{w_{2}}\right)  }  \right|^{2}  \leqslant \left(2\sum^{\prime }_{\left| I\right|  =s,\left| J\right|  =t+1} c_{I,J}\int_{V} \frac{\left| f_{I,J}\right|^{2}  }{c}\cdot e^{-\varphi }\,\mathrm{d}P\right)\cdot \frac{\left| \left|\widetilde{T}^{\ast }g\right|  \right|^{2}_{L^{2}_{\left( s,t\right)  }\left( V,\widetilde{w_{1}}\right)  }}{c_0^{s,t}\cdot(t+1)}.\label{general second inequality}
\end{eqnarray}

For any $g\in D_{\widetilde{T}^{\ast }}$, write $g=g_{1}+g_{2}$, where $ g_{1}\in N^{\bot }_{\widetilde{S}}$ and $g_{2}\in N_{\widetilde{S}}$. Since $N^{\bot }_{\widetilde{S}}\subset R_{\widetilde{T}}^{\perp}=N_{{\widetilde{T}}^{\ast }}\subset D_{{\widetilde{T}}^*}$ and $g\in D_{{\widetilde{T}}^{\ast }}$, we have $g_{1},g_{2}\in D_{{\widetilde{T}}^{\ast }}$ and hence
$$
	\begin{aligned}
		\left| \left( g,f\right)_{L^{2}_{(s,t+1)}\left( V,\widetilde{w_{2}}\right)  }  \right|^{2}  &=\left| \left( g_{2},f\right)_{L^{2}_{(s,t+1)}\left( V,\widetilde{w_{2}}\right)  }  \right|^{2}  \\&\leqslant \left(2\sum^{\prime }_{\left| I\right|  =s,\left| J\right|  =t+1} c_{I,J}\int_{V} \frac{\left| f_{I,J}\right|^{2}  }{c}\cdot e^{-\varphi }\,\mathrm{d}P\right)\cdot \frac{\left| \left| \widetilde{T}^{\ast }g_2\right|  \right|^{2}_{L^{2}_{\left( s,t\right)  }\left( V,\widetilde{w_{1}}\right)}  }{c_0^{s,t}\cdot(t+1)}  \\
&=\left(2\sum^{\prime }_{\left| I\right|  =s,\left| J\right|  =t+1} c_{I,J}\int_{V} \frac{\left| f_{I,J}\right|^{2}  }{c}\cdot e^{-\varphi }\,\mathrm{d}P\right)\cdot \frac{\left| \left| \widetilde{T}^{\ast }g-\widetilde{T}^{\ast }g_1\right|  \right|^{2}_{L^{2}_{\left( s,t\right)  }\left( V,\widetilde{w_{1}}\right)  }}{c_0^{s,t}\cdot(t+1)} \\
&=\left(2\sum^{\prime }_{\left| I\right|  =s,\left| J\right|  =t+1} c_{I,J}\int_{V} \frac{\left| f_{I,J}\right|^{2}  }{c}\cdot e^{-\varphi }\,\mathrm{d}P\right)\cdot \frac{\left| \left| \widetilde{T}^{\ast }g\right|  \right|^{2}_{L^{2}_{\left( s,t\right)  }\left( V,\widetilde{w_{1}}\right)  }}{c_0^{s,t}\cdot(t+1)} ,
	\end{aligned}
$$
where the first inequality follows from \eqref{general second inequality} and the last equality follows from the fact that $\widetilde{T}^{\ast }g_{1}=0$.
If  $\sum^{\prime }\limits_{\left| I\right|  =s}\sum^{\prime }\limits_{\left| J\right|  =t+1}c_{I,J} \int_{V} \frac{\left| f_{I,J}\right|^{2}  }{c}\cdot e^{-\varphi }\,\mathrm{d}P<\infty$, then by the Hahn-Banach theorem, there exists $u_a\in L^{2}_{(s,t)}\left( V,\widetilde{w_{1}}\right)$ such that $\left( g,f\right)_{L^{2}_{(s,t+1)}\left( V,\widetilde{w_{2}}\right)  } =\left( \widetilde{T}^*g,u_a\right)_{L^{2}_{(s,t)}\left( V,\widetilde{w_{1}}\right)  } $ for all $g\in D_{\widetilde{T}^*}$ and
$$
\sum^{\prime }_{\left| I\right|  =s,\left| L\right|  =t} c_{I,L}\int_{V} \left| (u_a)_{I,L}  \right|^{2} \cdot e^{-\widetilde{w_{1}}}\,\mathrm{d}P\leqslant    \frac{ 2\sum^{\prime }\limits_{\left| I\right|  =s,\left| J\right|  =t+1} c_{I,J}\int_{V} \frac{\left| f_{I,J}\right|^{2}  }{c}\cdot e^{-\varphi }\,\mathrm{d}P }{c_0^{s,t}\cdot(t+1)}.
$$
Since $\widetilde{T}^{**}=\widetilde{T}$ (e.g., \cite[Exercise 2, p. 178]{Kra}), we conclude that $\overline{\partial}u_a=\widetilde{T}u_a=f$. Since  $\widetilde{w_{1}}(\textbf{z})=\varphi(\textbf{z})$ for all $\textbf{z}\in V_{a}$, we have
$$
\sum^{\prime }_{\left| I\right|  =s,\left| L\right|  =t} c_{I,L}\int_{V_{a}} \left| \left( u_{a}\right)_{I,L}  \right|^{2} \cdot e^{-\varphi }\,\mathrm{d}P\leqslant \frac{ 2\sum^{\prime }\limits_{\left| I\right|  =s,\left| J\right|  =t+1} c_{I,J}\int_{V} \frac{\left| f_{I,J}\right|^{2}  }{c} \cdot e^{-\varphi }\,\mathrm{d}P }{c_0^{s,t}\cdot(t+1)} .
$$
For every $b\geqslant a$, by a similar procedure, we can find $ u_{b}\in L^{2}_{(s,t)}\left( V,\widehat{w_{1}} \right)$ (where $\widehat{w_{1}}$ is the corresponding weight function for $b$) such that $\overline{\partial}u_{b}=f$ and
\begin{eqnarray}\label{general estimation for ub}
\sum^{\prime }_{\left| I\right|  =s,\left| L\right|  =t} c_{I,L}\int_{V_{b}} \left| \left( u_{b}\right)_{I,L}  \right|^{2} \cdot e^{-\varphi }\,\mathrm{d}P\leqslant \frac{ 2\sum^{\prime }\limits_{\left| I\right|  =s,\left| J\right|  =t+1} c_{I,J}\int_{V} \frac{\left| f_{I,J}\right|^{2}  }{c} \cdot e^{-\varphi }\,\mathrm{d}P }{c_0^{s,t}\cdot(t+1)}.
\end{eqnarray}
Note that for each $b\in[a,+\infty)$, by Corollary \ref{cylinder function is dense}, $L^ {2}_ {(s, t)} \left (V_{b}, \varphi \right) $ is a separable Hilbert space.

By the Banach-Alaoglu Theorem (e.g., \cite[Theorem 3.15 at p. 66 and Theorem 3.17 at p. 68]{Rud91}), there exists $\{ b_{j}\}^{\infty }_{j=1}  \subset \left( a,+\infty \right)$ and $\tilde{u}_{a}\in L^{2}_{(s,t)}\left( V_{a},\varphi \right)$ such that $\lim\limits_{j\rightarrow \infty } b_{j}=\infty $ and $\lim\limits_{j\rightarrow \infty } u_{b_{j}}\rightharpoonup \tilde{u}_{a} $ weakly in $L^{2}_{(s,t)}\left( V_{a},\varphi \right)$.
For $a+1$, by the same argument, there exists a subsequence $\{b_{j_i}\}_{i=1}^{\infty}$ of $\{ b_{j}\}^{\infty }_{j=1}$ and $\tilde{u}_{a+1}\in L^{2}_{(s,t)}\left( V_{a+1},\varphi \right)$  such that $\lim\limits_{i\rightarrow \infty }u_{b_{j_i}}\rightharpoonup \tilde{u}_{a+1} $ weakly in $L^{2}_{(s,t)}\left( V_{a+1},\varphi \right)$.
By induction and the diagonal process, there exists $\{\tilde{b}_j\}_{j=1}^{\infty}\subset (a,+\infty)$ and $\tilde{u}_{a+k}\in L^{2}_{(s,t)}\left( V_{a+k},\varphi \right)$ for each nonnegative integer $k$ such that
\begin{itemize}
\item[(1)]$\lim\limits_{j\rightarrow \infty } \tilde{b}_{j}=\infty $;
\item[(2)]$\lim\limits_{j\rightarrow \infty }u_{\tilde{b}_{j}}= \tilde{u}_{a+k} $ weakly in $L^{2}_{(s,t)}\left( V_{a+k},\varphi \right)$ for each nonnegative integer $k$;
\item[(3)]$(\tilde{u}_{a+k_1})_{I,L}(\textbf{z}) =(\tilde{u}_{a+k_2})_{I,L}(\textbf{z})$ for all $\textbf{z}\in V_{a+k_2}$, positive integers $k_1$ and $k_2$ with $k_1\geqslant k_2$, and strictly increasing multi-indices $I$ and $L$ with $\left| I\right|  =s$ and $\left| L\right|  =t$.
\end{itemize}
Define $u_{I,L}(\textbf{z})\triangleq (\tilde{u}_{a+k})_{I,L}(\textbf{z})$ for $\textbf{z}\in V_{a+k}$, $k\in\mathbb{N}$ and strictly increasing multi-indices $I$ and $L$ with $\left| I\right|  =s$ and $\left| L\right|  =t$. Set $u=\sum\limits^{\prime }_{\left| I\right|  =s,\left| L\right|  =t }u_{I,L}dz_{I}\wedge d\overline{z_{L}}$. Then $u\in L^{2}_{(s,t)}\left( V,loc\right)$.

We claim that $\overline{\partial}u=f$. Indeed, since $\overline{\partial}u_{\tilde{b}_{j}}=f$ for each $j\in\mathbb{N}$, for every $\phi \in C^{\infty }_{0,F}\left( V\right)$ and strictly increasing multi-indices $I$ and $J$ with $\left| I\right|  =s$ and $\left| J\right|  =t+1$, we have
$$
\left( -1\right)^{s+1}  \int_{V} \sum^{\infty }_{i=1} \sum^{\prime }_{\left| L\right|  =t} \varepsilon^{J}_{iL}\cdot  ( u_{\tilde{b}_{j}} )_{I,L}  \cdot\overline{\delta_{i} \phi }\,\mathrm{d}P=\int_{V} f_{I,J}\cdot\overline{\phi }\,\mathrm{d}P.
$$
Since $\supp \phi \stackrel{\circ}{\subset} V$, by the conclusion (3) of Proposition \ref{properties on pseudo-convex domain}, there exists $j_{0}\in\mathbb{N}$ such that $\supp \phi \subset V_{a+j_{0}}$. Hence, for $j>j_{0}$, we have
$$
\int_{V} f_{I,J}\cdot\overline{\phi }\,\mathrm{d}P=\left( -1\right)^{s+1}  \int_{V_{a+j_{0}}} \sum^{\infty }_{i=1} \sum^{\prime }_{\left| L\right|  =t} \varepsilon^{J}_{iL}\cdot ( u_{\tilde{b}_{j}})_{I,L}\cdot  \overline{\delta_{i} \phi } \,\mathrm{d}P.
$$
Letting $j\rightarrow \infty$ in the above equality, we arrive at
$$
\int_{V} f_{I,J}\cdot\overline{\phi }\,\mathrm{d}P=\left( -1\right)^{s+1}  \int_{V_{a+j_{0}}} \sum^{\infty }_{i=1} \sum^{\prime }_{\left| L\right|  =t} \varepsilon^{J}_{iL}\cdot u_{I,L}\cdot\overline{\delta_{i} \phi } \,\mathrm{d}P=\left( -1\right)^{s+1}  \int_{V} \sum^{\infty }_{i=1} \sum^{\prime }_{\left| L\right|  =t} \varepsilon^{J}_{iL}\cdot u_{I,L}\cdot\overline{\delta_{i} \phi } \,\mathrm{d}P,
$$
which implies that $\overline{\partial}u=f$.

Now, since $\lim\limits_{j\rightarrow \infty }u_{\tilde{b}_{j}}= \tilde{u}_{a+k} $ weakly in $L^{2}_{(s,t)}\left( V_{a+k},\varphi \right)$ for each $k\in\mathbb{N}$, using \eqref{general estimation for ub}, we conclude that
$$
\left| \left| u\right|  \right|^{2}_{L^{2}_{(s,t)}\left( V_{a+k},\varphi \right)  }  \leqslant \varliminf_{j\rightarrow \infty } || u_{\tilde{b}_{j}}||^{2}_{L^{2}_{(s,t)}\left( V_{a+k},\varphi \right)  }  \leqslant\frac{ 2\sum^{\prime }\limits_{\left| I\right|  =s,\left| J\right|  =t+1} c_{I,J}\int_{V} \frac{\left| f_{I,J}\right|^{2}  }{c}\cdot e^{-\varphi }\,\mathrm{d}P }{c_0^{s,t}\cdot(t+1)}.
$$
Letting $k\to\infty$ in the above, we obtain
the desired estimate (\ref{230423e1}).
The proof of Theorem \ref{general L} is completed.
\end{proof}

Finally, we show the following infinite-dimensional analogue of \cite[Theorem 4.4.2, p. 94]{Hor90}.
\begin{corollary}\label{general Theorem 4.4.2Hor90}
Suppose that Conditions \ref{230424c1}, \ref{230423ass1}, \ref{230423ass2} and \ref{230423ass3} hold, and the real-valued function $\varphi$ in \eqref{weight function} is chosen so that for each $n\in\mathbb{N}$,
$$
\sum_{1\leqslant i,j\leqslant n}\left( \partial_{i} \overline{\partial_{j} } \varphi\left(\textbf{z}\right)\right)\cdot\zeta_{i}\cdot\overline{\zeta_{j}} \geqslant 0,\quad\forall\; (\textbf{z},\zeta_{1},\cdots,\zeta_n)\in V\times\mathbb{C}^n.
$$
Then, for any $f=\sum\limits^{\prime }_{\left| I\right|  =s,\left| J\right|  =t+1 }f_{I,J}dz_{I}\wedge d\overline{z_{J}}\in L^{2}_{(s,t+1)}\left( V,\varphi \right)$ with $\overline{\partial}f=0$, there exists $u=\sum\limits^{\prime }_{\left| I\right|  =s,\left| L\right|  =t }u_{I,L}dz_{I}\wedge d\overline{z_{L}}$ $\in L^{2}_{(s,t)}\left( V,\varphi \right)$ such that $\overline{\partial}u=f$ and
\begin{equation}\label{230423e2}
\sum^{\prime }_{\left| I\right|  =s,\left| L\right|  =t} c_{I,L}\int_{V} \left( 1+\left| \left| \textbf{z}\right|  \right|^{2}_{\ell^{2} }  \right)^{-2} \cdot \left| u_{I,L}\right|^{2}\cdot  e^{-\varphi }\,\mathrm{d}P\leqslant  \frac{\sum^{\prime }\limits_{\left| I\right|  =s,\left| J\right|  =t+1} c_{I,J}\int_{V} \left| f_{I,J}\right|^{2} \cdot e^{-\varphi }\,\mathrm{d}P}{c_0^{s,t}\cdot(t+1)}.
\end{equation}
Furthermore, if $V$ is bounded, then
\begin{equation}\label{230423e31}
\sum^{\prime }_{\left| I\right|  =s,\left| L\right|  =t} c_{I,L}\int_{V} \left| u_{I,L}\right|^{2}\cdot  e^{-\varphi }\,\mathrm{d}P\leqslant  \left( 1+\sup\limits_{\textbf{z}\in V} \left| \left| \textbf{z}\right|  \right|^{2}_{\ell^{2} }  \right)^{2}  \cdot  \frac{\sum^{\prime }\limits_{\left| I\right|  =s,\left| J\right|  =t+1} c_{I,J}\int_{V} \left| f_{I,J}\right|^{2} \cdot e^{-\varphi }\,\mathrm{d}P}{c_0^{s,t}\cdot(t+1)}.
\end{equation}
\end{corollary}
\begin{proof}
Let $\tilde{\varphi } \left( \textbf{z}\right)\triangleq\varphi \left(\textbf{z}\right)  +2\ln \left( 1+\left| \left| \textbf{z}\right|  \right|^{2}_{\ell^{2} }  \right)$ for $\textbf{z}\in V$.
Then $\tilde{\varphi }\in C^{2}_{F}\left( V\right)$ and for each $n\in\mathbb{N}$, it hold that
$$
	 		\begin{aligned}
	 			\sum_{1\leqslant i,j\leqslant n}\left( \partial_{i}\overline{\partial_{j}}\tilde{\varphi }\left( \textbf{z}\right) \right)\cdot \zeta_{i}\cdot\overline{\zeta_{j}}&\geqslant 2\sum_{1\leqslant i,j\leqslant n} \left(\partial_{i} \overline{\partial_{j} } \ln \left( 1+\left| \left| \textbf{z}\right|  \right|^{2}_{\ell^{2} }  \right) \right)\cdot\zeta_{i}\cdot\overline{\zeta_{j}} \\
&= 2\sum_{1\leqslant i,j\leqslant n} \frac{z_{j}\cdot\overline{z_{i}} }{\left( 1+\left| \left| \textbf{z}\right|  \right|^{2}_{\ell^{2} }  \right)^{2}  } \cdot\zeta_{i}\cdot\overline{\zeta_{j}} + \frac{2}{1+\left| \left| \textbf{z}\right|  \right|^{2}_{\ell^{2} }  } \cdot\left(\sum_{i=1}^{n}\left| \zeta_{i}\right|^{2} \right) \\
&\geqslant \frac{2}{\left( 1+\left| \left| \textbf{z}\right|  \right|^{2}_{\ell^{2}}  \right)^{2}  }\cdot\left( \sum^{n}_{i=1} \left| \zeta_{i}\right|^{2}\right),\quad\forall\; (\zeta_{1},\cdots,\zeta_n)\in \mathbb{C}^n,\,\textbf{z}=(z_i)_{i=1}^{\infty}\in V,
	 		\end{aligned}
$$
and
\begin{eqnarray*}
\sum^{\prime }_{\left| I\right|  =s,\left| J\right|  =t+1} c_{I,J}\int_{V}   \left| f_{I,J}\right|^{2} \cdot e^{-\tilde{\varphi } }\,\mathrm{d}P
&\leqslant &\sum^{\prime }_{\left| I\right|  =s,\left| J\right|  =t+1} c_{I,J}\int_{V} \left( 1+\left| \left| \textbf{z}\right|  \right|^{2}_{\ell^2}  \right)^{2} \cdot \left| f_{I,J}\right|^{2} \cdot e^{-\tilde{\varphi } }\,\mathrm{d}P\\
&=& \sum^{\prime }_{\left| I\right|  =s,\left| J\right|  =t+1} c_{I,J}\int_{V} \left| f_{I,J}\right|^{2} \cdot e^{-\varphi }\,\mathrm{d}P\\
&<&\infty.
\end{eqnarray*}
Thus $f=\sum\limits^{\prime }_{\left| I\right|  =s,\left| J\right|  =t+1 }f_{I,J}dz_{I}\wedge d\overline{z_{J}}\in L^{2}_{(s,t+1)}\left( V,\tilde{\varphi } \right)  $. By Theorem \ref{general L}, there exists $u=\sum\limits^{\prime }_{\left| I\right|  =s,\left| L\right|  =t }u_{I,L}dz_{I}\wedge d\overline{z_{L}}\in L^{2}_{(s,t)}\left( V,\tilde{\varphi} \right)\subset L^{2}_{(s,t)}\left( V,loc\right)$ such that $\overline{\partial}u=f$ and
$$
\sum^{\prime }_{\left| I\right|  =s,\left| L\right|  =t} c_{I,L}\int_{V} \left| u_{I,L}\right|^{2}\cdot  e^{-\tilde{\varphi } }\,\mathrm{d}P\leqslant \frac{\sum^{\prime }\limits_{\left| I\right|  =s,\left| J\right|  =t+1} c_{I,L}\int_{V} \left( 1+\left| \left| \textbf{z}\right|  \right|^{2}_{\ell^2}  \right)^{2} \cdot \left| f_{I,J}\right|^{2}\cdot  e^{-\tilde{\varphi } }\,\mathrm{d}P}{c_0^{s,t}\cdot(t+1)}<\infty,
$$
which gives (\ref{230423e2}).

Clearly, if $V$ is bounded, then we have the estimate (\ref{230423e31}).
This completes the proof of Corollary \ref{general Theorem 4.4.2Hor90}.
\end{proof}

\end{document}